\documentclass{amsart}
\usepackage{amsfonts}
\usepackage{amssymb}
\usepackage{graphicx}

\newtheorem{theorem}{Theorem}[section]
\theoremstyle{plain}

\newtheorem{corollary}{Corollary}[section]

\newtheorem{definition}{Definition}[section]

\newtheorem{lemma}{Lemma}[section]

\newtheorem{remark}{Remark}[section]

\numberwithin{equation}{section}
\numberwithin{figure}{section}

\begin{document}
\title[\textbf{The best bound of the area--length ratio in Ahlfors' theory
}]{\textbf{The best bound of the area--length ratio in Ahlfors
Covering surface theory (I)}}
\author{Guang Yuan Zhang }
\address{Department of Mathematical Sciences, Tsinghua University, Beijing
100084, P. R. China. \textit{Email:} \textit{gyzhang@math.tsinghua.edu.cn}}

\begin{abstract}
In Ahlfors' covering surface theory, it is well known that there exists a
positive constant $h$ such that for any nonconstant holomorphic mapping $f:%
\overline{\Delta }\rightarrow S,$ if $f(\Delta )\cap \{0,1,\infty
\}=\emptyset ,$ then%
\begin{equation*}
A(f,\Delta )\leq hL(f,\partial \Delta ),
\end{equation*}%
where $\Delta $ is the disk $|z|<1$ in $\mathbb{C},$ $S$ is the unit Riemann
sphere, $A(f,\Delta )$ is the area of the image of $\Delta $ and $%
L(f,\partial \Delta )$ is the length of the image of $\partial \Delta $,
both counting multiplicities.

In this paper, we will show that the best lower bound for $h$ is the number
\begin{equation*}
h_{0}=\max_{\tau \in \lbrack 0,1]}\left[ \frac{\sqrt{1+\tau ^{2}}\left( \pi
+\arcsin \tau \right) }{\mathrm{{arccot}\frac{\sqrt{1-\tau ^{2}}}{\sqrt{%
1+\tau ^{2}}}}}-\tau \right] =4.\,\allowbreak 034\,159\,790\,\allowbreak
51\dots ,
\end{equation*}%
and this is the exact estimation, i.e. there exists a sequence of
holomorphic mappings $f_{n}:\overline{\Delta }\rightarrow S$ such that $%
f_{n}(\Delta )\cap \{0,1,\infty \}=\emptyset $ and
\begin{equation*}
\lim_{n\rightarrow \infty }A(f_{n},\Delta )/L(f_{n},\partial \Delta )=h_{0}.
\end{equation*}
\end{abstract}

\subjclass[2000]{30D35, 30D45, 52B60}
\thanks{Project 10271063 and 10571009 supported by NSFC}
\maketitle

\section{Introduction\label{1SS-Intro}}

In this paper, the Riemann sphere $S$\ is the unit sphere
\begin{equation*}
S=\{(x_{1},x_{2},x_{3})\in \mathbb{R}^{3};\;x_{1}^{2}+x_{2}^{2}+x_{3}^{2}=1\}
\end{equation*}%
endowed with the stereographic projection
\begin{equation*}
P:\overline{\mathbb{C}}=\mathbb{C}\cup \{\infty \}\rightarrow S
\end{equation*}%
with $P(0)=(0,0,-1)$, $P(\infty )=(0,0,1).$ The lengths of curves and the
areas of domains in $S$ are defined in the usual way. Thus, $P$ induces the
spherical metric $ds=\rho (z)|dz|=\frac{2}{1+|z|^{2}}|dz|,z\in \mathbb{C}.$
For a set $V$ in $\overline{\mathbb{C}},$ we denote by $\partial V$ its
boundary and $\overline{V}$ its closure.

We will identify the extended plane $\overline{\mathbb{C}}=\mathbb{C}\cup
\{\infty \}$ with $S,$ via the stereographic projection $P$. So for any set $%
D\subset \mathbb{C}$, we will also write $D\subset S,$ but in the later
relation, $D$ in fact means the set $P(D).$ When we write $0\in S$, for
example, $0$ indicates the point $P(0)=(0,0,-1)$ in $S.$ In this way, some
notations in $\mathbb{C}$ will be used in $S$: we use the interval notation $%
[-1,1],[0,+\infty ]$ to denote the line segment $P([-1,1]),P([0,+\infty ])$
in $S,$ etc.

For a Jordan domain $U$ in $\mathbb{C}$ and a holomorphic mapping $g:%
\overline{U}\rightarrow S,$ we denote by $A(g,U)$ the spherical area of the
image of $U,$ counted with multiplicities, and denote by $L(g,\partial U)$
the spherical length of the image of $\partial U,$ counted with
multiplicities. If we regard $g$ as a mapping from from $\overline{U}$ into $%
\overline{\mathbb{C}}=\mathbb{C\cup \{\infty \}}$, via the stereographic
projection $P$, we have
\begin{equation*}
A(g,U)=\iint_{U}(\rho (g(z))|g(z)|)^{2}dxdy,\ z=x+iy;
\end{equation*}%
\begin{equation*}
L(g,\partial U)=\int_{\partial U}\rho (g(z))\left\vert g(z)\right\vert |dz|.
\end{equation*}

In Ahlfors' covering surface theory (\cite{Ah}, \cite{Ha}), it is well known
that there exists a positive constant $h$ such that for any holomorphic
mapping $f:\overline{\Delta }\rightarrow S,$ if $f(z)\neq 0,1,\infty $ for
any $z\in \Delta ,$ then%
\begin{equation}
A(f,\Delta )\leq hL(f,\partial \Delta ).  \label{a6}
\end{equation}%
The goal of this paper is to give the best lower bound for $h,$ and our main
result is the following theorem.

\begin{theorem}
\label{main}Let $f:\overline{\Delta }\rightarrow S$ be a nonconstant
holomorphic mapping such that $f(z)\neq 0,1,\infty $ for any $z\in \Delta $.
Then
\begin{equation}
A(f,\Delta )<h_{0}L(f,\partial \Delta ),  \label{1-1}
\end{equation}%
where
\begin{equation}
h_{0}=\max_{\tau \in \lbrack 0,1]}\left[ \frac{\sqrt{1+\tau ^{2}}\left( \pi
+\arcsin \tau \right) }{\mathrm{arccot}\frac{\sqrt{1-\tau ^{2}}}{\sqrt{%
1+\tau ^{2}}}}-\tau \right] =4.03415979051\dots ,  \label{best}
\end{equation}%
and $h_{0}$ is the best lower bound in the sense that there exists a
sequences of holomorphic mappings $f_{n}:\overline{\Delta }\rightarrow S$
such that $f_{n}(\Delta )\cap \{0,1,\infty \}=\emptyset $ and
\begin{equation*}
\lim_{n\rightarrow \infty }\frac{A(f_{n},\Delta )}{L(f_{n},\partial \Delta )}%
=h_{0}.
\end{equation*}
\end{theorem}

Consider the function
\begin{equation}
h(\tau )=\frac{\sqrt{1+\tau ^{2}}\left( \pi +\arcsin \tau \right) }{\mathrm{%
arccot}\frac{\sqrt{1-\tau ^{2}}}{\sqrt{1+\tau ^{2}}}}-\tau ,\tau \in \lbrack
0,1].  \label{1.1}
\end{equation}%
It is clear that
\begin{equation*}
h(0)=4,h(1)=3\sqrt{2}-1<4,
\end{equation*}%
and
\begin{equation*}
h^{\prime }(0)=\frac{4}{\pi }-1>0.
\end{equation*}%
Thus, $h$ takes its maximum $h_{0}$ at some point $\tau _{0}\in (0,1)$ and $%
h_{0}>4.$

For a domain $U$ in $S,$ we denote by $A(U)$ the area of $U.$ If $U\subset
\mathbb{C},$ we still use the notation $A(U)$ to denote the spherical area
of $U,$ which is the area of $P(U)$ given by
\begin{equation*}
A(U)=\iint_{U}(\rho (x+iy))^{2}dxdy.
\end{equation*}

For a curve $\Gamma =\Gamma (t),t\in \lbrack 0,1],$ in $S,$ we denote by $%
L(\Gamma )$ the length of the set $\Gamma =\{\Gamma (t);t\in \lbrack 0,1]\}.$
If $\Gamma =\Gamma (t),t\in \lbrack 0,1],$ is a curve in $\mathbb{C}$, we
still denote by $L(\Gamma )$ the spherical length of $\Gamma ,$ which is the
length of the set $P(\Gamma ),$ and in the case that $\Gamma $ is simple we
have
\begin{equation*}
L(\Gamma )=\int_{\Gamma }\rho (z)|dz|.
\end{equation*}

Now we explain the geometric meaning of the function $h(\tau )$ given by (%
\ref{1.1})$.$ Let $D$ be the disk in $S$ with diameter $\overline{1,\infty }%
, $ the shortest path from $1$ to $\infty $ in $S.$ Let $l\in \lbrack \pi ,%
\sqrt{2}\pi ]$ and let $D_{l}$ be a domain inside $D$ whose boundary is
composed of the two congruent circular arcs, each of which has endpoints $%
\{1,\infty \}$ and spherical length $\frac{l}{2}.$ Then we have $L(\partial
D_{l})=l.$ It is clear that $D_{l},$ regarded as a domain in $\mathbb{C},$
is an angular domain whose vertex is $1$ and bisector is the ray $[1,+\infty
)$ in $\mathbb{C}$. We denote by $2\theta _{l}$ the value of the angle of
this angular domain$.$ Then it is clear that $\theta _{l}<\frac{\pi }{2}.$

It is proved in Section \ref{ss-4classi} that the area $A(D_{l})$ and the
length $L(\partial D_{l})=l$ are real analytic functions of $\tau =\sin
\theta _{l},\theta _{l}\in \lbrack 0,\frac{\pi }{2}],$ and when we
understand $A(D_{l})$ and $l,$ in the ratio $\frac{4\pi +A(D_{l})}{l},$ as
functions of $\tau =\sin \theta _{l},$ we obtain the function $h(\tau )$
given by (\ref{1.1}):%
\begin{equation}
h(\tau )=\frac{A(S)+A(D_{l})}{l}=\frac{4\pi +A(D_{l})}{l},\tau \in \lbrack
0,1].  \label{1.4}
\end{equation}%
This is the geometrical meaning of the function $h(\tau ).$

Considering that $l\geq \pi $ and $A(D_{l})\leq A(D)=2\pi (1-\frac{\sqrt{2}}{%
2}),$ we have%
\begin{equation*}
h(\tau )\leq \frac{4\pi +2\pi (1-\frac{\sqrt{2}}{2})}{\pi }<4.6,
\end{equation*}%
and then%
\begin{equation*}
4<h_{0}<4.6.
\end{equation*}%
A numerical computation shows that%
\begin{equation*}
h_{0}=4.034\,159\,790\,51\dots
\end{equation*}

The inequality (\ref{a6}) directly follows from the fundamental theorem of
L. Ahlfors' covering surface theory (\cite{Ah}, \cite{Ha}) for a finite
number of points $a_{1},\dots ,a_{q}$:

\begin{theorem}[Ahlfors]
Let $a_{1},\dots ,a_{q}$ be distinct $q$ points in $S.$ Then there exists a
positive constant $h=h(a_{1},\dots ,a_{q})$ such that for any meromorphic
function defined on $\overline{\Delta }$%
\begin{equation}
(q-2)A(f,\Delta )/4\pi \leq \sum_{m=1}^{q}n(f,a_{m})+hL(f,\partial \Delta ),
\label{1.2}
\end{equation}%
where $n(f,a_{m})$ is the number of solutions of the equation $%
f(z)=a_{m},z\in \Delta ,$ ignoring multiplicities.
\end{theorem}

\begin{remark}
J. Dufresnoy's work \cite{Du} may be the first literature estimating the
number $h$ in (\ref{1.2}) explicitly, in which it is shown that the number $%
h $ in (\ref{1.2}) can be taken to be $h=h_{1}=\frac{3}{2\delta _{0}},$
where $\delta _{0}$ is the smallest spherical distance between the points $%
a_{m},m=1,\dots ,q$. When $f(z)\neq 0,1,\infty ,z\in \Delta ,$ Dufresnoy's
result is that
\begin{equation*}
A(f,\Delta )\leq 12L(f,\partial \Delta ).
\end{equation*}
\end{remark}

\begin{remark}
J. Dufresnoy's work \cite{Du} also studied the relationship between the
constant in (\ref{a6}) and some other classical constants, such as Landau's,
Bloch's and Schotkii's constants. This is also introduced in the book \cite%
{Ha} by Haymann.
\end{remark}

To prove the main theorem, the difficulty lies in the inequality (\ref{1-1}%
). It seems hard to estimate the best lower bound for the constant $h$ by
following Ahlfors' method in his covering surface theory. Fortunately, we
managed to re-understand Ahlfors's theory via the classical isoperimetric
inequality of the unit hemisphere which is obtained by F. Bernstein \cite%
{Ber} in 1905 (see Section \ref{ss-4classi}). The following is the outline
of the proof of the main theorem.

\noindent \textbf{(A). Observation for certain class of open mappings}. We
have been able to find that the area-length ratio is relative easy to figure
out for a special family $F$ of mappings from $\overline{\Delta }$ into $S$
such that for each $f\in F,$ $f$ satisfies the following conditions (a)--(e):

(a) $f$ is open, discrete\footnote{%
The term \emph{discrete} means that for each $q\in f(\overline{\Delta }),$ $%
f^{-1}(q)$ is a finite set.} and continuous, the boundary curve $\Gamma
_{f}=f(z),z\in \partial \Delta ,$ is a polygonal curve in $S$ and $f(\Delta
)\cap \{0,1,\infty \}=\emptyset .$

(b) Each natural edge\footnote{%
See Definition \ref{v-e} (2) and (3).} of $\Gamma $ has spherical length
strictly less than $\pi .$

(c) $\Gamma _{f}$ is locally convex everywhere except at $0,1,\infty .$

(d) All branched points of $f$ are located in $\{0,1,\infty \}.$

(e) $\Gamma _{f}\cap \lbrack 0,+\infty ]$ contains at most finitely many
points. Here $[0,+\infty ]$ denotes the line segment in $S$ from $0$ to $%
\infty $ passing through $1.$

It is clear that normal mappings defined in Section \ref{ss-3con-norm}
satisfy condition (a). Conversely, any mapping satisfying (a) that is
orientation preserved is a normal mapping\footnote{%
We will not introduce the proof for this conclusion, since it is not used in
this paper.}. It is relatively easy to estimate the area--length ratio for
mappings in the family $F:$ for each $f\in F\ $one can obtain the following
inequality by Lemmas \ref{nofat} and \ref{key1},%
\begin{equation*}
A(f,\Delta )\leq h_{0}L(g,\partial \Delta )-\min \{A(f,\Delta ),4\pi \},
\end{equation*}%
where $h_{0}$ is given by (\ref{best}).

On the other hand, it is fortunate that we are able to prove that, for any
holomorphic mapping $f:\overline{\Delta }\rightarrow S$ with $f(\Delta )\cap
\{0,1,\infty \}=\emptyset ,$ and for sufficiently small $\varepsilon >0,$
there exist a finite number of mappings $\{g_{1},\dots ,g_{n}\}$ in the
family $F,$ such that
\begin{equation}
\sum_{j=1}^{n}A(g_{j},\Delta )\geq A(f,\Delta )-\varepsilon ,\ \mathrm{and\ }%
\sum_{j=1}^{n}L(g_{j},\partial \Delta )\leq L(f,\partial \Delta
)+\varepsilon .  \label{1-2}
\end{equation}%
Summarizing the above two aspects, we obtain (\ref{1-1}).

The existence of the family $\{g_{1},\dots ,g_{n}\},$ which is given by
Theorem \ref{key}, is the first key step to prove the main theorem. Sections %
\ref{ss-7to<pi}--\ref{ss-10to-convex} is prepared for proving Theorem \ref%
{key}: we first prove Theorems \ref{1-br-2-map} and \ref{1-nconvex}, and
then we apply these two results to deduce Theorem \ref{key} in Section \ref%
{ss-9finite}. The ingredients of Sections \ref{ss-7to<pi} and \ref%
{ss-8move-br} are Theorem \ref{>pi}, Lemma \ref{move} and Lemma \ref{move-bd}%
, which are just used to prove Theorem \ref{1-br-2-map} and Theorem \ref%
{1-nconvex}. We will give the outline for the proof of Theorem \ref{key} in
the following part (B).

The content of Sections \ref{ss-4classi}--\ref{ss-p1} and \ref{ss14} is for
proving Lemmas \ref{nofat} and \ref{key1}, which, with the existence of the
family $\{g_{1},\dots ,g_{n}\},$ deduce the main theorem in the last
section, Section \ref{ss-p2}. In Section \ref{ss-4classi}, we introduce two
classical results, the Bernstein's isoperimetric inequality of the unit
hemisphere and the Lad\'{o}'s theorem, from which we prove Theorems \ref%
{good} and \ref{good2} that is used in Section \ref{ss-p2} for proving
Lemmas \ref{nofat} and \ref{key1}. Sections \ref{SS5-Sconvex-curves} and \ref%
{ss-6lift} are prepared for Section \ref{ss-p1}, and the ingredient of
Section \ref{ss-p1} is Theorem \ref{decom}, which is the second key step to
prove the main theorem: with Theorems \ref{good}, \ref{good2} and \ref{fat},
it deduces Lemmas \ref{nofat} and \ref{key1}. Theorem \ref{fat}, which is
proved just based on Lemma \ref{position} and Corollary \ref{by the way}, is
the third key step to prove the main theorem.

\noindent \textbf{(B). The existence of }$\{g_{1},\dots ,g_{n}\}$ \textbf{in
(A). }Now, we introduce the outline to prove the existence of $\{g_{1},\dots
,g_{n}\}.$ Let $f:\overline{\Delta }\rightarrow S$ be a holomorphic mapping
with $f(\Delta )\cap \{0,1,\infty \}=\emptyset .$ To show the existence of
the family $\{g_{1},\dots ,g_{n}\},$ for any $\varepsilon >0,$ we first
approximate $f$ by an open mapping $f_{1}$ such that $f_{1}$ satisfies (a)
and (b) in (A) and
\begin{equation*}
A(f_{1},\Delta )>A(f,\Delta )-\frac{\varepsilon }{2}\ \mathrm{and\ }%
L(f_{1},\partial \Delta )<L(f,\partial \Delta )-\frac{\varepsilon }{2}.
\end{equation*}

Then we are able to first show that there exist a finite number of mappings $%
\{G_{1},\dots ,G_{n}\}$ that satisfy (\ref{1-2}) and (a)--(d) as follows.

\noindent \textbf{Operation 1: (a)(b)}$\rightarrow $(\textbf{a)(b)(c).} We
can apply Theorem \ref{1-nconvex} several times to obtain a mapping $f_{2}$
such that $f_{2}$ satisfies (a)--(c) and%
\begin{equation*}
A(f_{2},\Delta )\geq A(f_{1},\Delta )\ \mathrm{and\ }L(f_{2},\partial \Delta
)\leq L(f_{1},\partial \Delta ).
\end{equation*}

If $f_{2}$ satisfies (d), then $\{G_{1}\}=\{f_{2}\}$ is the desired family.
Otherwise we turn to next operation.

\noindent \textbf{Operation 2: (a)(b)(c)}$\rightarrow $\textbf{(a)(b)(d).}
If $f_{2}$ does not satisfies (d)$,$ then we can apply Theorem \ref%
{1-br-2-map} a finite number of times to decompose $f_{2}$ into a finite%
\footnote{%
By Theorem \ref{1-br-2-map} (iv) we may assume $\sum_{j=1}^{m}V(f_{2j})\leq
V(f_{2})+2(m-1),$ where $V(f_{2})$ is the number of natural vertices (see
Definition \ref{v-e}) of the polygonal curve $\Gamma _{f_{2}}=f_{2}(z),z\in
\partial \Delta .$ Then by Lemma \ref{E>3} we have $3m\leq V(f_{2})+2(m-1),$
which implies $m\leq V(f_{2})-2,$ and then the finiteness follows.} number
of mappings $f_{2j},j=1,\dots ,m,$ that satisfy (a), (b), (d) and%
\begin{equation*}
\sum_{j=1}^{m}A(f_{2j},\Delta )\geq A(f_{2},\Delta )\ \mathrm{and\ }%
\sum_{j=1}^{m}L(f_{2j},\partial \Delta )\leq L(f_{2},\partial \Delta ).
\end{equation*}

Operation 2 may destroy condition (c)! We try to repair this by applying
Operation 1 to all the mappings $f_{2j}$ and obtain mappings $%
f_{12j},j=1,\dots ,m,$ that satisfy (a)--(c) and%
\begin{equation*}
A(f_{12j},\Delta )\geq A(f_{2j},\Delta )\ \mathrm{and\ }L(f_{12j},\partial
\Delta )\leq L(f_{2j},\partial \Delta ),j=1,\dots ,m.
\end{equation*}

But Operation 1 may destroy condition (d)! We try to repair this by applying
Operation 2 to each $f_{12j}$ that has ramification points in $\Delta $ and
obtain more mappings$.$ But then condition (c) may again be destroyed for
the mappings obtained from Operation 2.

It seems we are arguing in a circle! Luckily, we are able to prove that
Operations 1 and 2 can not be applied infinitely many times! This is the
ingredient of Theorem \ref{key}. Thus, we can execute Operations 1 and 2
alternatively with in a finite number of steps to obtain the desired
mappings $G_{j},j=1,\dots ,n.$

From the mappings $G_{j}$ we can easily obtain the mappings $%
g_{j},j=1,2,\dots ,n,$ by slightly perturb each $G_{j}$.

\begin{remark}
The method in this paper can also be used to estimate the best bound of the
constant $h$ in Ahlfors's fundamental theorem for any number $(\geq 3)$ of
points. We will discuss this in another paper.
\end{remark}

\section{Some notations and definitions related to curves in $S$ \label%
{ss-2appoint}}

In this section we introduce some notations, definitions and make some
conventions. \emph{Locally convex polygonal paths }and\emph{\ locally convex
polygonal curves} in the Riemann sphere $S$ defined in this section play a
central role in this paper.

Let $\Gamma =\Gamma (t),t\in \lbrack \alpha ,\beta ],$ be a curve in $%
\mathbb{C}$ or $S$. Then the orientation of the curve $\Gamma $ will be
regarded as the orientation as $t$ increases. Therefore, if $\Gamma $ is not
closed, the orientation of $\Gamma $ is from $\Gamma (\alpha )$ to $\Gamma
(\beta ),$ and we will denote by
\begin{equation*}
-\Gamma =\Gamma (t_{2}+t_{1}-t),t\in \lbrack t_{1},t_{2}],
\end{equation*}%
the same curve with opposite orientation.

If $\Gamma _{j}=\Gamma _{j}(t),t\in \lbrack t_{j1},t_{j2}],$ are two curves
in $\mathbb{C}$ (or $S$) and $\Gamma _{1}(t_{12})=\Gamma _{2}(t_{21}),$ we
will denote by $\Gamma _{1}+\Gamma _{2}$ the curve
\begin{equation*}
\Gamma (t)=\left\{
\begin{array}{ll}
\Gamma _{1}(t), & t\in \lbrack t_{11},t_{12}], \\
\Gamma _{2}(t+t_{21}-t_{12}), & t\in (t_{12},t_{12}+t_{22}-t_{21}].%
\end{array}%
\right.
\end{equation*}%
When $\Gamma _{1}+(-\Gamma _{2})$ makes sense, we will write it by $\Gamma
_{1}-\Gamma _{2}.$

Curves in this paper are always oriented and continuous curves. Some times a
curve $\Gamma $ will be regard as a set in $S.$ But this is only in the case
that the curve is involved in some set operations.

For a Jordan domain $D$ in $\mathbb{C},$ the boundary $\partial D$ of $D$ is
always regarded as an oriented curve with the anticlockwise orientation. If $%
D$ is a Jordan domain in $\mathbb{C}$ and $f:\overline{D}\rightarrow S$ is a
continuous mapping, then the the boundary curve
\begin{equation}
\Gamma _{f}=\Gamma _{f}(z),z\in \partial D,  \label{3-5}
\end{equation}%
of $f$ is always regarded as an oriented curve with the oreintation induced
by $\partial D$.

The notation $\Gamma _{f}$ will be used through out this paper, which alway
denotes the curve given by (\ref{3-5}) for any given Jordan domain $D$ of $%
\mathbb{C}$ and any mapping $f:\overline{D}\rightarrow S$.

An oriented great circle $C$ in $S$ divides the sphere into two hemispheres.
We will call the hemisphere that is on the left hand side of $C$ \emph{%
inside, or enclosed by,} $C$, in the sense that we are standing on the
sphere with our heads pointing to the center of $S,$ and going along $C$ in
the orientation of $C.$ For example, when $\Delta $ is regarded as a disk in
$S,$ $\Delta $ is the lower hemisphere of $S$ and $\Delta $ is inside the
oriented circle $\partial \Delta ,$ i.e. $P(\Delta )$ is inside the great
circle $P(\partial \Delta );$ and the upper hemisphere $\overline{\mathbb{C}}%
\backslash \overline{\Delta }$ in $S$ is inside the oriented circle $%
-\partial \Delta .$

If $\Gamma $ is a Jordan curve in $S$, then the domain in $S$ that is
bounded by $\Gamma $ and is inside $\Gamma $ is also called the domain
inside, or enclosed by, $\Gamma .$ Of course, here \textquotedblleft
inside\textquotedblright\ means \textquotedblleft on the left hand side
of\textquotedblright .

A section of a great circle in $S$ is called a \emph{line segment}. To
emphasize this, we also call it \emph{straight line segment }or \emph{%
geodesic line segment.}

The spherical distance of two points $p$ and $q$ in $S$ will be denoted by $%
d(p,q).$ In the case that $p$ and $q$ are not antipodal, we denote by $%
\overline{pq}$ the shortest (simple) path in $S$ from $p$ to $q,$ which is
unique and is in fact the shorter of the two arcs with end points $p$ and $q$
of the great circle of $S$ passing through $p$ and $q.$ We will write%
\begin{equation*}
\overline{q_{1}q_{2}\dots q_{n}}=\overline{q_{1}q_{2}}+\overline{q_{2}q_{3}}%
+\dots +\overline{q_{n-1}q_{n}},
\end{equation*}%
if each term of the right hand side makes sense, where $q_{1},\dots ,q_{n}$
are points in $S.$

We write $\overline{pq}$ by $\overline{p,q},$ if $p,$ or $q,$ or both$,$ is
replaced by explicit complex numbers. For example, we denote by the shortest
path from $p=1$ to $q=2$ by $\overline{pq}=\overline{1,2}.$ Note that we
identify $\overline{\mathbb{C}}=\mathbb{C}\cup \{\infty \}$ with $S,$ via
the stereographic projection $P.$

When $\overline{pq}$ makes sense, we will denote by $\overline{pq}^{\circ }$
the interior of the path.

\begin{definition}
\label{natural-cur}A closed curve
\begin{equation*}
\Gamma =f(z),z\in \partial \Delta ,
\end{equation*}
in $S$ is called a\emph{\ polygonal closed curve} if and only if there exist
a finite number of points $p_{j}\in \partial \Delta ,j=1,\dots ,n,$ with
\begin{equation}
\arg p_{1}<\arg p_{2}<\dots <\arg p_{n}<\arg p_{1}+2\pi  \label{c-2}
\end{equation}%
such that for each section\footnote{%
A section of a curve always inherits the orientation of the curve.} $\alpha
_{j}$ of $\partial \Delta $ from $p_{j}$ to $p_{j+1}$ $(p_{n+1}=p_{1}),$ the
section $\Gamma _{j}$ of $\Gamma $ restricted to $\alpha _{j}$ is a locally
simple and locally straight path, and in this case
\begin{equation*}
\Gamma =\Gamma _{1}+\dots +\Gamma _{n}
\end{equation*}%
is called a \emph{partition} of $\Gamma $.
\end{definition}

Note that the term \emph{partition} emphasizes that each term $\Gamma _{j}$
is locally simple and locally straight. A locally simple and locally
straight curve in $S$ must be contained in some great circle of $S.$ So, for
each $\alpha _{j}$ in the above definition, each $p_{0}\in \alpha _{j}$ has
a neighborhood $L_{p_{0}}$ in $\alpha _{j}$ such that $\Gamma $ restricted
to $L_{p_{0}}$ is a homeomorphism onto a line segment in $S.$

Through out this paper, we denote by $E$ the set $\{0,1,\infty \}$ in $S.$

\begin{definition}
\label{v-e}Let $\Gamma =f(z),z\in \partial \Delta ,$ be a polygonal closed
curve in $S$.

(1) A point $p_{0}\in \partial D$ is called a \emph{natural vertex} of $%
\Gamma $ if and only if one of the following conditions holds:

(a) $f(p_{0})\in E=\{0,1,\infty \}.$

(b) $f(p_{0})\notin E$ and for any neighborhood $I_{p_{0}}$ of $p_{0}$ in $%
\partial \Delta ,$ the restriction $\Gamma |_{I_{p_{0}}}=f(z),z\in
I_{p_{0}}, $ can not be a straight and simple path.

(2) In the case that $\Gamma $ has at least two natural vertices, a closed
interval $I$ in $\partial \Delta $ is called a \emph{natural edge} of $%
\Gamma $ if and only if the endpoints of $I$ are both natural vertices of $%
\Gamma $ but the interior of $I$ does not contain any natural vertex of $%
\Gamma .$

(3) If $I$ is a natural edge of $\Gamma ,$ then the restriction $\Gamma
|_{I}=f(z),z\in I,$ is also called a \emph{natural edge} of $\Gamma .$
\end{definition}

For the above definition (2), the reader should be aware that a natural edge
can not contain any point of $f^{-1}\left( E\right) =f^{-1}\left(
\{0,1,\infty \}\right) $ in its interior (in $\partial \Delta $), because by
definition each point in $f^{-1}\left( E\right) $ is a natural vertex. Thus,
one can not understand any natural edge to be a maximal interval on which $%
\Gamma $ is locally simple and locally straight. If we regard the great
circle $C$ determined\footnote{%
This is in the sense that $C$ contains $\overline{0,1}$ and is oriented by $%
\overline{0,1}.$} by $\overline{0,1}$ as a simple closed curve, it has three
natural edges $\overline{0,1},\overline{1,\infty }$ and $\overline{\infty
,-1,1}=\overline{\infty ,-1}+\overline{-1,1}$, but the whole curve $C$ is
simple and straight.

If $\Gamma $ has no any natural vertex, $\Gamma $ must be a closed curve
contained in some great circle $C_{1}$ of $S$ with $C_{1}\cap \{0,1,\infty
\}=\emptyset $ and $\Gamma $ is locally simple, and in this case, $\partial
\Delta $ is regard as a natural edge without endpoints.

If $\Gamma $ has only one natural vertex $p_{0}\in \partial \Delta ,$ then,
by the definition, $q_{0}=f(p_{0})=0,1$ or $\infty ,$ and $\Gamma $ must be
also contained in some great circle $C_{2}$ of $S$ so that $C_{2}\cap
\{0,1,\infty \}=\{q_{0}\}$ and $\Gamma $ must be a simple path from $q_{0}$
to $q_{0}.$ In this case $\partial \Delta $ will be regarded as a natural
edge with endpoints coinciding at the unique natural vertex $q_{0}$.

\begin{definition}
\label{natural}Let $\Gamma =f(z),z\in \partial \Delta ,$ be a\emph{\
polygonal closed curve }and assume that $p_{1}\in \partial \Delta $ is a
natural vertex of $\Gamma $. Then there uniquely exist a finite number of
points $p_{j}\in \partial \Delta ,j=1,\dots ,n,$ with (\ref{c-2}) such that $%
p_{1},\dots ,p_{n}$ is an enumeration of all natural vertices of $\Gamma .$
In this case,
\begin{equation}
\Gamma =\Gamma _{1}+\Gamma _{2}+\dots +\Gamma _{n}  \label{c-3}
\end{equation}%
is called a \emph{natural partition} of $\Gamma ,$ where each $\Gamma _{j}$
is the restriction of $\Gamma $ to the section $\alpha _{j}$ of $\partial
\Delta $ from $p_{j}$ to $p_{j+1}$ $(p_{n+1}=p_{n}),$ and%
\begin{equation}
\partial \Delta =\alpha _{1}+\alpha _{2}+\dots +\alpha _{n}  \label{c-4}
\end{equation}%
is also called a \emph{natural partition} of $\partial \Delta $
corresponding to (\ref{c-3}).
\end{definition}

\begin{remark}
\label{closed-c}For the sake of simplicity and avoiding confusions, we make
the following conventions.

(1) When we say that $\Gamma ^{\prime }$ is a natural edge of a polygonal
closed curve $\Gamma =f(z),z\in \partial \Delta ,$ we always mean that $%
\Gamma $ and $\partial \Delta $ have natural partitions (\ref{c-3}) and (\ref%
{c-4}), respectively, such that $\Gamma ^{\prime }$ is the restriction $%
\Gamma _{j}=f(z),z\in \alpha _{j},$ for some $j.$

(2) When we use (\ref{c-3}) to denote a natural partition of $\Gamma ,$ we
always mean that there is a natural partition (\ref{c-4}) corresponding to (%
\ref{c-3}). Then, in the above definition we also call $q_{j}=f(p_{j}),$
which should be understood to be the pair $(p_{j},q_{j}),$ a \emph{natural
vertex} of $\Gamma $ for $j=1,\dots ,n$.
\end{remark}

\begin{definition}
\label{permit}A partition
\begin{equation*}
\Gamma =\Gamma _{1}+\Gamma _{2}+\dots +\Gamma _{n}
\end{equation*}%
of a closed polygonal curve in $S$ is called a\emph{\ permitted partition}
if each $\Gamma _{j}$ is contained in some natural edge of $\Gamma .$
\end{definition}

A polygonal Jordan curve in $S$ that is composed of exactly three line
segments is called a \emph{triangle}. Note that a vertex of a triangle may
not be a natural vertex. Any great circle may be regarded as a triangle,
while it has no any natural vertex.

\begin{definition}
\label{convex1}Let $\Gamma =\Gamma (z),z\in \partial \Delta ,$ be a closed
polygonal curve in $S$.

(1) For a point $p_{0}\in \partial \Delta ,$ $\Gamma $ is called \emph{convex%
} at $p_{0}$, if $p_{0}$ has a neighborhood $I$ in $\partial \Delta $ such
that the following two conditions (a) and (b) hold.

(a) The restriction $\Gamma |_{I}$ of $\Gamma $ to $I$ is a simple path.

(b) Either $\Gamma |_{I}$ is straight or $\Gamma ^{\prime }=\Gamma |_{I}+%
\overline{p^{\prime \prime }p^{\prime }},$ in which $p^{\prime }\ $and $%
p^{\prime \prime }$ are the initial and terminal point of $\Gamma |_{I},$
respectively, is a triangle which encloses\footnote{%
By definition, \textquotedblleft encloses\textquotedblright\ means the
triangle domain is \textquotedblleft on the left hand side
of\textquotedblright\ of the triangle $\Gamma ^{\prime }$.} a convex
triangle domain in $S.$

(2) $\Gamma $ is called \emph{strictly convex} at $p_{0}\in \partial \Delta $
if $\Gamma $ is convex at $p_{0}$ and for any neighborhood $I$ of $p_{0}$ in
$\partial \Delta ,$ $\Gamma |_{I}$ is not straight.

(3). For a point $q_{0}\in S,\ \Gamma $ is called \emph{convex} at $q_{0}\in
S$ if and only if for each $p\in \partial \Delta $ with $\Gamma (p)=q_{0},$ $%
\Gamma $ is convex at $p.$

(4). For a set $T\subset S,$ the closed curve $\Gamma $ is called \emph{%
locally convex in }$T$ if and only if $\Gamma $ is convex at each point $%
q_{0}\in T.$
\end{definition}

It is clear that if $\Gamma $ is convex at $q_{0}\in S,$ then for some
neighborhood $T$ of $q_{0}$ in $S,$ $\Gamma $ is locally convex in $T.$

\begin{definition}
A path
\begin{equation*}
\Gamma =\Gamma (t),t\in \lbrack 0,1],
\end{equation*}%
in $S,$ is called a \emph{polygonal path} if and only if $[0,1]$ has a
partition
\begin{equation}
0=t_{0}<t_{1}<\dots <t_{n}=1,  \label{3-3}
\end{equation}%
such that the section $\Gamma _{j}=\Gamma (t),t\in \lbrack t_{j-1},t_{j}]$,
is a locally simple and locally straight path, $j=1,\dots ,n,$ and in this
case
\begin{equation*}
\Gamma =\Gamma _{1}+\dots +\Gamma _{n}
\end{equation*}%
is called a \emph{partition} of $\Gamma .$
\end{definition}

Natural vertices, natural edges, natural partition, permitted partitions and
convex vertices of a polygonal path $\Gamma =\Gamma (t),t\in \lbrack 0,1],$
in $S,$ can be defined as that for polygonal closed curves. But convex
vertices are only defined in the open interval $(0,1)$ of $[0,1]$ and we
don't call the endpoints $0$ and $1$ natural vertices. To avoid confusions
we write these definitions completely.

\begin{definition}
\label{n-v-path}Let $\Gamma =f(t),t\in \lbrack 0,1],$ be a polygonal path in
$S$.

(1) A point $p_{0}\in (0,1)$ is called a \emph{natural vertex} of $\Gamma $
if and only if one of the following conditions holds:

(a) $f(p_{0})\in E=\{0,1,\infty \}.$

(b) $f(p_{0})\notin E$ and for any neighborhood $I_{p_{0}}$ of $p_{0}$ in $%
(0,1),$ the restriction $\Gamma |_{I_{p_{0}}}=f(t),t\in I_{p_{0}},$ can not
be a straight and simple path.

(2) A closed interval $I$ in $[0,1]$ is called a \emph{natural edge} of $%
\Gamma $ if and only if each endpoint of $I$ is either $0,$ or $1,$ or a
natural vertex of $\Gamma ,$ and the interior of $I$ does not contain any
natural vertex of $\Gamma .$

(3) If $I$ is a natural edge of $\Gamma ,$ the restriction $\Gamma
|_{I}=f(t),t\in I,$ is also called a \emph{natural edge} of $\Gamma .$
\end{definition}

\begin{definition}
\label{n-p}For a polygonal path $\Gamma =f(t),t\in \lbrack 0,1],$ a
partition
\begin{equation}
\Gamma =\Gamma _{1}+\dots +\Gamma _{n}  \label{3-1}
\end{equation}%
is called a \emph{natural partition} of $\Gamma ,$ if and only if $[0,1]$
has a partition
\begin{equation}
0=t_{0}<t_{1}<\dots <t_{n}=1,  \label{3-2}
\end{equation}%
such that each $[t_{j-1},t_{j}]$ is a natural edge of $\Gamma ,$ and $\Gamma
_{j}$ is the restriction of $\Gamma $ to $[t_{j-1},t_{j}],j=1,\dots ,n$, in
this case (\ref{3-2}) is also called a \emph{natural partition} of $[0,1]$
corresponding to (\ref{3-1}).
\end{definition}

\begin{remark}
\label{path}We make similar conventions as in Remark \ref{closed-c} for
polygonal paths.

(1) When we say that $\Gamma ^{\prime }$ is a natural edge of a polygonal
path $\Gamma =f(t),t\in \lbrack 0,1],$ we always mean that $\Gamma $ and $%
[0,1]$ have natural partitions (\ref{3-1}) and (\ref{3-2}), respectively,
such that $\Gamma ^{\prime }$ is the restriction $\Gamma _{j}=f(t),t\in
\lbrack t_{j-1},t_{j}],$ for some $j.$

(2) When we use (\ref{3-1}) to denote a natural partition of $\Gamma ,$ we
always mean that there is a natural partition (\ref{3-2}) corresponding to (%
\ref{3-1}). Then, in the above definition we also call $q_{j}=f(t_{j}),$
which should be understood to be the pair $(t_{j},q_{j}),$ a natural vertex
of $\Gamma $ for $j=1,\dots ,n-1$.
\end{remark}

\begin{definition}
\label{convex}Let $\Gamma =\Gamma (t),t\in \lbrack 0,1],$ be a polygonal
path in $S$.

(1) For a point $p_{0}\in (0,1),$ $\Gamma $ is called \emph{convex} at $%
p_{0} $, if there is a closed interval $I\subset (0,1)$ such that (a) and
(b) in Definition \ref{convex1} (1) hold.

(2) $\Gamma $ is called \emph{strictly convex} at $p_{0}\in \partial \Delta $
if $\Gamma $ is convex at $p_{0}$ and for any neighborhood $I$ of $p_{0}$ in
$(0,1),$ $\Gamma |_{I}$ is not straight.

(3). For a point $q_{0}\in S,\ \Gamma $ is called \emph{convex} at $q_{0}\in
S$ if and only if for each $p\in (0,1)$ with $\Gamma (p)=q_{0},$ $\Gamma $
is convex at $p.$

(4). For a set $T\subset S,$ the closed curve $\Gamma $ is called \emph{%
locally convex in }$T$ if and only if $\Gamma $ is convex at each point $%
q_{0}\in T.$
\end{definition}


Geometrically, a locally convex path (or curve) has the property
that when we go ahead along the path (or curve) with our heads
pointing to the center of the sphere $S$, we always go straight or
turn left.

\begin{remark}
The term \textquotedblleft closed polygonal path\textquotedblright\ and
\textquotedblleft closed polygonal curve\textquotedblright\ have distinct
meaning in some sense. If a polygonal path $\Gamma $ given by its natural
partition (\ref{3-1}), the natural vertices mean $t_{1},\dots ,t_{n-1}$. But
when $f(0)=f(1)$ and $\Gamma $ is regarded as a \emph{closed curve, }$%
t_{1},\dots ,t_{n-1}$ are still natural vertices of $\Gamma ,$ $t_{0}=0$,
identified with $t_{n}=1,$ may or may not be a natural vertex of the \emph{%
closed curve} $\Gamma .$ Closed polygonal paths still emphasize the initial
and terminal points, while for a closed polygonal curve, there is no initial
and terminal points, all points on it have equality.
\end{remark}

\begin{remark}
A \emph{locally convex polygonal Jordan path} that is closed may
not be a \emph{locally convex polygonal Jordan curve}, by the
definition.
\end{remark}

\begin{definition}
A polygonal Jordan curve in $S$ that is either a great circle, or is
composed of exactly two straight edges is called a biangle. A biangle
divides the sphere $S$ into two biangle domains.
\end{definition}

Note that a biangle may contains more than two natural edges, in the case
that it contains $0$, $1$ or $\infty $ in its straight edges.

\begin{definition}
\label{gen-tri}A triangle in $S$ is called a \emph{generic} triangle if it
encloses a triangle domain whose three angles are all strictly less than $%
\pi .$
\end{definition}

\begin{definition}
A Jordan curve $\Gamma $ in $S$ is called convex if the domain $D_{\Gamma
}\subset S$ inside $\Gamma $ is a convex domain in the sense that for any
two points $q_{1}$ and $q_{2}$ in $D_{\Gamma },$ there is a line segment $%
L\subset S$ with endpoints $q_{1}$ and $q_{2}$ such that $L\subset D_{\Gamma
}.$
\end{definition}

\begin{remark}
\label{loc-con}By the definition, each locally convex polygonal Jordan curve
is a convex curve and is contained in some closed hemisphere, while any
locally convex curve that is not simple may not be contained in any closed
hemisphere.
\end{remark}

\begin{remark}
Any triangle $\Gamma $ in $S$ all of whose edges have length $\leq \pi $ has
a orientation so that $\Gamma $ is a convex polygonal Jordan curve. But when
a triangle $\Gamma $ in $S$ has an edge with length $>\pi ,$ $\Gamma ,$ with
either orientation, may not be a locally convex triangle.
\end{remark}

\begin{remark}
For any convex triangle $\Gamma $ in $S$, the triangle domain enclosed%
\footnote{%
By definition, \textquotedblleft enclosed\textquotedblright\ means
\textquotedblleft on the left hand side of\textquotedblright .} by $\Gamma $
is contained in some hemisphere of $S.$ Conversely, any triangle domain
whose closure is contained in some open hemisphere of $S$ is enclosed by a
generic convex triangle in the same open hemisphere.
\end{remark}

\section{Definition and some properties of Normal mappings \label%
{ss-3con-norm}}

The proof of the main theorem is based on the investigation of so called
normal mappings defined in this section, which are the mappings satisfying
condition (a) in Section \ref{1SS-Intro}. But we will use another definition.

\begin{definition}
\label{722-1}Let $D$ be a Jordan domain in $\mathbb{C}$. A mapping $f:%
\overline{D}\rightarrow S$ is called a \emph{normal mapping }if the
following five conditions are satisfied:

(a) The boundary curve $\Gamma _{f}=f(z),z\in \partial D,$ is a polygonal
closed curve.

(b) For each $p\in D,$ there exist a neighborhood $U\subset D$ of $p$, a
disk $V\ $in $S$ centered at $q=f(p)$ and homeomorphisms $h_{1}:U\rightarrow
\Delta $ and $h_{2}:V\rightarrow \Delta ,$ such that
\begin{equation*}
h_{2}\circ f|_{U}\circ h_{1}^{-1}(\zeta )=\zeta ^{d},\zeta \in \Delta
\end{equation*}%
for some positive integer $d.$

(c) For each $p\in \partial D,$ there exists a neighborhood $U$ of $p$ in $%
\overline{D}$, a disk $V\ $in $S$ centered at $q=f(p)$ and homeomorphisms $%
h_{1}:\overline{U}\rightarrow \overline{\Delta ^{+}}$ and $h_{2}:\overline{V}%
\rightarrow \overline{\Delta },$ such that
\begin{eqnarray*}
h_{1}\left( \overline{U}\cap \partial D\right) &=&[-1,1], \\
h_{2}\circ f|_{\overline{U}}\circ h_{1}^{-1}(\zeta ) &=&\zeta ^{d},\zeta \in
\overline{\Delta ^{+}},
\end{eqnarray*}%
for some positive integer $d$, where $\Delta ^{+}$ is the upper half disk $%
\{\zeta \in \Delta ,\mathrm{Im}\zeta >0\}.$

(d) $f(D)\cap \{0,1,\infty \}=\emptyset .$

(e) $f$ is orientation preserved in the sense that $P^{-1}\circ f$ is
orientation preserved, where $P$ is the stereographic projection.
\end{definition}

The reader should be aware of that a normal mapping satisfies condition (a)
in Section \ref{1SS-Intro}. Conversely, a mapping that is orientation
preserved and satisfies (a) in Section \ref{1SS-Intro} must be a normal
mapping, but this is not important for us.

In the above definition if for some point $p\in D,$ the corresponding $d\geq
2,$ then $p$ is called a \emph{ramification point}, $f(q)$ is called a \emph{%
branched point}, $v_{f}(p)=d$ is called the \emph{multiplicity} of $f$ at $%
p, $ and $b_{f}(p)=d-1$ is called the \emph{branched number} of $f$ at $p.$

If for some $p\in \partial D,$ the corresponding $d\geq 3,$ then $p$ is
called a \emph{ramification point}, $f(q)$ is called a \emph{branched point}%
, $v_{f}(p)=\left[ \frac{d}{2}\right] $ is called the multiplicity of $f$ at
$p,$ and $b_{f}(p)=\left[ \frac{d+1}{2}\right] -1$ is called the branched
number of $f$ at $p.$

In the definition, \textquotedblleft orientation
preserved\textquotedblright\ means that for any regular point $p\in
\overline{\Delta }$ of $f,$ there is a closed Jordan domain $K_{p}$ in $%
\overline{\Delta }$ that is a neighborhood\footnote{%
This means that if $p\in \Delta ,$ $p$ is contained in the interior of $%
K_{p} $ in $\mathbb{C},$ and if $p\in \partial \Delta ,$ $p$ is contained in
the interior of the arc $K_{p}\cap \partial \Delta $ in $\partial \Delta .$}
of $p$ in $\overline{\Delta }$ such that $\widetilde{f}=P^{-1}\circ f$ or $%
\frac{1}{\widetilde{f}}$ maps $K_{p}$ homeomorphically onto a Jordan domain $%
K^{\prime }$ in $\mathbb{C}$ such that when $z$ goes along $\partial K_{p}$
anticlockwise, $\widetilde{f}(z)$ goes along $\partial K^{\prime }$
anticlockwise.

For a normal mapping $f:\overline{D}\rightarrow S$, $f$ has only finitely
many ramification points. $p\in \overline{D}$ is called a regular point of $%
f $ if $v_{p}(f)=1.$

The reader may be puzzled by the definition of $v_{f}(p)$ and $b_{f}(p)$
when $p\in \partial D.$ As a matter of fact, the definition in this case
follows from the fact that we can extend the mapping $f$ to be a normal
mapping so that $p$ becomes an interior ramification point with multiplicity
$\left[ \frac{d+1}{2}\right] .$

For a Jordan domain $D$ and a normal mapping $f:\overline{D}\rightarrow S,$
the boundary curve $\Gamma _{f}=f(z),z\in \partial D,$ is a polygonal closed
curve. Then the term \emph{natural vertex}, \emph{natural edge,} \emph{%
natural partition, permitted partition, }etc. introduced in Section \ref%
{ss-2appoint} are well defined for $\Gamma _{f}.$

\begin{definition}
For a Jordan domain $D$ and a normal mapping $f:\overline{D}\rightarrow S.$
We define $V(f)$ to be the number of natural vertices of the boundary curve $%
\Gamma _{f}=f(z),z\in \partial D;$ define $V_{E}(f)$ to be the number of
natural vertices of $\Gamma _{f}$ that is contained in $E$ and define%
\begin{equation*}
V_{NE}(f)=V(f)-V_{E}(f),
\end{equation*}%
which is the number of natural vertices of $\Gamma _{f}$ that is not
contained in $E.$
\end{definition}

Recall that $E$ always denotes the set $\{0,1,\infty \}$ in $S.$

Let $f:\overline{D}\rightarrow S$ be a normal mapping. Then by the
definition, $\overline{D}$ has a triangulation such that each ramification
point of $f$ is a vertex of the triangulation and $f$ restricted to each
triangle of the triangulation of $\overline{D}$ is a homeomorphism onto a
real triangle on $S,$ i.e., each edge of the triangle is straight. Then $f$
and the triangulation of $\overline{D}$ induce a triangulation of the
Riemann surface of $f;$ which is consisted of real triangles in $S.$
Therefore, the following two lemmas are obvious.

\begin{lemma}
\label{p3}Let $D$ be a Jordan domain in $\mathbb{C}$ and let $f:\overline{D}%
\rightarrow S$ be a normal\emph{\ }mapping. Then for any Jordan domain $%
D_{1} $ contained in $D,$ the restriction of $f$ to $\overline{D_{1}}$ is a
normal mapping, provided that the curve $f(z),z\in \partial D_{1},$ is a
polygonal curve.
\end{lemma}

\begin{lemma}
\label{patch}Let $D$ be a Jordan domain in $\mathbb{C}$, let $\alpha $ be a
Jordan path in $\overline{D}$ such that the interior\footnote{%
This means the curve $\alpha $ without endpoints.} of $\alpha $ is contained
in $D$ and $\alpha $ has two distinct endpoints lying on $\partial D$, let $%
D_{1}$ and $D_{2}$ be the two components of $D\backslash \alpha $, and let $%
f_{j}:\overline{D_{j}}\rightarrow S$ be two normal mappings, $j=1,2$. If $%
f_{1}(z)=f_{2}(z)$ for each $z\in \alpha ,$ then the mapping
\begin{equation*}
F=\left\{
\begin{array}{l}
f_{1}(z),z\in \overline{D_{1}}, \\
f_{2}(z),z\in D\backslash \overline{D_{1}},%
\end{array}%
\right.
\end{equation*}%
is a normal mapping defined on $\overline{D}.$
\end{lemma}

\begin{lemma}
\label{cov-1}Let $f:\overline{\Delta }\rightarrow S$ be a normal mapping and
let $q\in f(\overline{\Delta }).$ Then, for sufficiently small disk $D(q)$
in $S$ centered at $q,$ $f^{-1}(\overline{D(q)})$ is a union of disjoint
closed domains $\overline{U_{j}}$ in $\overline{\Delta },j=1,2,\dots ,n,$
such that for each $j,$ $\overline{U_{j}}$ is the closure of a (relatively)
open subset $U_{j}$ of $\overline{\Delta },$ $U_{j}\cap f^{-1}(q)$ contains
exactly one point $x_{j}$ and the followings holds:

(i). If $x_{j}\in \Delta ,$ then $f$ restricted to $\overline{U_{j}}$ is a
branched covering mapping onto $\overline{D(q)}$ such that $x_{j}$ is the
unique possible ramification point.

(ii). If $x_{j}\in \partial \Delta ,\ $then $f(\overline{U_{j}})=\overline{%
D(q)}$ or $f(\overline{U_{j}})$ is a closed sector of $\overline{D(q)}$, and
there exist homeomorphisms $\phi _{j}$ from $\overline{U_{j}}$ onto the
closed half disk $\overline{\Delta ^{+}}$ and $\psi _{j}$ from $D(q)$ onto $%
\overline{\Delta }$ such that%
\begin{equation*}
\phi _{j}(x_{j})=0,\ \phi _{j}(\overline{U_{j}}\cap \partial \Delta )=[-1,1],
\end{equation*}%
and%
\begin{equation*}
\psi _{j}\circ f\circ \phi _{j}^{-1}(\xi )=\xi ^{d_{j}},\xi \in \overline{%
\Delta ^{+}},
\end{equation*}%
for some positive integer $d_{j}$.
\end{lemma}

\begin{proof}
The proof is quite simple and standard. Note that in (ii) $f(\overline{U_{j}}%
\cap \Delta )$ may be the disk $D(q)$ omitting a radius.
\end{proof}

\begin{corollary}
\label{branched-lift}Let $f:\overline{\Delta }\rightarrow S$ be a normal
mapping that has a ramification point $p_{0}\in \partial \Delta $. Then $%
\partial \Delta $ has a section $\alpha _{1}$ from $p_{0}$ to some point in $%
\partial \Delta \backslash \{p_{0}\}$ such that $\beta =f(z),z\in \alpha
_{1},$ is a simple path in $S$ starting from $f(p_{0})$ and, lifted by $f,$ $%
\beta $ has $b=b_{f}(p_{0})$ lifts $\alpha _{2},\dots ,\alpha _{b+1}$ that
start from $p_{0}$ and satisfy
\begin{equation*}
\alpha _{j}\backslash \{p_{0}\}\subset \Delta ,j=2,\dots ,b+1.
\end{equation*}
\end{corollary}

\begin{lemma}
\label{glue}Let $D$ be a Jordan domain in $\mathbb{C}$ and let $\alpha
_{j}=\alpha _{j}(t),t\in \lbrack 0,1],$ be two paths contained in $\partial
D $ such that $\alpha _{1}(0)=\alpha _{2}(0)$ and $\alpha _{1}\cap \alpha
_{2}$ contains at most two points. Let $f:\overline{D}\rightarrow S$ be a
normal mapping such that%
\begin{equation*}
f(\alpha _{1}(t))=f(\alpha _{2}(t)),t\in \lbrack 0,1].
\end{equation*}%
If $\alpha _{1}(1)\neq \alpha _{2}(1),$ then $f$ can be regarded as a normal
mapping $g:\overline{\Delta }\rightarrow S$ such that
\begin{equation*}
A(g,\Delta )=A(f,D),L(g,\partial \Delta )=L(f,\left( \partial D\right)
\backslash \{\alpha _{1}\cup \alpha _{2}\}),
\end{equation*}%
and $\Gamma _{g}=g(z),z\in \partial \Delta ,$ is the same as the closed
curve
\begin{equation*}
\Gamma _{f}=f(z),z\in \left\{ \left( \partial D\right) \backslash \left[
\alpha _{1}\cup \alpha _{2}\right] \right\} \cup \{\alpha _{1}(1)\},
\end{equation*}%
ignoring a parameter transformation$.$

If $\alpha _{1}(1)=\alpha _{2}(1),$ then $f$ can be regard as an open
continuous mapping $g$ from the sphere $S$ onto itself$.$ And so, $f$ takes
every value in $S.$
\end{lemma}

\begin{proof}
The proof is the standard gluing argument that glue the domain $D$ by
identifying $\alpha _{1}(t)$ and $\alpha _{2}(t)$ for each $t\in \lbrack
0,1] $.
\end{proof}

\begin{lemma}
\label{cut-3}Let $f:\overline{\Delta }\rightarrow S$ be a normal mapping and
let $p_{0}\in \overline{\Delta }$ be a ramification point of $f$. Assume
that $\beta =\beta (t),t\in \lbrack 0,1],$ is a polygonal Jordan path in $S$
such that the followings hold.

(a) $\beta (0)=f(p_{0}),$ $\beta $ has two distinct lifts $\alpha
_{j}=\alpha _{j}(t),t\in \lbrack 0,1],$ in $\overline{\Delta }$ by $f,$ with
$\alpha _{j}(0)=p_{0}$ and%
\begin{equation}
f(\alpha _{1}(t))=f(\alpha _{2}(t))=\beta (t),t\in \lbrack 0,1],j=1,2.
\label{boli-1}
\end{equation}

(b) The interior $\alpha _{j}^{\circ }=\alpha _{j}(t),t\in (0,1),$ of $%
\alpha _{j}$ is contained in $\Delta ,j=1,2,\ $and
\begin{equation*}
\{\alpha _{1}(1),\alpha _{2}(1)\}\subset \partial \Delta .
\end{equation*}

(c) $f$ has no ramification point in the interior of $\alpha _{1}$ and $%
\alpha _{2}.$

Then $\alpha _{1}(1)\neq \alpha _{2}(1).$
\end{lemma}

\begin{proof}
Since there is no ramification point in the interiors of $\alpha _{1}$ and $%
\alpha _{2},$ we have%
\begin{equation}
\alpha _{1}\cap \alpha _{2}\subset \{\alpha _{1}(0),\alpha _{1}(1)\}.
\label{2}
\end{equation}%
By (\ref{boli-1}) and (b),
\begin{equation}
\beta (t)\neq 0,1,\infty ,t\in (0,1).  \label{3}
\end{equation}

If $\alpha _{1}(1)=\alpha _{2}(1),$ then $\alpha _{1}-\alpha _{2},$ or $%
\alpha _{2}-\alpha _{1},$ encloses a Jordan domain $D$ in $\overline{\Delta }%
,$ and then by Lemma \ref{glue}, $f(\overline{D})=S.$ But $f$ is a normal
mapping, and then $f(D)\subset f(\Delta )\subset S\backslash \{0,1,\infty
\}, $ and then by (\ref{2}) and (\ref{3}) we have $f^{-1}(\{0,1,\infty
\})\subset \{\alpha _{1}(0),\alpha _{2}(1)=\alpha _{2}(1)\}.$ Therefore we
have $f(\overline{D})\neq S$. This is a contradiction.
\end{proof}

\section{A classical isoperimetric inequality of the unit hemisphere\label%
{ss-4classi}}

In this section we use Bernstein's isoperimetric inequality to prove
theorems \ref{good} and \ref{good2}, which will be used in Section \ref%
{ss-p2}.

The following result is obtained by Bernstein in 1905.

\begin{theorem}[Bernstein inequality \protect\cite{Ber}]
Let $\Gamma $ be a simple curve in some hemisphere $S^{\ast }$ of $S.$ Then
the length $L=L(\Gamma )$ and the area $A$ of the domain in $S^{\ast }$
enclosed by $\Gamma $ satisfy%
\begin{equation*}
L^{2}\geq 4\pi A-A^{2},
\end{equation*}%
equality holds if and only if $\Gamma $ is a circle.
\end{theorem}

The following inequality is another version of Bernstein inequality.

\begin{corollary}
\label{ber}Under the same hypothesis and additional condition $L(\Gamma
)\leq 2\pi ,$
\begin{equation*}
A\leq 2\pi \left( 1-\sqrt{1-R^{2}}\right) ,
\end{equation*}%
equality holds if and only if $\Gamma $ is a circle, where $R=\frac{L(\Gamma
)}{2\pi }.$
\end{corollary}

In fact, any circle in $S$ with Euclidian radius $R$ divides the sphere into
two spherical disks with areas $2\pi \left( 1\pm \sqrt{1-R^{2}}\right) .$
The following result is obtained by Lad\'{o} in 1935.

\begin{theorem}[Lad\'{o} \protect\cite{Rad}]
\label{Rad}Any closed curve in $S$ with length less than $2\pi $ is
contained in some open hemisphere.
\end{theorem}

\begin{corollary}
\label{ber0}Let $l$ be a given positive number with
\begin{equation*}
\pi <l<\sqrt{2}\pi ,
\end{equation*}%
let $l_{1}$ and $l_{2}$ be positive numbers with
\begin{equation*}
l_{1}+l_{2}=l\ \mathrm{and\ }l_{j}\geq \frac{\pi }{2},j=1,2,
\end{equation*}%
and, for $j=1,2,$ let $\gamma _{j}$ be a circular path in $S$ such that $%
\gamma _{j}$ has endpoints $\{0,1\}$, $L(\gamma _{j})=l_{j},$ and $\gamma
=\gamma _{1}+\gamma _{2}$ is a Jordan curve that encloses a domain $%
D_{\gamma }$ in some hemisphere of $S.$ Then the area of $D_{\gamma }$
assumes the maximum if and only if $l_{1}=l_{2}=\frac{1}{2}l$ and $D_{\gamma
}$ is convex.
\end{corollary}

By this corollary, $D_{\gamma }$ assume the maximum if and only $D_{\gamma }$
is congruent with the domain $D_{l}$ defined in Section \ref{1SS-Intro}.

\begin{proof}
This follows from Corollary \ref{ber} and Theorem \ref{Rad} directly. Let $%
\Gamma _{1}$ be a circle passing through $0$ and $1$ in $S$ so that the
length of the section $\alpha _{1}$ of $\Gamma _{1}$ from $0$ to $1$ is $%
\frac{l}{2}$ and $\Gamma _{1}$ is convex in the sense that the disk inside $%
\Gamma _{1}$ is contained in some open hemisphere of $S$. Then by the
assumption, we have $L(\alpha _{1})<L(\Gamma _{1}\backslash \alpha _{1}),$
and then there is a point $p\in \Gamma _{1}\backslash \alpha _{1}$ so that
the section $\alpha _{2}$ of $\Gamma _{1}$ from $1$ to $p$ has length $\frac{%
l}{2}$ as well.

We replace $\alpha _{j}$ with $\gamma _{j}^{\prime }$ so that $\gamma
_{j}^{\prime }$ is congruent with $\gamma _{j},j=1,2,$ and that the circle $%
\Gamma _{1}$ becomes a Jordan curve $\Gamma _{2}$ that is convex everywhere,
except at $0,1$ and $p,$ in the sense that the triangle $\overline{0,1,p,0}$
is inside the closure of the domain inside $\Gamma _{2}$. It is clear that $%
L(\Gamma _{1})=L(\Gamma _{2})<2\pi ,$ and thus by Theorem \ref{Rad}, $\Gamma
_{j}$ is contained in some hemisphere $S_{j}$ of $S,j=1,2$. Then by Theorem %
\ref{ber} $A_{\Gamma _{1}}\geq A_{\Gamma _{2}},$ the equality holds if and
only if $\Gamma _{2}$ is a circle, where $A_{\Gamma _{j}}$ is the area
enclosed by $\Gamma _{j}$ in $S_{j},j=1,2.$ From this, the conclusion
follows.
\end{proof}

\begin{lemma}
\label{extre}Let $l<2\pi $ be a positive number and let $l_{1},l_{2},\dots
,l_{n}$ be nonnegative numbers with
\begin{equation*}
0\leq l_{1}\leq l_{2}\leq \dots \leq l_{n}\ \mathrm{and\ }l_{1}+l_{2}+\dots
+l_{n}=l.
\end{equation*}%
Then%
\begin{equation*}
\sum_{k=1}^{n}\left( 2\pi -\sqrt{(2\pi )^{2}-l_{k}^{2}}\right) \leq 2\pi -%
\sqrt{(2\pi )^{2}-l^{2}},
\end{equation*}%
the equality holds if and only if $l_{1}=\dots =l_{n-1}=0$ and $l_{n}=l.$
\end{lemma}

\begin{proof}
There is a standard way in calculus to prove this. In fact, it also follows
from Bernstein's inequality and Lad\'{o}'s theorem directly.
\end{proof}

\begin{theorem}
\label{good}Let $f:\overline{\Delta }\rightarrow S$ be a normal mapping such
that
\begin{equation*}
L(f,\partial \Delta )<2\pi .
\end{equation*}%
Then
\begin{equation}
A(f,\Delta )\leq A(D_{f})=\frac{1-\sqrt{1-R_{f}^{2}}}{R_{f}}L(f,\partial
\Delta )<L(f,\partial \Delta )  \label{a4}
\end{equation}%
with $R_{f}=\frac{L(f,\partial \Delta )}{2\pi }$ and $D_{f}$ is a disk in
some open hemisphere of $S$ with $L(\partial D_{f})=L(f,\partial \Delta ).$

If, in addition, $L(f,\partial \Delta )\geq \sqrt{2}\pi ,$ then
\begin{equation}
4\pi +A(f,\Delta )<4L(f,\partial \Delta ).  \label{a5}
\end{equation}
\end{theorem}

\begin{proof}
We first show that%
\begin{equation}
A(f,\Delta )\leq 2\pi -\sqrt{4\pi ^{2}-\left( L(f,\partial \Delta )\right)
^{2}}.  \label{sss2}
\end{equation}%
We may assume that

(a) The boundary curve $\Gamma _{f}=f(z),z\in \partial \Delta ,$ has
finitely many multiple points, i.e. there is a finite set $Q\subset \partial
\Delta ,$ such that $f$ restricted to $\left( \partial \Delta \right)
\backslash Q$ is injective.

If (a) fails, we may conside the restriction $f_{1}=f|_{\overline{D}}$ of $f$
to some closed Jordan domain $\overline{D}\subset \overline{\Delta },$ such
that the boundary curve
\begin{equation*}
\Gamma _{f_{1}}=f_{1}(z)=f(z),z\in \partial D,
\end{equation*}
of $f_{1}$ satisfies (a), while $|A(f,\Delta )-A(f_{1},D)|$ and $%
|L(f,\partial \Delta )-L(f_{1},\partial D)|$ may be made arbitrarily small.
Then we prove (\ref{sss2}) for $f_{1},$ which implies (\ref{sss2}) for $f$.

By Theorem \ref{Rad}, $f(\partial \Delta )$ is contained in some open
hemisphere $S^{\prime }$ of $S,$ and then $f(\partial \Delta )\cap \left(
S\backslash S^{\prime }\right) =\emptyset .$ If $f(\Delta )\cap \left(
S\backslash S^{\prime }\right) \neq \emptyset ,$ then, since $f$ is normal
and a normal mapping is an open mapping, it is clear that $S\backslash
S^{\prime }\subset f(\Delta ),$ which implies that $f(\Delta )\cap E\neq
\emptyset $ (recall that $E=\{0,1,\infty \}),$ for $S\backslash S^{\prime }$
is a closed hemisphere of $S$ and a closed hemisphere of $S$ must contain at
least one point of $E.$ But this contradicts that $f$ is a normal mapping$.$
Thus, the followings holds.

(b) $f(\overline{\Delta })$ is contained in $S^{\prime }.$

For each positive integer $j,$ let $\Delta _{j}$ be the set that for each
point $p\in \Delta _{j},$ $f(z)=p$ has at least $j$ solutions in $\Delta $,
counted with multiplicities. Since $f$ is normal, there exists a positive
integer $n$ such that $n$ is the largest number with $\Delta _{n}\neq
\emptyset .$ Then
\begin{equation}
A(f,\Delta )=\sum_{j=1}^{n}A(\Delta _{j}),  \label{aaa6}
\end{equation}%
and by (a), considering that $f$ is a normal mapping, it is clear that for
any pair $\{j,k\}$ with $j\neq k,$ $\left( \partial \Delta _{j}\right) \cap
\left( \partial \Delta _{k}\right) $ is a finite set and
\begin{equation}
L(f,\partial \Delta )=\sum_{j=1}^{n}L(\partial \Delta _{j}).  \label{aaa7}
\end{equation}

For each $j\leq n,$ $\Delta _{j}$ is a union of finitely many components $%
\Delta _{jk},k=1,2,\dots ,k_{j},$ each of which is a domain also contained
in $S^{\prime }$ (by (b)) and is enclosed by a finite number of polygonal
Jordan curves. For each $j$ and each $k\leq k_{j},$ Let $\Delta _{jk}^{\ast
} $ be the domain which is the complement of the component of $S\backslash
\Delta _{jk}$ in $S$ that contains $\partial S^{\prime }.$ Then $\Delta
_{jk}^{\ast }$ is a polygonal Jordan domain with%
\begin{equation*}
\partial \Delta _{jk}^{\ast }\subset \partial \Delta _{jk}\mathrm{\ but\ }%
\Delta _{jk}^{\ast }\supset \Delta _{jk},
\end{equation*}%
and then by Corollary \ref{ber} we have%
\begin{equation*}
A(\Delta _{jk})\leq A(\Delta _{jk}^{\ast })\leq 2\pi -\sqrt{4\pi
^{2}-L(\partial \Delta _{jk}^{\ast })^{2}}\leq 2\pi -\sqrt{4\pi
^{2}-L(\partial \Delta _{jk})^{2}},
\end{equation*}%
i.e.%
\begin{equation*}
A(\Delta _{jk})\leq 2\pi -\sqrt{4\pi ^{2}-L(\partial \Delta _{jk})^{2}},
\end{equation*}%
for each $j\leq n$ and each $k\leq k_{j}.$ Then, by Lemma \ref{extre} we have%
\begin{eqnarray*}
\sum_{j=1}^{n}\sum_{k=1}^{k_{j}}A(\Delta _{jk}) &\leq
&\sum_{j=1}^{n}\sum_{k=1}^{k_{j}}(2\pi -\sqrt{4\pi ^{2}-L(\partial \Delta
_{jk})^{2}}) \\
&\leq &2\pi -\sqrt{4\pi ^{2}-\left(
\sum_{j=1}^{n}\sum_{k=1}^{k_{j}}L(\partial \Delta _{jk})\right) ^{2}},
\end{eqnarray*}%
the second equality holds if and only if $n=1$ and $\Delta _{1}=f(\Delta ).$
By (\ref{aaa6}) and (\ref{aaa7}), considering that
\begin{equation*}
\sum_{j=1}^{n}A(\Delta _{j})=\sum_{j=1}^{n}\sum_{k=1}^{k_{j}}A(\Delta _{jk}),
\end{equation*}%
and%
\begin{equation*}
\sum_{j=1}^{n}L(\partial \Delta
_{j})=\sum_{j=1}^{n}\sum_{k=1}^{k_{j}}L(\partial \Delta _{jk}),
\end{equation*}%
we have (\ref{sss2})$.$

Let $R_{f}=\frac{L(f,\partial \Delta )}{2\pi }.$ Then by (\ref{sss2}),
considering that $R_{f}<1,$ we have%
\begin{eqnarray*}
A(f,\Delta ) &\leq &2\pi (1-\sqrt{1-R_{f}^{2}}) \\
&=&\frac{1-\sqrt{1-R_{f}^{2}}}{R_{f}}L(f,\partial \Delta ) \\
&<&L(f,\partial \Delta ),
\end{eqnarray*}%
and (\ref{a4}) is proved.

On the other hand, under the additional assumption $L(f,\partial \Delta
)\geq \sqrt{2}\pi ,$ we have $\frac{4\pi }{L(f,\partial \Delta )}\leq 2\sqrt{%
2},$ and then by (\ref{a4}), we have
\begin{equation*}
A(f,\Delta )+4\pi <L(f,\partial \Delta )+4\pi \leq \left( 1+2\sqrt{2}\right)
L(f,\partial \Delta )<4L(f,\partial \Delta ).
\end{equation*}%
This completes the proof.
\end{proof}

\begin{corollary}
\label{good0}Let $f:\overline{\Delta }\rightarrow S$ be a normal mapping
such that $f$ maps the diameter $I=[-1,1]$ of $\overline{\Delta }$
homeomorphically onto the line segment $\gamma =\overline{0,1}$ in $S$ and%
\begin{equation}
L(f,\partial \Delta )<\sqrt{2}\pi .  \label{4.1}
\end{equation}%
Then%
\begin{equation*}
A(f,\Delta )\leq A(D_{l}),
\end{equation*}%
where $D_{l}$ is the convex Jordan domain in $S$ which is contained in the
spherical disk in $S$ with diameter $\overline{0,1}$ and is enclosed by the
two circular arcs in $S,$ each of which has endpoints $\{0,1\}$ and length $%
l=\frac{1}{2}L(f,\partial \Delta ).$
\end{corollary}

\begin{remark}
The domain $D_{l}$ defined here is congruent with the domain $D_{l}$ defined
in Section \ref{1SS-Intro} and the convexity of $D_{l}$ is ensured by (\ref%
{4.1}).
\end{remark}

\begin{proof}
This follows from Corollaries \ref{ber0} and Theorem \ref{good}. Without
loss of generality, we assume that the orientation of $f([-1,1])\subset S$
is from $0$ to $1.$

Let
\begin{equation*}
\alpha ^{+}=\{z\in \partial \Delta ;\mathrm{Im}z\geq 0\},\ \alpha
^{-}=\{z\in \partial \Delta ;\mathrm{Im}z\leq 0\},
\end{equation*}%
\begin{equation*}
\Delta ^{+}=\{z\in \Delta ;\mathrm{Im}z>0\},\ \Delta ^{-}=\{z\in \Delta ;%
\mathrm{Im}z<0\}.
\end{equation*}%
Then by (\ref{4.1}) there uniquely exists a circle $C$ in $S$ passing
through $0$ and $1$ such that the interior $\overline{0,1}^{\circ }$ of $%
\overline{0,1}$ is contained in the disk $K$ enclosed by $C$ and the section
$c_{1}$ of $C$ from $1$ to $0$ has length $L(f,\alpha ^{+}).$ Then, by the
assumption, it is clear that
\begin{equation*}
L(c_{1})=L(f,\alpha ^{+})\geq \frac{\pi }{2}.
\end{equation*}

If $f(\alpha ^{+})=\frac{\pi }{2},$ then by the assumption we have $%
f(\partial \Delta ^{+})=\overline{0,1},$ and then $f(\Delta ^{+})$ must
contains $\infty ,$ for normal mappings are open mappings. But this
contradicts the assumption that $f$ is normal and as a normal mapping $%
f(z)\neq 0,1,\infty $ for all $z\in \Delta $. Thus we have $L(f,\alpha ^{+})>%
\frac{\pi }{2},$ which implies
\begin{equation}
L(\partial K)=L(C)<2\pi .  \label{4.1+1}
\end{equation}%
Thus $\overline{0,1}+c_{1}$ encloses a Jordan domain $D_{1}$ and $%
D_{2}=K\backslash \overline{D_{1}}$ is also a Jordan domain.

We may extend $f|_{\overline{\Delta ^{+}}}$ to be a continuous mapping $F:%
\overline{\Delta }\rightarrow S$ such that $F$ restricted to $\overline{%
\Delta ^{-}}$ is a homeomorphism onto $\overline{D_{2}}$ and restricted to $%
\alpha ^{-}$ is a homeomorphism onto $C\backslash c_{1}.$ Then, we have
\begin{equation*}
L(F,\partial \Delta )=L(f,\alpha ^{+})+L(F,\alpha
^{-})=L(c_{1})+L(C\backslash c_{1})=L(C),
\end{equation*}%
which, with (\ref{4.1+1}), implies
\begin{equation}
L(F,\partial \Delta )=L(\partial K)=L(C)<2\pi .  \label{4.2}
\end{equation}

$F$ is not a normal mapping, and so we can not apply Theorem \ref{good} to $%
F $ directly. But by (\ref{4.2}) we can apply Theorem \ref{good} to a normal
mapping $g$ so that $|L(g,\partial \Delta )-L(F,\partial \Delta )|$ and $%
|A(g,\Delta )-A(F,\Delta )|$ can be made arbitrarily small, and finally
obtain%
\begin{equation}
A(F,\Delta )\leq A(D_{F}),  \label{4.3}
\end{equation}%
where $D_{F}$ is a disk in some hemisphere of $S$ with%
\begin{equation}
L(\partial D_{F})=L(F,\partial \Delta ).  \label{4.5}
\end{equation}

$F$ is not normal just because the boundary curve $\Gamma _{F}=F(z),z\in
\partial \Delta ,$ is not polygonal. Since $F(\alpha ^{+})=f(\alpha ^{+})$
is already polygonal, $F(\alpha ^{-})=C\backslash c_{1}$ and $\partial
\Delta =\alpha ^{+}\cup \alpha ^{-},$ the mapping $g$ mentioned above can be
obtained by restricting $F$ to a domain $\Delta _{g}\subset \Delta $ with $%
\Delta _{g}\supset \Delta ^{+}.$

By (\ref{4.2}) and (\ref{4.5}) we have $L(\partial D_{F})=L(\partial K),$
which implies $A(D_{F})=A(K).$ Thus, by (\ref{4.3}) we have%
\begin{equation*}
A(F,\Delta )\leq A(K).
\end{equation*}%
Therefore, by the facts $A(K)=A(D_{1})+A(D_{2})$ and $A(F,\Delta
)=A(f,\Delta ^{+})+A(D_{2})$ we have $A(f,\Delta ^{+})\leq A(D_{1}).$

Similarly, we can show that $A(f,\Delta ^{-})\leq A(D_{1}^{\prime }),$ where
$D_{1}^{\prime }$ is the convex domain in some hemisphere of $S$ and is
enclosed by $\overline{1,0}$ and the circular arc $c_{2}$ from $0$ to $1$
with $L(c_{2})=L(f,\alpha ^{-}).$ Then $\gamma =c_{1}+c_{2}$ encloses a
Jordan domain $D_{\gamma }$ with $A(D_{\gamma })=A(D_{1})+A(D_{1}^{\prime })$
and%
\begin{equation*}
A(f,\Delta )\leq A(D_{1})+A(D_{1}^{\prime })=A(D_{\gamma }),
\end{equation*}%
and by Corollary \ref{ber0}, the desired result follows. This completes the
proof.
\end{proof}

Let $\alpha $ be a circular path in the upper half plane $\mathrm{Im}z\geq 0$
from $1$ to $0$ and let $\mathfrak{A}_{\alpha }$ be the domain in $\mathbb{C}
$ enclosed by $\alpha $ and the interval $[0,1]$ and assume $L(\alpha )\leq
\frac{\sqrt{2}}{2}\pi ,$ which means that $\alpha $ is contained in the
closed half-disk
\begin{equation*}
\{z\in \mathbb{C};\ \mathrm{Im}z\geq 0\ \mathrm{and\ }|z-\frac{1}{2}|<\frac{1%
}{2}\}.
\end{equation*}%
Then
\begin{equation*}
\frac{\pi }{2}\leq L(\alpha )\leq \frac{\sqrt{2}}{2}\pi .
\end{equation*}%
We want to find the relation between the spherical length $L(\alpha )$ and
the spherical area $A(\mathfrak{A}_{\alpha }).$ We will show that both $%
L(\alpha )$ and $A(\mathfrak{A}_{\alpha })$ is a real analytical function of
\begin{equation*}
\tau =\sin \theta _{a},0\leq \theta _{\alpha }\leq \frac{\pi }{2}.
\end{equation*}%
where $\theta _{\alpha }$ is the value of the angle between $\alpha $ and
the interval $[0,1]$ at $0.$


\begin{lemma}
\label{mid}In the above setting, we have%
\begin{equation}
L(\alpha )=\zeta _{0}(\tau ):=\frac{2}{\sqrt{1+\tau ^{2}}}(\frac{\pi }{2}%
-\arctan \frac{\sqrt{1-\tau ^{2}}}{\sqrt{1+\tau ^{2}}}),\tau \in \lbrack
0,1],  \label{1110-1}
\end{equation}%
and%
\begin{equation}
A(\mathfrak{A}_{\alpha })=\zeta _{1}(\tau ):=2\arcsin \tau -\tau \zeta
_{0}(\tau ),\tau \in \lbrack 0,1].  \label{1110-2}
\end{equation}
\end{lemma}

\begin{proof}
Let $c_{\alpha }\in \mathbb{C}$ be the center of the circle containing $%
\alpha .$ Then $\mathrm{Re}c_{\alpha }=\frac{1}{2},$ and since $L(\alpha
)\leq \frac{\sqrt{2}}{2}\pi ,$ $\mathrm{Im}c_{\alpha }\leq 0.$ Let $%
d_{\alpha }=2c_{\alpha }.$ Then the triangle in $\mathbb{C}$ with vertices $%
0,1$ and $d_{\alpha }$ is a right-angled triangle and $\theta _{\alpha }$ is
the value of the angle at $d_{\alpha }.$

It is clear that
\begin{equation*}
|d_{\alpha }|=\frac{1}{\sin \theta _{\alpha }}.
\end{equation*}%
On the other hand, for any point $z\in \alpha ,$ it is clear that
\begin{equation*}
|z|=\sin (\theta _{\alpha }-t)|d_{\alpha }|,
\end{equation*}%
where $t=\arg z.$ Then we obtain a parameter expression of the circular path
$\alpha $:%
\begin{equation*}
\alpha =\alpha (t)=\frac{\sin (\theta _{\alpha }-t)}{\sin \theta _{\alpha }}%
e^{it},t\in \lbrack 0,\theta _{\alpha }],
\end{equation*}

Then we have%
\begin{equation*}
|d\alpha (t)|=\frac{|-e^{it}\cos (\theta _{\alpha }-t)+ie^{it}\sin (\theta
_{\alpha }-t)|}{\sin \theta _{\alpha }}dt=\frac{dt}{\sin \theta _{\alpha }},
\end{equation*}%
and%
\begin{eqnarray*}
L(\alpha ) &=&\int_{\alpha }\frac{2|dz|}{1+|z|^{2}}=\int_{0}^{\theta
_{\alpha }}\frac{2|d\alpha (t)|}{1+|\alpha (t)|^{2}} \\
&=&\int_{0}^{\theta _{\alpha }}\frac{2\sin \theta _{\alpha }}{\sin
^{2}\theta _{\alpha }+\sin ^{2}(\theta _{\alpha }-t)}dt \\
&=&\int_{0}^{\theta _{\alpha }}\frac{2\sin \theta _{\alpha }}{\sin
^{2}\theta _{\alpha }+\sin ^{2}x}dx \\
&=&\frac{2}{\sqrt{1+\sin ^{2}\theta _{\alpha }}}\left( \frac{\pi }{2}%
-\arctan \frac{\sqrt{1-\sin ^{2}\theta _{\alpha }}}{\sqrt{1+\sin ^{2}\theta
_{\alpha }}}\right) ,
\end{eqnarray*}%
and we have (\ref{1110-1}).

On the other hand, we have%
\begin{eqnarray*}
A(\mathfrak{A}_{\alpha }) &=&\iint\limits_{\mathfrak{A}_{\alpha }}\frac{4dxdy%
}{\left( 1+|z|^{2}\right) ^{2}} \\
&=&\int_{0}^{\theta _{a}}dt\int_{0}^{|\alpha (t)|}\frac{4rdr}{\left(
1+r^{2}\right) ^{2}}=2\int_{0}^{\theta _{a}}\left( 1-\frac{1}{1+|\alpha
(t)|^{2}}\right) dt \\
&=&2\theta _{\alpha }-2\int_{0}^{\theta _{a}}\frac{dt}{1+|\alpha (t)|^{2}} \\
&=&2\theta _{a}-2\int_{0}^{\theta _{a}}\frac{\sin ^{2}\theta _{\alpha }dx}{%
\sin ^{2}\theta _{\alpha }+\sin ^{2}x} \\
&=&2\theta _{a}-\sin \theta _{\alpha }L(\alpha ),
\end{eqnarray*}%
and we have (\ref{1110-2}).
\end{proof}

It is clear from the geometrical sense that the function
\begin{equation*}
\zeta _{0}(\tau )=\frac{2}{\sqrt{1+\tau ^{2}}}\left( \frac{\pi }{2}-\arctan
\frac{\sqrt{1-\tau ^{2}}}{\sqrt{1+\tau ^{2}}}\right) ,\tau \in \lbrack 0,1]
\end{equation*}%
is an injective mapping.

\begin{corollary}
\label{good00}Let $f:\overline{\Delta }\rightarrow S$ be a normal mapping
such that $f$ maps the diameter $[-1,1]$ of $\overline{\Delta }$
homeomorphically onto the line segment $\overline{0,1}$ in $S$ and%
\begin{equation*}
L(f,\partial \Delta )<\sqrt{2}\pi .
\end{equation*}%
Then
\begin{equation*}
A(f,\Delta )\leq 2\zeta _{1}(\tau )=4\arcsin \tau -2\tau \zeta _{0}(\tau ),
\end{equation*}%
where $\tau =\zeta _{0}^{-1}(\frac{1}{2}L(f,\partial \Delta ))$.
\end{corollary}

\begin{proof}
This follows from Lemma \ref{mid} and Corollary \ref{good0} directly.
\end{proof}

\begin{theorem}
\label{good2}Let $f:\overline{\Delta }\rightarrow S$ be a normal mapping
such that $f$ maps the diameter $[-1,1]$ of $\overline{\Delta }$
homeomorphically onto the interval $[0,1]$ in $S$ and%
\begin{equation*}
L(f,\partial \Delta )<\sqrt{2}\pi .
\end{equation*}%
Then
\begin{equation*}
4\pi +A(f,\Delta )\leq h_{0}L(f,\partial \Delta ),
\end{equation*}%
where $h_{0}$ is given by (\ref{best}), i.e.
\begin{equation*}
h_{0}=\max_{\tau \in \lbrack 0,1]}\left[ \frac{\sqrt{1+\tau ^{2}}\left( \pi
+\arcsin \tau \right) }{\mathrm{arccot}\frac{\sqrt{1-\tau ^{2}}}{\sqrt{%
1+\tau ^{2}}}}-\tau \right] .
\end{equation*}
\end{theorem}

\begin{proof}
Let $\tau =\zeta _{0}^{-1}(\frac{1}{2}L(f,\partial \Delta )).$ Then
\begin{equation*}
L(f,\partial \Delta )=2\zeta _{0}(\tau )
\end{equation*}%
and by Corollary \ref{good00}.
\begin{equation*}
A(f,\Delta )\leq 2\zeta _{1}(\tau ).
\end{equation*}%
Then%
\begin{eqnarray*}
4\pi +A(f,\Delta ) &\leq &4\pi +2\zeta _{1}(\tau ) \\
&=&\frac{4\pi +2\zeta _{1}(\tau )}{L(f,\partial \Delta )}L(f,\partial \Delta
) \\
&=&\frac{4\pi +2\zeta _{1}(\tau )}{2\zeta _{0}(\tau )}L(f,\partial \Delta )
\\
&=&\frac{2\pi +\zeta _{1}(\tau )}{\zeta _{0}(\tau )}L(f,\partial \Delta
)\leq h_{0}L(f,\partial \Delta ),
\end{eqnarray*}%
where
\begin{eqnarray*}
h_{0} &=&\max_{\tau \in \lbrack 0,1]}\frac{2\pi +\zeta _{1}(\tau )}{\zeta
_{0}(\tau )}=\max_{\tau \in \lbrack 0,1]}\frac{2\pi +2\arcsin \tau -\tau
\zeta _{0}(\tau )}{\zeta _{0}(\tau )} \\
&=&\max_{\tau \in \lbrack 0,1]}\left[ \frac{2\pi +2\arcsin \tau }{\frac{2}{%
\sqrt{1+\tau ^{2}}}\left( \frac{\pi }{2}-\arctan \frac{\sqrt{1-\tau ^{2}}}{%
\sqrt{1+\tau ^{2}}}\right) }-\tau \right] \\
&=&\max_{\tau \in \lbrack 0,1]}\left[ \frac{\sqrt{1+\tau ^{2}}\left( \pi
+\arcsin \tau \right) }{\mathrm{arccot}\frac{\sqrt{1-\tau ^{2}}}{\sqrt{%
1+\tau ^{2}}}}-\tau \right] .
\end{eqnarray*}
\end{proof}

\section{Locally convex polygonal paths and curves in the Riemann sphere
\label{SS5-Sconvex-curves}}

The goal of Sections \ref{SS5-Sconvex-curves}--\ref{ss-p1} is to prove
Theorem \ref{decom}, which is the second key step to prove the main theorem.

In this section we prove some results about locally convex polygonal Jordan
paths and curves, which is used in Sections \ref{ss-6lift} and \ref{ss-p1}.

It is clear that a generic convex triangle\footnote{%
See Definition \ref{gen-tri}.} in $S$ is contained in some open hemisphere
of $S,$ and then we have the followings.

\begin{lemma}
\label{convextri}Let $T\subset S$ be a triangle domain inside\footnote{%
By definition, \textquotedblleft inside\textquotedblright\ means
\textquotedblleft on the left hand side of\textquotedblright .} a generic
convex triangle $\overline{q_{1}q_{2}q_{3}q_{1}}$ in $S.$ Then for any $q\in
T$, the notation $\overline{q_{1}qq_{3}q_{1}}$ makes sense and denotes a
generic convex triangle.
\end{lemma}

The following result is easy to see.

\begin{lemma}
\label{con-hemi0}For any polygonal convex Jordan curve $\Gamma $ in $S$ and
any natural edge\footnote{%
By definition, natural edges are oriented by the polygonal curve.} $l$ of $%
\Gamma ,$ the domain inside $\Gamma $ is contained in the open hemisphere of
$S$ which is inside the great circle determined\footnote{%
This means that the great circle contains $l$ and is oriented by $l$.} by $%
l. $
\end{lemma}

\begin{lemma}
\label{con-hemi}Let $\Gamma =\Gamma (z),z\in \partial \Delta ,$ be a
polygonal Jordan curve in $S.$

(i) If $\Gamma $ is convex and contains a pair of antipodal points, then $%
\Gamma $ is a biangle, and moreover, if in addition $\Gamma $ has a straight
edge with length $>\pi ,$ then $\Gamma $ is a great circle of $S.$

(ii) If $\Gamma $ is convex and has at least three vertices in the usual
sense, i.e. $\Gamma $ can be expressed as
\begin{equation*}
\Gamma =l_{1}+l_{2}+\dots +l_{m},m\geq 3,
\end{equation*}%
where each $l_{j}$ is a straight edge\footnote{%
Note that the interior of $l_{j}$ may contain points in $E,$ and so $l_{j}$
may not be a natural edge of $\Gamma ,$ by the definition of natual edges.}
of $\Gamma $ whose endpoints are both strictly convex vertices of $\Gamma $,
$j=1,2,\dots ,m;$ then for each $j,$ $j=1,2,\dots ,m,$%
\begin{equation}
L(l_{j})<\pi  \label{6.1}
\end{equation}%
and for the great circle $C_{l_{j}}$ in $S$ determined by $l_{j},$%
\begin{equation}
\Gamma \cap C_{l_{j}}=l_{j},  \label{6.2}
\end{equation}%
and therefore, $\Gamma $ is contained in some open hemisphere of $S.$
\end{lemma}

\begin{proof}
Assume that $\Gamma $ has a pair of antipodal points $q_{1}$ and $q_{2}$. We
show that the section $\Gamma ^{\prime }$ of $\Gamma $ from $q_{1}$ to $%
q_{2}\ $is straight.

For any natural edge $e$ of $\Gamma ^{\prime }$, by Lemma \ref{con-hemi0}, $%
q_{1}$ and $q_{2}$ are both contained in the closed hemisphere inside the
great circle $C_{e}$ determined by $e,$ and thus $q_{1}$ and $q_{2}$ are
both contained in $C_{e}.$ By the arbitrariness of $e,$ $\Gamma ^{\prime }$
must be a straight path from $q_{1}$ to $q_{2}.$ For the same reason, we can
show that $\Gamma \backslash \Gamma ^{\prime }$ is also a straight path$.$
Thus, $\Gamma $ is a biangle, which implies the second conclusion of (i),
and (i) is proved.

Now, we prove (ii). It is clear that (i) implies (\ref{6.1}) directly, for
otherwise $\Gamma $ is a biangle which contains at most two edges in the
usual sense. So we may write%
\begin{equation*}
\Gamma =\overline{q_{1}q_{2}}+\overline{q_{2}q_{3}}+\dots +\overline{%
q_{m}q_{1}},m\geq 3,
\end{equation*}%
where $l_{j}=\overline{q_{j}q_{j+1}}$ with $q_{m+1}=q_{1}.$

We denote by $C_{l_{m}}$, $C_{l_{1}}$ and $C_{l_{2}}$ the great circles in $%
S $ determined by $l_{m}=\overline{q_{m}q_{1}},l_{1}=\overline{q_{1}q_{2}}\ $%
and $l_{2}=\overline{q_{2}q_{3}},$ and denote by $D_{l_{m}}$, $D_{l_{1}}$
and $D_{l_{2}}$ the domains inside $C_{l_{m}}$, $C_{l_{1}}$ and $C_{l_{2}},$
respectively. Then, $q_{m}$ and $q_{3}$ must be both contained in $%
D_{l_{1}}, $ since, by the assumption, $\Gamma $ is strictly convex at $%
q_{1} $ and $q_{2}$; and then, it is clear that $K=\overline{D_{l_{m}}}\cap
\overline{D_{l_{1}}}\cap \overline{D_{l_{2}}}$ is a closed triangle domain
whose three angles are all strictly less than $\pi $, and then $K$ has a
vertex in $D_{l_{1}}$ and%
\begin{equation*}
l_{1}=\overline{q_{1}q_{2}}=K\cap C_{l_{1}}.
\end{equation*}%
On the other hand, it is clear that $\Gamma \cap C_{l_{1}}\supset l_{1}$
and, by Lemma \ref{con-hemi0}, $K\supset \Gamma .$ Therefore, we have (\ref%
{6.2}) for $j=1$. This completes the proof.
\end{proof}

\begin{lemma}
\label{con-path-hemi}Let $\Gamma $ be a locally convex polygonal Jordan path
with initial and terminal point at $q_{1}.$ Assume $\Gamma $ has the
following natural partition\footnote{%
By definition, here "locally convex" means that for each $j=1,2,\dots ,m-1,$
$\overline{q_{j}q_{j+1}q_{j+2}}$ is a convex path from $q_{j}$ to $q_{j+2}.$
So, as a closed curve, $\Gamma $ may not be convex at $q_{1}$, i.e., $%
\overline{q_{m}q_{1}q_{2}}$ may not be a convex path.}
\begin{equation}
\Gamma =\overline{q_{1}q_{2}}+\overline{q_{2}q_{3}}+\dots +\overline{%
q_{m}q_{1}},m\geq 3,  \label{3.001}
\end{equation}%
such that
\begin{equation}
\overline{q_{1}q_{2}\dots q_{m}}\cap \lbrack 0,+\infty ]=\{q_{1}\}.
\label{3.002}
\end{equation}

Then the followings hold.

(i) For each $j=1,\dots ,m-2,$ $L_{j}=\overline{q_{1}q_{j+1}q_{j+2}q_{1}}$
is a generic convex triangle.

(ii) The closure $\overline{T_{\Gamma }}$ of the domain $T_{\Gamma }$ inside
$\Gamma $ is contained in some open hemisphere of $S.$

(iii) For each triangle domain $T_{j}$ inside the triangle $L_{j},$
\begin{equation*}
T_{j}\cap T_{k}=\emptyset ,1\leq j<k\leq m-2,
\end{equation*}%
and
\begin{equation*}
\overline{T_{\Gamma }}=\cup _{j=1}^{m-2}\overline{T_{j}}.
\end{equation*}
\end{lemma}

\begin{remark}
\label{u}(1). Condition (\ref{3.002}) is used just to ensure that each
vertex $q_{j},j=2,3,\dots ,m,$ of $\Gamma $ is a \emph{strictly} convex
vertex$.$ Thus, (\ref{3.002}) can be replaced by the condition that $\Gamma $
is strictly convex at $q_{2},\dots ,q_{m}.$ By Definition \ref{convex}, (\ref%
{3.002}) can also be replaced by
\begin{equation*}
\{q_{2},\dots q_{m}\}\cap E=\emptyset ,
\end{equation*}%
which, with the assumption that $\Gamma $ is locally convex, implies that $%
\Gamma $ is strictly convex at $q_{2},\dots ,q_{m}.$

(2). The reader should notice that (\ref{3.001}) makes sense if and only if $%
d(q_{j},q_{j+1})<\pi $ for all $j=1,\dots ,m-1,$ by the appointment.

(3). By conclusion (ii), in the case that $\overline{q_{m}q_{1}}+\overline{%
q_{1}q_{m}}$ is straight, we have $L(\overline{q_{m-1}q_{m}}+\overline{%
q_{m}q_{1}})<\pi .$ Thus if we regard $\Gamma $ as a closed polygonal Jordan
curve, each edge, in the usual sense, of $\Gamma $ has length $<\pi ,$ and
thus, each natural edge of $\Gamma $ has length $<\pi .$
\end{remark}

\begin{proof}
We regard $\Gamma $ as a closed curve. Then $\Gamma $ is locally convex
everywhere, with at most one exceptional point at $q_{1}.$

Let $T_{\Gamma }$ be the polygonal domain inside $\Gamma .$ Then, it is easy
to see that there is a path $l$ in $\overline{T_{\Gamma }}$ from $q_{1}$ to
some point $q^{\prime }\in \Gamma $ such that the followings hold.

(a) $l\cap \overline{q_{s}q_{s+1}}^{\circ }=\{q^{\prime }\}$ for some
natural edge $\overline{q_{s}q_{s+1}}$ of $\Gamma ,$ where $\overline{%
q_{s}q_{s+1}}^{\circ }$ is the interior of $\overline{q_{s}q_{s+1}}$ (if $%
s=m,$ $q_{s+1}=q_{1})$.

(b) The interior of $l$ is in the domain $T_{\Gamma }.$

(c) $l$ divides the angle $\Theta _{q_{1}}$ of the polygonal domain $%
T_{\Gamma }$ at $q_{1}$ into two angles, each of which has value $<\pi .$

By the fact that any two distinct straight lines in the sphere $S$ only
intersect at a pair of antipodal points, and that (\ref{3.001}) implies $L(%
\overline{q_{1}q_{2}})<\pi $ and $L(\overline{q_{m}q_{1}})<\pi ,$ we have
that
\begin{equation}
l\cap \overline{q_{1}q_{2}}=l\cap \overline{q_{m}q_{1}}=\{q_{1}\},
\label{3.003}
\end{equation}%
which implies%
\begin{equation}
2\leq s\leq m-1.  \label{3.004}
\end{equation}

It is easy to see from (a)--(c) that
\begin{equation*}
\Gamma _{1}=\overline{q_{1}q_{2}\dots q_{s}q^{\prime }}-l
\end{equation*}
is strictly convex at $q_{1}$ and $q^{\prime },$ and then by (\ref{3.002})
and the assumption that $\Gamma $ is a locally convex path and that $%
q_{2},\dots ,q_{m}$ are the all natural vertices, $\Gamma _{1}$ is a
polygonal convex Jordan curve that is strictly convex at all points $%
q_{1},\dots ,q_{s},q^{\prime }.$ On the other hand, since $\Gamma $ is
simple, by (\ref{3.003}) and (\ref{3.004}) we conclude that $q_{1},q_{2}$
and $q^{\prime }$ are distinct each other. Therefore, $\Gamma _{1}$ is a
convex polygonal Jordan curve that has at least three strictly convex
vertices$,$ and thus, by Lemma \ref{con-hemi} (ii), $\Gamma _{1}\backslash
\overline{q_{s}q^{\prime }}$ is contained in the open hemisphere $S^{\prime
} $ inside the great circle determined by $\overline{q_{s}q_{s+1}}\supset
\overline{q_{s}q^{\prime }}$, and for the same reason,
\begin{equation*}
\Gamma _{2}=\overline{q^{\prime }q_{s+1}\dots q_{m}q_{1}}+l
\end{equation*}
is also a convex polygonal Jordan curve that has at least three strictly
convex vertices and $\Gamma _{2}\backslash \overline{q^{\prime }q_{s+1}}$ is
also contained in $S^{\prime }.$ Thus, $\Gamma \backslash \overline{%
q_{s}q_{s+1}}$ is contained in $S^{\prime },$ and, considering that $L(%
\overline{q_{s}q_{s+1}})<\pi ,$ we have proved (ii).

(i) follows from (ii) and the convexity of $\Gamma _{1}$ and $\Gamma _{2}$;
and (iii) follows from (i) and (ii) directly. This completes the proof.
\end{proof}

In the rest of this section we assume that $\gamma _{0}$ is a locally convex
polygonal Jordan path that has the natural partition
\begin{equation}
\gamma _{0}=\overline{q_{1}q_{2}}+\overline{q_{2}q_{3}}+\dots +\overline{%
q_{m-1}q_{m}},m\geq 3,  \label{6-1}
\end{equation}%
with%
\begin{equation}
\gamma _{0}\cap \lbrack 0,+\infty ]=\{q_{1},q_{m}\}.  \label{3.2+1}
\end{equation}

Then $q_{2},\dots ,q_{m-1}$ are natural vertices of $\gamma _{0},$ at which $%
\gamma _{0}$ is convex, and none of $q_{2},\dots ,q_{m-1}$ is contained in $%
E.$ Thus, by Definitions \ref{n-v-path} and \ref{convex} we have that

(a) $\gamma _{0}$ is strictly convex at all its natural vertices, the points
$q_{2},\dots ,q_{m-1}.$

\begin{lemma}
\label{L6-1}Assume $q_{1}\neq q_{m}$ and let $I_{q_{1}q_{m}}$ be the section
of $[0,+\infty ]$ from $q_{1}$ to $q_{m}.$ Then the followings hold.

(i) $\Gamma =\gamma _{0}-I_{q_{1}q_{m}}$ is a polygonal Jordan curve that is
convex everywhere, with at most one exceptional point at $q_{1}$ or $q_{m}$.

(ii) $L(I_{q_{1}q_{m}})<\pi ,$ and $\Gamma $ and the closure $\overline{%
T_{\Gamma }}$ of the domain $T_{\Gamma }$ enclosed by
\begin{equation*}
\Gamma =\gamma _{0}-I_{q_{1}q_{m}}=\gamma _{0}+\overline{q_{m}q_{1}}
\end{equation*}%
is contained in some open hemisphere of $S$.

(iii) If, in addition, $q_{1}=0,$ then $\Gamma $ is strictly convex at $%
q_{m}.$
\end{lemma}

\begin{proof}
It is clear that $\Gamma =\gamma _{0}-I_{q_{1}q_{m}}$ is simple, and by (a)
we have

(b) $q_{2},\dots ,q_{m-1}$ are strictly convex vertices of $\Gamma .$

Thus the possible nonconvex vertices of $\Gamma $ are $q_{1}$ and $q_{m}.$
We show that $\Gamma $ is convex at $q_{1}$ or $q_{m}.$ We assume the
contrary that both $q_{1}$ and $q_{m}$ are nonconvex vertices and without
loss of generality, we assume
\begin{equation}
q_{1}<q_{m}.  \label{aaa1}
\end{equation}
Then we have\footnote{%
Note that under this contrary assumption, $\Gamma $ does not go straight at $%
q_{1},$ nor at $q_{m}$, but turn right at both $q_{1}$ and $q_{m}$. On the
other hand, $\overline{q_{1}q_{2}},\overline{q_{m-1}q_{m}}$ make sense if
and only if $d\{q_{1},q_{2}\}<\pi $ and $d\{q_{m-1},q_{m}\}<\pi .$ Thus, $%
\overline{q_{1}q_{2}}\cap C_{1m}=\{q_{1}\}\ $and and $\overline{q_{m-1}q_{m}}%
\cap C_{1m}=\{q_{m}\}$, which, with the assumption $q_{1}<q_{m},$ implies
(c).}

(c) Both $q_{2}$ and $q_{m-1}$ are contained in the open hemisphere $%
S^{\prime }$ inside the great circle $C_{1m}$ determined by $I_{q_{1}q_{m}}.$

Then by (\ref{3.2+1}), $I_{q_{1}q_{m}}$ has a neighborhood $J_{1}$ in the
great circle $C_{1m}$ determined by $I_{q_{1}q_{m}}\subset \lbrack 0,+\infty
]$ such that $J_{1}^{\circ }\supset \lbrack 0,+\infty ]$ and $%
J_{1}\backslash I_{q_{1}q_{m}}\subset T_{\Gamma },$ where $T_{\Gamma }$ is
the domain inside $\Gamma $ and $J_{1}^{\circ }$ is the interior of $J_{1}.$

It is clear that there are only two cases need to discuss:

\noindent \textbf{Case 1.} $C_{1m}\cap \Gamma =I_{q_{1}q_{m}}.$

\noindent \textbf{Case 2. }$\left( C_{1m}\cap \Gamma \right) \backslash
I_{q_{1}q_{m}}\neq \emptyset .$

Assume Case 1 occurs. Then $C_{1m}\cap \gamma _{0}=\{q_{1},q_{m}\},$ and for
the section\footnote{%
Recall that, by the appointment, a section of a curve inherits the
orientation of the curve, and so $I_{q_{m}q_{1}}^{\prime }$ is the
complementary of $I_{q_{1}q_{m}}^{\circ }$ in $C_{1m}.$} $%
I_{q_{m}q_{1}}^{\prime }$ of $C_{1m}$ from $q_{m}$ to $q_{1},$ by (\ref{aaa1}%
) and (c) we conclude that
\begin{equation*}
\Gamma ^{\prime }=\gamma _{0}+I_{q_{m}q_{1}}^{\prime }
\end{equation*}%
is a Jordan curve that is strictly convex at $q_{1}$ and $q_{m},$ and then
by (a) and Remark \ref{loc-con} we can conclude that $\Gamma ^{\prime }$ is
a convex polygonal Jordan curve in $S$ and is strictly convex at $%
q_{1},\dots ,q_{m},$ and thus we have by Lemma \ref{con-hemi} that $%
L(I_{q_{m}q_{1}}^{\prime })<\pi ,$ but on the other hand
\begin{equation*}
L(I_{q_{m}q_{1}}^{\prime })=L(C_{1m})-L(I_{q_{1}q_{m}})\geq 2\pi
-L([0,+\infty ])=\pi ,
\end{equation*}%
which is a contradiction. Thus, Case 1 can not occur, and then, Case 2 must
occur.

Then, we can extend the path $J_{1}$ past both sides to be a longer path $J$
from $q^{\prime }$ to $q^{\prime \prime }$ such that

(d) $\{q^{\prime },q^{\prime \prime }\}\subset \Gamma ,\ J$ is oriented by $%
I_{q_{1}q_{m}},$ the interior of the section of $J$ from $q^{\prime }$ to $%
q_{1}$ and the interior of the section of $J$ from $q_{m}$ to $q^{\prime
\prime }$ are both contained in $T_{\Gamma }.$

Then
\begin{equation}
L(J)>\pi ,  \label{qqqq}
\end{equation}
for $J\supset J_{1}^{\circ }\supset \lbrack 0,+\infty ].$

We first show that $q^{\prime }\neq q^{\prime \prime }.$ We assume the
contrary that $q^{\prime }=q^{\prime \prime }.$ Then it is clear that $%
q^{\prime }$ is in the interior $\gamma _{0}^{\circ }$ of $\gamma _{0}$ and,
by (d), we have
\begin{equation}
C_{1m}\cap \gamma _{0}^{\circ }=\{q^{\prime }\}.  \label{aaa}
\end{equation}%
Then, by (c), (\ref{aaa}) and the fact that $\gamma _{0}$ is simple and
connected, we have

\begin{equation}
\gamma _{0}^{\circ }\backslash \{q^{\prime }\}\subset S^{\prime }.
\label{qqq}
\end{equation}

Since $q^{\prime }\neq q_{1},q_{m},$ $\gamma _{0}$ is convex at $q^{\prime }$
by the assumption. Then by (\ref{qqq}), $\gamma _{0}$ is strictly convex at $%
q^{\prime },$ and thus the domain $T_{\Gamma }$ is a polygonal Jordan domain
with an angle at $q^{\prime }$ strictly less than $\pi .$ But by (d) and the
assumption $q^{\prime }=q^{\prime \prime },$ $J\backslash I_{q_{1}q_{m}}$ is
a neighborhood of $q^{\prime }$ in $C_{1m}$ and $\left( J\backslash
I_{q_{1}q_{m}}\right) \backslash \{q^{\prime }\}\subset T_{\Gamma }$. This
is a contradiction. Thus, $q^{\prime }\neq q^{\prime \prime }.$

Let $\gamma _{0}^{\prime }$ be the section of $\gamma _{0}$ from $q^{\prime
} $ to $q^{\prime \prime }.$ Then
\begin{equation*}
\Gamma ^{\prime }=\gamma _{0}^{\prime }-J
\end{equation*}%
is a polygonal Jordan curve that is strictly convex at $q^{\prime }\ $and $%
q^{\prime \prime },$ for $\gamma _{0}$ is a locally convex path, $%
\{q^{\prime },q^{\prime \prime }\}\subset \gamma _{0}^{\prime }\subset
\gamma _{0}^{\circ }$, $q^{\prime }$ and $q^{\prime \prime }$ have
neighborhoods in $J$ contained in $T_{\Gamma }.$ Thus, $\Gamma ^{\prime }$
is convex everywhere by the assumption on $\gamma _{0}$, and then $\Gamma
^{\prime }$ is convex by Remark \ref{loc-con}, and then by (\ref{qqqq}) and
Lemma \ref{con-hemi} (i), $\Gamma ^{\prime }$ is a great circle, which is a
contradiction since $\Gamma ^{\prime }$ strictly convex at $q^{\prime }.$

Summarizing the above argument, we can conclude that $\Gamma $ must be
convex at $q_{1}$ or $q_{m}$, and (i) is proved.

To prove the inequality in (ii), assume the contrary, that is, $%
L(I_{q_{1}q_{m}})\geq \pi .$ Then
\begin{equation*}
q_{1}=0\ \mathrm{and\ }q_{m}=\infty ,
\end{equation*}
and so $L(I_{q_{1}q_{m}})=\pi .$ Without loss of generality, by (i), we may
assume that

(e) $\Gamma =\gamma _{0}-I_{q_{1}q_{m}}$ is convex at $q_{1}=0.$

If $\Gamma $ is also convex at $q_{m},$ then $\Gamma $ is a convex curve in $%
S,$ and then by Lemma \ref{con-hemi} (i), $\Gamma $ is a biangle with
vertices $0$ and $\infty $, and then $[0,+\infty ]$ and $\gamma _{0}$ should
be the two straight edges of the biangle $\Gamma ;$ but by (b) this is a
contradiction.

We first assume that $\Gamma $ is not convex at $q_{m}.$ Then we can extend $%
I_{q_{1}q_{m}}$ past $q_{m}$ to obtain a longer line segment $J^{\prime }$
from $q_{1}$ to $q^{\prime }$ so that
\begin{equation}
q^{\prime }\in \gamma _{0}\ \mathrm{and\ }\left( J^{\prime }\backslash
I_{q_{1}q_{m}}\right) \backslash \{q^{\prime }\}\subset T_{\Gamma }.
\label{aaa4}
\end{equation}%
If $q^{\prime }=q_{1},$ then we have $J^{\prime }=C_{1m}\ $and then
\begin{equation}
J^{\prime }\backslash I_{q_{1}q_{m}}=C_{1m}\backslash \lbrack 0,+\infty
]\subset T_{\Gamma }.  \label{aaa2}
\end{equation}
But on the other hand, by (e), $\overline{q_{1}q_{2}}\backslash \{q_{1}\}$
is either contained in the open hemisphere $S\backslash \overline{S^{\prime }%
}$ outside $C_{1m},$ or $\overline{q_{1}q_{2}}\subset C_{1m}.$ Then, in the
case $q^{\prime }=q_{1},$ we have $\overline{q_{1}q_{2}}\backslash
\{q_{1}\}\subset S\backslash \overline{S^{\prime }}$ by (\ref{aaa2}), and
then $q^{\prime }=q_{1}$ has a neighborhood in $J^{\prime }$ that is outside
$T_{\Gamma },$ which contradicts (\ref{aaa2}). Thus, $q^{\prime }\neq q_{1}.$

Then $-J^{\prime }$ and the segment of $\gamma _{0}$ from $q_{1}$ to $%
q^{\prime }$ compose a polygonal Jordan curve $\Gamma ^{\prime },\ $and $%
\Gamma ^{\prime }$ is strictly convex at $q^{\prime },$ since $J^{\prime
}\backslash I_{q_{1}q_{m}}$ intersects $\Gamma $ at $q^{\prime }$ from the
left hand side of $\Gamma ,$ by (\ref{aaa4}), and $\Gamma $ is convex at $%
q^{\prime }(\neq q_{1},q_{m}).$ Hence, by (b) and (e), $\Gamma ^{\prime }$
is locally convex polygonal Jordan curve with the straight edge $-J^{\prime
} $ with $L(J^{\prime })>\pi ,$ which implies that $\Gamma ^{\prime }$ is a
great circle in $S$ by Lemma \ref{con-hemi} (i). But this contradicts that $%
\Gamma ^{\prime }$ is strictly convex at $q^{\prime },$ and we obtain a
contradiction again.

Summarizing the above discussion, we have proved
\begin{equation}
L(I_{q_{1}q_{m}})<\pi ,  \label{aaa5}
\end{equation}
the inequality in (ii). Then, we can write $-I_{q_{1}q_{m}}=-\overline{%
q_{1}q_{m}}=\overline{q_{m}q_{1}},$ and%
\begin{equation*}
\Gamma =\overline{q_{1}q_{2}}+\overline{q_{2}q_{2}}+\dots \overline{%
q_{m-1}q_{m}}+\overline{q_{m}q_{1}}.
\end{equation*}

Now we prove that $\Gamma $ is contained in some open hemisphere of $S.$

If $\{q_{1},q_{m}\}\subset (0,\infty ),$ then by (\ref{3.2+1}), neither $%
q_{2},$ nor $q_{m-1}$ can lie in $C_{1m}$ and thus by (i) $\Gamma $ is
strictly convex at $q_{1}$ or $q_{m}.$

Assume $q_{1}=0$. Then, by (\ref{aaa5}), $q_{m}\in (0,\infty )$ and by (\ref%
{3.2+1})%
\begin{equation}
q_{m-1}\notin C_{1m}.  \label{zz1}
\end{equation}
If $\Gamma $ is not convex at $0,$ then by (i) and by (\ref{3.2+1}), $\Gamma
$ is strictly convex at $q_{m}.$ If $\Gamma $ is straight near $q_{1},$ then
$\overline{q_{1}q_{2}}\subset C_{1m}$ and by (b), $\overline{q_{2}q_{3}}%
\backslash \{q_{2}\}\subset S\backslash \overline{S^{\prime }},$ and then we
can extend $\overline{q_{3}q_{2}}$ past $q_{2}$ to a point $q_{2}^{\prime }$
so that $\overline{q_{3}q_{2}^{\prime }}$ makes sense and $\overline{%
q_{1}q_{2}^{\prime }q_{3}}$ is still strictly convex at $q_{2}^{\prime }.$
Then the curve $\gamma _{0}^{\ast }=\overline{q_{1}q_{2}^{\prime }q_{3}\dots
q_{m}}$ satisfies all the assumptions of $\gamma _{0}$ but the curve $\Gamma
^{\ast }=\gamma _{0}^{\ast }-I_{q_{1}q_{m}}$ is not convex at $q_{1},$ and
thus by (i) and (\ref{zz1}) $\Gamma ^{\ast }$ is strictly convex at $q_{m}.$
But $\Gamma ^{\ast }$ and $\Gamma $ coincide near $q_{m},$ and thus $\Gamma $
is strictly convex at $q_{m}.$ If $q_{m}=\infty ,$ the discussion is similar.

Summarizing the above discussion, we can conclude that $\Gamma $ is strictly
convex at $q_{1}$ or $q_{m},$ and thus, by (b), either $\overline{%
q_{1}q_{2}\dots q_{m}q_{1}}$, or $\overline{q_{m}q_{1}\dots q_{m-1}q_{m}},$
is a locally convex path that is strictly convex at each natural vertices,
and then $\Gamma =\overline{q_{1}q_{2}\dots q_{m}q_{1}}$ is contained in
some open hemisphere of $S.$ The second part of (ii) is proved, and (ii) is
proved completely.

Now, assume $q_{1}=0.$ Then $q_{m}\in (0,+\infty )$ by the assumption and
(ii). Thus, by the assumption, $q_{m-1}$ is either contained in the open
hemisphere $S^{\prime }$ inside the great circle determined by $[0,+\infty
], $ or $q_{m-1}\in S\backslash \overline{S^{\prime }}.$ If $q_{m-1}\in
S^{\prime },$ then $\Gamma $ is not convex at $q_{m}$ and the open interval
of the great circle $C$ determined by $[0,+\infty ]$ from $q_{m}$ to $\infty
$ is contained in $T_{\Gamma },$ and then we can obtain a contradiction as
the above argument involving $J^{\prime }$. Thus $q_{m-1}\in S\backslash
\overline{S^{\prime }},$ i.e. $\Gamma $ is strictly convex at $q_{m},$ and
(iii) is proved.
\end{proof}

\begin{lemma}
\label{con-path}If $q_{1}=0$ and $q_{m}\in (0,+\infty ),$ then for each $%
j=1,\dots ,m-2,$ $L_{j}=\overline{q_{1}q_{j+1}q_{j+2}q_{1}}$ is a generic
convex triangle and for the triangle domain $T_{j}$ inside $L_{j},$
\begin{equation*}
T_{j}\cap T_{k}=\emptyset ,1\leq j<k\leq m-2,
\end{equation*}%
\begin{equation*}
\overline{T_{j}}\cap \lbrack 0,+\infty ]=\{0\},\ \mathrm{for\ }j=1,\dots
,m-3,
\end{equation*}%
\begin{equation*}
\overline{T_{m-2}}\cap \lbrack 0,+\infty ]=\overline{q_{m}q_{1}}.
\end{equation*}
\end{lemma}

\begin{proof}
By (b) in the above proof and by Lemma \ref{L6-1} (ii) and (iii),
\begin{equation*}
\Gamma =\gamma _{0}+\overline{q_{m}q_{1}}=\overline{q_{1}q_{2}}+\overline{%
q_{2}q_{3}}+\dots +\overline{q_{m-1}q_{m}}+\overline{q_{m}q_{1}}
\end{equation*}%
is contained in some open hemisphere of $S$ and is strictly convex at $%
q_{2},\dots ,q_{m}$. Then by Lemma \ref{con-path-hemi} (iii) and Remark \ref%
{u} (1), the conclusion follows.
\end{proof}

\begin{lemma}
\label{con-path1}If $q_{2}$ is contained in the open hemisphere $S^{\prime }$
inside the great circle $C$ determined by $[0,+\infty ]$, $q_{m-1}$ is
contained in $S\backslash \overline{S^{\prime }}$, and if
\begin{equation}
\left\{ q_{1},q_{m}\right\} \subset (0,+\infty ),  \label{3.2+2}
\end{equation}%
then the followings hold.

(i) $\Gamma =\gamma _{0}+\overline{q_{m}q_{1}}=\overline{q_{1}q_{2}\dots
q_{m}q_{1}}$ is a Jordan curve such that $0$ is contained in the domain $%
T_{\Gamma }$ inside $\Gamma $.

(ii) For $q_{0}=0$ and $j=1,\dots ,m-1,$ $L_{j}=\overline{%
q_{0}q_{j}q_{j+1}q_{0}}$ is a generic convex triangle; and for the triangle
domain $T_{j}$ inside $L_{j},$
\begin{equation*}
\overline{T_{j}}\cap \lbrack 0,+\infty ]=\{0\},\ \mathrm{for\ }j=2,\dots
,m-2,
\end{equation*}%
\begin{equation*}
\overline{T_{1}}\cap \lbrack 0,+\infty ]=\overline{q_{0}q_{1}},\ \overline{%
T_{m-1}}\cap \lbrack 0,+\infty ]=\overline{q_{m}q_{0}}.
\end{equation*}
\end{lemma}

\begin{proof}
For any point $q_{1}^{\prime }$ that is in the interior of $\overline{%
q_{1}q_{2}}$ and is sufficient close to $q_{1},$ by the assumption of the
lemma, the polygonal curve%
\begin{equation*}
\gamma _{0}^{\prime }=\overline{0q_{1}^{\prime }q_{2}\dots q_{m}}=\overline{%
0q_{1}^{\prime }}+\dots +\overline{q_{m-1}q_{m}}
\end{equation*}%
is a locally convex Jordan path and satisfies the assumption on $\gamma _{0}$
just with more edges, then applying Lemma \ref{con-path} to $\gamma
_{0}^{\prime }$ and taking $q_{1}^{\prime }\rightarrow q_{1},$ we can obtain
(i) and (ii).
\end{proof}

\begin{lemma}
\label{con-path2}If%
\begin{equation}
\left\{ q_{1},q_{m}\right\} \subset (0,+\infty ),  \label{6--2}
\end{equation}
and%
\begin{equation}
\{q_{2},q_{m-1}\}\subset S^{\prime },  \label{6--1}
\end{equation}%
where $S^{\prime }$ is the open hemisphere inside the great circle
determined by $[0,+\infty ]$, then the curve $\Gamma =\gamma _{0}+\overline{%
q_{m}q_{1}}=\overline{q_{1}q_{2}\dots q_{m}q_{1}}$ is a convex polygonal
Jordan curve, $q_{m}\leq q_{1}$ and $\Gamma $ is strictly convex at $%
q_{j},j=1,2,\dots ,m.$
\end{lemma}

\begin{proof}
We first assume $q_{1}=q_{m}.$ Then
\begin{equation*}
\Gamma =\gamma _{0}=\overline{q_{1}q_{2}}+\dots +\overline{q_{m-1}q_{1}}
\end{equation*}%
is a locally convex Jordan path, and then $m\geq 4$ and, by (a), $\Gamma $
is strictly convex at $q_{2},\dots ,q_{m-1}$, and considering that in this
case, (\ref{3.2+1}) is reduced to (\ref{3.002}), we can conclude by Lemma %
\ref{con-path-hemi} that the closure $\overline{T_{\Gamma }}$ of the domain $%
T_{\Gamma }$ enclosed $\Gamma $ is contained in some open hemisphere of $S.$
On the other hand, by (\ref{3.2+1}), (\ref{6--2}) and (\ref{6--1}) and the
assumption that $q_{1}=q_{m},$ it is easy to see that, if $\overline{%
q_{m-1}q_{1}q_{2}}$ is not convex at $q_{1},$ then $[0,+\infty ]\backslash
\{q_{1}\}$ will be contained in $T_{\Gamma }$, and then $T_{\Gamma }$ can
not be contained in any open hemisphere of $S.$ This is a contradiction.
Thus, by (\ref{6--2}) and (\ref{6--1}), $\Gamma $ is strictly convex at $%
q_{1},$ and then by (a), $\Gamma $ is strictly convex at $q_{j},j=1,2,\dots
,m.$

Now, we assume $q_{1}\neq q_{m}.$ Then by Lemma \ref{L6-1}, $\Gamma $ is
convex at $q_{1}$ or $q_{m}.$ If $q_{1}<q_{m},$ then $\Gamma $ is neither
convex at $q_{1},$ nor at $q_{m}.$ Thus, we must have $q_{m}<q_{1}.$ Then,
by (\ref{6--2}) and (\ref{6--1}), $\Gamma $ is strictly convex at $q_{1}$
and $q_{m},$ and then by (a), $\Gamma $ is strictly convex at $%
q_{j},j=1,2,\dots ,m.$
\end{proof}

\section{Lifting Lemmas for normal mappings\label{ss-6lift}}

In this section, we prove Theorem \ref{pre-key-2} that is used to prove
Theorem \ref{decom}. Theorem \ref{decom} is the second key step to prove the
main theorem.

\begin{lemma}
\label{continue0}Let $f:\overline{\Delta }\rightarrow S$ be a normal mapping
and let $D$ be a polygonal Jordan domain in $S$ such that $f^{-1}$ has a
univalent branch\footnote{%
"univalent branch" always means that the branch is a homeomorphism.} $g$
defined on $D.$ Then $g$ can be extended to be a homeomorphism $\widetilde{g}
$ from $\overline{D}$ onto $\widetilde{g}(\overline{D})$.
\end{lemma}

\begin{proof}
There is a simple and standard way to prove this by Lemma \ref{cov-1}.
\end{proof}

The following result is obvious but useful.

\begin{lemma}
\label{b-in}Let $D_{1}$ and $D_{2}$ be Jordan domains in $\mathbb{C}$ and
let $f:\overline{D_{1}}\rightarrow \overline{D_{2}}$ be a mapping such that $%
f:\overline{D_{1}}\rightarrow f(\overline{D_{1}})$ is a homeomorphism. If $%
f(\partial D_{1})\subset \partial D_{2},$ Then $f(\overline{D_{1}})=%
\overline{D_{2}}$.
\end{lemma}

\begin{lemma}
\label{position}Let $p_{1}$ and $p_{2}$ be two distinct points in $\partial
\Delta ,$ let $\alpha $ be the section\footnote{%
Recall that $\partial \Delta $ is always orientated anticlockwise, and a
section of a curve inherits the orientation of the curve.} of $\partial
\Delta $ from $p_{1}$ to $p_{2}$ and let $\beta $ be a Jordan path in $%
\overline{\Delta }$ from $p_{2}$ to $p_{1}$ such that $\alpha $ and $\beta $
have a common point $p_{0}$ with $p_{0}\neq p_{1},p_{2}.$ Assume that $f:%
\overline{\Delta }\rightarrow S$ is a normal mapping such that the
followings hold.

(a) The curve $\Gamma _{\alpha }=f(z),z\in \alpha ,$ and $\Gamma _{\beta
}=f(z),z\in \beta ,$ are polygonal paths and are both convex at $p_{0}$.

(b) $f$ is regular\footnote{%
This means that $f$ is homeomorphic in a neighborhood of $p_{0}.$} at $%
p_{0}. $

Then $p_{0}$ has a neighborhood $\beta ^{\prime }$ in $\beta $ such that $%
\beta ^{\prime }\subset \alpha \subset \partial \Delta $ and $f$ restricted
to $\beta ^{\prime }$ is a line segment in $S$.
\end{lemma}

\begin{proof}
By the assumption, $p_{0}$ has a neighborhood $\alpha ^{\prime \prime }$ in $%
\alpha $ and a neighborhood $\beta ^{\prime \prime }$ in $\beta ,$
such that the curves $f(\alpha ^{\prime \prime })$ and $f(\beta
^{\prime \prime })$ intersect "tangently".

\end{proof}

\begin{lemma}
\label{m-tri}Let $p_{j}=e^{i\theta _{j}}$ be a number of $m$ distinct points
in $\partial \Delta $ with%
\begin{equation*}
\theta _{1}<\theta _{2}<\dots <\theta _{m}<\theta _{1}+2\pi ,
\end{equation*}%
let $\alpha _{j}$ be the section of $\partial \Delta $ from $p_{j}$ to $%
p_{j+1},j=1,\dots ,m-1,$ let $f:\overline{\Delta }\rightarrow S$ be a normal
mapping and let
\begin{equation}
q_{j}=f(p_{j}),j=1,\dots m.  \label{x}
\end{equation}
Assume that the followings hold.

(a) The section
\begin{equation*}
\Gamma _{0}=f(z),z\in \alpha _{0}=\alpha _{1}+\dots +a_{m}
\end{equation*}%
of the boundary curve $\Gamma _{f}=f(z),z\in \partial \Delta ,$ is a
polygonal Jordan path and each section $\Gamma _{j}=f(\alpha _{j})$ of $%
\Gamma _{0}$ is a natural edge of $\Gamma _{0}$ with
\begin{equation}
L(\Gamma _{j})<\pi ,j=1,\dots ,m.  \label{y}
\end{equation}

(b) $L_{j}=\overline{q_{1}q_{j+1}q_{j+2}q_{1}},j=1,\dots ,m-2,$ are generic
convex triangles in $S$, the triangle domains $T_{j}$ enclosed by $L_{j}$
are disjoint each other, and
\begin{equation*}
\Gamma =\Gamma _{0}+\overline{q_{m}q_{1}}
\end{equation*}%
is a polygonal Jordan curve.

(c) For the domain $T$ enclosed by $\Gamma ,$ $f$ has no branched point in $%
\overline{T}\backslash \overline{q_{m}q_{1}}.$

(d) The boundary curve $\Gamma _{f}=f(z),z\in \partial \Delta ,$ is locally
convex in $T.$

Then, the followings hold true.

(i) $f^{-1}$ has a univalent branch $g$ defined on $\overline{T}$ such that $%
g$ maps $\Gamma _{0}=\overline{q_{1}q_{2}\dots q_{m}}$ onto $\alpha
_{0}=\alpha _{1}+\dots +\alpha _{m-1}$ with $g(q_{j})=p_{j},j=1,2,\dots ,m.$

(ii) If in addition, for some\emph{\ open} interval $\gamma $ of $\overline{%
q_{m}q_{1}},$ $f$ has no branched point in $\gamma $ and $\Gamma
_{f}=f(z),z\in \partial \Delta ,$ is locally convex in $\gamma ,$ then
either $g(\gamma )\subset \partial \Delta $ or $g(\gamma )\subset \Delta .$
\end{lemma}


We first prove the following lemma under the same assumption as that in
Lemma \ref{m-tri}. Note that by (\ref{x}) and (\ref{y}), we can write
\begin{equation*}
\Gamma _{j}=\overline{q_{j}q_{j+1}},j=1,2,\dots ,m-1.
\end{equation*}

\begin{lemma}
\label{c1}(i). $f^{-1}$ has a univalent branch $g_{1}$ defined on $\overline{%
T_{1}},$ such that $g_{1}$ restricted to $\overline{q_{1}q_{2}q_{3}}$ is a
homeomorphisms onto $\alpha _{1}+\alpha _{2}$ with $%
g_{1}(q_{j})=p_{j},j=1,2,3.$

(ii). If $m>3,$ then $\beta _{1}=g_{1}(\overline{q_{1}q_{3}})$ is a Jordan
path in $\overline{\Delta }$ from $p_{1}$ to $p_{3}$ and the interior of $%
\beta _{1}$ is contained in $\Delta .$

(iii). If $m=3$ and for some \emph{open} interval $\gamma $ contained $%
\overline{q_{m}q_{1}},$ $f$ has no branched point in $\gamma $ and $\Gamma
_{f}=f(z),z\in \partial \Delta ,$ is locally convex in $\gamma ,$ then
either $g(\gamma )\subset \partial \Delta $ or $g(\gamma )\subset \Delta .$
\end{lemma}

\begin{proof}
Write $c_{0}=\alpha _{1}+\alpha _{2}$, $\gamma _{0}=\Gamma _{1}+\Gamma _{2}=%
\overline{q_{1}q_{2}q_{3}}$ and $\gamma
_{1}=\overline{q_{1}q_{3}}$ .

Let $v_{1}$ be an interior point of $\gamma _{1}=\overline{q_{1}q_{3}}$ and
let $v=v_{s}=v(s)$, $s\in \lbrack 0,1],$ be a Jordan path that represents
the straight path from $q_{2}$ to $v_{1}$ in the closed triangle domain $%
\overline{T_{1}}$ enclosed by the triangle $L_{1}=\overline{%
q_{1}q_{2}q_{3}q_{1}}=\gamma _{0}-\gamma _{1}$ $.$ Then, by (b)
and Lemma \ref{convextri}, for each $s\in (0,1),$ the polygonal
Jordan path $\gamma _{s}=\overline{q_{1}v_{s}q_{3}}$ in
$\overline{T_{1}}$ is strictly convex at $v_{s}$, and $\gamma
_{s},s\in \lbrack 0,1],$ is a family of curves exhausting the
closed domain $\overline{T_{1}}$ and satisfying the following
condition (e).

(e) For each $s\in (0,1],$ the domain $T_{s}$ inside $\gamma _{0}-\gamma
_{s} $ is a (spherical) quadrilateral domain contained in $T_{1}$ and for
any pair $s_{1},s_{2}\in (0,1]$ with $s_{1}<s_{2},$%
\begin{equation*}
T_{s_{1}}\cup (\gamma _{s_{1}}\backslash \{q_{1},q_{3}\})=\overline{T_{s_{1}}%
}\backslash \gamma _{0}\subset T_{s_{2}}.
\end{equation*}

Since $f$ is normal, by the definition, there exists a point $q_{1}^{\prime
} $ in the interior of $\Gamma _{1}=\overline{q_{1}q_{2}}$ and there exists
a point $q^{\prime }$ in the domain $T_{1}$ such that $f^{-1}$ has a
univalent branch defined on the closure of the triangle domain inside the
triangle $\overline{q_{1}q_{1}^{\prime }q^{\prime }q_{1}}\subset \overline{%
T_{1}}$ and this branch restricted to $\overline{q_{1}q_{1}^{\prime }}$ is a
homeomorphism onto a section of $\alpha _{1}\ $from $p_{1}$ to some interior
point of $\alpha _{1}.$ At $q_{3}$ we can do this similarly. On the other
hand, considering that $\overline{q_{1}q_{2}q_{3}}$ is simple and $f$ is a
normal mappings, by (b) and (c), we can conclude that for each $q_{0}$
contained in the interior\footnote{%
Note that the interior of $\gamma _{0}$ does not intersects $\overline{%
q_{m}q_{1}},$ and thus $f$ has no branched point in the interior of $\gamma
_{0}.$} of $\gamma _{0}=\Gamma _{1}+\Gamma _{2}=\overline{q_{1}q_{2}q_{3}},$
there exists a disk $V_{q_{0}}$ in $S$ such that $f^{-1}$ has a univalent
branch defined on $V_{q_{0}}\cap \overline{T_{1}}$ and this branch maps $%
\gamma _{0}\cap V_{q_{0}}$ onto a section of $c_{0}.$ Summarizing these
discussion, we conclude that, for sufficiently small $\delta >0,$ $\delta $
satisfies the following property:

(f) $f^{-1}$ has a univalent branch $g_{\delta }$ defined on $\overline{%
T_{\delta }}$ with
\begin{equation*}
g_{\delta }(q_{j})=p_{j},j=1,2,3,
\end{equation*}%
$g_{\delta }$ restricted to $\gamma _{0}=\overline{q_{1}q_{2}q_{3}}$ is a
homeomorphism onto $c_{0}=\alpha _{1}+\alpha _{2}$ and $c_{\delta
}=g_{\delta }(\gamma _{\delta })$ is a Jordan path from $p_{1}$ to $p_{3}$
whose interior is contained in $\Delta .$

If $\delta $ satisfies (f) and $\delta <1$, then $c_{0}-c_{\delta }$ is a
Jordan curve, the domain $\widetilde{\Delta }_{\delta }$ inside $%
c_{0}-c_{\delta }$ is a Jordan domain, $\overline{\Delta }\backslash
\widetilde{\Delta }_{\delta }$ is a closed Jordan domain\footnote{%
By (f), $\alpha _{\delta }$ divides $\Delta $ into two Jordan domains.} and $%
f$ restricted $\overline{\Delta }\backslash \widetilde{\Delta }_{\delta }$
is a normal mapping (note that $f(\partial (\overline{\Delta }\backslash
\widetilde{\Delta }_{\delta }))$ is polygonal). In this case, replacing $%
\overline{\Delta }$ by $\overline{\Delta }\backslash \widetilde{\Delta }%
_{\delta },$ $c_{0}$ by $c_{\delta },$ $\gamma _{0}$ by $\gamma _{\delta }$
and applying the above argument once more, we can also prove the following
property for $\delta :$

(g) For each $\delta \in (0,1),$ if $\delta $ satisfies (f), then for
sufficiently small $\varepsilon >0$, $\delta +\varepsilon $ satisfies (f) as
well.

On the other hand, it is clear that, if $\delta $ satisfies (f), then each
positive number $\delta ^{\prime }<\delta $ satisfies (f) as well. Thus, for%
\begin{equation*}
\delta _{0}=\sup \{\delta \in (0,1);\ \delta \ \mathrm{satisfies}\ \mathrm{%
(f)}\},
\end{equation*}%
we have

(h) Each $\delta \in (0,\delta _{0})$ satisfies (f).

To show $\delta _{0}=1,$ we first show that $\delta _{0}$ satisfies (f) if $%
\delta _{0}<1$.

By (e), (f) and (h), $f^{-1}$ has a univalent branch $\widetilde{g}_{\delta
_{0}}$ defined on $T_{\delta _{0}}\cup \gamma _{0}$. By Lemma \ref{continue0}%
, $\widetilde{g}_{\delta _{0}}$ can be extended to be a homeomorphism $%
g_{\delta _{0}}$ defined on $\overline{T_{\delta _{0}}}.$ Thus, $\gamma
_{\delta _{0}}$ has a lift $c_{\delta _{0}}=g_{\delta _{0}}(w),w\in \gamma
_{\delta _{0}},$ by $f,$ and $c_{\delta _{0}}$ is a Jordan path from $p_{1}$
to $p_{3}$ in $\overline{\Delta }.$ Let $\overline{\widetilde{\Delta }%
_{\delta _{0}}}=g_{\delta _{0}}(\overline{T_{\delta _{0}}}),$ then $f$
restricted to $\overline{\widetilde{\Delta }_{\delta _{0}}}$ is a
homeomorphism onto $\overline{T_{\delta _{0}}}$, and maps $c_{\delta _{0}}$
onto $\gamma _{\delta _{0}}$.

Now, we show that the following hold.

(j) If $\delta _{0}<1,$ then the interior of $c_{\delta _{0}}$ is contained
in $\Delta .$

Assume $\delta _{0}<1$ and let $p_{0}\in c_{\delta _{0}}$ be any interior
point of $c_{\delta _{0}}\ $with $p_{0}\in \partial \Delta .$ Then $p_{0}\in
\alpha :=\left( \partial \Delta \right) \backslash c_{0}$ and $f(p_{0})$ is
in the interior of $\gamma _{\delta _{0}},$ and then $f(p_{0})\in T_{1}.$
Thus, by (d), the curves $\Gamma _{\alpha }=f(z),z\in \alpha =\left(
\partial \Delta \right) \backslash c_{0},$ and $\Gamma _{\beta }=f(z),z\in
\beta =c_{\delta _{0}},$ are both convex at $p_{0}$ (note that $\Gamma
_{\beta }$ is the path $\gamma _{\delta _{0}})$. Therefore, by (c) and Lemma %
\ref{position}, $p_{0}$ has a neighborhood $\beta ^{\prime }$ in $\beta
=c_{\delta _{0}}$ such that $\beta ^{\prime }\subset \alpha =\left( \partial
\Delta \right) \backslash c_{0}$ and $f(\beta ^{\prime })$ is straight. But
then, with a continuation argument, we can prove that the whole of $%
f(c_{\delta _{0}})$ is also straight, which contradicts the fact that $%
\gamma _{\delta _{0}}=f(c_{\delta _{0}})$ is not straight if $\delta _{0}<1.$
Thus, the interior of $c_{\delta _{0}}$ must be in $\Delta $ and (j) is
proved.

(j) implies that $\delta _{0}$ satisfies (f) if $\delta _{0}<1.$ This, with
(g), implies that if $\delta _{0}<1,$ then $\delta _{0}+\varepsilon $
satisfies (f) for sufficiently small $\varepsilon >0$. This contradicts the
definition of $\delta _{0}.$ Thus we have proved $\delta _{0}=1.$

Now that $\delta _{0}=1,$ by (e)--(h), $f^{-1}$ has a univalent branch $%
\widetilde{g}_{1}$ defined on $T_{1}\cup \gamma _{0},$ and by Lemma \ref%
{continue0}, $\widetilde{g}_{1}$ can be extended to be a homeomorphism $%
g_{1} $ defined on $\overline{T_{1}}$. Thus, (i) holds. (ii) can be proved
as the proof of (j), by (i) and Lemma \ref{position}, and (iii) can be
proved similarly.
\end{proof}

\begin{proof}[Proof of Lemma \protect\ref{m-tri}]
If $m=3,$ then Lemma \ref{m-tri} follows from (i) and (iii) of Lemma \ref{c1}%
. So we may assume $m\geq 4.$ But, without loss of generality, we complete
the proof only for the case $m=4.$

We continue the proof of Lemma \ref{c1}. Let $\beta _{1}=g_{1}(\overline{%
q_{1}q_{3}}).$ Then by Lemma \ref{c1} (ii), $\beta _{1}^{\circ }\subset
\Delta ,$ where $\beta _{1}^{\circ }$ is the interior of $\beta _{1}.$ Then $%
\beta _{1}$ divides $\Delta $ into two Jordan domains. We denote by $\Delta
_{1}$ the component of $\Delta \backslash \beta _{1}$ that is on the left
hand side of $\beta _{1},$ i.e. $\Delta _{1}=\Delta \backslash g_{1}(%
\overline{T_{1}}).$

Then, by the assumption $m=4,$ $\Delta _{1}$ is enclosed by
\begin{equation*}
\beta _{1}+\alpha _{3}+\alpha ^{\ast }
\end{equation*}%
where $\alpha ^{\ast }$ is the section of $\partial \Delta $ from $%
p_{m}=p_{4}$ to $p_{1}.$

Again by (a)--(d) and Lemma \ref{c1} (i), $f^{-1}$ has a univalent branch $%
g_{2}$ defined on $\overline{T_{2}}$ such that $g_{2}:\overline{T_{2}}%
\rightarrow g_{2}(\overline{T_{2}})$ is a homeomorphism, restricted to $%
\overline{q_{1}q_{3}q_{4}}$ is a homeomorphism onto $\beta _{1}+\alpha _{3}$
and
\begin{equation*}
g_{2}(q_{j})=p_{j},j=1,3,4.
\end{equation*}

Since $f$ has no branched point in $\overline{T}\backslash \overline{%
q_{4}q_{1}}$ (note that $m=4),$ $f$ has no branched point on $\overline{%
q_{3}q_{1}}\backslash \{q_{1}\}.$ Thus, $g_{1}$ and $g_{2}$ must be
identical on $\overline{q_{3}q_{1}}.$ Then $g_{1}$ and $g_{2}$ make up a
univalent branch $g$ of $f^{-1},$ such that $g:\overline{T}=\overline{%
T_{1}\cup T_{2}}\rightarrow g(\overline{T})$ is a homeomorphism with $g(%
\overline{q_{1}q_{2}q_{3}q_{4}})=\alpha _{1}+\alpha _{2}+\alpha _{3}$. (i)
is proved.

(ii) can be proved as the proof of (j). This completes the proof of Lemma %
\ref{m-tri}.
\end{proof}

\begin{remark}
Lemma \ref{m-tri} implies an interesting proposition: let $f:\overline{%
\Delta }\rightarrow \mathbb{C}$ be an open mapping that is orientation
preserved and is locally homeomorphism. Then, $f$ is a homeomorphism,
provided that the boundary curve $\Gamma _{f}=f(z),z\in \partial \Delta ,$
is locally convex.

Here \textquotedblleft locally convex\textquotedblright\ means that when $z$
goes around $\partial \Delta $ anticlockwise, $f(z)$ always go straight or
turn left. For example, if we assume that the curve $\Gamma _{f}$ is smooth
and is locally straight, or locally on the left hand side of its tangent
line, then $\Gamma _{f}$ is locally convex.

For later use, we only prove this in a special version for normal mappings,
which is the following corollary.
\end{remark}

\begin{corollary}
\label{bv-1}Let $\alpha _{0}\ $be a section of $\partial \Delta $ from $%
p_{1} $ to $p_{m}$ with
\begin{equation}
p_{1}\neq p_{m},  \label{5.1}
\end{equation}%
let $f:\overline{\Delta }\rightarrow S$ be a normal mapping such that the
section $\gamma _{0}=f(z),z\in \alpha _{0},$ is a closed Jordan path that
has the natural partition%
\begin{equation}
\gamma _{0}=\overline{q_{1}q_{2}}+\overline{q_{2}q_{3}}+\dots +\overline{%
q_{m-1}q_{1}},  \label{7-5}
\end{equation}%
with $q_{1}=f(p_{1})=f(p_{m})$ and%
\begin{equation*}
\{q_{2},\dots ,q_{m-1}\}\cap E=\emptyset .
\end{equation*}%
Assume that for the domain $T_{\gamma _{0}}\subset S$ enclosed by $\gamma
_{0},$ the boundary curve $\Gamma _{f}=f(z),z\in \partial \Delta ,$ is
locally convex in $\overline{T_{\gamma _{0}}}\backslash \{q_{1}\}.$ Then $f$
has a branched point in $\overline{T_{\gamma _{0}}}\backslash \{q_{1}\}.$
\end{corollary}

\begin{proof}
Since $\gamma _{0}$ is a closed Jordan path, by (\ref{7-5}) we have\footnote{%
Note that (\ref{7.5}) makes sense iff each term on the right hand side has
spherical length $<\pi .$} $m\geq 4.$ Since $\Gamma _{f}$ is locally convex
in $\overline{T_{\gamma _{0}}}\backslash \{q_{1}\},$ $\gamma _{0}$ is a
locally convex path, and then by Lemma \ref{con-path-hemi} and Remark \ref{u}
(1), for each $j=1,\dots ,m-3,$ $L_{j}=\overline{q_{1}q_{j+1}q_{j+2}q_{1}}$
is a generic convex triangle such that the triangle domains $T_{j}$ enclosed
by $L_{j}$ are disjoint each other and $\overline{T_{\gamma _{0}}}=\cup
_{j=1}^{m-3}\overline{T_{j}}$.

Assume $f$ has no branched point in $\overline{T_{\gamma _{0}}}\backslash
\{q\}$. Then, Lemma \ref{m-tri} applies, i.e. $f^{-1}$ has a univalent
branch $g$ defined on $\overline{T_{\gamma _{0}}}$ such that $g$ restricted
to $\overline{q_{1}q_{2}\dots q_{m-1}}$ is a homeomorphism onto a section $%
\alpha _{0}^{\prime }$ of $\alpha _{0}$ from $p_{1}$ to some point $%
p_{m-1}^{\prime }\in \alpha _{0}^{\circ },$ here $\alpha _{0}^{\circ }$ is
the interior $\alpha _{0}\backslash \{p_{1},p_{2}\}$ of $\alpha _{0}.$

Let $\alpha _{0}^{\prime \prime }$ be the section of $\alpha _{0}$ from $%
p_{m-1}^{\prime }$ to $p_{m},$ then, by the assumption, it is clear that $f$
maps $\alpha _{0}^{\prime \prime }$ homeomorphically onto $\overline{%
q_{m-1}q_{m}}=\overline{q_{m-1}q_{1}}.$ Since $f$ has no branched point on $%
\overline{T_{\gamma _{0}}}\backslash \{q_{1}\}$ and $q_{m-1}\in \overline{%
T_{\gamma _{0}}}\backslash \{q_{1}\},$ after an argument of uniqueness of
the lifting, we have $g(\overline{q_{m-1}q_{1}})=\alpha _{0}^{\prime \prime
}\subset \partial \Delta .$ Then we have $g(\gamma _{0})\subset \partial
\Delta ,$ and then $g(\gamma _{0})=\partial \Delta $ by Lemma \ref{b-in}.
Thus $f$ is a homeomorphism, and $\gamma _{0}$ is the whole curve $\Gamma
_{f},$ which contradicts (\ref{5.1}). The proof is completed.
\end{proof}

In the rest of this section, let $p_{j}=e^{i\theta _{j}}$ be $m$ distinct
points in $\partial \Delta ,j=1,\dots ,m,$ with%
\begin{equation*}
m\geq 3\ \mathrm{and\ }\theta _{1}<\theta _{2}<\dots <\theta _{m}\leq \theta
_{1}+2\pi ,
\end{equation*}%
let $\alpha _{j}$ be the section of $\partial \Delta $ from $p_{j}$ to $%
p_{j+1},j=1,\dots ,m-1,$ and let
\begin{equation*}
\alpha _{0}=\alpha _{1}+\alpha _{1}+\dots +\alpha _{m-1}.
\end{equation*}

\begin{definition}
\label{FM}The family $\mathcal{F}_{m}$ is defined to be the family of all
normal mappings $f:\overline{\Delta }\rightarrow S$ that satisfies all the
following conditions (A)--(E).

(A) The section $\gamma _{0}=f(z),z\in \alpha _{0},$ of the boundary curve $%
\Gamma _{f}=f(z),z\in \partial \Delta ,$ is a Jordan path.

(B) $\gamma _{0}$ has the natural partition
\begin{equation}
\gamma _{0}=\overline{q_{1}q_{2}}+\overline{q_{2}q_{3}}+\dots +\overline{%
q_{m-1}q_{m}},  \label{7-6}
\end{equation}%
with%
\begin{equation}
\gamma _{0}\cap \lbrack 0,+\infty ]=\{q_{1},q_{m}\},  \label{7-3}
\end{equation}%
where, $q_{j}=f(p_{j}),$ $j=1,\dots ,m,$ and $\overline{q_{j}q_{j+1}}$ is
the section%
\begin{equation*}
\Gamma _{j}=f(z),z\in \alpha _{j},j=1,\dots ,m-1.
\end{equation*}

(C) The boundary curve $\Gamma _{f}=f(z),z\in \partial \Delta ,$ is locally
convex in $S\backslash \{0,\infty \}.$

(D) $f$ has no ramification point in $\overline{\Delta }.$

(E) $f(\Delta )\cap \lbrack 0,+\infty ]=\emptyset .$
\end{definition}

Each $f\in \mathcal{F}_{m}$ will be endowed with all the notations in the
definition. By (A) and (B) the curve%
\begin{equation}
\Gamma =\gamma _{0}+\overline{q_{m}q_{1}}  \label{7.6}
\end{equation}
is a polygonal Jordan curve$.$ Here it is permitted that $q_{1}=q_{m},$ and
in this case $\Gamma =\gamma _{0}.$

Note that by (A), (B), (C) and Definition \ref{convex}, we have

(F) $\gamma _{0}$ is a locally convex polygonal Jordan path that is strictly
convex at $q_{2},\dots ,q_{m-1}.$

Then by (B) and Lemma \ref{L6-1} (ii), $\overline{q_{m}q_{1}}$ in (\ref{7.6}%
) makes sense. On the other hand, if $q_{1}=q_{m},$ then, by (A) and (\ref%
{7-6}), $m\geq 4.$ Therefore, by Lemma \ref{con-path-hemi} (for the case $%
q_{1}=q_{m}$ here) and Lemma \ref{L6-1} (for the case $q_{1}\neq q_{m})$ the
following holds true.

(G) The closure $\overline{T_{\Gamma }}$ of the domain $T_{\Gamma }$
enclosed by $\Gamma =\gamma _{0}+\overline{q_{m}q_{1}}$ is contained in some
open hemisphere of $S.$

\begin{theorem}
\label{pre-key-2}Let $f\in \mathcal{F}_{m}$ and denote by $T_{\Gamma }$ the
domain enclosed by $\Gamma .$ Then the followings hold.

(i) The restriction $f|_{\Delta }:\Delta \rightarrow T_{\Gamma }\backslash
\lbrack 0,+\infty ]$ is a homeomorphism. (ii) $f(\overline{\Delta })$ is
contained in some open hemisphere of $S$.

(iii) For $\alpha _{0}^{\circ }=\alpha _{0}\backslash \{p_{1},p_{m}\},$%
\begin{equation}
f(\alpha _{0}^{\circ })\cap \lbrack 0,+\infty ]=\emptyset ,  \label{ooo}
\end{equation}%
\begin{equation}
f(\left( \partial \Delta \right) \backslash \alpha _{0}^{\circ })\subset
\lbrack 0,+\infty ],  \label{o}
\end{equation}%
and
\begin{equation}
L(f,\alpha _{0})>L(f,\left( \partial \Delta \right) \backslash \alpha _{0}).
\label{oo}
\end{equation}
\end{theorem}

\begin{proof}
By (A) and (B), it is clear that (\ref{ooo}) holds true. To complete the
remained proof, it suffices to consider three cases.

\noindent \textbf{Case 1. }%
\begin{equation}
0=q_{1}\leq q_{m}<+\infty .  \label{7.5}
\end{equation}

If $q_{1}=q_{m}=0,$ then by (G) we have%
\begin{equation*}
\overline{T_{\Gamma }}\cap \lbrack 0,+\infty ]=\{0\},
\end{equation*}%
and then by (C), (D) and Corollary \ref{bv-1} we have $p_{1}=p_{m},$ and
then, by (A), $f$ maps $\alpha _{0}=\partial \Delta $ homeomorphically onto
the closed Jordan curve $\Gamma =\gamma _{0},$ and since $f$ is normal we
conclude that $f:\Delta \rightarrow T_{\Gamma }=T_{\Gamma }\backslash
\lbrack 0,+\infty ]$ is a homeomorphism, and other conclusions of Theorem %
\ref{pre-key-2} is trivially hold with $\alpha _{0}=\partial \Delta $, by
(G).

If $q_{1}\neq q_{m},$ i.e. $q_{1}=0$ and $q_{m}\in (0,+\infty ),$ then by
Lemma \ref{con-path} the triangles $L_{j}=\overline{q_{1}q_{j+1}q_{j+2}q_{1}}
$ are generic convex for $j=1,2,\dots ,m-2,$ the domains $T_{j}$ enclosed by
$L_{j}$ are disjoint each other and for the domain $T_{\Gamma }$ enclosed by
$\Gamma =\gamma _{0}+\overline{q_{m}q_{1}}$ we have%
\begin{equation*}
\overline{T_{\Gamma }}=\cup _{j}^{m-2}\overline{T_{j}},\ 0\notin \overline{%
T_{\Gamma }},
\end{equation*}%
and
\begin{equation*}
\overline{T_{\Gamma }}\backslash \overline{q_{m}q_{1}}=\overline{T_{\Gamma }}%
\backslash \overline{q_{m}0}\subset S\backslash \{0,\infty \}.
\end{equation*}

Then, by (C) and (D), Lemma \ref{m-tri} applies, and then, $f^{-1}$ has a
univalent branch $g$ defined on $\overline{T_{\Gamma }}$ such that $g$
restricted to $\gamma _{0}$ is a homeomorphism onto $\alpha _{0}.$ Let $%
\alpha ^{\ast }=g(\overline{q_{m}q_{1}}).$ Then $\alpha ^{\ast }$ is a
Jordan path in $\overline{\Delta }$ from $p_{m}$ to $p_{1}$ and by (E) we
have $\alpha ^{\ast }\subset \partial \Delta ,$ and then $\alpha ^{\ast
}=\left( \partial \Delta \right) \backslash \alpha _{0}^{\circ }$. This
implies that $g(\partial T_{\Gamma })=\partial \Delta ,$ and then $f:%
\overline{\Delta }\rightarrow \overline{T_{\Gamma }}$ and $f:\Delta
\rightarrow T_{\Gamma }=T_{\Gamma }\backslash \lbrack 0,+\infty ]$ are
homeomorphisms, with
\begin{equation*}
f(\left( \partial \Delta \right) \backslash \alpha _{0}^{\circ })=f(\alpha
^{\ast })=\overline{q_{m}q_{1}}\subset \lbrack 0,+\infty ],
\end{equation*}%
and%
\begin{equation*}
L(f,\alpha _{0})=L(\gamma _{0})>L(\overline{q_{m}q_{1}})=L(f,\left( \partial
\Delta \right) \backslash \alpha _{0}).
\end{equation*}%
Then, by (G), the proof is complete for Case 1.

\noindent \textbf{Case 2. }%
\begin{equation}
\{q_{1},q_{m}\}\subset (0,+\infty ),  \label{7.0}
\end{equation}%
and%
\begin{equation}
\{q_{2},q_{m-1}\}\subset S^{\prime },  \label{7-1}
\end{equation}%
where $S^{\prime }$ is the open hemisphere inside the great circle
determined by $[0,+\infty ].$

By (A), (B), (C), (\ref{7.0}), (\ref{7-1}) and Lemma \ref{con-path2}, we have

(H) $\Gamma =\gamma _{0}+\overline{q_{m}q_{1}}$ is a convex Jordan curve
that is strictly convex at all vertices $q_{1},q_{2},\dots ,q_{m}.$

We first assume $q_{1}=q_{m}.$ Then the closed curve $\Gamma =\gamma _{0}=%
\overline{q_{1}q_{2}}+\overline{q_{2}q_{3}}+\dots +\overline{q_{m-1}q_{1}}$
is strictly convex at all its vertices $q_{1},\dots ,q_{m-1},$ and, by Lemma %
\ref{con-hemi} (ii)
\begin{equation*}
\overline{T_{\gamma _{0}}}\backslash \{q_{1}\}\subset S^{\prime },
\end{equation*}%
where $T_{\gamma _{0}}$ is the domain enclosed by $\gamma _{0}.$ Then by
(C), (D) and Corollary \ref{bv-1}, $\alpha _{0}=\alpha _{1}+\alpha
_{1}+\dots +\alpha _{m-1}=\partial \Delta ,$ i.e. $p_{1}=p_{m}.$ This
implies that $f$ restricted to $\partial \Delta $ is a homeomorphism onto $%
\gamma _{0}$ and then $f$ is a homeomorphism, and the other conclusions are
trivial in this setting.

Now, we assume $q_{1}\neq q_{m}.$ Then $\Gamma =\gamma _{0}+\overline{%
q_{m}q_{1}}$ has the following natural partition
\begin{equation*}
\Gamma =\overline{q_{1}q_{2}}+\overline{q_{2}q_{3}}+\dots +\overline{%
q_{m-1}q_{m}}+\overline{q_{m}q_{1}},
\end{equation*}%
and by (H) and Lemma \ref{con-hemi} (ii), $\overline{T_{\Gamma }}\subset
S^{\prime }\cup \overline{q_{m}q_{1}},$ which, with (\ref{7.0}), implies
that
\begin{equation}
\overline{T_{\Gamma }}\cap \{0,\infty \}=\emptyset \ \mathrm{and\ }\overline{%
T_{\Gamma }}\cap \lbrack 0,+\infty ]=\overline{q_{m}q_{1}}.  \label{7-2}
\end{equation}

Then again by (H), the triangles $L_{j}=\overline{q_{1}q_{j+1}q_{j+2}q_{1}}$
are all generic convex triangles and the domains $T_{j}$ enclosed by $L_{j}$
are disjoint each other, and $\overline{T_{\Gamma }}=\cup _{j}^{m-2}%
\overline{T_{j}}.$ By (C) and (\ref{7-2}), $\Gamma _{f}=f(z),z\in \partial
\Delta ,$ is locally convex in $\overline{T_{\Gamma }}$ and by (D), $f$ has
no branched point in $\overline{T_{\Gamma }}.$ Thus, by Lemma \ref{m-tri}, $%
f^{-1}$ has a univalent branch $g$ defined on $\overline{T_{\Gamma }}$ such
that $g$ maps $\gamma _{0}=\overline{q_{1}q_{2}\dots q_{m}}$ onto $\alpha
_{0}$.

Let $\alpha ^{\ast }=g(\overline{q_{m}q_{1}}).$ Then $\alpha ^{\ast }$ is a
Jordan path in $\overline{\Delta }$ from $p_{m}$ to $p_{1}.$ By (E), we have
$\alpha ^{\ast }\subset \partial \Delta ,$ and then we have $\alpha ^{\ast
}=\partial \Delta \backslash \alpha _{0}^{\circ }$ and $g(\partial T_{\Gamma
})=g(\Gamma )\subset \partial \Delta $, which, with Lemma \ref{b-in},
implies that $g(\overline{T_{\Gamma }})=\overline{\Delta },$ and then $f:%
\overline{\Delta }\rightarrow \overline{T_{\Gamma }}$ is a homeomorphism.

Thus $f:\Delta \rightarrow T_{\Gamma }=T_{\Gamma }\backslash \lbrack
0,+\infty ]$ is a homeomorphism%
\begin{equation*}
f(\left( \partial \Delta \right) \backslash \alpha _{0})=\overline{q_{m}q_{1}%
}\subset \lbrack 0,+\infty ],
\end{equation*}%
and, by the fact that $L(\gamma _{0})>L(\overline{q_{m}q_{1}}),$ we have%
\begin{equation*}
L(f,\alpha _{0})>L(f,\alpha ^{\ast })=L(f,\left( \partial \Delta \right)
\backslash \alpha _{0}).
\end{equation*}%
Then, by (G), The proof is complete for Case 2.

\noindent \textbf{Case 3.}%
\begin{equation}
\{q_{1},q_{m}\}\subset (0,+\infty ),  \label{7.1}
\end{equation}%
and
\begin{equation}
q_{2}\in S^{\prime },q_{m-1}\in S\backslash \overline{S^{\prime }}.
\label{7.2}
\end{equation}


By (A), (B), (C), (\ref{7.1}) and (\ref{7.2}), Lemma \ref{con-path1} apply
to $\gamma _{0}$, and then we have the following.

(I) $0$ is contained in the domain $T_{\Gamma }$, $L_{j}=\overline{%
0q_{j}q_{j+1}0}$ is a generic convex triangle for $j=1,\dots ,m-1$; and for
the triangle domain $T_{j}$ enclosed by $L_{j},$
\begin{equation*}
\overline{T_{j}}\cap \lbrack 0,+\infty ]=\{0\},\ \mathrm{for\ }j=2,\dots
,m-2,
\end{equation*}%
\begin{equation*}
\overline{T_{1}}\cap \lbrack 0,+\infty ]=\overline{0q_{1}},\ \overline{%
T_{m-1}}\cap \lbrack 0,+\infty ]=\overline{q_{m}0}.
\end{equation*}

By (I), we can extend $\overline{q_{1}0}$ past $0$ to some point $q^{\prime
}\in \Gamma $ such that the open line segment $\overline{q^{\prime }0}%
^{\circ }$ is contained in $T_{\Gamma }$ (note that by (G) the notations $%
\overline{q^{\prime }0}$ and $\overline{q^{\prime }q_{1}}=\overline{%
q^{\prime }0}+\overline{0q_{1}}$ make sense, i.e. $d(q^{\prime },q_{1})<\pi
).$ By (G) and (I), $\overline{q^{\prime }q_{1}}$ divides $T_{\Gamma }$ into
two polygonal Jordan domains $T_{1}^{\ast }$ and $T_{2}^{\ast }$ with $%
q_{2}\in \partial T_{1}^{\ast }$ and $q_{m-1}\in \partial T_{2}^{\ast }$,
both $T_{1}^{\ast }$ and $T_{2}^{\ast }$ are strictly convex at $q^{\prime
}, $ $T_{1}^{\ast }$ is on the left hand side of $\overline{q^{\prime }q_{1}}
$ and $T_{2}^{\ast }$ is on the right hand side of $\overline{q^{\prime
}q_{1}},$ $q_{1}$ is a strictly convex vertex of of $T_{1}^{\ast }$ and $%
q_{m}$ is a strictly convex vertex of $T_{2}^{\ast }$. Thus, by (F), both $%
T_{1}^{\ast }$ and $T_{2}^{\ast }$ are polygonal convex Jordan domains.

Considering that $q^{\prime }$, $q_{1}$ and $q_{2}$ are strictly convex
vertices of $T_{1}^{\ast },$ by Lemma \ref{con-hemi}, we have
\begin{equation}
\overline{T_{1}^{\ast }}\backslash \overline{q^{\prime }q_{1}}\subset
S^{\prime }.  \label{7-4}
\end{equation}

Let $\gamma _{1}$ be the section of $\gamma _{0}$ from $q_{1}$ to $q^{\prime
}$, $p^{\prime }$ the unique point in $\alpha _{0}$ such that $f(p^{\prime
})=q^{\prime }$, $\alpha _{0}^{1}$ the section of $\alpha _{0}$ from $p_{1}$
to $p^{\prime }$ and let $\alpha _{0}^{2}$ be the section of $\alpha _{0}$
from $p^{\prime }$ to $p_{m}.$ We may assume $q^{\prime }\in \overline{%
q_{s}q_{s+1}}^{\circ }$ (in the case $q^{\prime }=q_{s}$ or $q_{s+1},$ the
proof is the same). Then
\begin{equation*}
\partial T_{1}^{\ast }=\gamma _{1}+\overline{q^{\prime }q_{1}}=\overline{%
q_{1}q_{2}}+\dots +\overline{q_{s}q^{\prime }}+\overline{q^{\prime }q_{1}},
\end{equation*}%
and $T_{1}^{\ast }$ is strictly convex at $q_{1},q_{2},\dots
,q_{s},q^{\prime }$, and then $\overline{q_{1}q_{2}q_{3}q_{1}},\dots ,%
\overline{q_{1}q_{s-1}q_{s}q_{1}},$ $\overline{q_{1}q_{s}q^{\prime }q_{1}}$
are generic convex triangles that triangulate $\overline{T_{1}^{\ast }}$.
Hence, by (C), (D), (\ref{7-4}), Lemma \ref{m-tri} applies to $\overline{%
T_{1}^{\ast }}$, and then $f^{-1}$ has a univalent branch $g_{1}$ defined on
$\overline{T_{1}^{\ast }}$ such that $g_{1}$ restricted to $\gamma _{1}$ is
a homeomorphism onto $\alpha _{0}^{1}$ with
\begin{equation*}
g_{1}(q_{1})=p_{1}\text{\textrm{and\ }}g_{1}(q^{\prime })=p^{\prime }.
\end{equation*}

For the same reason, $f^{-1}$ has a univalent branch $g_{2}$ defined on $%
\overline{T_{2}^{\ast }}$ such that $g_{2}$ restricted to $\gamma _{2}=%
\overline{q^{\prime }q_{s+1}}+\dots +\overline{q_{m-1}q_{m}}$ is a
homeomorphism onto $\alpha _{0}^{2}$ with
\begin{equation*}
g_{2}(q_{m})=p_{m}\mathrm{\ and\ }g_{2}(q^{\prime })=p^{\prime }.
\end{equation*}%
Considering that $f$ has no branched point in $S$ and $g_{1}(q^{\prime
})=g_{2}(q^{\prime }).$ We have
\begin{equation}
g_{1}(w)=q_{2}(w),w\in \overline{q^{\prime }0},  \label{7.3}
\end{equation}%
and we denote by $\alpha =g_{1}(\overline{q^{\prime }0})=g_{2}(\overline{%
q^{\prime }0}).$ By (E) we have%
\begin{equation*}
g_{1}(0)=g_{2}(0)\in \partial \Delta ,
\end{equation*}%
and thus, the initial and terminal points of $\alpha $, the points $%
p^{\prime }$ and $g_{1}(0)=g_{2}(0),$ are contained in $\partial \Delta .$

Then we can glue $g_{1}$ and $g_{2}$ along $\overline{q^{\prime }0}$ to be a
multivalent function $G$ such that $G$ restricted to $T_{\Gamma }\backslash
\left( \overline{0q_{1}}\cap \overline{0q_{m}}\right) $ is a homeomorphism
and restricted to $\overline{T_{j}^{\ast }}$ is the homeomorphism $g_{j}$, $%
j=1,2.$ Then, it is clear that the interior of $\alpha =g_{1}(\overline{%
q^{\prime }0})=g_{2}(\overline{q^{\prime }0})$ is contained in $\Delta ,$
and thus $\alpha $ divides $\Delta $ into two Jordan domains $\Delta _{1}$
and $\Delta _{2},$ and we assume $\Delta _{1}$ is on the left hand side of $%
\alpha .$

Let $\alpha ^{\prime }=g_{2}(\overline{q_{m}0})$ and $\alpha ^{\prime \prime
}=g_{1}(\overline{0q_{1}}).$ Then by (E), $\alpha ^{\prime }$ is a section
of $\partial \Delta $ from $p_{m}$ to $g_{2}(0)$ and $\alpha ^{\prime \prime
}$ is a section of $\partial \Delta $ from $g_{1}(0)=g_{2}(0)$ to $p_{1}$,
since $g_{1}(q_{1})=p_{1}$ and $g_{2}(q_{m})=p_{m}.$ Thus, we can conclude
that $f$ maps $\alpha _{0},\alpha ^{\prime },\alpha ^{\prime \prime }$
homeomorphically onto $\gamma _{0},\overline{q_{m}0},\overline{0q_{1}},$
respectively, and
\begin{equation*}
\partial \Delta =\alpha _{0}+\alpha ^{\prime }+\alpha ^{\prime \prime }\text{%
\textrm{with\ }}\alpha _{0}^{\circ }\cap \left( \alpha ^{\prime }+\alpha
^{\prime \prime }\right) =\emptyset .
\end{equation*}%
This implies that $f$ maps $\Delta $ homeomorphically onto $T_{\Gamma
}\backslash \overline{0q_{1}}=T_{\Gamma }\backslash \overline{0q_{m}},$
since $f$ is normal.

On the other hand, it is clear that $L(\gamma _{1})>L(\overline{0q_{1}})$
and $L(\gamma _{2})>L(\overline{q_{m}0}).$ Thus, we have
\begin{eqnarray*}
L(f,\alpha _{0}) &=&L(f(\alpha _{0}))=L(\gamma _{0})=L(\gamma _{1})+L(\gamma
_{2})>L(\overline{0q_{1}})+L(\overline{q_{m}0}) \\
&=&L(f,\alpha ^{\prime \prime }+\alpha ^{\prime })=L(f,(\partial \Delta
)\backslash \alpha _{0}).
\end{eqnarray*}%
By (G), $f(\overline{\Delta })\subset \overline{T_{\Gamma }}$ is contained
in some open hemisphere of $S$ and it is clear that
\begin{equation*}
f((\partial \Delta )\backslash \alpha _{0}^{\circ })=\overline{q_{m}0}\cup
\overline{0q_{1}}\subset \lbrack 0,+\infty ].
\end{equation*}%
This completes the proof for Case 3, and we have finally proved Theorem \ref%
{pre-key-2}.
\end{proof}

\section{Cutting Riemann surfaces along $[0,+\infty ]$ \label{ss-p1}}

In this section we prove the following theorem, which is the second key step
to prove the main theorem in Section \ref{ss-p2} and is also used to prove
Theorem \ref{fat}. Recall that we denote by $[0,+\infty ]$ the line segment
in $S$ from $0$ to $\infty $ that passes through $1.$

\begin{theorem}
\label{decom}Let $f:\overline{\Delta }\rightarrow S$ be a normal mapping and
assume that the followings hold.

(a) Each natural edge of the boundary curve $\Gamma _{f}=f(z),z\in \partial
\Delta ,$ has length strictly less than $\pi $.

(b) $\Gamma _{f}=f(z),z\in \partial \Delta ,$ is locally convex in\footnote{%
See Definition \ref{convex1}.} $S\backslash E,E=\{0,1,\infty \}.$

(c) $f$ has no branched point in $S\backslash E.$

(d) $\Gamma _{f}\cap \lbrack 0,+\infty ]$ contains at most finitely many
points.

Then, in the case $\Delta \cap f^{-1}\left( [0,+\infty ]\right) =\emptyset ,$
$f(\overline{\Delta })$ is contained in some open hemisphere of $S,$ $f:%
\overline{\Delta }\rightarrow f(\overline{\Delta })$ is a homeomorphism and%
\begin{equation*}
\left( \partial \Delta \right) \cap f^{-1}([0,+\infty ])
\end{equation*}%
contains at most one point; and in the case $\Delta \cap f^{-1}\left(
[0,+\infty ]\right) \neq \emptyset ,$ the following (i)--(v) hold:

(i) Each component of $f^{-1}\left( [0,+\infty ]\right) \cap \Delta $ is a
Jordan path with distinct endpoints contained in $\partial \Delta $ and
divides $\Delta $ into two Jordan domains.

(ii) Any pair of two distinct components of $f^{-1}([0,+\infty ])\cap \Delta
$ have at most one common endpoint.

(iii) For each component $D$ of $\Delta \backslash f^{-1}\left( [0,+\infty
]\right) ,$ $D$ is a Jordan domain and $f$ restrict to $D$ is a
homeomorphism.

(iv) For each component $D$ of $\Delta \backslash f^{-1}\left( [0,+\infty
]\right) ,$ $\left( \partial D\right) \cap (\partial \Delta )$ is consisted
of a connected open subset $\alpha _{0}$ and a number of finite points such
that%
\begin{equation*}
f(\alpha _{0})\cap \lbrack 0,+\infty ]=\emptyset \ \mathrm{and\ }f(\partial
D\backslash \alpha _{0})\subset \lbrack 0,+\infty ],
\end{equation*}%
and $f$ restricted to $\alpha _{0}$ is a homeomorphism.

(v) For $\alpha _{0}$ in (iv), if $\alpha _{0}\neq \emptyset $, then $f(%
\overline{D})$ is contained in some hemisphere of $S$ and
\begin{equation*}
L(f(\alpha _{0}))>L(f,\partial D\backslash \alpha _{0}),
\end{equation*}%
that is%
\begin{equation*}
L(f,\partial D\cap (\partial \Delta ))>L(f,\left( \partial D\right)
\backslash (\partial \Delta )).
\end{equation*}
\end{theorem}

This theorem has very simple geometrical explanation: when we cut the
Riemann surface of $f$ along $[0,+\infty ]$ in the case
\begin{equation*}
\Delta \cap f^{-1}\left( [0,+\infty ]\right) \neq \emptyset ,
\end{equation*}%
we obtain a finite number of pieces, each of which is either the whole
sphere $S$ with folded boundary $[0,+\infty ],$ or is contained in some open
hemisphere of $S$ such that the length of the boundary located in $%
S\backslash \lbrack 0,+\infty ],$ which is a part of the original boundary
of the Riemann surface of $f,$ is larger than the length of the boundary
located in $[0,+\infty ],$ which is a part of the the new boundary, the cut
edges.

This geometrical understand of the mapping in the theorem plays an important
role in this paper. We first prove the following lemma.

\begin{lemma}
\label{ok}Let $g:\overline{\Delta }\rightarrow S$ be a normal mapping that
satisfies (a)--(c) of the previous theorem and
\begin{equation}
g(\Delta )\subset S\backslash \lbrack 0,+\infty ].  \label{12.1}
\end{equation}%
Then,

(i) $g$ restricted to $\Delta $ is a homeomorphism onto $g(\Delta ).$

(ii) $\partial \Delta $ has an open connected subset $\alpha _{0}$ of $%
\partial \Delta ,$ such that
\begin{equation*}
g(\alpha _{0})\cap \lbrack 0,+\infty ]=\emptyset \ \mathrm{and\ }g(\partial
\Delta \backslash \alpha _{0})\subset \lbrack 0,+\infty ].
\end{equation*}

(iii) If in (ii) $\alpha _{0}\neq \emptyset $, then $g$ restricted to $%
\alpha _{0}$ is a homeomorphism onto the curve $g(\alpha _{0})$ in $S$ and $%
L(g(\alpha _{0}))>L(g,\partial \Delta \backslash \alpha _{0})$.

(iv) If in (ii), $\alpha _{0}\neq \emptyset ,$ then $g(\overline{\Delta })$
is contained in some open hemisphere of $S$.
\end{lemma}

\begin{proof}
By condition (c) and (\ref{12.1}) we have

(e) $g$ has no ramification point in $\overline{\Delta }$.

Let $p$ be any point in $\partial \Delta $ such that
\begin{equation*}
f(p)=1\in (0,+\infty )\subset S.
\end{equation*}
If $\Gamma _{g}=g(z),z\in \partial \Delta ,$ is not convex at $p,$ then,
since $f$ is normal, there is an open interval $I\subset (0,+\infty )$ whose
one endpoint is $1$ such that $I\subset g(\Delta ),$ which contradicts (\ref%
{12.1}). Thus, $\Gamma _{g}$ is convex at\footnote{%
This means that $\Gamma _{g}=g(z),z\in \partial \Delta ,$ is convex at each
point $p\in \left( \partial \Delta \right) \cap g^{-1}(1),$ by Definition %
\ref{convex1}.} $1,$ and then by (b), we have

(f) $\Gamma _{g}$ is locally convex in $S\backslash \{0,\infty \}.$

If $g(\partial \Delta )\cap \lbrack 0,+\infty ]=\emptyset ,$ then by (f) $%
\Gamma _{g}$ is locally convex everywhere, which implies that $\Gamma _{g}$
is locally simple by the definition, and then by Corollary \ref{bv-1} and
(e), $\Gamma _{g}$ is a simple curve and then $g$ is a homeomorphism from $%
\overline{\Delta }$ onto $g(\overline{\Delta }).$ On the other hand, in this
case, by (a), (f) and Definition \ref{convex1}, $\Gamma _{g}$ is a locally
convex curve and has at least three natural vertices, at each of which $%
\Gamma _{g}$ is strictly convex. Thus, by Lemma \ref{con-path-hemi} (ii),
the closure $\overline{T_{\Gamma _{g}}}$ of the domain $T_{\Gamma _{g}}$
enclosed by $\Gamma _{g}$ is contained in some open hemisphere of $S,$ and
thus, $g(\overline{\Delta })\subset \overline{T_{\Gamma _{g}}}$ is contained
in some open hemisphere of $S.$ Hence, putting $\alpha _{0}=\partial \Delta
, $ (i)--(iv) hold.

Consider the case $g(\partial \Delta )\subset \lbrack 0,+\infty ].$ Then $%
g(\partial \Delta )$ must be a closed interval in $[0,+\infty ]$. If $%
g(\partial \Delta )\neq \lbrack 0,+\infty ],$ then by the fact that $g$ is
normal, $g(\Delta )$ contains $0$ or $\infty ,$ but this contradicts the
assumption. Thus, $g(\partial \Delta )=[0,+\infty ],$ and then by (\ref{12.1}%
), $g$ restricted to $\Delta $ is a covering onto $S\backslash \lbrack
0,+\infty ]$, which, together with (e), implies that $g$ restricted to $%
\Delta $ is a homeomorphism, and putting $\alpha _{0}=\emptyset $, we have
(ii). Then, in the case
\begin{equation*}
g(\partial \Delta )\cap \lbrack 0,+\infty ]=\emptyset \ \mathrm{or\ }%
g(\partial \Delta )\subset \lbrack 0,+\infty ],
\end{equation*}%
the lemma is proved.

Now, we assume that $g(\partial \Delta )\cap \lbrack 0,+\infty ]\neq
\emptyset $ and $g(\partial \Delta )\backslash \lbrack 0,+\infty ]\neq
\emptyset $.

Then by (a), $\Gamma _{g}$ has a section%
\begin{equation}
\gamma _{0}=\overline{q_{1}q_{2}}+\overline{q_{2}q_{3}}+\dots +\overline{%
q_{m-1}q_{m}},m\geq 3,  \label{12.1+1}
\end{equation}%
such that each $q_{2},\dots ,q_{m-1}$ are natural vertices of $\Gamma _{g},$
the edges $\overline{q_{2}q_{3}},\dots ,\overline{q_{m-2}q_{m-1}}$ are
natural edges of $\Gamma _{g}$,%
\begin{equation}
\left[ \{q_{2},\dots ,q_{m-1}\}\cup \cup _{j=2}^{m-2}\overline{q_{j}q_{j+1}}%
\right] \cap \lbrack 0,+\infty ]=\emptyset  \label{8-1}
\end{equation}%
and%
\begin{equation}
\{q_{1},q_{m}\}\subset \lbrack 0,+\infty ]\ \mathrm{but\ }\gamma
_{0}\backslash \{q_{1},q_{m}\}\subset S\backslash \lbrack 0,+\infty ].
\label{12.2}
\end{equation}

By (f) and (\ref{12.1+1})--(\ref{12.2}), $\gamma _{0}$ is locally simple,
i.e. $\overline{q_{j}q_{j+1}}\cap \overline{q_{j+1}q_{j+2}}=\{q_{j+1}\}$ for
$j=1,\dots ,m-2.$ Then, by (\ref{12.1+1})--(\ref{12.2}), in the case that $%
\gamma _{0}$ is not simple, there exist integers $s$ and $t$ with $1\leq
s<s+1<t\leq m-1$ and a point
\begin{equation*}
q_{s}^{\prime }\in \left( \overline{q_{s}q_{s+1}}\cap \overline{q_{t}q_{t+1}}%
\right) \backslash \{q_{1},q_{m}\}
\end{equation*}%
such that
\begin{equation*}
\Gamma ^{\prime }=\overline{q_{s}^{\prime }q_{s+1}}+\overline{q_{s+1}q_{s+2}}%
+\dots +\overline{q_{t-1}q_{t}}+\overline{q_{t}q_{s}^{\prime }}
\end{equation*}%
is a section of $\gamma _{0}$ that is a simple path from $q_{s}^{\prime }$
to $q_{s}^{\prime },$ and
\begin{equation*}
\Gamma ^{\prime }\cap \lbrack 0,+\infty ]=\emptyset .
\end{equation*}%
Therefore, by (f) and Definition \ref{convex}, $\Gamma ^{\prime }$ is a
locally convex Jordan path\footnote{%
This does not mean that as a closed curve $\Gamma ^{\prime \prime }$ is
convex at $q_{s}^{\prime },$ by the definition of locally convex path and
locally convex closed curves.} that is strictly convex at $%
q_{s+1},q_{s+2},\dots ,q_{t}$, and then, by (a), $\Gamma ^{\prime }$ has at
least three strictly convex vertices. Therefore, by Lemma \ref{con-path-hemi}
(ii), $\Gamma ^{\prime }$ is contained in some open hemisphere, which
implies that $[0,+\infty ]\cap \overline{T_{\Gamma ^{\prime }}}=\emptyset ,$
where $T_{\Gamma ^{\prime }}$ is the domain inside $\Gamma ^{\prime },$ and
then by (f), $\Gamma _{g}=g(z),z\in \partial \Delta ,$ is locally convex in $%
\overline{T_{\Gamma ^{\prime }}}.$ Then, $g$ and $\Gamma ^{\prime }$
satisfies the assumption of Corollary \ref{bv-1}, and then $g$ has a
branched point in $\overline{T_{\Gamma ^{\prime }}}$, which contradicts (e).
Thus we have proved

(g) $\gamma _{0}$ is a Jordan path.

Then $\partial \Delta $ has an open section $\alpha _{0}$ such that $g$
restricted to $\alpha _{0}$ is a homeomorphism onto $\gamma _{0}^{\circ }.$
Then, by (e), (f), (g), (\ref{12.1}), (\ref{12.1+1}) and (\ref{12.2}), we
have $g\in \mathcal{F}_{m},$ and Theorem \ref{pre-key-2} applies. Then $g$
restricted to $\Delta $ is a homeomorphism onto the domain $T_{\Gamma
}\backslash \lbrack 0,+\infty ]\subset S,$ where $T_{\Gamma }$ is the domain
enclosed by $\Gamma =\gamma _{0}+\overline{q_{m}q_{1}},$%
\begin{equation*}
g(\left( \partial \Delta \right) \backslash \alpha _{0})\subset \lbrack
0,+\infty ],
\end{equation*}%
\begin{equation*}
L(g,\alpha _{0})>L(g,\left( \partial \Delta \right) \backslash \alpha _{0}),
\end{equation*}%
and $g(\overline{\Delta })$ is contained in some open hemisphere of $S.$
Thus, (i)--(v) hold, and the proof is complete.
\end{proof}

\begin{proof}[Proof of Theorem \protect\ref{decom}]
We first assume
\begin{equation}
\Delta \cap f^{-1}([0,+\infty ])=\emptyset .  \label{end-2}
\end{equation}%
Then Lemma \ref{ok} applies, and then $f:\Delta \rightarrow f(\Delta )$ is a
homeomorphism, $f(\overline{\Delta })$ is contained in some open hemisphere
of $S$ and there exists a connected open subset $\alpha _{0}\subset \partial
\Delta $ such that%
\begin{equation}
f(\alpha _{0})\cap \lbrack 0,+\infty ]=\emptyset ,  \label{end-3}
\end{equation}%
\begin{equation}
f(\left( \partial \Delta \right) \backslash \alpha _{0})\subset \lbrack
0,+\infty ],  \label{end-4}
\end{equation}%
and $f$ restricted to $\alpha _{0}$ is also a homeomorphism onto some curve
in $S.$ Then, $\left( \partial \Delta \right) \backslash \alpha _{0}$ is
also a connected section of $\partial \Delta $, and $f$ restricted to $%
\Delta \cup \alpha _{0}$ is also a homeomorphism.

By (d) and (\ref{end-4}),
\begin{equation*}
f(\left( \partial \Delta \right) \backslash \alpha _{0})\subset f(\partial
\Delta )\cap \lbrack 0,+\infty ]
\end{equation*}%
is a finite set. Then, since $\left( \partial \Delta \right) \backslash
\alpha _{0}$ is connected, $f(\left( \partial \Delta \right) \backslash
\alpha _{0})$ is a singleton, or is empty, which implies that $\left(
\partial \Delta \right) \backslash \alpha _{0}$ is a singleton, or is empty,
which, with (\ref{end-3}) and the above argument, implies that $\left(
\partial \Delta \right) \cap f^{-1}([0,+\infty ])$ contains at most one
point and $f:\overline{\Delta }\rightarrow f(\overline{\Delta })$ is a
homeomorphism. The theorem is proved under the assumption (\ref{end-2}).

Now, we assume%
\begin{equation*}
f(\Delta )\cap \lbrack 0,+\infty ]\neq \emptyset .
\end{equation*}%
Then by (c) and (d) and the assumption that $f$ is normal, each component of
$f^{-1}([0,+\infty ])\cap \Delta $ is a simple path in $\Delta $ whose
endpoints are distinct and contained in $\partial \Delta ,$ i.e. (i) holds
true, and $f^{-1}([0,+\infty ])\cap \Delta $ has only a finite number of
components. This implies that $\Delta \backslash f^{-1}([0,+\infty ])$ has a
finite number of components, each of which is a Jordan domain.

(ii) follows from Lemma \ref{cut-3}.

Let $D$ be any component of $\Delta \backslash f^{-1}([0,+\infty ]).$ Then
by (i), $D$ is a Jordan domain. Let $g$ be the restricted mapping $g=f|_{%
\overline{D}}.$ Then $g$ is a normal mapping and each natural edge of $%
\Gamma _{g}=g(z),z\in \partial D,$ is either a natural edge of $\Gamma
_{f}=f(z),z\in \partial \Delta ,$ or a section of some natural edge of $%
\Gamma _{f},$ or an interval contained in $\overline{0,1}$ or $\overline{%
1,\infty }.$ Thus (a) is satisfied by $g.$ By (b) and (c), $g$ also
satisfies (b) and (c), and it is clear that%
\begin{equation*}
g(D)\cap \lbrack 0,+\infty ]=\emptyset .
\end{equation*}%
Thus $g$ satisfies all the assumption of Lemma \ref{ok}, by ignoring a
coordinate transform that maps $\overline{D}$ homeomorphically onto $%
\overline{\Delta }.$ Thus, Lemma \ref{ok} applies to $g$ and (iii) follows.
By Lemma \ref{ok}, $\partial D$ has a connected open subset $\alpha _{0}$ of
$\partial D,$ such that (\ref{end-3}) and (\ref{end-4}) still hold.

It is clear, by (\ref{end-3}) and (\ref{end-4}), that
\begin{equation*}
\alpha _{0}=(\partial D)\backslash f^{-1}([0,+\infty ]),
\end{equation*}%
and, by (i) and (ii), that
\begin{equation*}
\left( \partial D\right) \cap \Delta \subset f^{-1}([0,+\infty ].
\end{equation*}%
Then, we have
\begin{equation*}
\alpha _{0}=\left( (\partial D)\backslash \Delta \right) \backslash
f^{-1}([0,+\infty ]),
\end{equation*}%
and, considering that $(\partial D)\backslash \Delta =\left( \partial
D\right) \cap (\partial \Delta ),$ we have%
\begin{equation*}
\alpha _{0}=\left[ \left( \partial D\right) \cap (\partial \Delta )\right]
\backslash f^{-1}([0,+\infty ]).
\end{equation*}%
Then, considering that, by (d), $(\partial \Delta )\cap f^{-1}([0,+\infty ])$
is a finite set, we conclude that $\left( \partial D\right) \cap (\partial
\Delta )\backslash \alpha _{0}$ is a finite set, and thus $\alpha _{0}$ is
the interior of $\left( \partial D\right) \cap (\partial \Delta )$ in $%
\partial \Delta .$ Therefore, by (\ref{end-3}) and (\ref{end-4}), we have
(iv).

(v) follows from (iv)\ and Lemma \ref{ok}. This completes the proof.
\end{proof}

In the above two proofs, we have also proved that:

\begin{corollary}
\label{by the way}Let $f:\overline{\Delta }\rightarrow S$ be a normal
mapping that satisfies all assumptions of Theorem \ref{decom} and let $%
\Delta _{1}$ be any component of $\Delta \backslash f^{-1}([0,+\infty ]).$
Then the restriction $g=f|_{\overline{\Delta _{1}}}$ satisfies all the
assumptions of Lemma \ref{ok}, and, furthermore, if $g(\partial \Delta
_{1})\subset \lbrack 0,+\infty ],$ then $g(\partial \Delta _{1})=[0,+\infty
] $ and $g$ restricted to $\Delta _{1}$ is a homeomorphism onto $S\backslash
\lbrack 0,+\infty ].$\newpage
\end{corollary}

The condition (d) in Theorem \ref{decom} may be removed by the following
lemma.

\begin{lemma}
\label{pertur}Let $f:\overline{\Delta }\rightarrow S$ be a normal mapping
satisfying the conditions (a)--(c) of Theorem \ref{decom}.

Then for any $\varepsilon >0,$ there exists a normal mapping $g:\overline{%
\Delta }\rightarrow S$ such that%
\begin{equation*}
A(g,\Delta )\geq A(f,\Delta ),L(g,\partial \Delta )<L(f,\partial \Delta
)+\varepsilon ,
\end{equation*}%
and $g$ satisfies all conditions (a)--(d) of Theorem \ref{decom}.
\end{lemma}

\begin{proof}
This can be proved by perturb the natural edges of $f$ lying on $[0,+\infty
] $ slightly.


Let $\Gamma _{1}$ be any natural edge of $\Gamma _{f}$ such that $\Gamma
_{1}\cap \lbrack 0,+\infty ]$ contains more that one point. Then since $%
\Gamma _{1}$ is a natural edge, it is either contained in the interval $%
[0,1] $ in $S$, or in the interval $[1,\infty ].$ Without loss of generality
we assume $\Gamma _{1}\subset \lbrack 0,1]$ and the orientation of $\Gamma
_{1}$ is the same as $\overline{0,1}.$ Then there are four cases:

(i) $\Gamma _{1}=\overline{0,1}.$

(ii) $\Gamma _{1}=\overline{0,t_{0}}$ for some $t_{0}\in (0,1).$

(iii) $\Gamma _{1}=\overline{t_{0},t_{1}}$ for some $t_{0},t_{1}\in (0,1).$

(iv) $\Gamma _{1}=\overline{t_{0},1}$ for some $t_{0}\in (0,1).$

In these cases, we can extend the Riemann surface of $f$ by
patching a closed triangle domain along $\Gamma _{1}$
so that the vertex $p_{1}^{\prime }$ is very close the middle point of $%
\Gamma _{1}$ and is on the right hand side of $\Gamma _{1}.$ By Lemma \ref%
{patch}, the new Riemann surface can be realized by a normal mapping $f_{1}:%
\overline{\Delta }\rightarrow S.$ It is clear that when $p_{1}^{\prime }$ is
sufficiently close to $\frac{1}{2}$ in case (i), or $\frac{t_{0}}{2}$ in
case (ii), or $\frac{t_{0}+t_{1}}{2}$ in case (iii), or $\frac{t_{0}+1}{2}$
in case (iv), $f_{1}$ satisfies (a)--(c) of the lemma and%
\begin{equation*}
\left\vert L(f_{1},\partial \Delta )-L(f,\partial \Delta )\right\vert <\frac{%
\varepsilon }{V(f)},A(f_{1},\Delta )\geq A(f,\Delta ),
\end{equation*}%
while the number of natural edges that lie on $[0,+\infty ]$ is dropped by
one.

Then repeating the above argument for $f_{1},$ and so on, and finally we can
obtain the desired mapping.
\end{proof}

\section{Deformation of edges of normal mappings with length larger than $%
\protect\pi $ \label{ss-7to<pi}}

In this section we will prove the following theorem, which is prepared for
proving Theorem \ref{1-br-2-map} and Theorem \ref{1-nconvex}.

\begin{theorem}
\label{>pi}Let $f:\overline{\Delta }\rightarrow S$ be a normal mapping whose
boundary curve $\Gamma _{f}=f(z),z\in \partial \Delta ,$ has the natural
partition%
\begin{equation*}
\Gamma _{f}=\Gamma _{1}+\Gamma _{2}+\dots +\Gamma _{n},n=V(f),
\end{equation*}%
that satisfies one of the following conditions (a)--(d).

(a) $\pi \leq L(\Gamma _{1})<2\pi ,$ $L(\Gamma _{j})<\pi $ for all $j\geq 2,$
and $\Gamma _{1}$ has an endpoint contained in $E=\{0,1,\infty \}.$

(b) $\pi \leq L(\Gamma _{1})<2\pi ,$ $L(\Gamma _{j})<\pi $ for all $j\geq 2,$
and $\Gamma _{1}$ has no endpoint contained in $E.$

(c) $\pi \leq L(\Gamma _{1})<2\pi $ and $\pi \leq L(\Gamma _{j_{0}})<2\pi $
for some $j_{0}\geq 2,$ $L(\Gamma _{j})<\pi $ for each $j\neq 1,j_{0};$ $%
\Gamma _{1}$ has an endpoint contained in $E,$ and so does $\Gamma _{j_{0}}$.

(d) $2\pi \leq L(\Gamma _{1})<3\pi ,$ while $L(\Gamma _{j})<\pi $ for all $%
j\geq 2.$

Then, there exists a normal mapping $g:\overline{\Delta }\rightarrow S$ such
that

(i) $L(g,\partial \Delta )\leq L(f,\partial \Delta )$ and $A(g,\Delta )\geq
A(f,\Delta )$

(ii) Each natural edge of $g$ has spherical length strictly less than $\pi ,$

(iii) In case (a), $V_{NE}(g)\leq V_{NE}(f),\mathrm{\ }V_{E}(g)\geq
V_{E}(f)+1,$ and $V(g)\leq V(f)+1;$

In case (b), $V_{NE}(g)\leq V_{NE}(f),\mathrm{\ }V_{E}(g)\geq V_{E}(f)+1,$
and $V(g)\leq V(f)+2;$

In case (c), $V_{NE}(g)\leq V_{NE}(f),\mathrm{\ }V_{E}(g)\geq V_{E}(f)+2,$
and $V(g)\leq V(f)+2;$

In case (d), $V_{NE}(g)=V_{NE}(f)+2,V_{E}(g)=V_{E}(f)+1,\ $and $V(g)=V(f)+3.$
\end{theorem}

The proof is divided into four parts: Lemmas \ref{ini-1-pi}--\ref{>=2pi}.

\begin{lemma}
\label{pre-large}Let $\Gamma $ be a line segment in $S$ with endpoints $%
q_{1} $ and $q_{2}$ and $\pi \leq L(\Gamma )<2\pi ,$ and let $q_{0}$ be any
point in $S\backslash \Gamma .$ Then
\begin{equation}
d(q_{0},q_{1})<\pi ,d(q_{0},q_{2})<\pi ,  \label{9-1}
\end{equation}%
and%
\begin{equation*}
L(\overline{q_{1}q_{0}})+L(\overline{q_{0}q_{2}})\leq L(\Gamma ).
\end{equation*}
\end{lemma}

By the first two inequalities, $\overline{q_{1}q_{0}}$ and $\overline{%
q_{0}q_{2}}$ make sense, which is the shortest paths.

\begin{proof}
Since $L(\Gamma )\geq \pi ,$ the antipodal points of $q_{1}$ and $q_{2}$ are
both contained in $\Gamma ,$ and thus, neither $q_{1}$, nor $q_{2},$ can be
an antipodal point of $q_{0}\in S\backslash \Gamma .$ This implies (\ref{9-1}%
).

Let $q_{1}^{\prime }$ be the antipodal point of $q_{1}$ in $S$. Then $%
q_{1}^{\prime }\in \Gamma .$ Let $\Gamma _{1}^{\prime }$ be the section of $%
\Gamma $ from $q_{1}$ to $q_{1}^{\prime }$ and let $\Gamma _{1}^{\prime
\prime }$ be the section of $\Gamma $ from $q_{1}^{\prime }$ to $q_{2}.$
Then it is clear that $\overline{q_{1}q_{0}}+\overline{q_{0}q_{1}^{\prime }}$
is a straight path in $S$ from $q_{1}$ to $q_{1}^{\prime }.$ Thus we have%
\begin{eqnarray*}
L(\Gamma ) &=&L(\Gamma _{1}^{\prime })+L(\Gamma _{1}^{\prime \prime })=\pi
+L(\Gamma _{1}^{\prime \prime }) \\
&=&L(\overline{q_{1}q_{0}}+\overline{q_{0}q_{1}^{\prime }})+L(\Gamma
_{1}^{\prime \prime }) \\
&=&L(\overline{q_{1}q_{0}})+L(\overline{q_{0}q_{1}^{\prime }}+\Gamma
_{1}^{\prime \prime }) \\
&\geq &L(\overline{q_{1}q_{0}})+L(\overline{q_{0}q_{2}}),
\end{eqnarray*}%
and equality holds if and only if $q_{1}$ and $q_{2}$ are a pair of
antipodal points of $S.$
\end{proof}

\begin{lemma}
\label{parti-ini}Let $f:\overline{\Delta }\rightarrow S$ be a normal mapping
and assume that $\Gamma _{f}$ has a natural partition
\begin{equation}
\Gamma _{f}=\Gamma _{1}+\Gamma _{2}+\dots +\Gamma _{n},n=V(f),  \label{a8}
\end{equation}%
such that $L(\Gamma _{1})\geq \pi $ and the initial point of $\Gamma _{1}$
is in $E.$ Then, there exists a normal mapping $f_{1}:\overline{\Delta }%
\rightarrow S$ such that
\begin{equation*}
L(f_{1},\partial \Delta )\leq L(f,\partial \Delta ),A(f_{1},\Delta )\geq
A(f,\Delta ),
\end{equation*}%
and the boundary curve $\Gamma _{f_{1}}$ has a permitted partition%
\begin{equation}
\Gamma _{f_{1}}=\Gamma _{1}^{\prime }+\Gamma _{1}^{\prime \prime }+\Gamma
_{2}+\dots +\Gamma _{n},  \label{1021-1}
\end{equation}%
such that the end point of $\Gamma _{1}^{\prime }$, which is also the
initial point of $\Gamma _{1}^{\prime \prime },$ is contained in $E,$ and%
\begin{equation*}
L(\Gamma _{1}^{\prime })<\pi ,L(\Gamma _{1}^{\prime \prime })<\pi .
\end{equation*}
\end{lemma}

See definitions in Section \ref{ss-2appoint} for the terms \emph{natural
partition }and \emph{permitted partition}. Since (\ref{a8}) is a natural
partition and the initial point of $\Gamma _{1}$ is in $E,$ it is clear by (%
\ref{1021-1}) that the initial point of $\Gamma _{1}^{\prime },$ which is
the terminal point of $\Gamma _{n},$ is still the initial point of $\Gamma
_{1},$ and then $\Gamma _{3},\dots ,\Gamma _{n}$ are still natural edges of $%
\Gamma _{f_{1}}.$ Then the two endpoints of $\Gamma _{1}^{\prime }$ are both
in $E,$ and so $\Gamma _{1}^{\prime }$ is a natual edge$.$ But $\Gamma _{2}$
may not be a natural edge of $\Gamma _{f_{1}}$ and $\Gamma _{2}$ is a
natural edge if and only if $\Gamma _{1}^{\prime }$ is. In the case $\Gamma
_{2}$ is not a natural edge of $\Gamma _{f_{1}},$ $\Gamma _{1}^{\prime
\prime }+\Gamma _{2}$ must be a natural edge.

\begin{proof}
Let $q_{1}$ and $q_{2}$ be the initial and terminal point of $\Gamma _{1},$
respectively. Then, $\Gamma _{1}$ contains the antipodal point $%
q_{1}^{\prime }$ of $q_{1}.$ Let $C$ be the great circle determined\footnote{%
Recall that this means that $C$ contains $\Gamma _{1}$ and is oriented by $%
C. $} by $\Gamma _{1}$ and let $S^{\prime }$ be the open hemisphere outside%
\footnote{%
Recall that this means $S^{\prime }$ is on the right hand side of $C.$} $C$.

There are only two cases (note that we assumed $q_{1}\in E$ in the lemma):

\noindent \textbf{Case 1.} $q_{1}\in E,$ $q_{1}^{\prime }\notin E.$

\noindent \textbf{Case 2.} $q_{1}\in E,q_{1}^{\prime }\in E.$

Assume the first case occurs. Then we must have $q_{1}=1$ and $q_{1}^{\prime
}=-1.$ Then $C$ must separate $0$ and $\infty .$ Without loss of generality,
we assume $0\in S^{\prime }.$ Let $\Gamma _{1}^{\prime }=\overline{q_{1}0}=%
\overline{1,0},$ which is the shortest path in $S$ from $q_{1}=1$ to $0$ and
let $\Gamma _{1}^{\prime \prime }=\overline{0q_{2}}.$ Then the curve $\Gamma
_{1}^{\prime }+\Gamma _{1}^{\prime \prime }$ is a simple path from $q_{1}=1$
to $q_{2}$ and $\Gamma _{1}^{\prime }+\Gamma _{1}^{\prime \prime }-\Gamma
_{1}$ is a Jordan curve that encloses a domain $T$ in $S^{\prime }$ such
that $T$ is on the right hand side of $\Gamma _{1},$
\begin{equation*}
T\cap E=\emptyset ,
\end{equation*}
\begin{equation*}
L(\Gamma _{1}^{\prime })=\frac{\pi }{2},
\end{equation*}%
and by Lemma \ref{pre-large},%
\begin{equation*}
L(\Gamma _{1}^{\prime \prime })<\pi \ \mathrm{and\ }L(\Gamma _{1}^{\prime
})+L(\Gamma _{1}^{\prime \prime })\leq L(\Gamma _{1}).
\end{equation*}

Let
\begin{equation*}
\partial \Delta =\alpha _{1}+\dots +\alpha _{n}
\end{equation*}%
be a natural partition of $\partial \Delta $ for $\Gamma _{f},$ corresponding%
\footnote{%
See Definition \ref{natural} and Remark \ref{closed-c} (2).} to (\ref{a8}),
let $V$ be a Jordan domain outside $\Delta $ with $\left( \partial V\right)
\cap \partial \Delta =\alpha _{1},$ and let $g$ be a homeomorphism from $%
\overline{V}$ onto $\overline{T}$ such that%
\begin{equation*}
f|_{\alpha _{1}}=g|_{\alpha _{1}}.
\end{equation*}%
Then, by Lemma \ref{patch},
\begin{equation*}
g_{1}=\left\{
\begin{array}{l}
f(z),z\in \overline{\Delta }, \\
g(z),z\in \overline{V}\backslash \overline{\Delta },%
\end{array}%
\right.
\end{equation*}%
is a normal mapping defined on the closure of the Jordan domain $D=\Delta
\cup \alpha _{1}^{\circ }\cup V,$ where $\alpha _{1}^{\circ }$ is the
interior of $\alpha _{1}.$ Then the boundary curve $\Gamma _{g_{1}}$ of $%
g_{1}$ has the permitted partition%
\begin{equation}
\Gamma _{g_{1}}=\Gamma _{1}^{\prime }+\Gamma _{1}^{\prime \prime }+\Gamma
_{2}+\dots +\Gamma _{n}  \label{121-1}
\end{equation}%
and $A(g_{1},D)=A(f,\Delta )+A(T).$ Then we have
\begin{equation*}
L(g_{1},\partial D)\leq L(f,\partial \Delta )\ \mathrm{and\ }%
A(g_{1},D)>A(f,\Delta ).
\end{equation*}

Let $h$ be any homeomorphism from $\overline{D}$ onto $\overline{\Delta }.$
Then $f_{1}=g_{1}\circ h^{-1}$ satisfies all the desired conditions in (ii).

Assume Case 2 occur. Then $q_{1}$ and $q_{1}^{\prime }$ must be the pair $%
\{0,\infty \}$, $q_{1}^{\prime }=q_{2},$ and there is no point in $E$
located in the interior of $\Gamma _{1},$ for $\Gamma _{1}$ is a natural
edge of $f$. Without loss of generality, we assume that $q_{1}=0$ and $%
q_{2}=q_{1}^{\prime }=\infty .$ Let $L$ be the straight path from $0$ to $%
\infty $ that passes through $1.$ Then, the domain $T$ enclosed by $\Gamma
_{1}$ and $L$ that is on the right hand side of $\Gamma _{1}$ does not
contains point in $E\ $and $\{\Gamma _{1}\cup T\}\cap E=\{0,\infty \}.$ Let $%
\Gamma _{1}^{\prime }$ be the section of $L$ from $0$ to $1$ and $\Gamma
_{1}^{\prime \prime }$ be the section of $L$ from $1$ to $\infty .$ Then%
\begin{equation*}
L(\Gamma ^{\prime })=\frac{\pi }{2},L(\Gamma _{1}^{\prime \prime })=\frac{%
\pi }{2},
\end{equation*}%
and repeating the process in Case 1, we can obtain a desired $f_{1}$
satisfies all the conditions.
\end{proof}

\begin{lemma}
\label{ini-1-pi}Let $f:\overline{\Delta }\rightarrow S$ be a normal mapping,
and let
\begin{equation}
\Gamma _{f}=\Gamma _{1}+\dots +\Gamma _{n},n=V(f),  \label{9-2}
\end{equation}%
be the natural partition of $\Gamma _{f}=f(z),z\in \partial \Delta .$ Assume
that

(a) At least one endpoint of $\Gamma _{1}$ is contained in $E.$

(b) $\pi \leq L(\Gamma _{1})<2\pi ,$ but $L(\Gamma _{j})<\pi ,j=2,\dots ,n.$

Then, there exists a normal mapping $g:\overline{\Delta }\rightarrow S$ such
that

(i) $L(g,\partial \Delta )\leq L(f,\partial \Delta )$ and $A(g,\Delta )\geq
A(f,\Delta )$

(ii) Each natural edge of $g$ has spherical length strictly less than $\pi ,$

(iii) $V_{NE}(g)\leq V_{NE}(f)$ and $V_{E}(g)\geq V_{E}(f)+1.$

(iv) $V(g)\leq V(f)+1.$
\end{lemma}

\begin{proof}
Without loss of generality, we assume the initial point $q_{1}$ of $\Gamma
_{1}$ is in $E.$ Then by Lemma \ref{parti-ini}, there exists a normal
mapping $f_{1}:\overline{\Delta }\rightarrow S$ such that
\begin{equation}
L(f_{1},\partial \Delta )\leq L(f,\partial \Delta ),A(f_{1},\Delta )\geq
A(f,\Delta ),  \label{a12}
\end{equation}%
and the boundary curve $\Gamma _{f_{1}}$ has a permitted partition%
\begin{equation}
\Gamma _{f_{1}}=\Gamma _{1}^{\prime }+\Gamma _{1}^{\prime \prime }+\Gamma
_{2}+\dots +\Gamma _{n}  \label{125-1}
\end{equation}%
such that
\begin{equation}
L(\Gamma _{1}^{\prime })=\frac{\pi }{2},L(\Gamma _{1}^{\prime \prime })<\pi ,
\label{125-2}
\end{equation}%
the end point of $\Gamma _{1}^{\prime }$, which is also the initial point of
$\Gamma _{1}^{\prime \prime },$ is contained in $E.$ It is clear that the
initial points of $\Gamma _{1}^{\prime }$ and $\Gamma _{1}$ are the same
point and so is in $E,$ for they are both the terminal point of $\Gamma
_{n}, $ by (\ref{9-2}) and (\ref{125-1}), and thus, $\Gamma _{1}^{\prime }$
is a natural edge of $\Gamma _{f_{1}}$ whose two endpoints are in $E,$ the
terminal point of $\Gamma _{1}^{\prime \prime }$ is a natural vertex of $%
\Gamma _{f}$, which may not be a natural vertex of $\Gamma _{f_{1}}$. On the
other hand, by (\ref{9-2}) and (\ref{125-1}), the terminal points of $\Gamma
_{2},\dots ,\Gamma _{n}$ are still natural vertices of $\Gamma _{f_{1}}.$
Then we have%
\begin{equation}
V_{NE}(f_{1})\leq V_{NE}(f),V_{E}(f_{1})=V_{E}(f)+1,V(f_{1})\leq V(f)+1.
\label{125-3}
\end{equation}

\noindent \textbf{Case 1.} If the terminal point of $\Gamma _{1}$ is also in
$E,$ then both $\Gamma _{1}^{\prime }$ and $\Gamma _{1}^{\prime \prime }$
have initial and terminal points in $E,$ and then they are natural edges of $%
f_{1}$ and then (\ref{125-1}) is a natural partition$;$ therefore, by (\ref%
{a12})--(\ref{125-3}) and (b), $g=f_{1}$ is the desired mapping.

\noindent \textbf{Case 2. }Now assume that the terminal point of $\Gamma
_{1} $ is not in $E.$

If\textbf{\ }$\Gamma _{1}^{\prime \prime }$ is still a natural edge, then (%
\ref{125-1}) is still a natural partition, and $g=f_{1}$ is the desired
mapping by (b) and (\ref{125-3}).

Assume $\Gamma _{1}^{\prime \prime }$ is not a natural edge. Then $\Gamma
_{1}^{\prime \prime }+\Gamma _{2}$ will be a natural edge, and then $\Gamma
_{f_{1}}$ has the natural partition%
\begin{equation}
\Gamma _{f_{1}}=(\Gamma _{1}^{\prime \prime }+\Gamma _{2})+\Gamma _{3}+\dots
+\Gamma _{n}+\Gamma _{1}^{\prime },  \label{1212-4}
\end{equation}%
the initial point of $\Gamma _{2},$ which is also the terminal point of $%
\Gamma _{1}^{\prime \prime }$ is now in the interior of the the natural edge
$\left( \Gamma _{1}^{\prime \prime }+\Gamma _{2}\right) $ and then we have
by (\ref{1212-4})
\begin{equation}
V_{NE}(f_{1})=V_{NE}(f)-1,V_{E}(f_{1})=V_{E}(f)+1,V(f_{1})=V(f).
\label{1212-3}
\end{equation}%
On the other hand, by (b) and (\ref{125-2}) we have
\begin{equation}
L(\Gamma _{1}^{\prime \prime }+\Gamma _{2})<2\pi .  \label{1212-1}
\end{equation}

If $L(\Gamma _{1}^{\prime \prime }+\Gamma _{2})<\pi ,$ then $g=f_{1}$ is the
desired mapping.

If $L(\Gamma _{1}^{\prime \prime }+\Gamma _{2})\geq \pi ,$ then $f_{1}$
satisfies all the assumption of the lemma with natural partition (\ref%
{1212-4}), and the above argument applies to $f_{1}.$ By (\ref{1212-4})--(%
\ref{1212-1}), if we repeat the above argument once, and if we do not arrive
at the desired mapping, then $V_{NE}(\cdot )$ drops by one, $V_{E}(\cdot )$
increases by one and $V(\cdot )$ keep invariant. But $V_{NE}(\cdot )\geq 0$
in any case, and so we can reach the desired mapping by repeating the above
argument finitely many times. This completes the proof.
\end{proof}

\begin{lemma}
\label{bd-bd>pi}Let $f:\overline{\Delta }\rightarrow S$ be a normal mapping,
and let
\begin{equation*}
\Gamma _{f}=\Gamma _{1}+\Gamma _{2}+\dots +\Gamma _{n},n=V(f),
\end{equation*}%
be a natural partition of $\Gamma _{f}=f(z),z\in \partial \Delta .$ Assume
that for some positive integer $k_{0}$ with $1<k_{0}\leq n$ the followings
hold.

(a) $\Gamma _{1}$ has an endpoint contained in $E,$ and so does $\Gamma
_{k_{0}}$.

(b) $\pi \leq L(\Gamma _{1})<2\pi ,\pi \leq L(\Gamma _{k_{0}})<2\pi ,$ but $%
L(\Gamma _{j})<\pi ,j\neq 1,k_{0}.$

Then, there exists a normal mapping $g:\overline{\Delta }\rightarrow S$ such
that

(i) $L(g,\partial \Delta )\leq L(f,\partial \Delta )$ and $A(g,\Delta )\geq
A(f,\Delta )$

(ii) Each natural edge of $g$ has spherical length strictly less than $\pi ,$

(iii) $V_{NE}(g)\leq V_{NE}(f)$ and $V_{E}(g)\geq V_{E}(f)+2.$

(iv) $V(g)\leq V(f)+2.$
\end{lemma}

\begin{proof}
Without loss of generality, we assume that the initial point $q_{1}$ of $%
\Gamma _{1}$ is in $E.$

By Lemma \ref{parti-ini}, there exists a normal mapping $f_{1}:\overline{%
\Delta }\rightarrow S$ such that
\begin{equation}
L(f_{1},\partial \Delta )\leq L(f,\partial \Delta )\ \mathrm{and\ }%
A(f_{1},\Delta )\geq A(f,\Delta ),  \label{q1}
\end{equation}
and the boundary curve $\Gamma _{f_{1}}$ has a permitted partition%
\begin{equation}
\Gamma _{f_{1}}=\Gamma _{1}^{\prime }+\Gamma _{1}^{\prime \prime }+\Gamma
_{2}+\dots +\Gamma _{n}  \label{11}
\end{equation}%
such that
\begin{equation}
L(\Gamma _{1}^{\prime })=\frac{\pi }{2},L(\Gamma _{1}^{\prime \prime })<\pi ,
\label{12}
\end{equation}%
and the end point of $\Gamma _{1}^{\prime }$, which is also the initial
point of $\Gamma _{1}^{\prime \prime },$ is contained in $E.$ Then, $\Gamma
_{1}^{\prime }$ is a natural edge of $\Gamma _{f_{1}},$ because its
endpoints are both in $E.$

If $\Gamma _{1}^{\prime \prime }$ is a natural edge of $f_{1},$ then (\ref%
{11}) is a natural partition and by (a), (b), (\ref{11}) and (\ref{12}), $%
f_{1}$ satisfies all assumptions of Lemma \ref{ini-1-pi}, with
\begin{equation*}
V_{NE}(f_{1})=V_{NE}(f),V_{E}(f_{1})=V_{E}(f)+1,V(f_{1})=V(f)+1,
\end{equation*}%
and then, by (\ref{q1}), we can apply Lemma \ref{ini-1-pi} to deform $f_{1}$
to be another normal mapping $g$ satisfying (i)--(iv) with%
\begin{equation*}
\left\{
\begin{array}{l}
V_{NE}(g)\leq V_{NE}(f_{1})=V_{NE}(f), \\
V_{E}(g)\geq V_{E}(f_{1})+1=V_{E}(f)+2, \\
V(g)\leq V(f_{1})+1=V(f)+2.%
\end{array}%
\right.
\end{equation*}

Now, assume that

(c) $\Gamma _{1}^{\prime \prime }$ is not a natural edge of $\Gamma _{f_{1}}$%
.

We complete the proof by induction on $k_{0}\geq 2.$ We first assume $%
k_{0}=2.$

Then by the assumption (c), $\Gamma _{1}^{\prime \prime }+\Gamma _{2}$ must
be a natural edge, $\Gamma _{f_{1}}$ has the natural partition%
\begin{equation}
\Gamma _{f_{1}}=\Gamma _{1}^{\prime }+\left( \Gamma _{1}^{\prime \prime
}+\Gamma _{2}\right) +\Gamma _{3}+\dots +\Gamma _{n},  \label{9-3}
\end{equation}%
with%
\begin{equation}
V_{NE}(f_{1})=V_{NE}(f)-1,V_{E}(f_{1})=V_{E}(f)+1,V(f_{1})=V(f),  \label{13}
\end{equation}%
and the initial point of $\Gamma _{2}$ is not contained in $E,$ for,
otherwise, $\Gamma _{1}^{\prime \prime }$ has two endpoints in $E$, which
implies that $\Gamma _{1}^{\prime \prime }$ is a natural edge. Therefore, by
(a), the initial and terminal points of the natural edge $\Gamma
_{1}^{\prime \prime }+\Gamma _{2}$ of $\Gamma _{f_{1}}$ are both contained
in $E.$

By the definition, each natural edge of a closed polygonal curve with
initial and terminal points in $E$ is simple and has length $\frac{\pi }{2},$
$\pi ,$ or $2\pi .$ Then by (b) and the fact that $L(\Gamma _{1}^{\prime
\prime }+\Gamma _{2})>L(\Gamma _{2})\geq \pi $ we have%
\begin{equation*}
L(\Gamma _{1}^{\prime \prime }+\Gamma _{2})=2\pi .
\end{equation*}%
Note that we are in the situation that $\Gamma _{1}^{\prime \prime }+\Gamma
_{2}$ is a natural edge of $\Gamma _{f_{1}}$ and a natural edge never
contains any point of $E$ in its interior, and then we conclude that
\begin{equation*}
C=\Gamma _{1}^{\prime \prime }+\Gamma _{2}
\end{equation*}%
is a great circle passing through $1$ with $C\cap E=\{1\}$ (so, $1$ is the
initial and terminal point of $C).$

Then $C$ separates $0$ and $\infty ,$ without loss of generality we assume $%
0 $ is on the right hand side of $C.$ Let
\begin{equation*}
\partial \Delta =\alpha _{1}+\alpha _{2}+\dots +\alpha _{n}
\end{equation*}%
be a natural partition of $\partial \Delta $ corresponding to (\ref{9-3})
(see Definition \ref{natural}). Then $\Gamma _{1}^{\prime \prime }+\Gamma
_{2}$ is the section of $\Gamma _{f_{1}}$ restricted to $\alpha _{2}$. Let $%
V $ be a bounded Jordan domain in $\mathbb{C}$ that is outside $\Delta $
with $\left( \partial \Delta \right) \cap \left( \partial V\right) =\alpha
_{2}$, let $p_{1}$ and $p_{2}$ be the initial and terminal points of $\alpha
_{2},$ respectively, and let $T$ be the hemisphere on the right hand side of
$C$ with the path $\overline{0,1}$ being removed. Then, there exists a
continuous mapping $\tau $ from $\overline{V}$ onto $\overline{T}$ such that
$\tau |_{\alpha _{2}}=f_{1}|_{\alpha _{2}},$ $\tau $ restricted to $\alpha
_{2}\cup V$ is a homeomorphism onto $(\Gamma _{1}^{\prime \prime }+\Gamma
_{2})\cup T$ with $\tau (\alpha _{2})=\Gamma _{1}^{\prime \prime }+\Gamma
_{2},$ and $\tau $ restricted to $\left( \partial V\right) \backslash \alpha
_{2}=\left( \partial V\right) \backslash \overline{\Delta }$ is a folded $2$
to $1$ mapping onto $\overline{0,1}$. Then by Lemma \ref{patch}, the mapping
\begin{equation*}
f^{\ast }=\left\{
\begin{array}{l}
f_{1}(z),z\in \overline{\Delta }, \\
\tau (z),z\in \overline{V}\backslash \overline{\Delta },%
\end{array}%
\right.
\end{equation*}%
is a normal mapping defined on the closure of the Jordan domain
\begin{equation*}
\Delta ^{\ast }=\Delta \cup V\cup \alpha _{2}\backslash \{p_{1},p_{2}\},
\end{equation*}%
with
\begin{equation*}
A(f^{\ast },\Delta ^{\ast })=A(f_{1},\Delta )+A(T),\
\end{equation*}%
and%
\begin{eqnarray*}
L(f^{\ast },\partial \Delta ^{\ast }) &=&L(f_{1},\left( \left( \partial
\Delta \right) \backslash \alpha _{2}\right) )+L(f,\left( \partial V\right)
\backslash \alpha _{2}) \\
&=&L(f_{1},\partial \Delta )-L(f_{1},\alpha _{2})+L(f,\left( \partial
V\right) \backslash \alpha _{2}) \\
&=&L(f_{1},\partial \Delta )-L(\Gamma _{1}^{\prime \prime }+\Gamma _{2})+L(%
\overline{1,0})+L(\overline{0,1}) \\
&=&L(f_{1},\partial \Delta )-2\pi +\frac{\pi }{2}+\frac{\pi }{2} \\
&<&L(f_{1},\partial \Delta ),
\end{eqnarray*}%
and the boundary curve $\Gamma _{f^{\ast }}$ has a natural partition%
\begin{equation*}
\Gamma _{f^{\ast }}=\Gamma _{1}^{\prime }+\Gamma _{2}^{\prime }+\Gamma
_{2}^{\prime \prime }+\Gamma _{3}+\dots +\Gamma _{n},
\end{equation*}%
and we have $V_{E}(f^{\ast })=V_{E}(f_{1})+1$ and $V(f^{\ast })=V(f_{1})+1,$
and then by (\ref{13}) we have
\begin{equation*}
V_{E}(f^{\ast })\geq V_{E}(f)+2,V(f^{\ast })\leq V(f)+2.
\end{equation*}

Considering that $V(f^{\ast })=V_{E}(f^{\ast })+V_{NE}(f^{\ast })$ and
regarding $\Delta ^{\ast }$ as a disk, we obtained the desired mapping $%
g=f^{\ast }$ that satisfies (i)--(iv)$.$ The proof is complete for the case $%
k_{0}=2$ under the assumption (c).

Then we have in fact prove the lemma in the case $k_{0}=2.$

Now, assume that for some positive integer $m$ with $2\leq m<n=V(f),$ Lemma %
\ref{bd-bd>pi} holds true for all $k_{0}$ with $2\leq k_{0}\leq m.$ We prove
that Lemma \ref{bd-bd>pi} holds true for $k_{0}=m+1.$

To prove the lemma for $k_{0}=m+1,$ it is suffices to prove the lemma under
the assumption (c).

By the assumption (c), $\Gamma _{1}^{\prime \prime }+\Gamma _{2}$ is still a
natural edge of $\Gamma _{f_{1}}$ and since $k_{0}=m+1\geq 3$ we have, by
(b) and (\ref{12}), that%
\begin{equation*}
L(\Gamma _{1}^{\prime \prime }+\Gamma _{2})<2\pi ,
\end{equation*}%
the initial point of $\Gamma _{1}^{\prime \prime }+\Gamma _{2}$ is in $E$,
and (\ref{13}) still holds.

If $L(\Gamma _{1}^{\prime \prime }+\Gamma _{2})<\pi ,$ then $f_{1}$ also
satisfies all assumptions of Lemma \ref{ini-1-pi} and by (\ref{13}) we can
again deform $f_{1}$ to be another normal mapping $g$ such that (i)--(iv)
hold.

If
\begin{equation*}
\pi \leq L(\Gamma _{1}^{\prime \prime }+\Gamma _{2})<2\pi ,
\end{equation*}%
then, considering that by (\ref{11}) $\Gamma _{f_{1}}$ also has the
following natural partition%
\begin{equation*}
\Gamma _{f_{1}}=\left( \Gamma _{1}^{\prime \prime }+\Gamma _{2}\right)
+\Gamma _{3}+\dots +\Gamma _{n}+\Gamma _{1}^{\prime },
\end{equation*}%
$f_{1}$ satisfies all the assumption of Lemma \ref{bd-bd>pi}, with $k_{0}=m.$
Then by the induction hypothesis, the proof is complete.
\end{proof}

\begin{lemma}
\label{no-end-pt}Let $f:\overline{\Delta }\rightarrow S$ be a normal mapping
and let
\begin{equation*}
\Gamma _{f}=\Gamma _{1}+\dots +\Gamma _{n},n=V(f),
\end{equation*}%
be a natural partition of $\Gamma _{f}=f(z),z\in \partial \Delta .$ Assume
that the following hold.

(a) $\pi \leq L(\Gamma _{1})<2\pi $ but $L(\Gamma _{j})<\pi $ for all $%
j=2,\dots ,n.$

(b) The two endpoints of $\Gamma _{1}$ are outside $E.$

Then, there exists a normal mapping $g:\overline{\Delta }\rightarrow S$ such
that

(i) $L(g,\partial \Delta )\leq L(f,\partial \Delta )$ and $A(g,\Delta )\geq
A(f,\Delta )$

(ii) Each natural edge of $g$ has spherical length strictly less than $\pi ,$

(iii) $V_{NE}(g)\leq V_{NE}(f)\ \mathrm{and\ }V_{E}(g)\geq V_{E}(f)+1.$

(iv) $V(g)\leq V(f)+2.$
\end{lemma}

\begin{proof}
Let $C$ be the great circle determined by $\Gamma _{1}.$ Then there are two
cases:

\noindent \textbf{Case 1.} $C\cap E\neq \emptyset .$

\noindent \textbf{Case 2.} $C\cap E=\emptyset .$

Assume Case 1 occurs. Then by (a)$,$ $C$ contains only one point $p_{0}$ in $%
E$ and this point must be $1.$ Otherwise, $C$ contains the antipodal points $%
0$ and $\infty ,$ and either $0$ or $\infty $ is in the interior of $\Gamma
_{1}$ by (a) and (b), which contradicts the assumption that $\Gamma _{1}$ is
a natural edge. Then $C$ must separates $0$ and $\infty ,$ without loss of
generality, assume $0$ is on the right hand side of $C.$ Let $q_{j}$ be the
initial point of $\Gamma _{j},j=1,\dots ,n.$ Then $q_{j+1}$ is the endpoint
of $\Gamma _{j},j=1,\dots ,n,$ where $q_{n+1}=q_{1}.$ Let
\begin{equation}
\Gamma _{1}^{\prime }=\overline{q_{1}0}\ \mathrm{and\ }\Gamma _{1}^{\prime
\prime }=\overline{0q_{2}}.  \label{9-8}
\end{equation}%
Then, by Lemma \ref{pre-large}, $\Gamma _{1}^{\prime }$ and $\Gamma
_{1}^{\prime \prime }$ make sense and
\begin{equation}
L(\Gamma _{1}^{\prime })<\pi ,L(\Gamma _{1}^{\prime \prime })<\pi \ \mathrm{%
and\ }L(\Gamma _{1}^{\prime })+L(\Gamma _{1}^{\prime \prime })\leq L(\Gamma
_{1}),  \label{9-4}
\end{equation}%
and $\Gamma _{1}^{\prime }+\Gamma _{1}^{\prime \prime }-\Gamma _{1}$
encloses a domain $T$ that is on the right hand side of $C$ and $(\Gamma
_{1}\cup T)\cap E=\emptyset .$

By Lemma \ref{patch}, ignoring a coordinate transform, there exists a normal
mapping $g_{1}:\overline{\Delta }\rightarrow S$, which will be regarded as
an extension of $f,$ such that $\Gamma _{g_{1}}$ has the permitted partition%
\begin{equation*}
\Gamma _{g_{1}}=\Gamma _{1}^{\prime }+\Gamma _{1}^{\prime \prime }+\Gamma
_{2}+\dots +\Gamma _{n}
\end{equation*}%
and%
\begin{equation}
L(g_{1},\partial \Delta )\leq L(f,\partial \Delta ),A(g_{1},\Delta )\geq
A(f,\Delta ).  \label{125-5}
\end{equation}%
It is clear that we have%
\begin{equation}
V_{NE}(g_{1})\leq V_{NE}(f),V_{E}(g_{1})=V_{E}(f)+1,V(g_{1})\leq V(f)+1,
\label{9-7}
\end{equation}%
and we can rewrite the permitted partition of $\Gamma _{g_{1}}$ as%
\begin{equation*}
\Gamma _{g_{1}}=\Gamma _{n}+\Gamma _{1}^{\prime }+\Gamma _{1}^{\prime \prime
}+\Gamma _{2}+\dots +\Gamma _{n-1}.
\end{equation*}

Then there are three cases:

\noindent \textbf{Case 1.1.} Both $\Gamma _{1}^{\prime }$ and $\Gamma
_{1}^{\prime \prime }$ are natural edges of $\Gamma _{g_{1}}.$

\noindent \textbf{Case 1.2.} One of $\Gamma _{1}^{\prime }$ and $\Gamma
_{1}^{\prime \prime }$ is a natural edge, while the other is not.

\noindent \textbf{Case 1.3.} Neither $\Gamma _{1}^{\prime }$ nor $\Gamma
_{1}^{\prime \prime }$ is a natural edge.

In Case 1.1, it is clear that $g=g_{1}$ satisfies all the desired
conclusions with
\begin{equation}
V_{NE}(g_{1})=V_{NE}(f),\mathrm{\ }V_{E}(g_{1})=V_{E}(f)+1,V(g_{1})=V(f)+1.
\label{125-9}
\end{equation}

Assume Case 1.2 occurs. Without loss of generality, assume that $\Gamma
_{1}^{\prime \prime }$ is a natural edge. Then $\Gamma _{g_{1}}$ has the
natural partition%
\begin{equation*}
\Gamma _{g_{1}}=\left( \Gamma _{n}+\Gamma _{1}^{\prime }\right) +\Gamma
_{1}^{\prime \prime }+\Gamma _{2}+\dots +\Gamma _{n-1}
\end{equation*}%
where $\left( \Gamma _{n}+\Gamma _{1}^{\prime }\right) \ $is a natural edge$%
, $ and (\ref{9-7}) becomes
\begin{equation}
V_{NE}(g_{1})=V_{NE}(f)-1,\mathrm{\ }V_{E}(g_{1})=V_{E}(f)+1\ \mathrm{and\ }%
V(g_{1})=V(f).  \label{128-1}
\end{equation}

Then, in the case $L(\Gamma _{n}+\Gamma _{1}^{\prime })<\pi ,$ by (a) and (%
\ref{9-4}), $g=g_{1}$ satisfies (i)--(iv) with (\ref{128-1}); and in the
case $L(\Gamma _{n}+\Gamma _{1}^{\prime })\geq \pi ,$ by (a), (\ref{9-8})
and (\ref{9-4}), $\pi \leq L(\Gamma _{n}+\Gamma _{1}^{\prime })<2\pi $ and $%
g_{1}$ satisfies the assumption of Lemma \ref{ini-1-pi} with (\ref{128-1}),
and then, by (\ref{125-5}), there exists a normal mapping $g:\overline{%
\Delta }\rightarrow S$ that satisfies (i) and (ii), and
\begin{equation*}
V_{NE}(g)\leq V_{NE}(g_{1}),\mathrm{\ }V_{E}(g)\geq V_{E}(g_{1})+1\ \mathrm{%
and\ }V(g)\leq V(g_{1})+1,
\end{equation*}%
and so by (\ref{128-1}), (iii) and (iv) are satisfied by $g\ $with
\begin{equation}
V_{NE}(g)\leq V_{NE}(f)-1,\mathrm{\ }V_{E}(g)\geq V_{E}(f)+2\ \mathrm{and\ }%
V(g)\leq V(f)+1.  \label{125-11}
\end{equation}

Assume Case 1.3 occurs. Then both $\Gamma _{n}+\Gamma _{1}^{\prime }$ and $%
\Gamma _{1}^{\prime \prime }+\Gamma _{2}$ are natural edges of $g_{1},$ $%
\Gamma _{g_{1}}$ has the natural partition
\begin{equation*}
\Gamma _{g_{1}}=\left( \Gamma _{n}+\Gamma _{1}^{\prime }\right) +\left(
\Gamma _{1}^{\prime \prime }+\Gamma _{2}\right) +\Gamma _{3}+\dots +\Gamma
_{n-1}
\end{equation*}%
and (\ref{9-7}) becomes
\begin{equation}
V_{NE}(g_{1})=V_{NE}(f)-2,V_{E}(g_{1})=V_{E}(f)+1,V(g_{1})=V(f)-1.
\label{10501}
\end{equation}%
By (a) and (\ref{9-4}) we have
\begin{equation}
L(\Gamma _{n}+\Gamma _{1}^{\prime })<2\pi \ \mathrm{and\ }L(\Gamma
_{1}^{\prime \prime }+\Gamma _{2})<2\pi .  \label{9-10}
\end{equation}

Then, by (a), in the case
\begin{equation}
L(\Gamma _{n}+\Gamma _{1}^{\prime })<\pi \ \mathrm{and\ }L(\Gamma
_{1}^{\prime \prime }+\Gamma _{2})<\pi ,  \label{9-6}
\end{equation}%
$g=g_{1}$ is the desired mapping satisfying (i)--(iv) with (\ref{10501});
and in the case that (\ref{9-6}) fails, by (a), (\ref{9-8}) and (\ref{9-10}%
), Lemma \ref{ini-1-pi} or Lemma \ref{bd-bd>pi} applies, and then there
exists a normal mapping $g$ satisfies (i), (ii) and%
\begin{equation*}
V_{NE}(g)\leq V_{NE}(g_{1}),V_{E}(g)\geq V_{E}(g_{1})+1,V(g)\leq V(g_{1})+2,
\end{equation*}%
which, with (\ref{10501}), implies%
\begin{equation*}
V_{NE}(g)\leq V_{NE}(f),V_{E}(g)\geq V_{E}(f)+1,V(g)\leq V(f)+1,
\end{equation*}
i.e. (iii) and (iv) hold. This completes the proof in Case 1.3.

Now, assume Case 2 occurs. Then the hemisphere $S^{\prime }$ outside $C$
contains one or two points of $E.$ If $S^{\prime }$ contains only one point
of $E,$ the proof is exactly the same as the above arguments. So, we assume
that $S^{\prime }$ contains two points $q_{0}$ and $q_{0}^{\prime }$ of $E.$
Then either $\{q_{0},q_{0}^{\prime }\}=\{0,1\}$ or $\{1,\infty \},$ and then
there are two cases:


\noindent \textbf{Case 2.1. }The great circle of $S$ containing $q_{0}$ and $%
q_{0}^{\prime }$ intersects $C\backslash \Gamma _{1}$ $.$

\noindent \textbf{Case 2.2. }The great circle containing $q_{0}$ and $%
q_{0}^{\prime }$ does not intersects $C\backslash \Gamma _{1}$$.$

In Case 2.1, the argument for Case 1 exactly applies.

In Case 2.2, it is easy to show that the exists two points $r_{1}$ and $%
r_{1}^{\prime }$ on $\Gamma _{1}$ such that $r_{1}$ close to $q_{1}$ and $%
r_{1}^{\prime }$ close to $q_{2}$ (in $\Gamma _{1}),$ and $%
r_{1},q_{0},q_{0}^{\prime },r_{1}^{\prime }$ or $r_{1},q_{0}^{\prime
},q_{0},r_{1}^{\prime }$ are in order on the geodesic path from $r_{1}$ to $%
r_{1}^{\prime }$ in $S^{\prime },$ (then $r_{1}$ and $r_{1}^{\prime }$ are
antipodal). We assume $r_{1},q_{0},q_{0}^{\prime },r_{1}^{\prime }$ is
ordered in the orientation of the geodesic path from $r_{1}$ to $%
r_{1}^{\prime }$ in $S^{\prime }.$ It is clear that the notations
\begin{equation*}
\Gamma _{1}^{\prime }=\overline{q_{1}q_{0}},\gamma =\overline{%
q_{0}q_{0}^{\prime }},\Gamma _{1}^{\prime \prime }=\overline{q_{0}^{\prime
}q_{2}}.
\end{equation*}%
make sense. Then%
\begin{equation}
L(\Gamma _{1}^{\prime })<\pi ,L(\Gamma _{1}^{\prime \prime })<\pi ,L(\gamma
)=\frac{\pi }{2},  \label{18}
\end{equation}%
and it is also clear that%
\begin{equation}
L(\Gamma _{1}^{\prime })+L(\gamma )+L(\Gamma _{1}^{\prime \prime })=L(\Gamma
_{1}^{\prime })+\frac{\pi }{2}+L(\Gamma _{1}^{\prime \prime })<L(\Gamma
_{1})<2\pi ,  \label{19}
\end{equation}%
and $\Gamma _{1}^{\prime }+\gamma +\Gamma _{1}^{\prime \prime }-\Gamma _{1}$
encloses a Jordan domain $T$ in $S^{\prime }$ with $\left( \Gamma _{1}\cup
T\right) \cap E=\emptyset .$

By Lemma \ref{patch}, there exists a normal mapping $g_{1}$, such that $%
\Gamma _{g_{1}}$ has the permitted partition%
\begin{equation*}
\Gamma _{g_{1}}=\Gamma _{1}^{\prime }+\gamma +\Gamma _{1}^{\prime \prime
}+\Gamma _{2}+\dots +\Gamma _{n}
\end{equation*}%
and by (\ref{19}),%
\begin{equation*}
L(g_{1},\partial \Delta )\leq L(f,\partial \Delta ),A(f,\Delta )\geq
A(g_{1},\Delta ).
\end{equation*}%
It is clear that we have%
\begin{equation}
V_{NE}(g_{1})\leq V_{NE}(f),V_{E}(g_{1})=V_{E}(f)+2\ \mathrm{and\ }%
V(g_{1})\leq V(f)+2  \label{14}
\end{equation}%
and we rewrite the permitted partition of $\Gamma _{g_{1}}$ as%
\begin{equation}
\Gamma _{g_{1}}=\Gamma _{n}+\Gamma _{1}^{\prime }+\gamma +\Gamma
_{1}^{\prime \prime }+\Gamma _{2}+\dots +\Gamma _{n-1}.  \label{15}
\end{equation}%
Note that $\gamma $ is always a natural edge of $\Gamma _{g_{1}},$ because
the endpoints of $\gamma $ are both in $E.$

Then there are three cases:

\noindent \textbf{Case 2.2.1. }Both $\Gamma _{1}^{\prime }$ and $\Gamma
_{1}^{\prime \prime }$ are natural edges of $\Gamma _{g_{1}}$.

\noindent \textbf{Case 2.2.2. }One of $\Gamma _{1}^{\prime }$ and $\Gamma
_{1}^{\prime \prime }$ is a natural edge, while the other is not.

\noindent \textbf{Case 2.2.3. }Neither $\Gamma _{1}^{\prime }$ nor $\Gamma
_{1}^{\prime \prime }$ is a natural edge.

In Case 2.2.1, (\ref{15}) is a natural partition, and by (\ref{18}), $%
g=g_{1} $ satisfies (i)--(iv) with (\ref{14}).

In Case 2.2.2, we may assume $\Gamma _{1}^{\prime }$ is a natural edge, and
then by (\ref{15}), $\Gamma _{g_{1}}$ has the natural partition
\begin{equation*}
\Gamma _{g_{1}}=\Gamma _{n}+\Gamma _{1}^{\prime }+\gamma +\left( \Gamma
_{1}^{\prime \prime }+\Gamma _{2}\right) +\Gamma _{3}\dots +\Gamma _{n-1}.
\end{equation*}%
Then, by (\ref{14}) and (\ref{15}),%
\begin{equation}
V_{NE}(g_{1})\leq V_{NE}(f),V_{E}(g_{1})=V_{E}(f)+2\ \mathrm{and\ }%
V(g_{1})=V(f)+1.  \label{14+1}
\end{equation}%
Then, by (a) and (\ref{18}), in the case $L\left( \Gamma _{1}^{\prime \prime
}+\Gamma _{2}\right) <\pi ,$ $g=g_{1}$ satisfies (i)--(iv) with (\ref{14}),
and otherwise, Lemma \ref{ini-1-pi} applies to $g_{1},$ and then there
exists a normal mapping $g:\overline{\Delta }\rightarrow S$ satisfying
(i)--(iv), by (\ref{14+1}), with%
\begin{equation}
\left\{
\begin{array}{l}
V_{NE}(g)\leq V_{NE}(g_{1})\leq V_{NE}(f), \\
V_{E}(g)\geq V_{E}(g_{1})+1=V_{E}(f)+3, \\
V(g)\leq V(g_{1})+1=V(f)+2.%
\end{array}%
\right.  \label{16}
\end{equation}

In Case 2.2.3, $\Gamma _{n}+\Gamma _{1}^{\prime }$ and $\Gamma _{1}^{\prime
\prime }+\Gamma _{2}$ are two natural edges of $\Gamma _{g_{1}}$ with%
\begin{equation*}
L(\Gamma _{n}+\Gamma _{1}^{\prime })<2\pi \text{ and }L(\Gamma _{1}^{\prime
\prime }+\Gamma _{2})<2\pi ;
\end{equation*}%
$\Gamma _{g_{1}}$ has the natural partition%
\begin{equation*}
\Gamma _{g_{1}}=\left( \Gamma _{n}+\Gamma _{1}^{\prime }\right) +\gamma
+\left( \Gamma _{1}^{\prime \prime }+\Gamma _{2}\right) +\Gamma _{3}+\dots
+\Gamma _{n-1}.
\end{equation*}%
and (\ref{14}) becomes
\begin{equation}
V_{NE}(g_{1})\leq V_{NE}(f),V_{E}(g_{1})=V_{E}(f)+2\ \mathrm{and\ }%
V(g_{1})=V(f).  \label{9-5}
\end{equation}

Then, by (a) and (\ref{18}), in the case $L(\Gamma _{n}+\Gamma _{1}^{\prime
})<\pi $ and $L(\Gamma _{1}^{\prime \prime }+\Gamma _{2})<\pi ,$ $g=g_{1}$
satisfies (i)--(iv) with (\ref{9-5}), and in other cases, Lemma \ref%
{ini-1-pi} or Lemma \ref{bd-bd>pi} applies to $g_{1},$ and then there exists
a normal mapping $g:\overline{\Delta }\rightarrow S$ satisfying (i)--(iv)
with%
\begin{equation}
\left\{
\begin{array}{l}
V_{NE}(g)\leq V_{NE}(g_{1})\leq V_{NE}(f), \\
V_{E}(g)\geq V_{E}(g_{1})+1=V_{E}(f)+3, \\
V(g)\leq V(g_{1})+2=V(f)+2.%
\end{array}%
\right.  \label{17}
\end{equation}%
This completes the proof.
\end{proof}

\begin{lemma}
\label{>=2pi}Let $f:\overline{\Delta }\rightarrow S$ be a normal mapping,
and let
\begin{equation}
\Gamma _{f}=\Gamma _{1}+\Gamma _{2}+\dots +\Gamma _{n},n=V(f),  \label{25}
\end{equation}%
be a natural partition of $\Gamma _{f}=f(z),z\in \partial \Delta .$ Assume
that
\begin{equation}
2\pi \leq L(\Gamma _{1})<3\pi ,  \label{27}
\end{equation}%
and
\begin{equation}
L(\Gamma _{j})<\pi \ \mathrm{for\ }j=2,\dots ,n.  \label{28}
\end{equation}

Then, there exists a normal mapping $g:\overline{\Delta }\rightarrow S$ such
that

(i) $L(g,\partial \Delta )\leq L(f,\partial \Delta )$ and $A(g,\Delta )\geq
A(f,\Delta )$

(ii) Each natural edge of $g$ has spherical length strictly less than $\pi ,$

(iii) $V_{NE}(g)=V_{NE}(f)+2$, $V_{E}(g)=V_{E}(f)+1$ and $V(g)=V(f)+3.$
\end{lemma}

\begin{proof}
Let $C$ be the great circle determined by $\Gamma _{1}.$ Then by (\ref{27})
we have $C\subset \Gamma _{1},$ and by the definition of natural edges, in
the case that $L(\Gamma _{1})=2\pi ,$ the only possible point of $%
E=\{0,1,\infty \}$ contained in $C$ is $1,$ and in the case that $2\pi
<L(\Gamma _{1})<3\pi ,$ $C$ does not intsects $E,$ for otherwise the
interior $\alpha _{1}^{\circ }$ of $\alpha _{1}$ contains at least one point
of $f^{-1}(E).$

Let $S^{\prime }$ be the hemisphere outside\footnote{%
This means that $S^{\prime }$ is on the right hand side of $C.$ Note that $C$
is oriented by $\Gamma _{1}.$} $C$. Then%
\begin{equation}
1\leq \#\left( S^{\prime }\cap E\right) \leq 2.  \label{26}
\end{equation}%
Let
\begin{equation*}
\partial \Delta =\alpha _{1}+\dots +\alpha _{n}
\end{equation*}%
be a natural partition of $\partial \Delta $ corresponding the partition (%
\ref{25}) and let $p_{1}$ and $p_{4}$ be the initial and terminal point of $%
\alpha _{1}$, respectively$.$ Then by (\ref{27}) and (\ref{26}), summarizing
what we have, there exists $p_{2}$ and $p_{3}$ in the interior of $\alpha
_{1}$ such that $p_{1},p_{2},p_{3},p_{4}$ are in order anticlockwise and the
followings hold.

(a) $f$ restricted to each section $\alpha _{j}^{\prime }$ of $\alpha _{1}$
from $p_{j}$ and $p_{j+1}$ is a homeomorphism, $j=1,2,3.$

(b) For the sections $\Gamma _{j}^{\prime }=f(z),z\in \alpha _{j},j=1,2,3$
\begin{equation*}
L(\Gamma _{2}^{\prime })=\pi ,\ L(\Gamma _{1}^{\prime })<\pi ,\ L(\Gamma
_{3}^{\prime })<\pi .
\end{equation*}

(c) Any shortest path from $q_{2}\mathrm{\ }$to $q_{3}\ $contains at most
one point of $E$.

Then by the definition of natural edges,
\begin{equation*}
\Gamma _{2}^{\prime }\cap E=\emptyset ,
\end{equation*}%
and by (b), $q_{2}$ and $q_{3}\ $are antipodal. Therefore, by (\ref{26}) and
(c), there exists a unique shortest path $L$ from $q_{2}$ to $q_{3}$ such
that $L-\Gamma _{2}^{\prime }$ enclose a domain $T$ such that $\overline{T}%
\cap E$ contains exactly one point $q\in E$, which lies in $L\cap S^{\prime
}.$ We denote by $\Gamma ^{\prime }$ the section of $L$ from $q_{2}=f(p_{2})$
to $q=f(q)$ and by $\Gamma ^{\prime \prime }$ the section of $L$ from $q$ to
$q_{3}=f(p_{3}).$

Then we can extend the Riemann surface of $f$ to be a new Riemann surface so
that in the new Riemann surface, $\overline{T}$ is patched along $\Gamma
_{2}^{\prime }.$ By Lemma \ref{patch}, this can be realized by a normal
mapping $g:\overline{\Delta }\rightarrow S.$ Then the boundary curve $\Gamma
_{g}=g(z),\in \partial \Delta ,$ has the following natural partition%
\begin{equation}
\Gamma _{g}=\Gamma _{1}^{\prime }+\Gamma ^{\prime }+\Gamma ^{\prime \prime
}+\Gamma _{3}^{\prime }+\Gamma _{2}+\dots +\Gamma _{n},  \label{29}
\end{equation}%
because $\Gamma _{1}^{\prime },\Gamma _{3}^{\prime }$ is in $\Gamma _{1}$
and $\Gamma ^{\prime }$ and $\Gamma ^{\prime \prime }$ are clearly natural
edges.

It is clear that $\Gamma _{g}$ satisfies (ii) and%
\begin{equation}
L\left( \Gamma _{1}^{\prime }\right) +L\left( \Gamma ^{\prime }\right)
+L\left( \Gamma ^{\prime \prime }\right) +L\left( \Gamma _{3}^{\prime
}\right) =L\left( \Gamma _{1}^{\prime }\right) +L(\Gamma _{2}^{\prime
})+L\left( \Gamma _{3}^{\prime }\right) =L(f,\alpha _{1}),  \label{30}
\end{equation}%
and then by (\ref{29}) we have%
\begin{equation*}
L(g,\partial \Delta )=L(f,\partial \Delta ).
\end{equation*}%
On the other hand, it is also clear that
\begin{equation*}
A(g,\Delta )=A(f,\Delta )+A(T)>A(f,\Delta ).
\end{equation*}%
Thus, $g$ satisfies (i).

On the other hand, by (\ref{29}), considering that all the natural vertices
of $\Gamma _{f}$ are natural vertices of $\Gamma _{g}$ and $q_{2}^{\prime
},q_{3}^{\prime }$ and $q^{\prime }$ are the three new natural vertices of $%
g,$ we have%
\begin{equation*}
V_{NE}(g)=V_{NE}(f)+2,V_{E}(g)=V_{E}(f)+1,V(g)=V(f)+3.
\end{equation*}%
Thus, (iii) is satisfied by $g.$ This completes the proof.
\end{proof}

\section{Movement of branched points\label{ss-8move-br}}

This section is prepared for prove Theorem \ref{1-br-2-map}.

\begin{lemma}
\label{pre-move}Let $f:\overline{\Delta ^{+}}\rightarrow \overline{\Delta }$
be an orientation preserved open mapping that satisfies the following
conditions:

(a) f restricted to the upper half circle $\left( \partial \Delta \right)
^{+}=\{z\in \partial \Delta ;\mathrm{Im}z\geq 0\}$ is given by $f(e^{i\theta
})=e^{\phi (\theta )i},$ where $\phi $ is a strictly increasing function
defined on $[0,\pi ]$ with $\phi (0)=0$ and $\phi (\pi )=(2d+1)\pi $, where $%
d$ is a positive integer.

(b) $f$ maps the interval $[-1,-1]$ homeomorphically onto the interval $%
[-1,1].$

(c) $p_{0}\in \Delta ^{+}$ is the unique ramification point of $f$ in $%
\overline{\Delta ^{+}}.$

Then there exists an orientation preserved open mapping $g:\overline{\Delta
^{+}}\rightarrow \overline{\Delta }$ such that the followings hold.

(I) $z=0$ is the unique ramification point of $g$ in $\overline{\Delta ^{+}}$
and $g(0)=0.$

(II) $g|_{\left( \partial \Delta \right) ^{+}}=f|_{\left( \partial \Delta
\right) ^{+}}$ and $g$ restricted to the interval $[-1,1]$ is a
homeomorphism onto the interval $[-1,1].$
\end{lemma}

\begin{remark}
$g$ acts as an orientation preserved open mapping that moves the unique
ramification point $p_{0}$ of $f$ into the boundary of $\Delta ^{+}$ with
the same branched number, while none other ramification points appear.
\end{remark}

\begin{proof}
There exists an orientation preserved homeomorphism $f_{1}$ from $\overline{%
\Delta ^{-}}$ onto $\overline{\Delta ^{-}}$ such that
\begin{equation*}
f_{1}|_{[-1,1]}=f|_{[-1,1]}.
\end{equation*}

Then%
\begin{equation*}
f_{2}(z)=\left\{
\begin{array}{l}
f(z),z\in \overline{\Delta ^{+}}, \\
f_{1}(z),z\in \overline{\Delta ^{-}}\backslash \overline{\Delta ^{+}},%
\end{array}%
\right.
\end{equation*}%
is an orientation preserved $d+1$ to $1$ covering mapping from $\overline{%
\Delta }$ onto $\overline{\Delta }$ such that $p_{0}$ is the unique
ramification point of $f_{2},$ and then there exists another covering
mapping $f_{3}:\overline{\Delta }\rightarrow \overline{\Delta }$ such that
\begin{equation*}
f_{3}|_{\partial \Delta }=f_{2}|_{\partial \Delta }
\end{equation*}%
and that $0$ is the unique ramification point of $f_{3}$ with $f_{3}(0)=0.$

It is clear that the path $\beta =\beta (t),t\in \lbrack -1,1]$ in $%
\overline{\Delta }$ has a unique lift $\alpha =\alpha (t),t\in \lbrack
-1,1], $ in $\overline{\Delta }$ such that
\begin{equation*}
\alpha (-1)=-1,\alpha (1)=1,
\end{equation*}
$f_{3}$ restricted to $\alpha $ is a homeomorphism with%
\begin{equation*}
f_{3}(\alpha (t))=\beta (t)=t,t\in \lbrack -1,1],
\end{equation*}%
and $\alpha $ is a Jordan path and the interior of $\alpha $ is contained in
$\Delta .$ Thus $\alpha $ divides $\Delta $ into two Jordan domains and one
of these domains is enclosed by $\left( \partial \Delta \right) ^{+}$ and $%
\alpha ,$ and we denote this domain by $U^{+}.$ Then $f_{3}$ restricted to $%
\overline{\Delta ^{+}\backslash U^{+}}$ is a homeomorphism.

Now consider the restriction $f_{3}|_{\overline{U^{+}}}.$ Let $h$ be a
homeomorphism from $\overline{\Delta ^{+}}$ onto $\overline{U^{+}}$ such
that $h$ restricted to $\left( \partial \Delta \right) ^{+}$ is an identity
mapping, restricted to the interval $[-1,1]$ is a homeomorphism onto $\alpha
$ with $h(0)=0$ (note that $0\in \alpha ),$ and finally let%
\begin{equation*}
g=f_{3}\circ h(z),z\in \overline{\Delta ^{+}}.
\end{equation*}%
Then $g$ satisfies all the desired conditions.
\end{proof}

\begin{remark}
\label{first}By the proof we may construct that $g$ such that $g(0)=t$ for
any fixed $t\in (-1,1)$, $0$ is the unique ramification point of $g$ and all
other conclusions hold.
\end{remark}

\begin{lemma}
\label{move}Let $f:\overline{\Delta }\rightarrow S$ be a normal mapping, let
$p_{0}\in \Delta $ be a ramification point of $f$ with $v_{f}(p_{0})=d+1,$
and assume that $\beta =\beta (t),t\in \lbrack 0,1],$ is a polygonal Jordan
path in $S$ that satisfies the followings:

(a) $\beta (0)=f(p_{0}),$ $\beta (0)\neq \beta (1)$ and $\beta $ has a
number of $d+1$ lifts $\alpha _{j}=\alpha _{j}(t),t\in \lbrack 0,1],$ by $f$
in $\overline{\Delta },$ such that
\begin{equation*}
\cup _{j=2}^{d+1}\alpha _{j}\subset \Delta ,
\end{equation*}%
\begin{equation*}
f(\alpha _{j}(t))=\beta (t),t\in \lbrack 0,1],j=1,\dots ,d+1,
\end{equation*}%
and
\begin{equation*}
\alpha _{j}(0)=p_{0},\mathrm{\ }j=1,\dots ,d+1,
\end{equation*}

(b) $\alpha _{1}(t)\in \Delta $ for all $t\in \lbrack 0,1)\ $but $%
p_{1}=\alpha _{1}(1)\in \partial \Delta .$

(c) $f$ has no ramification point on $\cup _{j=1}^{d+1}\alpha _{j}(0,1],$
where $\alpha _{j}(0,1]$ is the curve $\alpha _{j}(t),t\in (0,1],$ which is
the curve $\alpha _{j}$ without initial point, $j=1,2,\dots d+1.$

(d) $f$ restricted to a neighborhood of $p_{1}=\alpha _{1}(1)$ in $\overline{%
\Delta }$ is a homeomorphism.\label{If folded then convex it ok}

Then, there exist a normal mapping $g:\overline{\Delta }\rightarrow S$ such
that
\begin{equation*}
A(g,\Delta )=A(f,\Delta ),L(g,\partial \Delta )=L(f,\partial \Delta ),
\end{equation*}%
and the followings hold.

(I) The ramification point $p_{0}$ of $f$ is no longer a ramification point
of $g,$ while the regular point $p_{1}$ of $f$ is a regular point of $g$ with%
\begin{equation*}
g(p_{1})=\beta (1)=f(p_{1})\mathrm{\ and\ }v_{f}(p_{0})=v_{g}(p_{1}).
\end{equation*}

(II) The boundary curves $\Gamma _{f}=f(e^{i\theta }),\theta \in \lbrack
0,2\pi ],$ and $\Gamma _{g}=g(e^{i\theta }),\theta \in \lbrack 0,2\pi ],$
are the same curves after a parameter transformation.

(III) The ramification point sets of $f$ and $g$ in $\overline{\Delta }%
\backslash \{p_{0},p_{1}\}$ are the same, and $f$ and $g$ coincide in a
neighborhood of this ramification point set.
\end{lemma}

\label{If folded, then first convex it}.

\begin{remark}
$g$ acts as a normal mapping that moves the ramification point $p_{0}$ of $f$
into the boundary $\partial \Delta $ with the same branched number, while
all other ramification points, as well as their branched number, remain
unchanged, and no other new ramification point appear, and the length and
the area also remains unchanged.
\end{remark}

\begin{proof}
Let $\delta <\frac{1}{2}$ be a positive number, let $D_{\delta }$ be the
disk $|z-p_{1}|<\delta $ in $\mathbb{C},$ let $c=\left( \partial \Delta
\right) \cap \overline{D_{\delta }}$, which will be regarded as a path from $%
p_{2}\in \partial \Delta $ to $p_{3}\in \partial \Delta $ (anticlockwise).
Then $p_{1}$ is the middle point of $c$. We write
\begin{equation*}
q_{j}=f(p_{j}),j=1,2,3;
\end{equation*}%
\begin{equation*}
D_{\delta }^{+}=D_{\delta }\backslash \overline{\Delta },
\end{equation*}%
\begin{equation*}
\Delta ^{\ast }=\Delta \cup D_{\delta }\cup c\backslash \{p_{2},p_{3}\};
\end{equation*}%
\begin{equation*}
\gamma =f(c);
\end{equation*}%
\begin{equation*}
e=\partial \Delta ^{\ast }\cap \overline{D_{\delta }}=\left( \partial
D_{\delta }\right) \backslash \Delta ,
\end{equation*}%
which is the boundary of $\Delta ^{\ast }$ outside $\Delta ,$ and is the
boundary of $\partial D_{\delta }$ outside $\Delta $ as well. It is clear
that $\Delta ^{\ast }$ is a Jordan domain (note that $\delta <\frac{1}{2})$
and
\begin{equation*}
\partial \Delta ^{\ast }=e\cup \left( \left( \partial \Delta \right)
\backslash c\right) ,\ e\cap \overline{\left( \partial \Delta \right)
\backslash c}=\{p_{2},p_{3}\}.
\end{equation*}

Since $f$ is normal and $\alpha _{2}\subset \Delta $ (by (a)), we have that
\begin{equation*}
f(p_{1})=f(\alpha _{1}(1))=f(\alpha _{2}(1))\notin E.
\end{equation*}%
On the other hand, by (c) and (d), we may take the number $\delta $
sufficiently small such that the following conditions (e)--(f) are satisfied.

(e) $f$ restricted to $c$ is a homeomorphism onto $\gamma =f(c)=\overline{%
q_{2}q_{1}q_{3}},$ i.e., $\gamma $ is a polygonal Jordan path with only one
possible vertex at $q_{1}.$

(f) There exists a point $q_{1}^{\prime }$ in $S\backslash \gamma $ such
that $q_{1}^{\prime }$ is very close to $\gamma $ and is on the right hand
side of $\gamma ,$ and the quadrangle $\overline{q_{2}q_{1}^{\prime
}q_{3}q_{1}q_{2}}$ encloses a domain $T$ that is on the right hand side of $%
\gamma ,$ with $\overline{T}\cap E=\emptyset ,$ and $f^{-1}(\overline{T})$
has $d$ components $A_{j}$ with $f(\alpha _{j}(1)\in A_{j}$ and $f$
restricted to each $A_{j}$ is a homeomorphism onto $\overline{T},$ for $%
j=2,\dots ,d+1.$

(g) $\beta $ intersects $\overline{T}$ only at $q_{1}=f(p_{1}).$

Let $f_{1}$ be an orientation preserved homeomorphism from $\overline{%
D_{\delta }^{+}}$ onto $\overline{T}$ such that $f_{1}$ and $f$ restricted
to $c$ are equal to each other. Then
\begin{equation*}
f_{2}(z)=\left\{
\begin{array}{l}
f(z),z\in \overline{\Delta }, \\
f_{1}(z),z\in \overline{D_{\delta }^{+}}\backslash \overline{\Delta },%
\end{array}%
\right.
\end{equation*}%
is a normal mapping defined on $\overline{\Delta ^{\ast }}$.

The above argument show that $T$ can be extended to be a polygonal Jordan
domain $T^{\ast }$ such that the followings hold.

(h) $\beta \subset T^{\ast }$, the path $\gamma ^{\prime }=\overline{%
q_{2}q_{1}^{\prime }q_{3}}$ is still a section of $\partial T^{\ast },$ and
\begin{equation*}
(\gamma \cup T)\backslash \{q_{2},q_{3}\}\subset T^{\ast }.
\end{equation*}

(i) $f_{2}$ restricted to the component $\overline{U}$ of $f_{2}^{-1}(%
\overline{T^{\ast }})$ with $p_{0}\in U$ is a $d+1$ to $1$ covering with the
unique ramification point $p_{0},$ and $f_{2}(\overline{U})=\overline{%
T^{\ast }}.$

(j) The boundary of $U$ is composed of $e$ and a Jordan path $\alpha $ in $%
\overline{\Delta }$ whose interior is in $U\ $and endpoints are $p_{2}$ and $%
p_{3}.$

Then $V=U\cap \Delta $ is also a Jordan domain. Let $h_{1}$ be a
homeomorphism from $\overline{V}$ onto $\overline{\Delta ^{+}}$ such that $%
h_{1}$ maps $\alpha $ homeomorphically onto $\left( \partial \Delta \right)
^{+},$ maps $c$ homeomorphically onto the interval $[-1,1]\ $with%
\begin{equation*}
h_{1}(p_{1})=0;
\end{equation*}%
let $h_{2}$ be a homeomorphism from $\overline{T^{\ast }}$ onto $\overline{%
\Delta }$ such that $h_{2}$ maps $\left( \partial T^{\ast }\right)
\backslash \{\gamma ^{\prime }\backslash \{q_{2},q_{3}\}\}$ homeomorphically
onto $\left( \partial \Delta \right) ^{+}$, maps $\gamma ^{\prime }$
homeomorphically onto $\left( \partial \Delta \right) ^{-},$ and maps $%
\gamma $ homeomorphically onto the interval $[-1,1]$ with%
\begin{equation*}
h_{2}(q_{1})=0;
\end{equation*}%
and finally let
\begin{equation*}
g_{1}=h_{2}\circ f_{2}|_{\overline{V}}\circ h_{1}^{-1}(\zeta ):\overline{%
\Delta ^{+}}\rightarrow \overline{\Delta }.
\end{equation*}%
Then $g_{1}$ is an orientation preserved open mapping that satisfies all the
assumptions of Lemma \ref{pre-move}, and then there exists an orientation
preserved open mapping $g_{2}:\overline{\Delta ^{+}}\rightarrow \overline{%
\Delta }$ such that the followings hold.

(k) $0$ is the unique ramification point of $g_{2}$ in $\overline{\Delta ^{+}%
}$ and $g_{2}(0)=0.$

(l) $g_{2}|_{\left( \partial \Delta \right) ^{+}}=f|_{\left( \partial \Delta
\right) ^{+}}$ and both $f$ and $g_{2}$ restricted to the interval $[-1,1]$
are homeomorphisms onto the interval $[-1,1].$

Let
\begin{equation*}
g_{3}=h_{2}^{-1}\circ g_{2}\circ h_{1}(z),z\in \overline{V}.
\end{equation*}%
Then $g_{3}$ restricted to a neighborhood of $\alpha $ in $\overline{V}$ is
a homeomorphism, $g_{3}$ maps $c$ homeomorphically onto $\gamma $ and $g_{3}$
restricted to $\alpha $ equals the restriction of $f$ to $\alpha $ and
\begin{equation*}
A(g_{3},V)=\left( d+1\right) A(T^{\ast })-A(T)=A(f,U)-A(f,D_{\delta
}^{+})=A(f,V).
\end{equation*}

Now,
\begin{equation*}
g(z)=\left\{
\begin{array}{l}
f(z),z\in \overline{\Delta }\backslash \overline{V}, \\
g_{3}(z),z\in \overline{V},%
\end{array}%
\right.
\end{equation*}%
is the desired mapping.
\end{proof}

\begin{lemma}
\label{move-bd}Let
\begin{equation*}
\alpha _{1}=\alpha _{1}(\theta )=e^{i\theta },\theta \in \lbrack \theta
_{1},\theta _{2}]
\end{equation*}%
with $\theta _{1}<\theta _{2}<\theta _{1}+2\pi ,$ be a section of $\partial
\Delta $ and let%
\begin{equation*}
p_{j}=\alpha _{1}(e^{i\theta _{j}}),j=1,2;
\end{equation*}%
let $f:\overline{\Delta }\rightarrow S$ be a normal mapping such that $p_{1}$
is a ramification point of $f$ with
\begin{equation*}
v_{f}(p_{1})=d.
\end{equation*}%
Assume that the section
\begin{equation*}
\beta =\beta (\theta )=f(e^{i\theta }),\theta \in \lbrack \theta _{1},\theta
_{2}],
\end{equation*}%
of $\Gamma _{f}=f(z),z\in \partial \Delta ,\ $is a Jordan path with
\begin{equation*}
\beta \cap E=\emptyset ,
\end{equation*}%
and $\beta $ has $d=v_{f}(p_{1})$ distinct lifts%
\begin{equation*}
\alpha _{j}=\alpha _{j}(\theta ),\theta \in \lbrack \theta _{1},\theta
_{2}],j=1,\dots ,d
\end{equation*}%
in $\overline{\Delta }$ by $f,$ such that

(a) For each $j=1,\dots ,d,$ $f(\alpha _{j}(\theta ))=f(\alpha _{1}(\theta
))=\beta (\theta )$ for $\theta \in \lbrack \theta _{1},\theta _{2}]\ $and $%
\alpha _{j}(\theta _{1})=p_{1}.$

(b) For each $j=2,\dots ,d,$ $\alpha _{j}(\theta )\in \Delta $\ for $\theta
\in (\theta _{1},\theta _{2}].$

(c) There is no ramification point of $f$ in $\cup _{j=1}^{d}\alpha _{j}$
other than $p_{1}.$

(d) $f$ restricted to a neighborhood of $p_{j}$ in $\partial \Delta $ is a
homeomorphism, for $j=1,2.$

Then, there exists a normal mapping $g:\overline{\Delta }\rightarrow S$ such
that%
\begin{equation*}
A(g,\Delta )=A(f,\Delta ),L(g,\partial \Delta )=L(f,\partial \Delta ),
\end{equation*}%
and the followings hold.

(I) The ramification point $p_{1}=e^{i\theta _{1}}$ of $f$ is no longer a
ramification point of $g,$ while the regular point $p_{2}=e^{i\theta _{2}}$
of $f$ is a ramification point of $g$ with $g(p_{2})=\beta (\theta
_{2})=f(p_{2})$ and
\begin{equation*}
b_{f}(p_{1})=b_{g}(p_{2}).
\end{equation*}

(II) The boundary curves $f(e^{i\theta })$ and $g(e^{i\theta })$ are the
same after a parameter transform$.$

(III) In $\overline{\Delta }\backslash \{p_{1},p_{2}\},$ $f$ and $g$ has the
same set of ramification points and $f$ and $g$ coincide in a neighborhood
of this ramification point set.
\end{lemma}

\label{If folded, then convex it, then ok.}

\begin{proof}
There are two ways to prove this lemma. One way is to use Remark \ref{first}%
. Here we use Lemma \ref{move} to give another proof.

We will first construct a normal mapping $f_{2}$ that is defined on some
closed Jordan domain $\overline{\Delta ^{\prime }}\ni p_{2}$ such that the
length and the area concerned in the lemma unchanged, the boundary curve $%
\Gamma _{f_{2}}$ of $f_{2}$ is the same as that of $f,$ $f$ and $f_{2}$ have
the same set $B$ of ramification points in $\overline{\Delta }\backslash
\{p_{1},p_{2}\}$, $f_{2}$ and $f$ coincide in a neighborhood of this
ramification point set, and $f_{2}$ has only one more ramification point $%
p_{1}^{\prime }$ outside $B,$ while $p_{1}^{\prime }$ is in the interior of
the domain $\Delta ^{\prime },$ and there is a path $\beta _{3}$ whose
interior and initial point are located in $\Delta ^{\prime }$ and the
terminal point is $p_{2}\in \partial \Delta ^{\prime },$ and $f_{2}$ and $%
\beta _{3}$ satisfies all assumptions of Lemma \ref{move} if $\Delta
^{\prime }$ is regarded as a disk. Then by applying Lemma \ref{move}, we
obtain the desired conclusion.

Let $\delta <\frac{1}{2}$ be a positive number, $D_{\delta }$ the disk $%
|z-p_{1}|<\delta ,$ $c$ the section of $\partial \Delta $ that is contained
in $\overline{D_{\delta }}$ and regarded as a path from $s_{1}$ to $s_{2}$
anticlockwise, $e$ the section of $\partial D_{\delta }$ that is outside $%
\Delta $ $V$ the part of $D_{\delta }$ outside $\Delta $ and write
.
\begin{eqnarray*}
\gamma &=&f(c), \\
t_{1} &=&f(s_{1}),t_{2}=f(s_{2}), \\
q_{1} &=&f(p_{1}),q_{2}=f(p_{2}), \\
V^{\ast } &=&\Delta \cup D_{\delta }, \\
\gamma &=&f(c),\gamma ^{\prime }=f(e).
\end{eqnarray*}

By the assumption, we may assume that $\delta $ is sufficiently small such
that the followings hold.


(e) $f$ can be extended to be a normal mapping $f_{1}$ defined on $\overline{%
\Delta ^{\ast }}.$

(f) $f_{1}$ restricted to $\overline{V}$ is a homeomorphism onto the closure
of a polygonal Jordan domain $T.$

(g) $q_{1}=f(p_{1})$ has a neighborhood $T^{\ast }$ such that $T^{\ast
}\supset T\cup \gamma \backslash \{t_{1},t_{2}\}$, $T^{\ast }$ is a
polygonal Jordan domain and for the component $U$ of $f_{1}^{-1}(\overline{%
T^{\ast }})$ with $p_{1}\in U,$ $f_{1}$ restricted to $\overline{U}$ is a $d$
to $1$ covering mapping onto $\overline{T^{\ast }},$ with the unique
ramification point at $p_{1}.$

(h) $\beta \cap \partial T^{\ast }=\{t_{2}\}$.

Then there is another normal mapping $f_{2}:\overline{\Delta ^{\ast }}%
\rightarrow S$ such that%
\begin{equation*}
f_{2}|_{\overline{\Delta ^{\ast }}\backslash U}=f_{1}|_{\overline{\Delta
^{\ast }}\backslash U},
\end{equation*}%
the restriction $f_{2}|_{\overline{U}}$ is also a $d$ to $1$ covering with a
unique ramification point $p_{1}^{\prime }$ in $U$ such that $p_{1}^{\prime
}\in \Delta $ and $q_{1}^{\prime }=f_{2}(p_{1}^{\prime })\in T^{\ast
}\backslash \overline{T}.$

Consider the lift of the path $\gamma =f(c)=f_{1}(c)$ by $f_{2}$. Since $%
f_{2}|_{\overline{U}}$ is a covering with the unique ramification point $%
p_{1}^{\prime }$ with $f_{2}(p_{1}^{\prime })=q_{1}^{\prime }\notin \gamma ,$
$\gamma =f(c)$ has a unique lift $\alpha $ in $\overline{U}$ by $f_{2}$ such
that the interior of $\alpha $ is in $U$ with endpoints $s_{1}$ and $s_{2}$
(note that $\overline{T}$ can be lifted by $f_{2}|_{\overline{U}},$ because $%
\overline{T}$ is simple connected and there is no branched point in $T).$
Then $\alpha $ divides $\Delta ^{\ast }$ in to two Jordan domains $\Delta
^{\prime }$ and $\Delta ^{\prime \prime }$ such that $f_{2}(\partial \Delta
^{\prime \prime })=\gamma \cup \gamma ^{\prime }=\partial T,$ and $f_{2}$
restricted to $\partial \Delta ^{\prime \prime }=\alpha \cup e$ is a
homeomorphism onto the boundary $\partial T.$ Then $f_{2}|_{\overline{\Delta
^{\prime \prime }}}$ is a homeomorphism from $\Delta ^{\prime \prime }$ onto
$T.$

Then the restriction of $f_{2}$ to $\overline{\Delta ^{\prime }}$ is a
normal mapping such that
\begin{equation*}
A(f_{2},\Delta ^{\prime })=A(f_{2},\Delta ^{\ast })-A(f_{2},\Delta ^{\prime
\prime })=A(f_{1},\Delta ^{\ast })-A(T)=A(f,\Delta ),
\end{equation*}%
and $\Gamma _{f}$ and $\Gamma _{f_{2}|_{\overline{\Delta ^{\prime }}}}$ are
the same, ignoring a parameter transformation.

By (h), there exists a unique $\theta _{1}^{\prime }\in (\theta _{1},\theta
_{2})$ such that $t_{2}=\beta (\theta _{1}^{\prime }),$ which is the unique
point in $\beta \cap \partial T^{\ast }.$ Let $\beta _{1}(\theta ),\theta
\in \lbrack \theta _{1},\theta _{1}^{\prime }],$ be a polygonal Jordan path
in $T^{\ast }\backslash T$ such that
\begin{equation*}
\beta _{1}(\theta _{1})=q_{1}^{\prime },\beta _{1}(\theta _{1}^{\prime
})=t_{2},
\end{equation*}%
and the interior of $\beta _{1}$ is contained in $T^{\ast }\backslash
\overline{T}.$ Then since $f_{2}|_{\overline{U}}$ is a covering with the
unique ramification point $p_{1}^{\prime }$ and $f_{2}(p_{1}^{\prime
})=q_{1}^{\prime },$ $\beta _{1}$ has $d$ lifts $\alpha _{j}^{\ast }=\alpha
_{j}^{\ast }(\theta ),\theta \in \lbrack \theta _{1},\theta _{1}^{\prime
}],j=1,\dots ,d,$ by $f_{2}|_{\overline{U}},$ such that

(i) $\cup _{j=2}^{d}\alpha _{j}^{\ast }\subset \Delta ^{\prime },$ the
interior of $\alpha _{1}^{\ast }$ is also contained in $\Delta ^{\ast },$
while $\alpha _{1}^{\ast }(\theta _{1}^{\prime })=s_{2}=\alpha _{1}(\theta
_{1}^{\prime })\in \partial \Delta ^{\prime }.$

(j) $\alpha _{j}^{\ast }(\theta _{1}^{\prime })=\alpha _{j}(\theta
_{1}^{\prime }),$and $\alpha _{j}^{\ast }(\theta _{1})=p_{1}^{\prime
},j=1,\dots ,d.$

Let
\begin{equation*}
\beta _{2}(\theta )=\left\{
\begin{array}{c}
\beta _{1}(\theta ),\theta \in \lbrack \theta _{1},\theta _{1}^{\prime }],
\\
\beta (\theta ),\theta \in \lbrack \theta _{1}^{\prime },\theta _{2}];%
\end{array}%
\right.
\end{equation*}%
and let%
\begin{equation*}
\alpha _{j}^{\prime \prime }=\left\{
\begin{array}{c}
\alpha _{j}^{\prime }(\theta ),\theta \in \lbrack \theta _{1},\theta
_{1}^{\prime }], \\
\alpha _{j}(\theta ),\theta \in \lbrack \theta _{1}^{\prime },\theta _{2}].%
\end{array}%
\right.
\end{equation*}%
Then, by the assumption of the lemma, we have%
\begin{equation*}
f_{2}(\alpha _{j}^{\prime \prime }(\theta ))=\beta _{2}(\theta ),\theta \in
\lbrack \theta _{1},\theta _{2}],j=1,\dots ,d
\end{equation*}%
\begin{equation*}
\alpha _{j}^{\prime \prime }\subset \Delta ^{\prime },j=2,\dots ,d,
\end{equation*}%
and $f_{2}$ has no ramification point in $\cup _{j=1}^{d}\alpha _{j}^{\prime
\prime }\backslash \{p_{1}^{\prime }\}.$ Since $\Gamma _{f_{2}|_{\overline{%
\Delta ^{\prime }}}}=f_{2}(z),z\in \partial \Delta ^{\prime }$, is
polygonal, there exists another polygonal Jordan path $\beta _{3}(\theta
),\theta \in \lbrack \theta _{1},\theta _{2}],$ such that $\beta _{3}\cap
\beta =\{q_{2}\}$, $\beta _{3}$ is so close to $\beta _{2}$ that $\beta _{3}$
has a number of $d$ lifts $\gamma _{j}=\gamma _{j}(\theta ),\theta \in
\lbrack \theta _{1},\theta _{2}]$ such that
\begin{equation*}
f_{2}(\gamma _{j}(\theta ))=\beta _{3}(\theta ),\theta \in \lbrack \theta
_{1},\theta _{2}],j=1,\dots ,d,
\end{equation*}%
\begin{equation*}
\gamma _{j}(\theta _{1})=p_{1}^{\prime },j=1,\dots ,d,
\end{equation*}%
\begin{equation*}
\cup _{j=2}^{d}\gamma _{j}\subset \Delta ^{\prime },\gamma \backslash
\{p_{2}\}\subset \Delta ^{\prime }
\end{equation*}%
and $f_{2}$ has no ramification point in $\cup _{j}^{d}\gamma _{j}\backslash
\{p_{1}^{\prime }\}.$ Then by Lemma \ref{move}, there exists a normal
mapping $f_{3}$ defined on $\Delta ^{\prime }$ such that the followings hold.

(k) $p_{2}=e^{i\theta _{2}}$ is a ramification point of $f_{3}$ with $%
f_{3}(p_{2})=q_{2}^{\prime }=\beta (\theta _{2}),\ $while $p_{1}$ is not a
ramification point of $f_{3}$, and
\begin{equation*}
b_{f_{2}}(p_{1})=b_{f_{3}}(p_{2}),
\end{equation*}

(l) The boundary curves $f_{3}(e^{i\theta })$ and $f_{2}(e^{i\theta })$ are
the same after a parameter transform$.$

(m) In $\overline{\Delta ^{\prime }}\backslash \{p_{1}^{\prime },p_{2}\},$ $%
f_{3}$ and $f_{2}$ has the same set of ramification points and $f$ and $g$
coincide in a neighborhood of this ramification point set.

(n) $A(f_{3},\Delta ^{\prime })=A(f_{2},\Delta ^{\prime }).$

Let $B$ be the set of all ramification points of $f_{3},$ then it is clear
that $B\backslash p_{2}\subset \Delta \cap \Delta ^{\prime }.$ Let $h$ be a
homeomorphism from $\overline{\Delta ^{\prime }}$ to $\overline{\Delta },$
such that $h$ restricted to a neighborhood of $B\backslash p_{2}$ is an
identity, and let $g=f_{2}\circ h^{-1}.$ Then $g$ is the desired mapping.
\end{proof}

\section{Cutting and Gluing Riemann surfaces of normal mappings\label%
{ss-9cut-glue}}

In this section, we will prove the following theorem, which is used in the
proof of Theorem \ref{key}.

\begin{theorem}
\label{1-br-2-map}Let $f:\overline{\Delta }\rightarrow S$ be a normal
mapping and assume that each natural edge of $f$ has spherical length
strictly less than $\pi $. If $f$ has a branched point in $S\backslash E$,
then there exist two normal mappings $f_{j}:\overline{\Delta }\rightarrow
S,j=1,2,$ such that the followings hold.

(i) Each natural edge of $f_{j}$ has spherical length strictly less than $%
\pi ,j=1,2.$

(ii) $\sum_{j=1}^{2}L(f_{j},\partial \Delta )\leq L(f,\partial \Delta
),\sum_{j=1}^{2}A(f_{j},\Delta )\geq A(f,\Delta )$.

(iii) $V_{NE}(f_{1})+V_{NE}(f_{2})\leq V_{NE}(f)+2,$ $%
V_{E}(f_{1})+V_{E}(f_{2})\geq V_{E}(f).$

(iv) $V(f_{1})+V(f_{2})\leq V(f)+2.$
\end{theorem}

The proof will be put to the end of this section, after we establish some
results for cutting and gluing the Riemann surface of $f.$

\begin{lemma}
\label{cut-1}Let $f:\overline{\Delta }\rightarrow S$ be a normal mapping,
let $p_{0}\in \Delta $ be a ramification point of $f$ and let $\beta =\beta
(t),t\in \lbrack 0,1],$ be a polygonal Jordan path in $S$ with distinct
endpoints. Assume that the followings hold.

(a) Each natural edge $f$ has spherical length strictly less that $\pi .$

(b) $\beta (0)=f(p_{0}),$ $\beta $ has two lifts $\alpha _{j}=\alpha
_{j}(t),t\in \lbrack 0,1],$ in $\overline{\Delta }$ by $f,$ with%
\begin{equation}
\alpha _{j}(0)=p_{0}\ \text{\textrm{and\ }}f(\alpha _{j}(t))=\beta (t),t\in
\lbrack 0,1],j=1,2.  \label{717-1}
\end{equation}

(c) $\alpha _{j}(t)\in \Delta $ for all $t\in \lbrack 0,1),$ $j=1,2,$ but $%
\{\alpha _{1}(1),\alpha _{2}(1)\}\subset \partial \Delta .$

(d) $f$ has no ramification point in the interior $\alpha _{j}(0,1)=\{\alpha
_{j}(t),t\in (0,1)\},j=1,2.$

Then there exist normal mappings $f_{1},f_{2}:\overline{\Delta }\rightarrow
S,$ such that the following conditions hold.

(i) Each natural edge $f_{j}$ has spherical length strictly less that $\pi ,$
$j=1,2$.

(ii) $L(f,\Delta )\geq L(f_{1},\Delta )+L(f_{2},\Delta )$\ and\textrm{\ }$%
A(f,\Delta )\leq A(f_{1},\Delta )+A(f_{2},\Delta ).$

(iii) $V_{NE}(f_{1})+V_{NE}(f_{2})\leq V_{NE}(f)+2$ and $%
V_{E}(f_{1})+V_{E}(f_{2})\geq V_{E}(f).$

(iv) $V(f_{1})+V(f_{2})\leq V(f)+2.$
\end{lemma}

\begin{proof}
By Lemma \ref{cut-3}, we have $\alpha _{1}(1)\neq \alpha _{2}(1),$ and by
(b) and (d), $\alpha _{1}$ and $\alpha _{2}$ are Jordan paths that intersect
only at $p_{0}.$ Then by the assumption, $\alpha =\alpha _{2}^{-}+\alpha
_{1} $ compose a Jordan path in $\overline{\Delta }$ from $\alpha _{2}(1)$
to $\alpha _{1}(1),$ such that the interior of $\alpha $ is contained in $%
\Delta .$ Thus, $\alpha $ divides the disk $\overline{\Delta }$ into two
parts, of which one is on the left hand side of $\alpha $ and is denoted by $%
\Delta _{1},$ and the other, denoted by $\Delta _{2},$ is on the right hand
side of $\alpha $.

Now, we consider the restrictions%
\begin{equation*}
f|_{\overline{\Delta _{j}}}:\overline{\Delta }\rightarrow S,j=1,2.
\end{equation*}%
By Lemma \ref{glue}, these two normal mapping can be regarded as two normal
mappings $g_{1}$ and $g_{2}$ defined on $\overline{\Delta }$ as follows.

Let $\gamma _{j}=\partial \Delta \cap \partial \Delta _{j},$ which is the
section of the boundary of $\Delta _{j}$ that is on the circle $\partial
\Delta ,$ and let $h_{j}:\overline{\Delta _{j}}\rightarrow \overline{\Delta }
$ be a continuous mapping such that $h_{j}|_{\Delta }$ is a homeomorphism
onto $\Delta _{\lbrack 0,1]}=\Delta \backslash \lbrack 0,1],$ in which $%
[0,1] $ is the interval of the real numbers, $h_{j}\ $maps the interior of $%
\gamma _{j}$ homeomorphically onto $\left( \partial \Delta \right)
\backslash \{1\}, $ and%
\begin{equation*}
h_{j}(\alpha _{j}(t))=t,t\in \lbrack 0,1].
\end{equation*}

Then define $g_{j}=f_{j}\circ h_{j}^{-1},$ and this is the glued mappings
from $\overline{\Delta }\ $into $S$. Then the followings hold.

(e) The boundary curves of $g_{1}$ and $g_{2}$ compose the boundary curve of
$f,$ i.e. the curve $\Gamma _{g_{1}}=g(z),z\in \partial \Delta ,$ and the
section of the curve $\Gamma _{f}=f(z),$ in which $z$ runs on $\partial
\Delta $ from $\alpha _{1}(1)$ to $\alpha _{2}(1)$ are the same and the
curve $\Gamma _{g_{2}}=g_{2}(z),z\in \partial \Delta ,$ and the the section
of the curve $\Gamma _{f}=f(z),$ in which $z$ runs on $\partial \Delta $
from $\alpha _{2}(1)$ to $\alpha _{1}(1)\ $are the same; and%
\begin{equation}
A(f,\Delta )=A(g_{1},\Delta )+A(g_{2},\Delta )\ \mathrm{and\ }L(f,\Delta
)=L(g_{1},\Delta )+L(g_{2},\Delta ).  \label{1062}
\end{equation}

Let $p_{j}=e^{i\theta _{j}},j=1,\dots ,n,$ be an enumeration of all natural
vertices of $f$ that are in order anticlockwise, let $c_{j}=e^{i\theta
},\theta \in \lbrack \theta _{j},\theta _{j+1}]$ ($\theta _{n+1}=\theta
_{1}+2\pi $) be the section of $\partial \Delta $ from $p_{j}$ to $p_{j+1}$
and write $q_{j}=f(p_{j}).$ Then%
\begin{eqnarray}
\Gamma _{f} &=&\Gamma _{1}+\Gamma _{2}+\Gamma _{3}+\dots +\Gamma _{n}
\label{12-3} \\
&=&\overline{q_{1}q_{2}}+\overline{q_{2}q_{3}}+\dots +\overline{q_{n-1}q_{n}}%
+\overline{q_{n}q_{1}}  \notag
\end{eqnarray}%
is a natural partition$\ $of the boundary curve $\Gamma _{f}=f(e^{i\theta
}),\theta \in \lbrack 0,2\pi ],$ with (by (a))%
\begin{equation*}
L(\Gamma _{j})<\pi ,j=1,2,\dots ,n,
\end{equation*}%
and
\begin{equation*}
n=V(f).
\end{equation*}

Without loss of generality, assume $\alpha _{1}(1)\in c_{1}$ and $\alpha
_{2}(1)\in c_{j_{0}}$ for some $j_{0}\leq n.$

Let
\begin{equation*}
q^{\prime }=f(\alpha _{1}(1))=f(\alpha _{2}(1)).
\end{equation*}%
Then, it is clear that the boundary curves $\Gamma _{g_{j}}(z),z\in \partial
\Delta ,j=1,2,$ have the permitted partitions%
\begin{eqnarray}
\Gamma _{g_{1}} &=&\overline{q_{j_{0}}q^{\prime }}+\overline{q^{\prime }q_{2}%
}+\overline{q_{2}q_{3}}+\dots +\overline{q_{j_{0}-1}q_{j_{0}}}  \label{12-1}
\\
&=&\Gamma _{11}+\Gamma _{12}+\Gamma _{2}+\dots +\Gamma _{j_{0}-1},  \notag
\end{eqnarray}%
and%
\begin{eqnarray}
\Gamma _{g_{2}} &=&\overline{q_{1}q^{\prime }}+\overline{q^{\prime
}q_{j_{0}+1}}+\overline{q_{j_{0}+1}q_{j_{0}+2}}+\dots +\overline{q_{n-1}q_{n}%
}+\overline{q_{n}q_{1}}  \label{12-2} \\
&=&\Gamma _{21}+\Gamma _{22}+\Gamma _{j_{0}+1}+\dots +\Gamma _{n},  \notag
\end{eqnarray}%
respectively, such that%
\begin{equation*}
L(\Gamma _{ij})<\pi ,i,j=1,2,
\end{equation*}%
where%
\begin{equation}
\Gamma _{11}=\overline{q_{j_{0}}q^{\prime }},\Gamma _{12}=\overline{%
q^{\prime }q_{2}},\Gamma _{21}=\overline{q_{1}q^{\prime }},\Gamma _{22}=%
\overline{q^{\prime }q_{j_{0}+1}}.  \label{41}
\end{equation}

If $\alpha _{1}(1)$ (or $\alpha _{2}(1))$ is one of the endpoint of $c_{1}$
(or $c_{j_{0}}),$ then the discussion is similar and easier than the
followings, since in this case some of edges in (\ref{41}) reduce to points,
and the discussion is left to the reader. So, we assume $\alpha _{1}(1)$ is
in the interior of $c_{1}$ and $\alpha _{2}(1)$ is in the interior of $%
c_{2}. $ Then
\begin{equation*}
q^{\prime }\notin E.
\end{equation*}%
and it is clear that%
\begin{equation}
\left\{
\begin{array}{l}
V_{NE}(g_{1})+V_{NE}(g_{2})\leq V_{NE}(f)+2, \\
V_{E}(g_{1})+V_{E}(g_{2})=V_{E}(f), \\
V(f_{1})+V(f_{2})\leq n+2,%
\end{array}%
\right.  \label{1061}
\end{equation}%
and%
\begin{eqnarray}
L(\Gamma _{1})+L(\Gamma _{j_{0}}) &=&L(\Gamma _{11}+\Gamma _{12})+L(\Gamma
_{21}+\Gamma _{22})  \label{a7} \\
&=&L(\Gamma _{11}+\Gamma _{22})+L(\Gamma _{21}+\Gamma _{12}).  \notag
\end{eqnarray}

Now, there are two cases need to discuss.

\noindent \textbf{Case 1. }$\Gamma _{1}^{\prime }=\Gamma _{11}+\Gamma _{12}=%
\overline{q_{j_{0}}q^{\prime }}+\overline{q^{\prime }q_{2}}$ is not a
natural edge of $\Gamma _{g_{1}}.$

\noindent \textbf{Case 2. }$\Gamma _{1}^{\prime }=\Gamma _{11}+\Gamma _{12}=%
\overline{q_{j_{0}}q^{\prime }}+\overline{q^{\prime }q_{2}}$ is a natural
edge of $\Gamma _{g_{1}}.$

In Case $1,$ $\Gamma _{2}^{\prime }=\Gamma _{21}+\Gamma _{22}=\overline{%
q_{1}q^{\prime }}+\overline{q^{\prime }q_{j_{0}+1}}$ not a natural as well.
Then the partitions (\ref{12-1}) and (\ref{12-2}) are natural partitions,
since (\ref{12-3}) is a natural partition, and then $g_{1},g_{2}$ are the
two desired mappings by (\ref{1062}) and (\ref{1061}).

In Case 2, $\Gamma _{2}^{\prime }=\Gamma _{21}+\Gamma _{22}$ is a natural
edge as well. Then%
\begin{equation*}
\Gamma _{g_{1}}=\Gamma _{1}^{\prime }+\Gamma _{2}+\dots +\Gamma _{j_{0}-1},
\end{equation*}%
and
\begin{equation}
\Gamma _{g_{2}}=\Gamma _{2}^{\prime }+\Gamma _{j_{0}+1}+\dots +\Gamma _{n},
\notag
\end{equation}%
are natural partition of $\Gamma _{g_{1}}$ and $\Gamma _{g_{2}},$
respectively, and then (\ref{1061}) changes into%
\begin{equation}
\left\{
\begin{array}{l}
V_{NE}(g_{1})+V_{NE}(g_{2})\leq V_{NE}(f), \\
V_{E}(g_{1})+V_{E}(g_{2})=V_{E}(f), \\
V(f_{1})+V(f_{2})=n.%
\end{array}%
\right.  \label{1063}
\end{equation}

If $L(\Gamma _{1}^{\prime })<\pi $ and $L(\Gamma _{2}^{\prime })<\pi ,$ then
$g_{1}$ and $g_{2}$ are the desired mappings with (\ref{1063}).

If $L(\Gamma _{1}^{\prime })\geq \pi ,$ then by (\ref{a7}) and the
assumption of the lemma, $L(\Gamma _{2}^{\prime })<\pi $. This is because
that by the assumption of the lemma, $L(\Gamma _{1}^{\prime })+L(\Gamma
_{2}^{\prime })=L(\Gamma _{1})+L(\Gamma _{j_{0}})<2\pi .$ Then applying
Theorem \ref{>pi}, there exists a normal mapping $f_{1}:\overline{\Delta }%
\rightarrow S$ such that
\begin{equation*}
L(f_{1},\Delta )\leq L(g_{1},\partial \Delta ),\mathrm{\ }A(f_{1},\Delta
)=A(g_{1},\Delta ),
\end{equation*}%
and
\begin{equation*}
V_{NE}(f_{1})\leq V_{NE}(g_{1}),\mathrm{\ }V_{E}(f_{1})\geq
V_{E}(g_{1})+1,V(f_{1})\leq V(g_{1})+2,
\end{equation*}%
and each natural edges of $\Gamma _{f_{1}}$ has spherical length strictly
less than $\pi .$ Then we have by (\ref{1062}) and (\ref{1063}) that%
\begin{equation*}
L(f_{1},\partial \Delta )+L(g_{2},\partial \Delta )\leq L(f,\partial \Delta
),A(f_{1}.\Delta )+A(g_{2},\Delta )=A(f,\Delta )
\end{equation*}%
\begin{equation*}
V_{NE}(f_{1})+V_{NE}(g_{2})\leq V_{NE}(g_{1})+V_{NE}(g_{2})\leq V_{NE}(f),
\end{equation*}%
\begin{equation*}
V_{E}(f_{1})+V_{E}(g_{2})\geq V_{E}(g_{1})+1+V_{E}(g_{2})=V_{E}(f)+1,
\end{equation*}%
and%
\begin{equation*}
V(f_{1})+V(g_{2})\leq V(g_{1})+2+V(g_{2})\leq V(f)+2.
\end{equation*}%
Thus, $f_{1}$ and $g_{2}$ is the desired mappings. This completes the proof.
\end{proof}

\begin{corollary}
\label{ccut-1}Assume that $f,\beta ,\alpha _{1}$ and $\alpha _{2}$ satisfy
all the assumptions in Lemma \ref{cut-1} and, in addition, $\beta (1)\in E.$
Then there exist normal mappings $g_{1},g_{2}:\overline{\Delta }\rightarrow
S,$ such that the followings hold.

(i) Each natural edge of $\Gamma g_{j}$ is a natural edge of $\Gamma
_{f},j=1,2,$ and each natural edge of $\Gamma _{f}$ is a natural edge of
either $g_{1}$ or $g_{2}$.

(ii) $L(g_{1},\Delta )+L(g_{2},\Delta )=L(f,\Delta )$\ and\textrm{\ }$%
A(g_{1},\Delta )+A(g_{2},\Delta )=A(f,\Delta ).$

(iii) $%
V_{NE}(g_{1})+V_{NE}(g_{2})=V_{NE}(f),V_{E}(g_{1})+V_{E}(g_{2})=V_{E}(f),$ $%
V(g_{1})+V(g_{2})=V(f).$

(iv)
\begin{equation}
\sum_{p\in \overline{\Delta }\backslash
g_{1}^{-1}(E)}b_{g_{1}}(p)+\sum_{p\in \overline{\Delta }\backslash
g_{2}^{-1}(E)}b_{g_{2}}(p)\leq \sum_{p\in \overline{\Delta }\backslash
f^{-1}(E)}b_{f}(p).  \label{1211-1}
\end{equation}
\end{corollary}

\begin{proof}
By repeating the above proof from the beginning to (\ref{41}) and
considering that, in current situation, $q^{\prime }=f(\alpha (1))=f(\alpha
_{2}(1))$ must be a natural vertex of $f,g_{1}$ and $g_{2},$ we can conclude
that all the conclusion follows, except the inequality (iv).

By the assumption and the definition of $g_{j}$s, it is clear that%
\begin{equation}
b_{g_{1}}(0)+b_{g_{2}}(0)=b_{f}(p_{0})-1.  \label{23}
\end{equation}

Next, we show that
\begin{equation}
b_{g_{1}}(1)+b_{g_{2}}(1)\leq b_{f}(\alpha _{1}(1))+b_{f}(\alpha _{2}(1))+1.
\label{22}
\end{equation}%
Let $l_{j}$ be the circular arc of the circle $C_{j}:|z-\alpha
_{j}(1)|=\varepsilon $ inside $\overline{\Delta }$ and $\gamma _{j}$ be the
section of $\partial \Delta $ inside $C_{j},j=1,2;$ let $l$ be the circular
arc of the circle $C:|z-1|$ inside $\overline{\Delta }$ and $\gamma ^{\prime
}$ be the section of $\partial \Delta $ inside $C,j=1,2;$ where $\varepsilon
$ is a sufficiently small positive number. Let $s_{1},s_{2},s_{1}^{\prime
},s_{2}^{\prime }$ be smooth and orientation preserved diffeomorphisms from
neighborhoods of $f(a_{1}(1))$, $f(\alpha _{2}(1)),g_{1}(1)$ and $g_{2}(1)$
onto the disk $\Delta $ with
\begin{equation*}
s_{j}(f(\alpha _{j}(1))=0,s_{j}^{\prime }(g_{j}(1))=0,
\end{equation*}%
such that they keep the angles at $f(a_{1}(1))$, $f(\alpha _{2}(1)),g_{1}(1)$
and $g_{2}(1)$ and maps $f(\gamma _{1}),f(\gamma _{2}),g_{1}(\gamma
_{1}^{\prime })$ and $g_{2}(\gamma _{2}^{\prime })$ onto angles (broken
lines) with vertices $f(a_{1}(1))$, $f(\alpha _{2}(1)),g_{1}(1)$ and $%
g_{2}(1),$ respectively, in $\mathbb{C}$. Then we can define the rotation
numbers
\begin{equation*}
\tau _{j}=\frac{1}{2\pi }\int_{l_{j}}\frac{d\left( s_{j}\circ f(z)\right) }{%
s_{j}\circ f(z)}\ \mathrm{and\ }\tau _{j}^{\prime }=\frac{1}{2\pi }\int_{l}%
\frac{d\left( s_{j}^{\prime }\circ g_{j}(z)\right) }{s_{j}^{\prime }\circ
g_{j}(z)},
\end{equation*}%
which is invariant for sufficiently small $\varepsilon $, independent of $%
s_{j}$s and $s_{j}^{\prime }$s by the assumption and all are positive
because $s_{j}$s and $s_{j}^{\prime }$s are orientation preserved and $g_{j}$%
s and $f$ are normal. It is clear that%
\begin{equation}
\tau _{1}^{\prime }+\tau _{2}^{\prime }=\tau _{1}+\tau _{2}.  \label{21}
\end{equation}

Then there exists $k_{j}$ and $k_{j}^{\prime },j=1,2,$ such that%
\begin{equation*}
k_{j}^{\prime }<\tau _{j}^{\prime }\leq k_{j}^{\prime }+1,k_{j}<\tau
_{j}\leq k_{j}+1,j=1,2,
\end{equation*}%
and then we have%
\begin{eqnarray*}
b_{g_{1}}(1)+b_{g_{2}}(1) &=&k_{1}^{\prime }+k_{2}^{\prime } \\
&<&\tau _{1}^{\prime }+\tau _{2}^{\prime }=\tau _{1}+\tau _{2}\leq
k_{1}+k_{2}+2 \\
&=&b_{f}(\alpha (1))+b_{f}(\alpha _{2}(1))+2,
\end{eqnarray*}%
but branched numbers are integers, we have (\ref{22}).

It is clear that, by the definition of $g_{j}$s
\begin{equation*}
\sum_{p\in \overline{\Delta }\backslash \{\alpha _{1}(1),\alpha
_{2}(1)\}}b_{f}(p)=\sum_{j=1}^{2}\sum_{p\in \overline{\Delta }\backslash
\{1\}}b_{g_{j}}(p),
\end{equation*}%
which, together (\ref{23}) and (\ref{22}), implies (\ref{1211-1}).
\end{proof}

\begin{lemma}
\label{cut-2}Let $f:\overline{\Delta }\rightarrow S$ be a normal mapping,
let $\alpha _{1}=\alpha _{1}(\theta )=e^{i\theta },\theta \in \lbrack \theta
_{1},\theta _{2}],$ be a section of $\partial \Delta $ and denote $\beta
=\beta (\theta )=f(e^{i\theta }),\theta \in \lbrack \theta _{1},\theta
_{2}]. $ Assume that the followings hold:

(a) Each natural edge $f_{j}$ has spherical length strictly less that $\pi .$

(b) $\beta $ is a Jordan path with distinct endpoints and
\begin{equation*}
\beta (\theta )\notin E\ \mathrm{for\ each\ }t\in \lbrack \theta _{1},\theta
_{2}).
\end{equation*}

(c) $\beta $ has a lift $\alpha _{2}=\alpha _{2}(\theta ),\theta \in \lbrack
\theta _{1},\theta _{2}],$ in $\overline{\Delta },$ with $\alpha _{2}(\theta
_{1})=\alpha _{1}(\theta _{1})=e^{i\theta _{1}},$ and%
\begin{equation*}
f(e^{i\theta })=f(\alpha _{1}(\theta ))=\beta (\theta ),\theta \in \lbrack
\theta _{1},\theta _{2}].
\end{equation*}

(d) $f$ has no ramification point in the interior of $\alpha _{1}$ and $%
\alpha _{2}.$

(e) The interior of $\alpha _{2},$ which means the open curve $\alpha
_{2}(\theta ),\theta \in (\theta _{1},\theta _{2}),$ is contained in $\Delta
,\ $but $\{\alpha _{2}(\theta _{2})\}\subset \partial \Delta .$

Then there exist two normal mappings $f_{1},f_{2}:\overline{\Delta }%
\rightarrow S,$ such that the followings hold.

(i) Each natural edge $f_{j}$ has spherical length strictly less that $\pi ,$
$j=1,2$.

(ii) $A(f,\Delta )\leq A(f_{1},\Delta )+A(f_{2},\Delta )\ \mathrm{and\ }%
L(f,\Delta )\geq L(f_{1},\Delta )+L(f_{2},\Delta ).$

(iii) $V_{NE}(f_{1})+V_{NE}(f_{2})\leq V_{NE}(f)+2,$ and $%
V_{E}(f_{1})+V_{E}(f_{2})\geq V_{E}(f).$

(iv) $V(f_{1})+V(f_{2})\leq V(f)+2.$
\end{lemma}

\begin{proof}
By Lemma \ref{cut-3}, we have $\alpha _{1}(\theta _{2})\neq \alpha
_{2}(\theta _{2}).$

$\alpha _{2}$ divides the disk $\overline{\Delta }$ into two parts, one of
which denoted by $\Delta _{1},$ is on the left hand side of $\alpha _{2},$
and the other, denoted by $\Delta _{2},$ is on the right hand side. Then, $%
\alpha _{1}$ is a section of $\partial \Delta _{2}.$ By ignoring a
coordinate transform, we may regard the restriction $f|_{\overline{\Delta
_{1}}}$ as a normal mapping $g_{1}$ defined on $\overline{\Delta }.$

Consider the Jordan domain $\Delta _{2}.$ By Lemma \ref{glue}, we can glue
the $\alpha _{1}$ and $\alpha _{2}$ so that the restriction $f|_{\overline{%
\Delta _{2}}}$ can be regarded as a normal mapping $g_{2}:\overline{\Delta }%
\rightarrow S,$ as we did in the proof of Lemma \ref{cut-1}. Then the
followings hold:

(e) The boundary curves of $g_{1}$ and $g_{2}$ compose the boundary curve of
$f,$ i.e. the curve $\Gamma _{g_{1}}=g(z),z\in \partial \Delta ,$ and the
section of the curve $\Gamma _{f}=f(z),$ in which $z$ runs on $\partial
\Delta $ from $\alpha _{1}(1)$ to $\alpha _{2}(1)$ are the same and the
curve $\Gamma _{g_{2}}=g_{2}(z),z\in \partial \Delta ,$ and the the section
of the curve $\Gamma _{f}=f(z),$ in which $z$ runs on $\partial \Delta $
from $\alpha _{2}(1)$ to $\alpha _{1}(1)\ $are the same; and%
\begin{equation}
A(f,\Delta )=A(g_{1},\Delta )+A(g_{2},\Delta )\ \mathrm{and\ }L(f,\Delta
)=L(g_{1},\Delta )+L(g_{2},\Delta ).  \label{1211-3}
\end{equation}

Then, as in the proof of Lemma \ref{cut-1}, there exist normal mappings $%
f_{j}:\overline{\Delta }\rightarrow S,j=1,2,$ satisfied the conclusions.
\end{proof}

\begin{corollary}
\label{ccut-2}Assume that $f,\beta ,\alpha _{1}$ and $\alpha _{2}$ satisfy
all the assumptions in Lemma \ref{cut-2} and, in addition, $\beta (1)\in E.$
Then all the conclusions of Corollary \ref{ccut-1} hold.
\end{corollary}

\begin{proof}
Repeat the above proof from the beginning to (\ref{1211-3}) and repeat the
proof of Corollary \ref{ccut-1}.
\end{proof}

\begin{lemma}
\label{move-to-bd}Let $f:\overline{\Delta }\rightarrow S$ be a normal
mapping such that each natural edge of $f$ has spherical length strictly
less than $\pi $. If $f$ has a ramification point in $\Delta ,$ then, one of
the following conditions (A) and (B) is satisfied.

(A) There exists a normal mappings $f_{1}:\overline{\Delta }\rightarrow S$
such that$,$ the followings hold.

(i) The boundary curve $\Gamma _{f_{1}}=f_{1}(z),z\in \partial \Delta $, is
the same as that of $f$.

(ii) $L(f_{1},\partial \Delta )=L(f,\partial \Delta ),A(f_{1},\Delta
)=A(f,\Delta ).$

(iii) $f_{1}$ has no ramification point in $\Delta .$

(iv) $f_{1}$ has at least one ramification point in $\left( \partial \Delta
\right) \backslash f_{1}^{-1}(E).$

(B) There exist normal mappings $f_{j}:\overline{\Delta }\rightarrow
S,j=1,2, $ such that the followings hold.

(i) Each natural edge of $f_{j}$ has spherical length strictly less than $%
\pi ,j=1,2$.

(ii) $\sum_{j=1}^{2}L(f_{j},\partial \Delta )\leq L(f,\partial \Delta
),\sum_{j=1}^{2}A(f_{j},\Delta )\geq A(f,\Delta ).$

(iii) $V_{NE}(f_{1})+V_{NE}(f_{2})\leq E(f)+2,$ $V_{E}(f_{1})+V_{E}(f_{2})%
\geq V_{E}(f).$

(iv) $V(f_{1})+V(f_{2})\leq V(f)+2.$
\end{lemma}

\begin{proof}
Let $p_{0}\in \Delta $ be any ramification point of $f$ and let $\beta
=\beta (t),t\in \lbrack 0,1],$ be a polygonal Jordan path in $S$ from $%
q_{0}=f(p_{0})$ to some point $q_{1}\in E=\{0,1,\infty \}$ such that the
interior of $\beta $ does not contain any point in $E$. Then $\beta (0)\neq
\beta (1),$ since $f$ is normal.

We may assume that

(a) There is no branched point of $f$ in the interior of $\beta $
(otherwise, we deform $\beta $ slightly$).$

Let $d=v_{f}(p_{0}).$ Then, by (a) and the fact that $f(\Delta )\cap
E=\emptyset $ (note that $f$ is normal) we conclude that there are only two
cases:

\noindent \textbf{Case 1.} There exists a positive number $t_{1}\leq 1,$
such that the followings hold.

(a1) The section $\beta \lbrack 0,t_{1}]=\{\beta (t);t\in \lbrack 0,t_{1}]\}$
of $\beta $ has two lifts $\alpha _{j}=\alpha _{j}(t),t\in \lbrack 0,t_{1}],$
in $\overline{\Delta }$ by $f,$ such that
\begin{equation*}
\alpha _{j}(0)=p_{0}\text{\textrm{and\ }}f(\alpha _{j}(t))=\beta (t),t\in
\lbrack 0,t_{1}];j=1,2.
\end{equation*}

(b1) For $j=1$ and $2,$ $\alpha _{j}(t)\in \Delta $ for all $t\in \lbrack
0,t_{1}],$ but $\{\alpha _{1}(t_{1}),\alpha _{2}(t_{1})\}\subset \partial
\Delta .$

(c1) $f$ has no ramification point in the interior of $\alpha _{1}$ and $%
\alpha _{2},$ i.e. $f$ has no ramification point on $\alpha _{1}(0,1)\cup
\alpha _{2}(0,1),$ where $\alpha _{j}(0,1)$ is the open curve $\alpha
_{j}(t),t\in (0,1),$ which is the curve $\alpha _{j}$ without end points, $%
j=1,2.$

\noindent \textbf{Case 2.} There exists a positive number $t_{1}\leq 1,$
such that the followings hold.

(a2) The section $\beta (t),t\in \lbrack 0,t_{1}],$ of $\beta $ has a number
of $d=v_{f}(p_{0})$ lifts $\alpha _{j}=\alpha _{j}(t),t\in \lbrack 0,t_{1}],$
in $\overline{\Delta },$ such that
\begin{equation}
\cup _{j=2}^{d}\alpha _{j}\subset \Delta ,  \label{122-2}
\end{equation}%
\begin{equation*}
f(\alpha _{j}(t))=\beta (t),t\in \lbrack 0,t_{1}],\mathrm{\ }j=1,\dots ,d,
\end{equation*}%
and
\begin{equation*}
\alpha _{j}(0)=p_{0},\mathrm{\ }j=1,\dots ,d.
\end{equation*}

(b2) $\alpha _{1}(t)\in \Delta $ for all $t\in \lbrack 0,t_{1})\ $but $%
p_{1}^{\prime }=\alpha _{1}(t_{1})\in \partial \Delta .$

In Case 1, by Lemmas \ref{cut-1}, (B) is satisfied.

Now, assume Case 2 occurs. Then, we must have $t_{1}<1.$ Otherwise, we have $%
f(\alpha _{2}(t_{1}))=\beta (1)=q_{1}\in E,$ and then by the fact $f(\Delta
)\cap E=\emptyset ,$ we have $\alpha _{2}(t_{1})\in \partial \Delta ,$
contradicting (\ref{122-2}). Thus by (a) we have:

(c2) $f$ has no ramification point on $\cup _{j=1}^{d}\alpha _{j}(0,t_{1}],$
where $\alpha _{j}(0,t_{1}]$ is the curve $\alpha _{j}(t),t\in (0,t_{1}],$
which is the curve $\alpha _{j}$ without initial point, $j=1,2,\dots d.$

Then, by (a2), (b2) and (c2), Lemma \ref{move} applies, and then, there
exists a normal mapping $g_{1}:\overline{\Delta }\rightarrow S$ such that
(recall that $p_{1}^{\prime }=\alpha _{1}(t_{1}))$
\begin{equation}
\#\{p\in \Delta ;b_{g_{1}}(p)>1\}=\#\{p\in \Delta ;b_{f}(p)>1\}-1,
\label{1118-11}
\end{equation}%
\begin{equation}
b_{g_{1}}(p_{1}^{\prime })=b_{f}(p_{0})  \label{122-1}
\end{equation}%
\begin{equation}
L(g_{1},\partial \Delta )=L(f,\partial \Delta ),A(g_{1},\Delta )=A(f,\Delta
),  \label{1128-1}
\end{equation}%
and

(d) The boundary curve of $g_{1}$ is the same as that of $f.$

Then, $g_{1}$ satisfies (i), (ii) of condition (A).

By (\ref{122-1}), $p_{1}^{\prime }\in \partial \Delta $ is a ramification
point of $g_{1}.$ On the other hand, by (b2) and (\ref{122-2})
\begin{equation*}
g_{1}(\partial \Delta )\ni g_{1}(p_{1}^{\prime })=f(p_{1}^{\prime
})=f(\alpha _{1}(t_{1}))=f(\alpha _{2}(t_{1}))\in f(\Delta ),
\end{equation*}%
,i.e. $g_{1}(p_{1}^{\prime })\in g_{1}(\partial \Delta )\cap f(\Delta ),$
and then $g_{1}(p_{1}^{\prime })\notin E$ since $f$ is normal. Thus, (iv) of
(A) hold.

If $g_{1}$ does not satisfies (iii) in (A), then $g_{1}$ satisfies all
assumptions of the lemma under proving, but the number of ramification
points of $g_{1}$ located in $\Delta $ is dropped by one (by (\ref{1118-11}%
)), and then apply the above argument, we again reach Case 1 or Case 2.
Since there are finitely many ramification point of $f$ in $\Delta ,$ after
repeating the above arguments finitely many time, we can show that either
(A) or (B) holds.
\end{proof}

\begin{lemma}
\label{move-to-3pts}Let $f:\overline{\Delta }\rightarrow S$ be a normal
mapping such that

(a) Each natural edge of $f$ has length strictly less than $\pi ,$ and

(b) $f$ has no ramification point in $\Delta $.

Assume that there exist $\theta _{1}$ and $\theta _{3}$ with $\theta
_{1}<\theta _{3}<\theta _{1}+2\pi $ such that the followings hold.

(c) $p_{1}=e^{i\theta _{1}}\in \partial \Delta \backslash f^{-1}(E)$ is a
ramification point of $f.\ $Note that $E=\{0,1,\infty \}.$

(d) $f(e^{i\theta _{3}})\in E\ $but $f(e^{i\theta })\notin E$ for each $%
\theta \in (\theta _{1},\theta _{3}).$

(e) Each point $e^{i\theta }\in \partial \Delta $ with $\theta \in (\theta
_{1},\theta _{3})$ is not a ramification point of $f$.

Then, there exist two normal mappings $f_{j}:\overline{\Delta }\rightarrow
S,j=1,2,$ such that the followings holds.

(i) Each natural edge of $f_{j}$ has length strictly less than $\pi $, $%
j=1,2.$

(ii) $\sum_{j=1}^{2}L(f_{j},\partial \Delta )\leq L(f,\partial \Delta
),\sum_{j=1}^{2}A(f_{j},\Delta )\geq A(f,\Delta ).$

(iii) $V_{NE}(f_{1})+V_{NE}(f_{2})\leq V_{NE}(f)+2,$ $%
V_{E}(f_{1})+V_{E}(f_{2})\geq V_{E}(f).$

(iv) $V(f_{1})+V(f_{2})\leq V(f)+2.$
\end{lemma}

\begin{proof}
By the assumption, there are only two cases need to discuss.

\noindent \textbf{Case 1.} There exist $\theta _{4},\theta _{5}\in (\theta
_{1},\theta _{2})\ $with $\theta _{4}<\theta _{5}$ such that

(f) $e^{i\theta _{4}}$ is a natural vertex of $f.$

(g) Both of the sections $\Gamma _{14}=f(e^{i\theta }),\theta \in \lbrack
\theta _{1},\theta _{4}],$ and $\Gamma _{45}=f(e^{i\theta }),\theta \in
\lbrack \theta _{4},\theta _{5}]$ are Jordan paths with $\Gamma _{14}(\theta
_{1})\neq \Gamma _{14}(\theta _{4})$ and $\Gamma _{45}(\theta _{4})\neq
\Gamma _{45}(\theta _{5})$, but $\Gamma _{14}+\Gamma _{45}$ is not a Jordan
curve.

\noindent \textbf{Case 2.} The section $\Gamma _{13}=f(e^{i\theta }),\theta
\in \lbrack \theta _{1},\theta _{3}],$ is a Jordan path$.$

Let $d=v_{f}(p_{0}).$ We first assume Case 1 occur.

Then, there exists a positive number $\delta $ and there exist a number of $%
d $ Jordan paths
\begin{equation*}
\alpha _{j,\delta }=\alpha _{j,\delta }(\theta ),\theta \in \lbrack \theta
_{1},\theta _{1+\delta }],j=1,\dots ,d,
\end{equation*}%
such that%
\begin{equation*}
\alpha _{1,\delta }(\theta )=e^{i\theta },\theta \in \lbrack \theta
_{1},\theta _{1+\delta }],
\end{equation*}%
\begin{equation*}
\alpha _{j,\delta }(\theta _{1})=p_{1},j=1,\dots ,d,
\end{equation*}%
\begin{equation*}
\alpha _{j,\delta }(\theta )\in \Delta ,\theta \in (\theta _{1},\theta
_{1+\delta }),j=2,3,\dots d,
\end{equation*}%
and
\begin{equation*}
f(\alpha _{j,\delta }(\theta ))=f(e^{i\theta }),\theta \in \lbrack \theta
_{1},\theta _{1+\delta }],j=1,\dots ,d.
\end{equation*}

Since $f$ has no ramification point in $\Delta ,$ there are only two further
cases for Case 1.

\noindent \textbf{Case 1.1.} Each $\alpha _{j,\delta }$ can be extended to
be a Jordan path $\alpha _{j}=\alpha _{j}(\theta ),\theta \in \lbrack \theta
_{1},\theta _{4}],$ such that%
\begin{equation*}
\alpha _{j}(\theta )\in \Delta ,\theta \in (\theta _{1},\theta
_{4}],j=2,3,\dots ,d,
\end{equation*}%
and
\begin{equation*}
f(\alpha _{j}(\theta ))=f(e^{i\theta }),\theta \in \lbrack \theta
_{1},\theta _{4}],j=2,3,\dots ,d.
\end{equation*}

\noindent \textbf{Case 1.2.} For some $j_{0}\in \{2,3,\dots ,d\},$ there
exists $\theta _{2}\in (\theta _{1},\theta _{4}]$ such that $\alpha
_{j_{0},\delta }$ can be extended to be a Jordan path $\alpha
_{j_{0}}=\alpha _{j_{0}}(\theta ),\theta \in \lbrack \theta _{1},\theta
_{2}],$ such that%
\begin{equation*}
\alpha _{j_{0}}(\theta _{2})\in \partial \Delta ,
\end{equation*}%
\begin{equation*}
\alpha _{j_{0}}(\theta )\in \Delta ,\theta \in (\theta _{1},\theta _{2}),
\end{equation*}%
\begin{equation*}
f(\alpha _{j_{0}}(\theta ))=f(e^{i\theta }),\theta \in \lbrack \theta
_{1},\theta _{2}].
\end{equation*}

In Case 1.1, since $f$ has no ramification point in $\Delta $, $f$ has no
ramification point in the interior of each $\alpha _{j},j=2,\dots ,d,$ and
then $f,\alpha _{1},\dots ,\alpha _{d}$ and $\beta =f(e^{i\theta }),\theta
\in \lbrack \theta _{1},\theta _{4}]$ satisfy all assumptions of Lemma \ref%
{move-bd} by (e)$.$ Then Lemma \ref{move-bd} apply, and then there exists a
normal mapping $g:\overline{\Delta }\rightarrow S$ such that%
\begin{equation*}
g(e^{i\theta })=f(e^{i\theta }),\theta \in \lbrack 0,2\pi ],
\end{equation*}%
\begin{equation*}
b_{g}(p_{1})=0,b_{g}(p_{2})=b_{f}(p_{1}),
\end{equation*}%
\begin{equation*}
b_{g}(p)=b_{f}(p)\ \mathrm{for\ a}\text{\textrm{l}}\mathrm{l\ }p\in
\overline{\Delta }\backslash \{p_{1},p_{2}\},
\end{equation*}%
and%
\begin{equation*}
L(g,\partial \Delta )=L(f,\partial \Delta ),A(g,\Delta )=A(f,\Delta ).
\end{equation*}%
Then $g$ satisfies all the assumptions in the lemma under proving by
replacing $\theta _{1}$ with $\theta _{4}$. But now, the number of loops of
the section $g(e^{i\theta }),\theta \in \lbrack \theta _{4},\theta _{3}],$
is dropped by one. Then, by repeating the same argument several times, we
can find a number $\theta _{1}^{\prime }\in \lbrack \theta _{1},\theta
_{3}), $ and a normal mapping $f_{1}:\overline{\Delta }\rightarrow S$ such
that
\begin{equation*}
f_{1}(e^{i\theta })=f(e^{i\theta }),\theta \in \lbrack 0,2\pi ],
\end{equation*}%
$f_{1}$, $\theta _{1}=\theta _{1}^{\prime }$ and $\theta _{3}$ satisfy all
assumptions of the lemma and fit case Case 2.

In Case 1.2, $f$ also has no ramification point in the interior of each $%
\alpha _{j},j=1,j_{0}.$ Without loss of generality, we assume $j_{0}=2.$
Then, $f,\alpha _{1},\alpha _{2}=\alpha _{j_{0}}$ and $\beta =\beta
=f(e^{i\theta }),\theta \in \lbrack \theta _{1},\theta _{2}]$, satisfy all
assumptions of Lemma \ref{cut-2}. Then by Lemmas \ref{cut-2}, There exist
two normal mappings $f_{1},f_{2}:\overline{\Delta }\rightarrow S,$
satisfying (i)--(ii).

Now assume Case 2 occurs. Then since $f(\Delta )\cap E=\emptyset $ (note
that $f$ is normal), by (a)--(e), there exists $\theta _{2}\in (\theta
_{1},\theta _{3}]$ such that the sections $\alpha _{1}=e^{i\theta }$ and $%
\beta =f(e^{i\theta })$ with $\theta \in \lbrack \theta _{1},\theta _{2}]$
satisfy all assumptions of Lemma \ref{cut-2}, and the arguments in Case 1.2
apply. This completes the proof.
\end{proof}

\begin{proof}[Proof of Theorem\protect\ref{1-br-2-map}]
Assume $f$ has a branched point in $f(\overline{\Delta })\backslash E.$
There are two cases:

\noindent \textbf{Case 1. }$f$ has a ramification point in $\Delta .$

\noindent \textbf{Case 2.} $f$ has no ramification point in $\Delta ,$ but
has a ramification point in $\partial \Delta \backslash f^{-1}(E).$

In Case 1, Lemma \ref{move-to-bd} applies, and we have the following
conclusions (A) or (B).

(A) There exists a normal mappings $g_{1}:\overline{\Delta }\rightarrow S$
such that$,$ the followings hold.

(1) The boundary curve $\Gamma _{g_{1}}=g_{1}(z),z\in \partial \Delta $, is
the same as that of $f$.

(2) $L(g_{1},\partial \Delta )=L(f,\partial \Delta ),A(g_{1},\Delta
)=A(f,\Delta ).$

(3) $g_{1}$ has no ramification point in $\Delta .$

(4) $g_{1}$ has at least one branched point in $f(\partial \Delta
)\backslash E.$

(B) The conclusions of Theorem \ref{1-br-2-map} hold true$.$

If (A) occurs, then $g_{1}$ has a ramification point $p_{1}\in \left(
\partial \Delta \right) \backslash f^{-1}(E),$ and then we can found $\theta
_{1}$ and $\theta _{3}$ such that $g_{1}$, $\theta _{1}$ and $\theta _{3}$
satisfy all assumptions of Lemma \ref{move-to-3pts}, and then (B) holds.

In Case 2, Lemma \ref{move-to-3pts}applies, and so, the conclusions of
Theorem \ref{1-br-2-map} hold true again$.$
\end{proof}

\section{Deformation of normal mappings that have nonconvex vertices\label%
{ss-10to-convex}}

In this section we will prove the following theorem, which is used to prove
Theorem \ref{key}. Theorem \ref{key} is the first key step to prove the main
theorem.

\begin{theorem}
\label{1-nconvex}Let $f:\overline{\Delta }\rightarrow S$ be a normal mapping
and assume that each natural edge of $\Gamma _{f}$ has length strictly less
than $\pi $. If $\Gamma _{f}$ is not convex at some natural vertex $q$ and $%
q\notin E.$ Then there exists a normal mapping $g:\overline{\Delta }%
\rightarrow S,$ such that%
\begin{equation*}
L(g,\partial \Delta )\leq L(f,\partial \Delta ),A(g,\Delta )\geq A(f,\Delta
),
\end{equation*}%
each natural edge of $g$ has spherical length strictly less than $\pi ,$ and%
\begin{equation*}
V_{NE}(g)\leq V_{NE}(f)-1,\mathrm{\ }V_{E}(g)\geq V_{E}(f)\ \mathrm{and}\
V(g)\leq V(f)+1.
\end{equation*}
\end{theorem}

\begin{proof}
We divide the proof into four parts, which is the coming Lemmas \ref{CA}--%
\ref{CD}.
\end{proof}

Before we introduce these lemmas, we first make some conventions.

We fix the normal mapping $f:\overline{\Delta }\rightarrow S$ and assume
\begin{equation}
\Gamma _{f}=\Gamma _{1}+\Gamma _{2}+\Gamma _{3}+\dots +\Gamma _{n}
\label{a-1}
\end{equation}%
is a natural partition of $\Gamma _{f},$ with $n=V(f)$,
\begin{equation*}
\partial \Delta =\gamma _{1}+\dots +\gamma _{n}
\end{equation*}%
is the corresponding natural partition of $\partial \Delta $ for $f,$ and
denote by $p_{j}=e^{i\theta _{j}},j=1,\dots ,n,$ the initial point of $%
\gamma _{j},j=1,\dots ,n$, with $\theta _{j+1}=\theta _{1}$ and%
\begin{equation*}
\theta _{1}<\theta _{2}<\dots <\theta _{n}<\theta _{1}+2\pi ;
\end{equation*}%
and assume that

(I) All natural edges of $f$ has spherical length strictly less than $\pi .$

Then, $q_{j}=f(p_{j})$ is the initial point of $\Gamma _{j}$ for each $%
j=1,2,\dots ,n,$ and by (I), the notation $\overline{q_{j}q_{j+1}}$ makes
sense, which is the unique shortest path from $q_{j}$ to $q_{j+1},$ and $%
\Gamma _{j}=\overline{q_{j}q_{j+1}},j=1,2,\dots ,n.$ Therefore, the natural
partition (\ref{a-1}) can be written%
\begin{equation*}
\Gamma _{f}=\overline{q_{1}q_{2}}+\overline{q_{2}q_{3}}+\dots +\overline{%
q_{n-1}q_{n}}.
\end{equation*}%
We will also assume that

(II) $\Gamma _{1}+\Gamma _{2}$ is not convex at $q_{2}\notin E.$

The assumption (II) means that either $\Gamma _{1}+\Gamma _{2}$ can be
regarded as a perigon angle, or the oriented triangle $\overline{%
q_{1}q_{3}q_{2}q_{1}}$ is a convex triangle. When $\Gamma _{1}+\Gamma _{2}$
is a perigon angle, there is only one case need to discuss.

\noindent \textbf{Case A. }$q_{3}\in \Gamma _{1}=\overline{q_{1}q_{2}}$ or $%
q_{1}\in \Gamma _{2}=\overline{q_{2}q_{3}}$.

When $\overline{q_{1}q_{3}q_{2}q_{1}}$ is a convex triangle, it encloses a
triangle domain $T$ that is on the right hand side of $\Gamma _{1}+\Gamma
_{2},$ and there are only three cases need to discuss:

\noindent \textbf{Case B. }$\left( \overline{\mathbf{T}}\backslash
\{q_{1},q_{3}\})\cap E\right) =\emptyset .$

\noindent \textbf{Case C. }There is only one point $q_{1}^{\prime }$ in $%
E=\{0,1,\infty \}$ that is located in $\overline{T}\backslash \{\Gamma
_{1}+\Gamma _{2}\}.$

\noindent \textbf{Case D. }There exist two points $q_{1}^{\prime }$ and $%
q_{1}^{\prime \prime }$ in $E=\{0,1,\infty \}$ that is located in the
triangle domain $T.$

Under these settings, we can execute deformations of $f\ $which will be
stated in the following Lemmas \ref{CA}--\ref{CD}.

\begin{lemma}
\label{CA}In Case A, there exists a normal mapping $g:\overline{\Delta }%
\rightarrow S$ such that
\begin{equation}
L(g,\partial \Delta )<L(f,\partial \Delta )\text{\textrm{\ and }}A(g,\Delta
)\geq A(f,\Delta ),  \label{1122-2}
\end{equation}%
each natural edge of $\Gamma _{g}$ has spherical length strictly less than $%
\pi ,$ and
\begin{equation}
V_{NE}(g)\leq V_{NE}(f)-1,V_{E}(g)\geq V_{E}(f),V(g)=V(f).  \label{1122-1}
\end{equation}
\end{lemma}

\begin{proof}
Assume Case A occurs. Then, without loss of generality, we may assume $%
q_{3}\in \Gamma _{1}.$ Let $p^{\prime }=e^{i\theta _{1}^{\prime }}$, $\theta
_{1}^{\prime }\in (\theta _{1},\theta _{2}),$ such that $f(p^{\prime
})=q_{3}.$ Then we can glue the section of $\partial \Delta $ from $%
p^{\prime }$ to $p_{2}$ and the section of $\partial \Delta $ from $p_{2}$
to $p_{3}$ and regard $f$ as a mapping $g$ of the glued closed set, which
can be regard as a closed disk, such that the boundary curve of $g$ has a
permitted partition%
\begin{equation}
\Gamma _{g}=\Gamma _{1}^{\prime }+\Gamma _{3}+\dots +\Gamma _{n},
\label{1122-3}
\end{equation}%
where $\Gamma ^{\prime }=\overline{q_{1}q^{\prime }}=\overline{q_{1}q_{3}}$
is the section of $\Gamma _{1}$ from $q_{1}=f(p_{1})$ to $q^{\prime
}=f(p^{\prime })=q_{3},$ and\textrm{\ }%
\begin{equation}
L(g,\partial \Delta )<L(f,\partial \Delta ),\ A(g,\Delta )=A(f,\Delta ).
\label{1122-4}
\end{equation}%
By (\ref{1122-3}), we have (\ref{1122-1}).

If (\ref{1122-3}) is a natural partition, then $g$ also satisfies (I) and
then $g$ is the desired mapping.

Assume that (\ref{1122-3}) is not a natural partition, which is only in the
case that
\begin{equation*}
\Gamma _{1}^{\prime \prime }=\Gamma _{1}^{\prime }+\Gamma _{3}=\overline{%
q_{1}q_{3}}+\overline{q_{3}q_{4}}
\end{equation*}%
is a natural edge of $\Gamma _{g}$ .

But in this case,
\begin{equation*}
\Gamma _{g}=\Gamma _{1}^{\prime \prime }+\Gamma _{4}+\dots +\Gamma _{n}
\end{equation*}%
is a natural partition$,$ and then we have%
\begin{equation}
V_{NE}(g)=V_{NE}(f)-2,V_{E}(g)=V_{E}(f),\ V(g)=V(f)-2.  \label{24}
\end{equation}%
By (I) we have%
\begin{equation*}
L(\Gamma _{1}^{\prime \prime })<2\pi .
\end{equation*}


If $L(\Gamma _{1}^{\prime \prime })<\pi ,$ then $g$ already satisfies all
the conclusions of Lemma \ref{CA}.

If $L(\Gamma _{1}^{\prime \prime })\geq \pi ,$ then $g$ satisfies (a) or (b)
of Theorem \ref{>pi}, and then there exists a normal mapping $f_{1}:%
\overline{\Delta }\rightarrow S$ such that
\begin{equation*}
L(f_{1},\partial \Delta )\leq L(g,\partial \Delta ),A(f_{1},\Delta )\geq
A(g,\Delta ),
\end{equation*}%
each natural edge of $f_{1}$ has spherical length strictly less than $\pi ,$
and
\begin{equation*}
V_{NE}(f_{1})\leq V_{NE}(g),V_{E}(f_{1})\geq V_{E}(g)+1,V(f_{1})\leq V(g)+2.
\end{equation*}

Then by (\ref{1122-4}) and (\ref{24}) we have
\begin{equation*}
L(f_{1},\partial \Delta )<L(f,\partial \Delta ),A(f_{1},\Delta )\geq
A(f,\Delta ),
\end{equation*}%
and
\begin{equation*}
V_{NE}(f_{1})\leq V_{NE}(f)-2,V_{E}(f_{1})\geq
V_{E}(f)+1>V_{E}(f),V(f_{1})\leq V(f).
\end{equation*}%
Thus, $f_{1}$ satisfies all the conclusion of Lemma \ref{CA}.
\end{proof}

\begin{lemma}
\label{CB}In Case B, there exists a normal mapping $g:\overline{\Delta }%
\rightarrow S$ such that
\begin{equation*}
L(g,\partial \Delta )<L(f,\partial \Delta )\text{\textrm{\ and }}A(g,\Delta
)>A(f,\Delta ),
\end{equation*}%
each natural edge of $\Gamma _{g}$ has spherical length strictly less than $%
\pi ,$ and%
\begin{equation*}
V_{NE}(g)\leq V_{NE}(f)-1,\text{\textrm{\ }}V_{E}(g)\geq V_{E}(f),V(g)\leq
V(f).
\end{equation*}
\end{lemma}


\begin{proof}
Putting
\begin{equation*}
\Gamma _{1}^{\prime }=\overline{q_{1}q_{3}},
\end{equation*}%
by (I) and (II), we have
\begin{equation}
L(\Gamma _{1}^{\prime })=L(\overline{q_{1}q_{3}})<\pi .  \label{33}
\end{equation}%
as in the previous proof, by Lemma \ref{patch}, there exists a normal
mapping $g$, which will be regarded as an extension of $f,$ such that $%
\Gamma _{g}$ has the permitted partition%
\begin{equation}
\Gamma _{g}=\Gamma _{1}^{\prime }+\Gamma _{3}+\dots +\Gamma _{n},
\label{128-2}
\end{equation}%
and%
\begin{equation*}
L(g,\partial \Delta )<L(f,\partial \Delta ),A(g,\Delta )>A(f,\Delta ).
\end{equation*}

Then
\begin{equation*}
V_{NE}(g)\leq V_{NE}(f)-1,V_{E}(g)=V_{E}(f),V(g)\leq V(f)-1,
\end{equation*}%
and there are four cases:

\noindent \textbf{Case 1. }Neither\textbf{\ }$\overline{q_{n}q_{1}q_{3}}$
nor $\overline{q_{1}q_{3}q_{4}}\ $is a natural edge of $\Gamma _{g}.$

\noindent \textbf{Case 2. }$\overline{q_{n}q_{1}q_{3}}\ $is a natural edge
of $\Gamma _{g},$ while $\overline{q_{1}q_{3}q_{4}}\ $is not$.$

\noindent \textbf{Case 3. }$\overline{q_{n}q_{1}q_{3}}$ is not a natural
edge of $\Gamma _{g}$, while $\overline{q_{1}q_{3}q_{4}}$ is.

\noindent \textbf{Case 4. }Both $\overline{q_{n}q_{1}q_{3}q_{4}}\ $is a
natural edge of $\Gamma _{g}$.

In Case 1, (\ref{128-2}) is a natural partition, and $g$ is the desired
mapping.

In Case 2, $g$ has a natural partition%
\begin{equation*}
\Gamma _{g}=\Gamma _{1}^{\prime \prime }+\Gamma _{3}+\dots +\Gamma _{n-1},
\end{equation*}%
where $\Gamma _{1}^{\prime \prime }=\Gamma _{n}+\Gamma _{1}^{\prime }=%
\overline{q_{n}q_{1}q_{3}},$ and it is clear that
\begin{equation*}
V_{NE}(g)=V_{NE}(f)-2,V_{E}(g)=V_{E}(f),V(g)=V(f)-2,
\end{equation*}%
and by (I) and (\ref{33}),
\begin{equation}
L(\Gamma _{1}^{\prime \prime })<2\pi .  \label{31}
\end{equation}

If $L(\Gamma _{1}^{\prime \prime })<\pi ,$ the $g$ satisfies all the
conclusions.

If $L(\Gamma _{1}^{\prime \prime })\geq \pi ,$ then by (I), (\ref{31}) and
Theorem \ref{>pi} for the cases (a) and (b), there exists a normal mapping $%
f_{1}:\overline{\Delta }\rightarrow S$ such that
\begin{equation*}
L(f_{1},\partial \Delta )\leq L(g,\partial \Delta ),A(f_{1},\Delta )\geq
A(g,\Delta ),
\end{equation*}%
each natural edge of $f_{1}$ has spherical length strictly less than $\pi ,$
and
\begin{equation*}
V_{NE}(f_{1})\leq V_{NE}(g_{1}),\mathrm{\ }V_{E}(f_{1})\geq
V_{E}(g)+1,V(f_{1})\leq V(g)+2.
\end{equation*}%
Then $f_{1}$ satisfies all the desired conditions in the lemma with%
\begin{equation*}
V_{NE}(f_{1})\leq V_{NE}(f)-2,V_{E}(f_{1})\geq V_{E}(f)+1\ \mathrm{and\ }%
V(f_{1})\leq V(f).
\end{equation*}

Case $3$ can be treated as Case $2.$

In case $4$ we have%
\begin{equation}
V_{NE}(g)=V_{NE}(f)-3,V_{E}(g)=V_{E}(f),V(g)=V(f)-3,  \label{35}
\end{equation}%
and $g$ has a natural partition%
\begin{equation}
\Gamma _{g}=\Gamma _{1}^{\prime \prime \prime }+\Gamma _{4}+\dots +\Gamma
_{n-1},  \label{32}
\end{equation}%
where
\begin{equation*}
\Gamma _{1}^{\prime \prime \prime }=\overline{q_{n}q_{1}}+\overline{%
q_{1}q_{3}}+\overline{q_{3}q_{4}}=\Gamma _{n}+\Gamma _{1}^{\prime }+\Gamma
_{3}.
\end{equation*}%
Then by (I) and (\ref{33}) we have%
\begin{equation}
\Gamma _{1}^{\prime \prime \prime }<3\pi .  \label{34}
\end{equation}

If $L(\Gamma _{1}^{\prime \prime \prime })<\pi ,$ then by (I), (\ref{35})
and (\ref{32}), $g$ is the desire mapping.

If $L(\Gamma _{1}^{\prime \prime \prime })\geq \pi ,$ then by (I), (\ref{32}%
) and Theorem \ref{>pi} (a), (b) and (d), there exists a normal mapping $%
g_{1}:\overline{\Delta }\rightarrow S$ such that
\begin{equation*}
L(g_{1},\partial \Delta )\leq L(g,\partial \Delta ),A(g_{1},\Delta )\geq
A(g,\Delta ),
\end{equation*}%
each natural edge of $g_{1}$ has spherical length strictly less than $\pi ,$
and
\begin{equation*}
V_{NE}(g_{1})\leq V_{NE}(g_{1})+2,\mathrm{\ }V_{E}(g_{1})\geq
V_{E}(g)+1,V(g_{1})\leq V(g)+3.
\end{equation*}%
Then $g_{1}$ satisfies all the desired conditions in the lemma with (by (\ref%
{35}))%
\begin{equation*}
V_{NE}(g_{1})\leq V_{NE}(f)-1,V_{E}(g_{1})\geq V_{E}(f)+1\ \mathrm{and\ }%
V(g_{1})\leq V(f).
\end{equation*}%
This completes the proof.
\end{proof}

\begin{lemma}
\label{CC}In Cases C, there exists a normal mapping $g:\overline{\Delta }%
\rightarrow S$ such that
\begin{equation*}
L(g,\partial \Delta )<L(f,\partial \Delta )\text{\textrm{\ and }}A(g,\Delta
)>A(f,\Delta ),
\end{equation*}%
each natural edge of $\Gamma _{g}$ has spherical length strictly less than $%
\pi ,$ and%
\begin{equation*}
V_{NE}(g)\leq V_{NE}(f)-1,V_{E}(g)\geq V_{E}(f)+1\ \mathrm{and}\ V(g)\leq
V(f).
\end{equation*}
\end{lemma}

\begin{proof}
Assume Case C occurs and let $\Gamma _{1}^{\prime }=\overline{%
q_{1}q_{1}^{\prime }}$ and $\Gamma _{2}^{\prime }=\overline{q_{1}^{\prime
}q_{2}}.$ Then, considering that $q_{1},q_{1}^{\prime },q_{2}$ are contained
in the closure of the triangle domain $T$ which in on the left hand side of
the convex triangle $\overline{q_{1}q_{3}q_{2}q_{1}},$ we have%
\begin{equation}
L(\Gamma _{1}^{\prime })<\pi ,L(\Gamma _{2}^{\prime })<\pi ,\ L(\Gamma
_{1}^{\prime }+\Gamma _{2}^{\prime })<L(\Gamma _{1}+\Gamma _{2}),  \label{36}
\end{equation}%
and it is clear that%
\begin{equation*}
\Gamma _{1}^{\prime }+\Gamma _{2}^{\prime }-\Gamma _{2}-\Gamma _{1}
\end{equation*}%
is a quadrilateral and encloses a domain $T^{\prime }$ in $T$ that is on the
right hand side of $\Gamma _{1}+\Gamma _{2}.$ Then, by (\ref{36}), replacing
the the domain $T$ in the proof of Lemma \ref{CB} by $T^{\prime }$ and
repeating the extension arguments, we can obtain a normal mapping $g:%
\overline{\Delta }\rightarrow S$ such that%
\begin{equation}
L(g,\partial \Delta )<L(f,\partial \Delta )\text{\textrm{\ and }}A(g,\Delta
)>A(f,\Delta ),  \label{914}
\end{equation}%
and the boundary curve $\Gamma _{g}$ of $g$ has the following permitted
partition%
\begin{equation*}
\Gamma _{g}=\Gamma _{1}^{\prime }+\Gamma _{2}^{\prime }+\Gamma _{3}+\dots
+\Gamma _{n},
\end{equation*}%
which implies another permitted partition%
\begin{eqnarray}
\Gamma _{g} &=&\Gamma _{n}+\Gamma _{1}^{\prime }+\Gamma _{2}^{\prime
}+\Gamma _{3}+\dots +\Gamma _{n-1}.  \label{128-3} \\
&=&\overline{q_{n}q_{1}}+\overline{q_{1}q_{1}^{\prime }}+\overline{%
q_{1}^{\prime }q_{2}}+\overline{q_{2}q_{3}}+\dots +\overline{q_{n-1}q_{n}}.
\notag
\end{eqnarray}%
But here the terminal point $q_{1}^{\prime }$ of $\Gamma _{1}^{\prime }$,
which is the initial point of $\Gamma _{2}^{\prime },$ is in $E,$ and so we
have%
\begin{equation*}
V_{NE}(g)\leq V_{NE}(f)-1,V_{E}(g)=V_{E}(f)+1\ \mathrm{and\ }V(g)\leq V(f).
\end{equation*}

Now, there are four cases need to discuss.

\noindent \textbf{Case 1. }Neither $\Gamma _{n}+\Gamma _{1}^{\prime }=%
\overline{q_{n}q_{1}q_{1}^{\prime }}\ $nor $\Gamma _{2}^{\prime }+\Gamma
_{3}=\overline{q_{1}^{\prime }q_{2}q_{3}}\ $is a natural edge of $\Gamma
_{g}.$

\noindent \textbf{Case 2. }$\overline{q_{n}q_{1}q_{1}^{\prime }}$ is a
natural edge of $\Gamma _{g},$ while $\overline{q_{1}^{\prime }q_{2}q_{3}}$
is not.

\noindent \textbf{Case 3. }$\overline{q_{n}q_{1}q_{1}^{\prime }}$ is a
natural edge of $\Gamma _{g}$, while $\overline{q_{1}^{\prime }q_{2}q_{3}}\ $%
is not$.$

\noindent \textbf{Case 4. }Both $\overline{q_{n}q_{1}q_{1}^{\prime }}\ $and $%
\overline{q_{1}^{\prime }q_{2}q_{3}}\ $are natural edges of $\Gamma _{g}.$

In Case 1, (\ref{128-3}) is a natural partition, and $g$ is the desired
mapping.

In Case 2, $g$ has the natural partition%
\begin{equation*}
\Gamma _{g}=\Gamma _{1}^{\prime \prime }+\Gamma _{2}^{\prime }+\Gamma
_{3}+\dots +\Gamma _{n-1},
\end{equation*}%
where $\Gamma _{1}^{\prime \prime }=\Gamma _{n}+\Gamma _{1}^{\prime }=%
\overline{q_{n}q_{1}q_{1}^{\prime }},$ and it is clear that
\begin{equation}
V_{NE}(g)=V_{NE}(f)-2,V_{E}(g)=V_{E}(f)+1\ \mathrm{and\ }V(g)=V(f)-1.
\label{916}
\end{equation}%
and by (I) and (\ref{36})
\begin{equation*}
L(\Gamma _{1}^{\prime \prime })<2\pi .
\end{equation*}

If $L(\Gamma _{1}^{\prime \prime })<\pi ,$ then $g$ is the desired mapping
with (\ref{916}).

If $\pi \leq L(\Gamma _{1}^{\prime \prime })<2\pi ,$ then by (I) and Theorem %
\ref{>pi} (a) (note that $q_{1}^{\prime }\in E$ is the terminal point of $%
\Gamma _{1}^{\prime \prime })$, there exists a normal mapping $g_{1}:%
\overline{\Delta }\rightarrow S$ such that
\begin{equation*}
L(g_{1},\partial \Delta )\leq L(g,\partial \Delta ),A(g_{1},\Delta )\geq
A(g,\Delta ),
\end{equation*}%
each natural edge of $g_{1}$ has spherical length strictly less than $\pi ,$
and

\begin{equation*}
V_{NE}(g_{1})\leq V_{NE}(g),\mathrm{\ }V_{E}(g_{1})\geq V_{E}(g)+1\ \mathrm{%
and\ }V(f_{1})\leq V(g)+1.
\end{equation*}%
Then, by (\ref{914}) and (\ref{916}), $g_{1}$ satisfies all the desired
conclusions in Lemma \ref{CC} with%
\begin{equation*}
V_{NE}(f_{1})\leq V_{NE}(f)-2,\mathrm{\ }V_{E}(f_{1})\geq V_{E}(f)+2\
\mathrm{and\ }V(f_{1})\leq V(f).
\end{equation*}

Case $3$ can be treated as Case $2.$

In case $4$ we have%
\begin{equation}
V_{NE}(g)=V_{NE}(f)-3,V_{E}(g)=V_{E}(f)+1\ \mathrm{and\ }V(g)\leq V(f)-3
\label{38}
\end{equation}%
and $g$ has a natural partition%
\begin{equation}
\Gamma _{g}=\Gamma _{1}^{\prime \prime }+\Gamma _{2}^{\prime \prime }+\Gamma
_{4}+\dots +\Gamma _{n-1},  \label{37}
\end{equation}%
where $\Gamma _{1}^{\prime \prime }=\Gamma _{n}+\Gamma _{1}^{\prime }=%
\overline{q_{n}q_{1}q_{1}^{\prime }}\ $and $\Gamma _{2}^{\prime \prime
}=\Gamma _{2}^{\prime }+\Gamma _{3}=\overline{q_{1}^{\prime }q_{2}q_{3}}.$
By (\ref{36}) and (I), we have%
\begin{equation}
L(\Gamma _{1}^{\prime \prime })<2\pi ,L(\Gamma _{2}^{\prime \prime })<2\pi .
\label{919}
\end{equation}

If
\begin{equation}
L(\Gamma _{1}^{\prime \prime })<\pi ,L(\Gamma _{2}^{\prime \prime })<\pi ,
\label{920}
\end{equation}
then by (I), (\ref{914}), (\ref{38}) and (\ref{37}), $g$ is the desired
mapping.

If (\ref{920}) does not hold, then by (I), (\ref{37}), (\ref{919}) and the
fact that both $\Gamma _{1}^{\prime \prime }$ and $\Gamma _{2}^{\prime
\prime }$ have endpoints in $E,$ Theorem \ref{>pi} (a) or (c) applies to $g$%
, and then, there exists a normal mapping $f_{1}:\overline{\Delta }%
\rightarrow S$ such that
\begin{equation*}
L(f_{1},\partial \Delta )\leq L(g,\partial \Delta ),A(f_{1},\Delta )\geq
A(g,\Delta ),
\end{equation*}%
each natural edge of $f_{1}$ has spherical length strictly less than $\pi ,$
and%
\begin{equation*}
V_{NE}(f_{1})\leq V_{NE}(g_{1}),\mathrm{\ }V_{E}(f_{1})\geq V_{E}(g)+1,\
\mathrm{and\ }V(f_{1})\leq V(g)+2.
\end{equation*}%
Then $f_{1}$ satisfies all the desired conclusions of Lemma \ref{CC} with
(by (\ref{38}))
\begin{equation*}
V_{NE}(f_{1})\leq V_{NE}(f)-3,\mathrm{\ }V_{E}(f_{1})\geq V_{E}(f)\ \mathrm{%
and}\ V(f_{1})\leq V(f)-1.
\end{equation*}%
This completes the proof.
\end{proof}

\begin{lemma}
\label{CD}In Case D, there exists a normal mapping $g:\overline{\Delta }%
\rightarrow S$ such that
\begin{equation*}
L(g,\partial \Delta )<L(f,\partial \Delta )\text{\textrm{\ and }}A(g,\Delta
)>A(f,\Delta ),
\end{equation*}%
each natural edge of $\Gamma _{g}$ has spherical length strictly less than $%
\pi ,$ and%
\begin{equation*}
V_{NE}(g)\leq V_{NE}(f)-1,V_{E}(g)\geq V_{E}(f)+1\ \mathrm{and\ }V(g)\leq
V(f)+1.
\end{equation*}
\end{lemma}


\begin{proof}
In Case D, $q_{1}^{\prime }\in T$ and $q_{2}^{\prime }\in T$ are
the only points in $\overline{T}\cap E.$ Let $L$ be the line
segment in $\overline{T}$ that passes through $q_{1}^{\prime }$ and $%
q_{2}^{\prime }$ and has endpoints in $\partial T.$ Then there are two cases:

\noindent \textbf{Case 1. }$L$ intersects $\overline{q_{1}q_{3}}$.

\noindent \textbf{Case 2. }$L$ does not intersect $\overline{q_{1}q_{3}}$.

Assume Case 1 occurs and, without loss of generality, assume $q_{2}^{\prime
} $ is closer to $\overline{q_{1}q_{3}}$ than $q_{1}^{\prime }.$ Let $\Gamma
_{1}^{\prime }=\overline{q_{1}q_{1}^{\prime }}$ and $\Gamma _{2}^{\prime }=%
\overline{q_{1}^{\prime }q_{2}}$ (a))$.$ Then $\Gamma _{1}^{\prime
}$ and $\Gamma _{2}^{\prime }$ satisfy all the conditions in the
proof of Lemma \ref{CC}, and in this case, we can prove Lemma
\ref{CD} by exactly repeating the proof of Lemma \ref{CC}.

Assume Case 2 occurs. Then one endpoint $q_{1}^{\prime \prime }$ of $L$ is
in the interior of $\Gamma _{1}$ and the other endpoint $q_{2}^{\prime
\prime }$ of $L$ is in the interior of $\Gamma _{2}.$ Without loss of
generality, assume $q_{1}^{\prime \prime },q_{1}^{\prime },q_{2}^{\prime }$
and $q_{2}^{\prime \prime }$ are arranged in order on $L.$ Let $\Gamma
_{1}^{\prime }=\overline{q_{1}q_{1}^{\prime }},\Gamma ^{\prime \prime }=%
\overline{q_{1}^{\prime }q_{2}^{\prime }}$ and $\Gamma _{2}^{\prime }=%
\overline{q_{2}^{\prime }q_{3}}$. Then,
considering that $T$ is on the left hand side of the convex triangle $%
\overline{q_{1}q_{3}q_{2}q_{1}},$ we have that%
\begin{equation*}
L(\Gamma _{1}^{\prime })<\pi ,L(\Gamma ^{\prime \prime })=\frac{\pi }{2}%
,L(\Gamma _{2}^{\prime })<\pi ,
\end{equation*}%
\begin{equation*}
L(\Gamma _{1}^{\prime }+\Gamma ^{\prime \prime }+\Gamma _{2}^{\prime
})<L(\Gamma _{1}+\Gamma _{2});
\end{equation*}%
and the domain $T$ enclosed by $\Gamma _{1}^{\prime }+\Gamma ^{\prime \prime
}+\Gamma _{2}^{\prime }-\Gamma _{2}-\Gamma _{1}$ is a polygonal Jordan
domain on the right hand side of $\Gamma _{1}+\Gamma _{2}$ with
\begin{equation*}
\overline{T}\cap E=\{q_{1},q_{2}\}.
\end{equation*}%
Then by Lemma \ref{patch} and the extension arguments, there exists a normal
mapping $g:\overline{\Delta }\rightarrow S$ such that%
\begin{equation*}
L(g,\partial \Delta )<L(f,\partial \Delta )\text{\textrm{\ and }}A(g,\Delta
)>A(f,\Delta ).
\end{equation*}%
and $\Gamma _{g}$ has a permitted partition
\begin{equation*}
\Gamma _{g}=\Gamma _{1}^{\prime }+\Gamma ^{\prime \prime }+\Gamma
_{2}^{\prime }+\Gamma _{3}+\dots +\Gamma _{n},
\end{equation*}%
which implies the following permitted partition%
\begin{eqnarray*}
\Gamma _{g} &=&\Gamma _{n}+\Gamma _{1}^{\prime }+\Gamma ^{\prime \prime
}+\Gamma _{2}^{\prime }+\Gamma _{3}+\dots +\Gamma _{n-1} \\
&=&\overline{q_{n}q_{1}}+\overline{q_{1}q_{1}^{\prime }}+\overline{%
q_{1}^{\prime }q_{2}^{\prime }}+\overline{q_{2}^{\prime }q_{3}}+\overline{%
q_{2}q_{3}}+\dots +\overline{q_{n-1}q_{n}}.
\end{eqnarray*}%
Since $q_{1}^{\prime },q_{2}^{\prime }\in E,$ it is clear that%
\begin{equation*}
V_{NE}(g)\leq V_{NE}(f)-1,V_{E}(g)\geq V_{E}(f)+2\ \mathrm{and\ }V(g)\leq
V(f)+1.
\end{equation*}%
Now, there are four cases:

\noindent \textbf{Case 2.1. }None of $\Gamma _{n}+\Gamma _{1}^{\prime }=%
\overline{q_{n}q_{1}q_{1}^{\prime }}$ and $\Gamma ^{\prime \prime }+\Gamma
_{2}^{\prime }=\overline{q_{1}^{\prime }q_{2}^{\prime }q_{3}}$ is a natural
edge of $\Gamma _{g}.$

\noindent \textbf{Case 2.2. }$\overline{q_{n}q_{1}q_{1}^{\prime }}\ $is a
natural edge of $\Gamma _{g},$ while $\overline{q_{1}^{\prime }q_{2}^{\prime
}q_{3}}\ $is not.

\noindent \textbf{Case 2.3. }$\overline{q_{n}q_{1}q_{1}^{\prime }}$ is not a
natural edge of $\Gamma _{g},$ while $\overline{q_{1}^{\prime }q_{2}^{\prime
}q_{3}}\ $is.

\noindent \textbf{Case 2.4. }Both $\overline{q_{n}q_{1}q_{1}^{\prime }}$ and
$\overline{q_{1}^{\prime }q_{2}^{\prime }q_{3}}$ are natural edges of $%
\Gamma _{g}.$

The discussion for these cases is almost the same as that for the four Cases
1--4 in the proof of Lemma \ref{CC}, just with a little difference which
leads to that the desired mapping may has a number of $V(f)+1$ natural edges.
\end{proof}

\section{Decomposition and deformation of Riemann surfaces of normal
mappings \label{ss-9finite}}

In this section, we prove the following theorem, which is the first key step
to prove the main theorem in Section \ref{ss-p2}.

\begin{theorem}
\label{key}Let $f:\overline{\Delta }\rightarrow S$ be a normal mapping and
assume that each natural edge of $f$ has spherical length strictly less than
$\pi $. Then, there exist a finite number of normal mappings $f_{j}:%
\overline{\Delta }\rightarrow S,j=1,\dots m,$ with $m\geq 1,$ such that%
\begin{equation*}
\sum_{j=1}^{m}L(f_{j},\partial \Delta )\leq L(f,\partial \Delta
),\sum_{j=1}^{m}A(f_{j},\Delta )\geq A(f,\Delta ),
\end{equation*}%
and for each $j\leq m$ the followings hold.

(i) Each natural edge of $f_{j}$ has spherical length strictly less than $%
\pi $.

(ii) The boundary curve $\Gamma _{f_{j}}=f_{j}(z),z\in \partial \Delta $, is
locally convex in $S\backslash E,$ where $E=\{0,1,\infty \}.$

(iii) $f_{j}$ has no branched point in $S\backslash E.$
\end{theorem}

We first prove several lemmas before we prove this theorem.

\begin{lemma}
\label{E>3}Let $f:\overline{\Delta }\rightarrow S$ be a normal mapping and
assume that each natural edge of $\Gamma _{f}=f(z),z\in \partial \Delta ,$
has spherical length strictly less than $\pi .$ Then $V(f)\geq 3,$ and if in
addition $V_{NE}(f)=0,$ then $V(f)\geq 4.$
\end{lemma}

\begin{proof}
If $V(f)=1$, then $\Gamma _{f}$ itself is a natural edge that is a straight
and closed curve in $S$, and $L(f,\partial \Delta )<\pi .$ This is
impossible.

Assume $V(f)=2$ and $\Gamma _{f}=\Gamma _{1}+\Gamma _{2}$ is a natural
partition. Since $\Gamma _{f}$ is a closed curve, $L(\Gamma _{j})<\pi $ and $%
\Gamma _{j}$ is straight, $j=1,2$, we have $\Gamma _{1}=-\Gamma _{2}$
(ignoring a transformation of parameter) with $L(\Gamma _{1})=L(\Gamma
_{2})<\pi .$ Then, $S\backslash \Gamma _{f}$ contains at least one point in $%
E=\{0,1,\infty \}.$ Considering that $f$ is normal, we conclude that $%
f(\Delta )\supset S\backslash \Gamma _{f}$ contains at least one point of $%
E, $ which contradicts the assumption that $f$ is normal. Thus, $V(f)\geq 3.$

If in addition $V_{NE}(f)=0,$ then by the assumption, each natural edge of $%
\Gamma _{f}$ must be $\overline{0,1},\overline{1,0},\overline{1,\infty }$ or
$\overline{\infty ,1},$ and then since $\Gamma _{f}$ is a closed curve and $%
V(f)\geq 3$, we have $V(f)\geq 4.$
\end{proof}

\begin{lemma}
\label{V=3}Let $f:\overline{\Delta }\rightarrow S$ be a normal mapping and
assume that the followings hold.

(a) each natural edge of $\Gamma _{f}=f(z),z\in \partial \Delta ,$ has
spherical length strictly less than $\pi .$

(b) $V(f)=3.$

Then $f:\overline{\Delta }\rightarrow f(\overline{\Delta })$ is a
homeomorphism and $\Gamma _{f}$ is a generic convex triangle.
\end{lemma}

\begin{proof}
Let
\begin{equation*}
\alpha =\alpha _{1}+\alpha _{2}+\alpha _{2}
\end{equation*}%
be a natural partition of $\partial \Delta $ for $f$ and let
\begin{equation*}
\Gamma _{f}=\overline{q_{1}q_{2}}+\overline{q_{2}q_{3}}+\overline{q_{3}q_{1}}
\end{equation*}%
be the corresponding natural partition of $\Gamma _{f}=f(z),z\in \partial
\Delta .$ Then by (a), $f$ restricted to each $\alpha _{j}$ is a
homeomorphism onto $\Gamma _{j}=\overline{q_{j}q_{j+1}},$ where $%
q_{4}=q_{1}. $

We first show that $\overline{q_{1}q_{2}q_{3}}$ can not be contained in any
great circle of $S.$ Otherwise, by (a) and the definition of natural edges,
either $q_{3}\in \overline{q_{1}q_{2}}^{\circ }$ or $q_{1}\in \overline{%
q_{2}q_{3}}^{\circ },\ $where $\overline{q_{1}q_{2}}^{\circ }$ denotes the
interior of $\overline{q_{1}q_{2}}.$ But in the first case, $q_{3}$ is not a
natural vertex of $\Gamma _{f}$ and in the second case, $q_{1}$ is not a
natural vertex of $\Gamma _{f}.$ Thus $\overline{q_{1}q_{2}q_{3}}$ is not
contained in any great circle of $S.$

Then $\Gamma _{f}$ must be a triangle that is contained in some open
hemisphere $S^{\prime }$ of $S$ and $f$ maps $\partial \Delta $
homeomorphically onto $\Gamma _{f}\ $and then, since $f$ is normal, $f:%
\overline{\Delta }\rightarrow \overline{T}$ is a homeomorphism, where $T$ is
the domain inside $\Gamma _{f}.$ Since $f$ is normal, we also have $f(\Delta
)\cap E=\emptyset .$ Thus, $\overline{T}\subset S^{\prime }$ and then $%
\Gamma _{f}$ is a generic convex triangle.
\end{proof}

\begin{lemma}
\label{NE=01}Let $f:\overline{\Delta }\rightarrow S$ be a normal mapping
such that each natural edge of $f$ has spherical length strictly less than $%
\pi $. If $V_{NE}(f)=0,\ $then $L(f,\partial \Delta )\geq 2\pi ,$ and if $%
V_{NE}(f)=1,$ then $L(f,\partial \Delta )\geq \pi .$
\end{lemma}

\begin{proof}
If $V_{NE}(f)=0,$ then by Lemma \ref{E>3}, $V(f)\geq 4,$ and in this case
each natural edge of $\Gamma _{f}=f(z),z\in \partial \Delta ,$ has spherical
length $\frac{\pi }{2},$ and then $L(f,\partial \Delta )\geq 2\pi $. If $%
V_{NE}(f)=1,$ then by Lemma \ref{E>3}, $V(f)\geq 3,$ and then $f(\partial
\Delta )$ contains at least two point of $E;$ and since $\Gamma _{f}$ is
closed, we have $L(f,\partial \Delta )\geq \pi .$
\end{proof}

\begin{lemma}
\label{splift}Let $f:\overline{\Delta }\rightarrow S$ be a normal mapping
and let $p_{0}$ be a ramification point of $f$. Assume $\beta =\overline{%
q_{0}q_{1}}\subset S$ satisfies the followings.

(a) $q_{0}=f(p_{0})$ and $q_{1}\in E=\{0,1,\infty \}$.

(b) The interior $\beta ^{\circ }$ of $\beta $ has a neighborhood $N$ in $S$
such that $N$ is a polygonal Jordan domain and $f$ has no branched point in $%
N$.

(c) The boundary curve $\Gamma _{f}=\Gamma _{f}(z),z\in \partial \Delta ,$
has no natural vertex in $N.$

(d) Either
\begin{equation}
f(\partial \Delta )\cap N=\emptyset ,  \label{11-1}
\end{equation}%
or
\begin{equation}
p_{0}\in \partial \Delta \ \mathrm{and}\ f(\partial \Delta )\cap N=\beta
^{\circ }.  \label{11-2}
\end{equation}

Then there exist normal mappings $g_{1},g_{2}:\overline{\Delta }\rightarrow
S,$ such that the followings hold.

(i) Each natural edge of $\Gamma g_{j}$ is a natural edge of $\Gamma
_{f},j=1,2,$ and each natural edge of $\Gamma _{f}$ is a natural edge of
either $g_{1}$ or $g_{2}$.

(ii) $L(g_{1},\Delta )+L(g_{2},\Delta )=L(f,\Delta )$\ and\textrm{\ }$%
A(g_{1},\Delta )+A(g_{2},\Delta )=A(f,\Delta ).$

(iii) $%
V_{NE}(g_{1})+V_{NE}(g_{2})=V_{NE}(f),V_{E}(g_{1})+V_{E}(g_{2})=V_{E}(f),$ $%
V(g_{1})+V(g_{2})=V(f).$
\end{lemma}

\begin{proof}
By Lemma \ref{cov-1} or Corollary \ref{branched-lift}, there exist a point $%
q_{2}$ in $\beta ^{\circ }$ such that the section $\overline{q_{0}q_{2}}$ of
$\beta =\overline{q_{0}q_{1}}$ has a lift $\gamma \subset \overline{\Delta }$
from $p_{0}$ to some point $p_{2}\in \Delta $ with $\gamma \backslash
\{p_{0}\}\subset \Delta .$ Let $q^{\ast }\in \beta $ be the closest point to
$q_{1}$ in $\beta $ such that the section $\overline{q_{0}q^{\ast }}$ has a
lift $\alpha _{2}$ that is an extension of the lift $\gamma $ and that $%
\alpha _{2}^{\circ }\subset \Delta .$ We show that $q^{\ast }=q_{1}.$

Let $p^{\ast }$ be the terminal point of $\alpha _{2}$. Assume $q^{\ast
}\neq q_{1},$ i.e. $q^{\ast }\in \beta ^{\circ }.$ If $p^{\ast }\in \Delta ,$
then by (b) and Lemma \ref{cov-1}, $\alpha _{2}$ can be extended past $%
p^{\ast }$ to be a longer lift so that the extended part is still in $\Delta
$, which contradicts the definition of $p^{\ast }$ and $q^{\ast }.$ Thus, we
have $p^{\ast }\in \partial \Delta .$

Then by (c) and the definition of natural vertices there is a neighborhood $%
A_{p^{\ast }}$ of $p^{\ast }$ in $\partial \Delta ,$ such that $f$
restricted to $A_{p^{\ast }}\ $is a homeomorphism onto a section of $\beta
^{\circ }.$ On the other hand, by (b), (c) and Lemma \ref{cov-1}, there is a
neighborhood $U_{p^{\ast }}$ of $p^{\ast }$ in $\overline{\Delta }$ such
that $f$ restricted to $U_{p^{\ast }}$ is a homeomorphism onto $f(U_{p^{\ast
}})$ with $f(U_{p^{\ast }})\subset N$ and $f(U_{p^{\ast }})$ is a half-disc
whose boundary diameter is contained in $\beta ^{\circ }\cap f(U_{p^{\ast
}}) $. Thus, by (d), $U_{p^{\ast }}\cap \partial \Delta =U_{p^{\ast }}\cap
f^{-1}([0,+\infty ]),$ and then $\alpha _{2}\cap U_{p^{\ast }}\subset
U_{p^{\ast }}\cap \partial \Delta .$ This is a contradiction, since $\alpha
_{2}^{\circ }\subset \Delta .$ Thus we have proved that $\alpha _{2}$ is a
lift of the whole path $\beta $ with $\alpha _{2}^{\circ }\subset \Delta .$

Since $p_{0}$ is a ramification point, in case $p_{0}\in \Delta $, by Lemma %
\ref{cov-1}, $\beta $ has another lift $\alpha _{1}$ starting from $p_{0}$
such that $\alpha _{1}^{\circ }\subset \Delta .$ Since $f(\Delta )\cap
E=\emptyset \ $and the terminal point $q_{1}$ of $\beta $ is in $E,$ the
terminal points of $\alpha _{1}$ and $\alpha _{2}$ must land on $\partial
\Delta ,$ and by Lemma \ref{cut-3}, these terminal points are distinct each
other. Thus, $f,\alpha _{1},\alpha _{2}$ and $\beta $ satisfy all
assumptions of Corollary \ref{ccut-1}, and then the desired $g_{1}$ and $%
g_{2}$ follow.

In case (\ref{11-2}), by (c) there is a section $\alpha _{1}$ of $\partial
\Delta $ starting from $p_{0}$ so that $\alpha _{1}$ is a lift of $\beta ,$
and by Lemma \ref{cut-3}, the terminal points of $\alpha _{1}$ and $\alpha
_{2}$ are also distinct. Then, $f,\alpha _{1},\alpha _{2}$ and $\beta $
satisfy all assumptions of Corollary \ref{ccut-2}, and then the desired $%
g_{1}$ and $g_{2}$ follow as well. This completes the proof of the lemma.
\end{proof}

Now, we can prove Theorem \ref{key} in some special cases.

\begin{lemma}
\label{Vne=0/1}Theorem \ref{key} holds true if $V_{NE}(f)=0$ or $%
V_{NE}(f)=1. $
\end{lemma}

\begin{proof}
Let $f:\overline{\Delta }\rightarrow S$ be a normal mapping that satisfies
the assumption of Theorem \ref{key} and $V_{NE}(f)=0$ or $V_{NE}(f)=1.$

If $\Gamma _{f}$ is locally convex in $S\backslash E$ and $f$ has no
branched point in $S\backslash E,$ then $f$ itself satisfies the conclusion
of Theorem \ref{key}, and there is nothing to proof.

If $\Gamma _{f}$ is not locally convex in $S\backslash E,$ then $V_{NE}(f)=1$
and by Theorem \ref{1-nconvex}, there exists a normal mapping $g:\overline{%
\Delta }\rightarrow S,$ such that each natural edge of $g$ has spherical
length strictly less than $\pi ,$%
\begin{equation*}
L(g,\partial \Delta )\leq L(f,\partial \Delta ),A(g,\Delta )\geq A(f,\Delta
),
\end{equation*}%
and%
\begin{equation*}
V_{NE}(g)\leq V_{NE}(f)-1=0.
\end{equation*}%
and then $V_{NE}(g)=0,$ and in this case $\Gamma _{g}$ is locally convex in $%
S\backslash E.$ If Theorem \ref{key} holds for $g,$ then it is clear that
Theorem \ref{key} holds for $f.$ Thus we may assume that

(a) $\Gamma _{f}$ is locally convex in $S\backslash E.$

If $f$ has no branched point in $S\backslash E,$ then there is nothing to
prove again. Thus, we may complete the proof under the assumption that

(b) $\Gamma _{f}$ is locally convex in $S\backslash E$ and $f$ has a
branched point in $S\backslash E.$

Let $p_{0}\in \overline{\Delta }\backslash f^{-1}(E)$ be a ramification
point of $f$ and let $q_{0}=f(p_{0}).$ Since $V_{NE}(f)=0$ or $1,$ $\Gamma
_{f}$ has at most one natural vertex $q^{\ast }$ outside $E.$ Then by (a),
there is a shortest path $\beta =\overline{q_{0}q_{1}}$ from $q_{0}$ to some
point $q_{1}\in E$ such that either $\beta \cap f(\partial \Delta
)=\{q_{1}\} $ or $\beta \cap f(\partial \Delta )=\overline{q_{0}q_{1}}.$ The
later case occurs if and only if $q_{0}\in \Gamma _{f}\backslash E.$

We may assume $f$ has no branched point in $\beta \backslash
\{q_{0},q_{1}\}, $ otherwise we take the branched point in the interior of $%
\beta $ that is closest to $q_{1}.$ Then $\beta ^{\circ }$ has a
neighborhood $N$ satisfying the hypothesis of Lemma \ref{splift}, and then,
by Lemma \ref{splift}, there exist two normal mappings $f_{1},f_{2}:%
\overline{\Delta }\rightarrow S,$ satisfying the following two conditions.

(c) Each natural edge of $f_{j}$ is a natural edge of $f,$ $j=1,2$, and each
natural edge of $\Gamma _{f}$ is a natural edge of $f_{1}$ or $f_{2}.$

(d) $L(f,\partial \Delta )=L(f_{1},\partial \Delta )+L(f_{2},\partial \Delta
)$\ and\textrm{\ }$A(f,\Delta )=A(f_{1},\Delta )+A(f_{2},\Delta ).$

By (c), we have

\begin{equation}
n=V(f)=V(f_{1})+V(f_{2}).  \label{52}
\end{equation}

It is clear by (c) that if $V_{NE}(f)=0,$ then $%
V_{NE}(f_{1})=V_{NE}(f_{2})=0,$ and if $V_{NE}(f)=1,$ then the the unique
natural vertex $q^{\ast }$ of $\Gamma _{f}$ outside $E$ can not be contained
in both $\Gamma _{f_{1}}$ and $\Gamma _{f_{2}},$ but $q^{\ast }$ must be a
convex natural vertex of $\Gamma _{f_{1}}$ or $\Gamma _{f_{2}}$ and both $%
f_{1}$ and $f_{2}$ satisfy the assumption of Theorem \ref{key}. Summarizing,
we may assume

(e) $V_{NE}(f_{1})=0,$ $V_{NE}(f_{2})=1,$ and $f_{1}$ and $f_{2}$ satisfy
(a).

On the other hand, by Lemma \ref{E>3} and (e) we have%
\begin{equation}
V(f_{1})\geq 4\ \mathrm{and\ }V(f_{2})\geq 3.  \label{11--3}
\end{equation}%
Thus, we have
\begin{equation*}
n=V(f)\geq 7,
\end{equation*}
and by (\ref{52}) we have%
\begin{equation}
V(f_{1})\leq V(f)-3\ \mathrm{and\ }V(f_{2})\leq V(f)-3.  \label{11--4}
\end{equation}

We have in fact proved that under the assumption (b), $n\geq 7.$ Thus,
Theorem \ref{key} holds true in case (a) with $n\leq 6.$ From this and the
above arguments for the existence of $f_{1}$ and $f_{2}$ satisfying (c),
(d), (e) and (\ref{11--4}) we can prove the theorem, under the assumption
(a), by induction on $n=V(f)$. This completes the proof.
\end{proof}

\begin{proof}[Proof of Theorem \protect\ref{key}]
We prove Theorem \ref{key} by induction on the sum $V_{NE}(f)+V(f).$

By Lemma \ref{E>3} we have $V(f)\geq 3,$ and then $V_{NE}+V(f)=3$ holds only
in the case $V_{NE}=0,$ but by Lemma \ref{E>3}, $V_{NE}=0$ implies $V(f)\geq
4.$ Thus
\begin{equation*}
V_{NE}(f)+V(f)\geq 4,
\end{equation*}
and equality holds if and only if $V_{NE}(f)=0$ and $V(f)=4$, or $%
V_{NE}(f)=1 $ and $V(f)=3.$ Thus, by Lemma \ref{Vne=0/1}, Theorem \ref{key}
holds true in the case $V_{NE}(f)+V(f)=4.$

Now, let $k>4$ be a positive integer and assume that we have proved Theorem %
\ref{key} for the case $4\leq V_{NE}(f)+V(f)\leq k.$ Let $f$ be any normal
mapping that satisfies the assumption of Theorem \ref{key} with
\begin{equation}
V_{NE}(f)+V(f)=k+1.  \label{50}
\end{equation}%
We call this that $f$ is at the level $k+1,$ and will show that Theorem \ref%
{key} holds true for $f.$

Then, there are only three cases need to be discussed.

\noindent \textbf{Case 1. }The boundary curve $\Gamma _{f}=f(z),z\in
\partial \Delta $, is locally convex in $S\backslash E$ and $f$ has no
branched point in $S\backslash E.$

\noindent \textbf{Case 2. }$\Gamma _{f}$ is not convex at some natural
vertex $p_{1}\in \left( \partial \Delta \right) \backslash f^{-1}(E)$ of $f.$

\noindent \textbf{Case 3.} $\Gamma _{f}$ is locally convex in $S\backslash
E, $ and $f$ has a branched point in $S\backslash E.$

If Case 1 occurs, then $f$ itself satisfies the conclusion of Theorem \ref%
{key}, and then there is nothing to proof.

\textbf{Discussion of Case 2.} In this case, by Theorem \ref{1-nconvex} it
is clear that there exists a normal mapping $f_{1}:\overline{\Delta }%
\rightarrow S,$ such that each natural edge of $f_{1}$ has spherical length
strictly less than $\pi ,$
\begin{equation}
L(f_{1},\partial \Delta )\leq L(f,\partial \Delta ),A(f_{1},\Delta )\geq
A(f,\Delta ),  \label{129-12}
\end{equation}%
and%
\begin{equation}
V_{NE}(f_{1})\leq V_{NE}(f)-1\text{\textrm{\ and\ }}V(f_{1})\leq V(f)+1.
\label{51}
\end{equation}

If one of the equalities of (\ref{51}) fails, then by (\ref{50}), $f_{1}$ is
at the level of $k,$ and by the induction hypothesis, Theorem \ref{key}
holds for $f_{1}$ and by (\ref{129-12}), Theorem \ref{key} holds for $f.$ If
both of the equalities in (\ref{51}) hold true, then $f_{1}$ is still at the
level of $k+1,$ but%
\begin{equation}
V_{NE}(f_{1})=V_{NE}(f)-1.  \label{129-13}
\end{equation}

Then, $f_{1}$ satisfies the assumption of Theorem \ref{key}, and then we
return to Cases 1, 2, or 3. If Case 1 occurs for $f_{1},$ then Theorem \ref%
{key} holds for $f_{1},$ and then Theorem \ref{key} holds for $f$ by (\ref%
{129-12}). If Case 2 occurs, then we can replace $f$ by $f_{1}$ and repeat
the discussion of Case 2. By (\ref{129-13}), we can not always return to
Case 2 from Case 2. Thus, Repeating discussion for Case 2 finitely many
times, we return to either Case 1 or Case 3. If we return to Case 1, the
proof is completed, and if we return to Case 3, we continue the following
discussion.

\textbf{Discussion of Case 3. }By Theorem \ref{1-br-2-map}, there exist two
normal mappings $g_{j}:\overline{\Delta }\rightarrow S,j=1,2,$ such that the
followings hold.

(a1) Each natural edge of $g_{j}$ has spherical length strictly less than $%
\pi ,j=1,2.$

(a2) $\sum_{j=1}^{2}L(g_{j},\partial \Delta )\leq L(f,\partial \Delta
),\sum_{j=1}^{2}A(g_{j},\Delta )\geq A(f,\Delta )$.

(a3) $V_{NE}(g_{1})+V_{NE}(g_{2})\leq V_{NE}(f)+2.$

(a4) $V(g_{1})+V(g_{2})\leq V(f)+2.$\medskip

By Lemma \ref{E>3} and (a1), $V(g_{j})\geq 3,j=1,2,$ and so by (a4) we have

(a5) $V(g_{j})\leq V(f)-1,j=1,2.$

Then there are only two cases:

\noindent \textbf{Case 2.1 }$V_{NE}(g_{j})\geq 2,j=1,2.$

\noindent \textbf{Case 2.2 }$V_{NE}(g_{1})\leq V_{NE}(g_{2})$ and $%
V_{NE}(g_{1})=0$ or $1.$

\textbf{Discussion of Case 2.1}. In this case, by (a3) we have $%
V_{NE}(g_{j})\leq V_{NE}(f),$ and then by (a5) and (\ref{50}) we have
\begin{equation*}
V_{NE}(g_{j})+V(g_{j})\leq V_{NE}(f)+V(f)-1=k,j=1,2,
\end{equation*}%
i.e. both $g_{1}$ and $g_{2}$ are at level $\leq k.$ Then by (a1) and the
induction hypothesis, Theorem \ref{key} holds for $g_{1}$ and $g_{2},$ and
then by (a2) and the induction hypothesis, Theorem \ref{key} holds for $f.$

\textbf{Discussion of Case 2.2}. Assume
\begin{equation}
V_{NE}(g_{1})=0\ \mathrm{or\ }1.  \label{129-10}
\end{equation}%
Then by Lemma \ref{NE=01}
\begin{equation}
L(g_{1},\partial \Delta )\geq \pi .  \label{129-11}
\end{equation}

We first show that%
\begin{equation}
V_{NE}(g_{2})+V(g_{2})\leq V_{NE}(f)+V(f)\leq k+1.  \label{129-14}
\end{equation}

If $V_{NE}(g_{1})=0,$ then by Lemma \ref{E>3} we have $V(g_{1})\geq 4,$ and
then by (a3) and (a4) we have (\ref{129-14}). If $V_{NE}(g_{1})=1,$ then by
(a3) and (a5) we still have (\ref{129-14})

If $V_{NE}(g_{2})+V(g_{2})<k+1,$ then Theorem \ref{key} holds for $g_{2}$ by
the induction hypothesis$;$ and on the other hand, by (\ref{129-10}) and
(a1), Theorem \ref{key} holds for $g_{1}$; and therefore, by (a2) Theorem %
\ref{key} holds for $f$.

Now, we assume
\begin{equation*}
V_{NE}(g_{2})+V(g_{2})=k+1,
\end{equation*}%
i.e., $g_{2}$ is in the level $k+1.$ Then we can return to Cases 1--3 and
repeat the same discussion for $g_{2}.$ We can prove Theorem \ref{key} by
repeating the above arguments finitely many times, since by (\ref{129-11}),
Case 2.2 can not occur infinitely times. This completes the proof.
\end{proof}

\section{Decomposition of fat mappings\label{ss14}}

In this section we prove Theorem \ref{fat}. This is the third key step to
prove the main theorem in Section \ref{ss-p2}.

A normal mapping $g:\overline{\Delta }\rightarrow S$ is called \emph{fat} if
and only if $\Delta \backslash f^{-1}([0,+\infty ])$ has a component $D$
such that $f:D\rightarrow S\backslash \lbrack 0,+\infty ]$ is a
homeomorphism.

By Corollary \ref{by the way}, if $g$ satisfies all assumptions of Theorem %
\ref{decom}, then $g$ is fat if and only if%
\begin{equation*}
f(\partial D)\subset \lbrack 0,+\infty ].
\end{equation*}

\begin{theorem}
\label{fat}Let $f:\overline{\Delta }\rightarrow S$ be a normal mapping that
satisfies (a)--(d) of Theorem \ref{decom}, that is, the following conditions
(a)--(d) hold.

(a) Each natural edge of the boundary curve $\Gamma _{f}=f(z),z\in \partial
\Delta ,$ has length strictly less than $\pi $.

(b) $\Gamma _{f}=f(z),z\in \partial \Delta ,$ is locally convex in $%
S\backslash E,E=\{0,1,\infty \}.$

(c) $f$ has no branched point in $S\backslash E.$

(d) $\Gamma _{f}\cap \lbrack 0,+\infty ]$ contains at most finitely many
points.

If $f$ is fat, then there exist two normal mappings $f_{j}:\overline{\Delta }%
\rightarrow S,j=1,2,$ such that $f_{1}$ and $f_{2}$ satisfy (a)--(d) and
\begin{equation*}
A(f_{1},\Delta )+A(f_{2},\Delta )=A(f,\Delta )-4\pi ,
\end{equation*}%
\begin{equation*}
L(f_{1},\partial \Delta )+L(f_{2},\partial \Delta )=L(f,\partial \Delta ),
\end{equation*}%
$f_{1}$ maps $[-1,1]\subset \overline{\Delta }$ homeomorphically onto $%
\overline{0,1}\subset S$ and $f_{2}$ maps $[-1,1]$ homeomorphically onto $%
\overline{1,\infty }\subset S.$
\end{theorem}

The geometrical meaning of this theorem is that we can cut off the whole
Riemann sphere $S,$ with $[0,+\infty ]$ being removed, from the interior of
the Riemann surface of $f$ and then sew up the cut edges along $[0,+\infty
]. $ Then we obtain two Riemann surfaces that are only joint at $1\in S,$
and the boundary curve of these two surfaces compose the boundary curve $%
\Gamma _{f}.$

\begin{proof}
Let $\Delta _{1}$ be a component of $\Delta \backslash f^{-1}([0,+\infty ])$
such that $f(\partial \Delta _{1})\subset \lbrack 0,+\infty ].$


Then by Corollary \ref{by the way}, $f|_{\overline{\Delta _{1}}}:\overline{%
\Delta _{1}}\rightarrow S$ is normal and surjective, $f(\partial \Delta
_{1})=[0,+\infty ]$ and $f$ restricted to $\Delta _{1}$ is a homeomorphism
onto $S\backslash \lbrack 0,+\infty ]$. Then the restriction
\begin{equation*}
f:\partial \Delta _{1}\rightarrow \lbrack 0,+\infty ]
\end{equation*}%
is a folded two to one mapping and we can express $\partial \Delta _{1}$ to
be%
\begin{equation*}
\partial \Delta _{1}=\beta _{1}+\beta _{2}+\beta _{3}+\beta _{4}
\end{equation*}%
such that $f$ maps $\beta _{1},\beta _{2},\beta _{3},\beta _{4}$
homeomorphically onto $\overline{0,1},\overline{1,\infty },\overline{\infty
,1},$ $\overline{1,0},$ respectively and $\beta _{1},\beta _{2},\beta
_{3},\beta _{4}$ are arranged anticlockwise in $\partial \Delta _{1}.$
Denote by $p_{j}\ $the initial points of $\beta _{j},j=1,2,3,4.$ Then
\begin{equation*}
f(p_{1})=0,f(p_{2})=f(p_{4})=1,f(p_{3})=\infty ,
\end{equation*}
which implies
\begin{equation*}
p_{j}\in \partial \Delta ,j=1,2,3,4,
\end{equation*}%
by the definition of normal mappings. We denote by $\alpha _{j}$ the section
of $\partial \Delta $ from $p_{j}$ to $p_{j+1},j=1,2,3,4$ ($p_{5}=p_{1}).$

We first show that the interior of $\beta _{j}$ is contained in $\Delta $
for $j=1,2,3,4.$ Assume $\beta _{1}$ has an interior point $p_{0}$ with $%
p_{0}\in \partial \Delta .$ Then it is clear that $p_{0}$ is in the interior
of $\alpha _{1}$ and $p_{0}\neq 0,1,\infty ,$ and then, by (b) the section $%
\Gamma _{\alpha _{1}}=f(z),z\in \alpha _{1},$ of $\Gamma _{f}$ is convex at $%
p_{0},$ and by (c), $f$ is regular at $p_{0}.$ On the other hand, $\Gamma
_{\beta _{1}}=f(z),z\in -\beta _{1},$ which is the simple path $\overline{1,0%
},$ is obviously convex by the definition$.$ Therefore, Lemma \ref{position}
applies to $p_{0},\alpha _{1},-\beta _{1}$ and $f$, and then $\alpha _{1}$
has a neighborhood contained in $\beta _{1}.$ This contradicts (d), for $%
\alpha _{1}\subset \partial \Delta $ and $\beta _{1}=\overline{0,1}\subset
\lbrack 0,+\infty ].$ Thus the interior of $\beta _{1}$ is contained in $%
\Delta .$ For the same reason, the interiors of $\beta _{2},\beta _{3}$ and $%
\beta _{4}$ are all contained in $\Delta $ $.$

We have proved that $\Delta \backslash \Delta _{1}$ contains four disjoint
Jordan domains $D_{j}$ enclosed by $\alpha _{j}-\beta _{j}$ with $\overline{%
D_{j}}\cap \overline{D_{j+1}}=\{p_{j+1}\},j=1,2,3,4$ ($D_{5}=D_{1}$ and $%
p_{5}=p_{1}).$

Now, we glue $\overline{D_{1}}$ and $\overline{D_{4}}$ along $\beta _{1}$
and $-\beta _{4}$ so that $x\in \beta _{1}$ and $y\in -\beta _{4}$ are
identified if and only if $f(x)=f(y).$ The glued closed domain can be
understood to be the unit disk $\overline{\Delta }$ so that $\beta _{1}$ and
$-\beta _{4}$ both become the diameter $[-1,1]$ of $\overline{\Delta }.$ In
this way we have in fact glued the restrictions $f|_{\overline{D_{1}}}$ and $%
f|_{\overline{D_{4}}}$ to be a normal mapping $f_{1}:\overline{\Delta }%
\rightarrow S$ such that $f_{1}$ maps $[-1,1]\subset \overline{\Delta }$
homeomorphically onto $\overline{0,1}\subset S.$

Similarly, we can glue the restrictions $f|_{\overline{D_{2}}}$ and $f|_{%
\overline{D_{3}}}$ to be a normal mapping $f_{2}:\overline{\Delta }%
\rightarrow S$ such that $f_{2}$ maps $[-1,1]$ homeomorphically onto $%
\overline{1,\infty }.$

It is clear that $f_{1}$ and $f_{2}$ satisfies all the conclusions of the
Theorem. As a matter of fact, the above process just cut off $f(\Delta
_{1}), $ the sphere $S$ with $[0,+\infty ]$ being removed, from the interior
of the Riemann surface of $f$ and then sew up the cut edges along $%
[0,+\infty ].$ Then we obtain two Riemann surfaces that are only joint at $%
1\in S,$ and the boundary curve of these two surfaces compose the boundary
curve $\Gamma _{f}. $ This completes the proof.
\end{proof}

\section{Proof of the main theorem\label{ss-p2}}

We first prove the main theorem under certain conditions.

\begin{lemma}
\label{nofat}Let $f:\overline{\Delta }\rightarrow S$ be a normal mapping
that is not fat and satisfies (a)--(d) of Theorem \ref{decom}. Then%
\begin{equation*}
A(f,\Delta )<2L(f,\partial \Delta ),
\end{equation*}%
and
\begin{equation*}
A(f,\Delta )<h_{0}L(f,\partial \Delta )-A(f,\Delta ),
\end{equation*}%
where $h_{0}$ is given by (\ref{best}).
\end{lemma}

\begin{proof}
The second inequality follows from the first one directly, for $h_{0}>4.$ So
we only prove the first inequality.

Since $f$ is not fat, for each component $D$ of $\Delta \backslash
f^{-1}([0,+\infty ]),$ $f(\partial D)\backslash \lbrack 0,+\infty ]\neq
\emptyset ,$ and so by (d) the interior $\alpha _{0}$ of $\left( \partial
D\right) \cap (\partial \Delta )$ is not empty, and then by Theorem \ref%
{decom},
\begin{equation*}
L(f,\left( \partial D\right) \cap (\partial \Delta ))>L(f,(\partial
D)\backslash \left( \partial \Delta \right) ),
\end{equation*}%
$f(\overline{D})$ is contained in some hemisphere of $S,$ which implies%
\begin{equation*}
A(f(D))<2\pi ,
\end{equation*}%
and $f$ restricted to $D$ is a homeomorphism.

Then we have%
\begin{equation}
2L(f,\left( \partial D\right) \cap (\partial \Delta ))>L(f,\left( \partial
D\right) \cap (\partial \Delta ))+L(f,(\partial D)\backslash \left( \partial
\Delta \right) )=L(f,\partial D).  \label{13-1}
\end{equation}

If $L(f,\partial D)\geq 2\pi ,$ then we have%
\begin{equation*}
A(f,D)=A(f(D))<2\pi \leq L(f,\partial D),
\end{equation*}%
and if $L(f,\partial D)<2\pi ,$ then by Theorem \ref{good} we have
\begin{equation*}
A(f,D)<L(f,\partial D),
\end{equation*}%
and then by (\ref{13-1}), in both cases we have%
\begin{equation}
A(f,D)<2L(f,\left( \partial D\right) \cap (\partial \Delta )).  \label{15-1}
\end{equation}

By Theorem \ref{decom} it is clear that for any pair $D_{1}$ and $D_{2}$ of
distinct components of $\Delta \backslash f^{-1}([0,+\infty ]),$ $\left(
\partial D_{1}\right) \cap (\partial \Delta )$ and $\left( \partial
D_{2}\right) \cap (\partial \Delta )$ contains at most two common points,
and on the other hand, by the same theorem, $\Delta \backslash
f^{-1}([0,+\infty ])$ contains only finitely many components. Thus, we have%
\begin{equation*}
L(f,\partial \Delta )=\sum_{D}L(f,\left( \partial D\right) \cap (\partial
\Delta )),
\end{equation*}%
where the sum runs over all components $D$ of $\Delta \backslash
f^{-1}([0,+\infty ])$. Then, summing up (\ref{15-1}), we have%
\begin{equation*}
A(f,\Delta )=\sum_{D}A(f,D)<2\sum_{D}L(f,(\partial D)\cap (\partial \Delta
))=2L(f,\partial \Delta ).
\end{equation*}%
This completes the proof.
\end{proof}

\begin{lemma}
\label{key1}Let $f:\overline{\Delta }\rightarrow S$ be a normal mapping that
is fat and satisfies (a)--(d) of Theorem \ref{decom}. Then%
\begin{equation*}
A(f,\Delta )\leq h_{0}L(f,\partial \Delta )-4\pi .
\end{equation*}%
where $h_{0}$ is given by (\ref{best}).
\end{lemma}

\begin{proof}
By Theorem \ref{fat}, there exist normal mappings $g_{j}:\overline{\Delta }%
\rightarrow S,j=1,2,\dots ,n+1,$ such that for each $j\leq n+1$ the
followings hold.

(e) Each $g_{j}$ satisfies all assumptions (a)--(d) of theorem \ref{decom}, $%
j=1,2,\dots ,n+1$.

(f) Each $g_{j}$ is not fat, $j=1,2,\dots ,n+1$.

(g) $g_{j}$ maps the diameter $[-1,1]$ of $\overline{\Delta }$
homeomorphically onto the real interval $\overline{0,1}$ in $S$ or onto $%
\overline{1,\infty }$ in $S.$

(h) $A(f,\Delta )=4n\pi +\sum_{j=1}^{n+1}A(g_{j},\Delta ),L(f,\partial
\Delta )=\sum_{j=1}^{n+1}L(g_{j},\partial \Delta ).$

Let $j$ be any positive integer with $j\leq n+1$ and consider the mapping $%
g_{j}.$ By (e) and (f), Lemma \ref{nofat} applies. Then we have%
\begin{equation}
A(g_{j},\Delta )<2L(g_{j},\partial \Delta ),  \label{13.5}
\end{equation}%
and then
\begin{eqnarray*}
4\pi +A(g_{j},\Delta ) &<&2L(g_{j},\partial \Delta )+4\pi \\
&=&(2+\frac{4\pi }{L(g_{j},\partial \Delta )})L(g_{j},\partial \Delta ),
\end{eqnarray*}%
and, considering that $2+\frac{4\pi }{L(g_{j},\partial \Delta )}\leq 4$ in
the case $L(g_{j},\partial \Delta )\geq 2\pi ,$ we have

(k) If $L(g_{j},\partial \Delta )\geq 2\pi ,$ then

\begin{equation*}
4\pi +A(g_{j},\Delta )<4L(g_{j},\partial \Delta ).
\end{equation*}

By Theorem \ref{good}, we have

(l) If $\sqrt{2}\pi \leq L(g_{j},\partial \Delta )<2\pi ,$ then%
\begin{equation*}
4\pi +A(g_{j},\Delta )<4L(g_{j},\partial \Delta ).
\end{equation*}

By (g) and Theorem \ref{good2}, we have

(m) If $L(g_{j},\partial \Delta )<\sqrt{2}\pi ,$ then%
\begin{equation*}
4\pi +A(g_{j},\Delta )\leq h_{0}L(g_{j},\partial \Delta ),
\end{equation*}%
where $h_{0}$ is given by (\ref{best}).

Summarizing (k)--(m) and the fact that $h_{0}>4$, we have in any case%
\begin{equation*}
4\pi +A(g_{j},\Delta )\leq h_{0}L(g_{j},\partial \Delta ),j=1,\dots ,n+1,
\end{equation*}%
and then by (h), we have

\begin{eqnarray*}
A(f,\Delta ) &=&4n\pi +\sum_{j=1}^{n+1}A(g_{j},\Delta ) \\
&=&\sum_{j=1}^{n+1}\left( 4\pi +A(g_{j},\Delta )\right) -4\pi \\
&\leq &h_{0}\sum_{j=1}^{n+1}L(g_{j},\partial \Delta )-4\pi \\
&=&h_{0}L(f,\partial \Delta )-4\pi .
\end{eqnarray*}

This completes the proof.
\end{proof}

\begin{proof}[Proof of the Main Theorem]
Let $f:\overline{\Delta }\rightarrow S$ be any nonconstant holomorphic
mapping such that $f(\Delta )\cap E=\emptyset .$ Then for any positive
number
\begin{equation}
\varepsilon <\frac{1}{4h_{0}}\min \{4\pi ,A(f,\Delta )\},  \label{15-2}
\end{equation}%
there exists a Jordan domain $D\subset \Delta $ such that $f$ restricted to $%
D$ is a normal mapping,%
\begin{equation}
A(f,D)\geq A(f,\Delta )-\varepsilon \mathrm{\ and\ }L(f,\partial
D)<L(f,\Delta )+\varepsilon ,  \label{13-4}
\end{equation}%
and the following condition holds.

(1) Each natural edge of the restricted mapping $f|_{\overline{D}}$ has
spherical length strictly less than $\pi $.

Let $h$ be a homeomorphism from $\overline{\Delta }$ onto $\overline{D}$ and
let $F=f\circ h.$ Then by (1) $F:\overline{\Delta }\rightarrow S$ is a
normal mapping satisfying the assumption of Theorem \ref{key} with $%
A(F,\Delta )=A(f,D)\ $and $L(F,\partial \Delta )=L(f,\partial D).$ Then by (%
\ref{13-4}) we have that
\begin{equation}
A(F,\Delta )\geq A(f,\Delta )-\varepsilon \ \mathrm{and}\ L(F,\partial
\Delta )<L(f,\Delta )+\varepsilon ;  \label{13-5}
\end{equation}%
and by Theorem \ref{key} there exist a number of $m$ normal mappings $f_{j}:%
\overline{\Delta }\rightarrow S,j=1,\dots m,$ such that%
\begin{equation}
\sum_{j=1}^{m}A(f_{j},\Delta )\geq A(F,\Delta )\ \mathrm{and\ }%
\sum_{j=1}^{m}L(f_{j},\partial \Delta )\leq L(F,\partial \Delta ),
\label{13-6}
\end{equation}%
and for each $j$ the followings hold.

(A) Each natural edge of the boundary curve $\Gamma _{f_{j}}=f_{j}(z),z\in
\partial \Delta $, has spherical length strictly less than $\pi $.

(B) $\Gamma _{f_{j}}=f_{j}(z),z\in \partial \Delta $, is locally convex in $%
S\backslash E.$

(C) $f_{j}$ has no branched point in $S\backslash E.$

Then, by Lemma \ref{pertur}, for the above $\varepsilon $ and each $j\leq m,$
there exists a normal mapping $g_{j}:\overline{\Delta }\rightarrow S$ such
that%
\begin{equation}
A(g_{j},\Delta )\geq A(f_{j},\Delta ),L(g_{j},\partial \Delta
)<L(f_{j},\partial \Delta )+\frac{\varepsilon }{m},  \label{13-7}
\end{equation}%
and $g_{j}$ satisfies (d) in Theorem \ref{decom} and (A)--(C), say, $g_{j}$
satisfies all hypotheses of Theorem \ref{decom}. Then by Lemmas \ref{nofat}
and \ref{key1} we have%
\begin{equation}
A(g_{j},\Delta )\leq h_{0}L(g_{j},\partial \Delta )-\min \{A(g_{j},\Delta
),4\pi \},j=1,\dots ,m.  \label{13-8}
\end{equation}%
On the other hand, if for some $j_{0}\leq m,$ $A(g_{j_{0}},\Delta )\geq 4\pi
,$ then we have
\begin{equation*}
\sum_{j=1}^{m}\min \{A(g_{j},\Delta ),4\pi \}\geq 4\pi ,
\end{equation*}%
and if $A(g_{j},\Delta )<4\pi $ for all $j\leq m,$ then we have, by (\ref%
{13-7}), (\ref{13-6}), (\ref{13-5}) and (\ref{15-2}), that%
\begin{eqnarray*}
&&\sum_{j=1}^{m}\min \{A(g_{j},\Delta ),4\pi \} \\
&=&\sum_{j=1}^{m}A(g_{j},\Delta )\geq \sum_{j=1}^{m}A(f_{j},\Delta )\geq
A(F,\Delta ) \\
&>&A(f,\Delta )-\varepsilon >\frac{4h_{0}-1}{4h_{0}}A(f,\Delta );
\end{eqnarray*}%
and thus, in both cases, we have%
\begin{equation}
\sum_{j=1}^{m}\min \{A(g_{j},\Delta ),4\pi \}\geq \min \{\frac{4h_{0}-1}{%
4h_{0}}A(f,\Delta ),4\pi \}.  \label{13-9}
\end{equation}

Summing up the inequalities of (\ref{13-8}), by (\ref{13-9}) we have%
\begin{eqnarray}
\sum_{j=1}^{m}A(g_{j},\Delta ) &\leq &\sum_{j=1}^{m}h_{0}L(g_{j},\partial
\Delta )-\sum_{j=1}^{m}\min \{A(g_{j},\Delta ),4\pi \}  \label{13-10} \\
&\leq &\sum_{j=1}^{m}h_{0}L(g_{j},\partial \Delta )-\min \{\frac{4h_{0}-1}{%
4h_{0}}A(f,\Delta ),4\pi \}.  \notag
\end{eqnarray}

By (\ref{13-7}), (\ref{13-6}) and (\ref{13-5}) we have
\begin{eqnarray*}
\sum_{j=1}^{m}h_{0}L(g_{j},\partial \Delta )
&<&\sum_{j=1}^{m}h_{0}L(f_{j},\partial \Delta )+\varepsilon h_{0}\leq
h_{0}L(F,\partial \Delta )+\varepsilon h_{0} \\
&<&h_{0}L(f,\Delta )+2\varepsilon h_{0},
\end{eqnarray*}%
i.e.%
\begin{equation}
\sum_{j=1}^{m}h_{0}L(g_{j},\partial \Delta )<h_{0}L(f,\Delta )+2\varepsilon
h_{0}.  \label{13-11}
\end{equation}

By (\ref{13-5})--(\ref{13-7}) we have

\begin{equation}
A(f,\Delta )\leq \sum_{j=1}^{m}A(g_{j},\Delta )+\varepsilon .  \label{13-12}
\end{equation}%
Therefore, we have by (\ref{13-10})--(\ref{13-12})%
\begin{eqnarray*}
A(f,\Delta ) &\leq &\sum_{j=1}^{m}A(g_{j},\Delta )+\varepsilon  \\
&\leq &\sum_{j=1}^{m}h_{0}L(g_{j},\partial \Delta )-\min \{\frac{4h_{0}-1}{%
4h_{0}}A(f,\Delta ),4\pi \} \\
&<&h_{0}L(f,\Delta )+2\varepsilon h_{0}+\varepsilon -\min \{\frac{4h_{0}-1}{%
4h_{0}}A(f,\Delta ),4\pi \} \\
&\leq &h_{0}L(f,\Delta )+\frac{2h_{0}+1}{4h_{0}}\min \{A(f,\Delta ),4\pi
\}-\min \{\frac{4h_{0}-1}{4h_{0}}A(f,\Delta ),4\pi \},
\end{eqnarray*}%
and considering that $h_{0}>4,$ we have%
\begin{equation*}
A(f,\Delta )<h_{0}L(f,\Delta ).
\end{equation*}%
It remains to show that the lower is sharp.

We give an example to show that $h_{0}$ given by (\ref{best}) is a sharp
lower bound of the constant $h$ in (\ref{a6}).

As in Section \ref{1SS-Intro}, we denote by $D$ the spherical disk in $S$
with diameter $\overline{1,\infty },$ the shortest path in $S$ from $1$ to $%
\infty ,$ and for $l\in \lbrack \pi ,\sqrt{2}\pi ]$ denote by $D_{l}$ the
domain contained in the disk $D$ such that the boundary $\partial D_{l}$ is
composed of two congruent circular arcs, each of which has endpoints $%
\{1,\infty \}$ and spherical length $\frac{l}{2}$. Then, $l=L(\partial D_{l})
$ and by (\ref{1.1}) and (\ref{1.4}), the number $h_{0}$ given by (\ref{best}%
) is the maximum value of the function $\frac{4\pi +A(D_{l})}{l}$ and%
\begin{equation*}
h_{0}=\frac{4\pi +A(D_{l_{0}})}{l_{0}}
\end{equation*}%
for some $l_{0}\in (\pi ,\sqrt{2}\pi )$.

It is clear that $D_{l_{0}},$ regarded as a domain in $\mathbb{C},$ is an
angular domain whose vertex is $1$ and bisector is the ray $[1,+\infty )$ in
$\mathbb{C}$. We denote by $2\theta _{0}$ the value of the angle of this
angular domain$.$ Then it is clear that $\theta _{0}<\frac{\pi }{2}.$

Let $M_{1}$ be the angular domain in $\mathbb{C}$ defined by%
\begin{equation*}
M_{1}=\{re^{i\theta };\ r>0,0<\theta <\frac{\theta _{0}}{m}\}
\end{equation*}%
and let $\Sigma _{1}$ be the angular domain in $\mathbb{C}$ defined by
\begin{equation*}
\Sigma _{1}=\{1+re^{i\theta };\ 0<r<+\infty ,-\theta _{0}<\theta <\pi \}.
\end{equation*}%
Then it is easy to construct a homeomorphism $f_{0}$ from the closure $%
\overline{M_{1}}$ of $M_{1}$ in $\overline{\mathbb{C}}$ onto the closure $%
\overline{\Sigma _{1}}$ of $\Sigma _{1}$ in $\overline{\mathbb{C}},$ such
that $f_{0}$ maps the ray $\arg z=\frac{\theta _{0}}{m}$ onto the ray $\arg
z=\pi ,$ maps the interval $[0,1]$ onto itself increasingly, maps the ray $%
[1,+\infty ]$ onto the ray
\begin{equation*}
z=1+re^{-i\theta _{0}},r\in \lbrack 0,+\infty ]
\end{equation*}%
and $f_{0}$ is holomorphic on $M_{1}.$

Let%
\begin{equation*}
M_{2}=e^{\frac{\theta _{0}}{m}}M_{1}=\{e^{\frac{\theta _{0}}{m}}z;\ z\in
M_{1}\}.
\end{equation*}%
Then, by the Schwarz symmetry principle, we can extend $f_{0}$ to be an open
and continuous mapping $f_{1}$ from the closed angular domain%
\begin{equation*}
A_{1}=\overline{M_{1}\cup M_{2}}=\{z\in \overline{\mathbb{C}};0\leq \arg
z\leq \frac{2\theta _{0}}{m}\}
\end{equation*}%
onto $\overline{\mathbb{C}}$ such that $f_{1}\ $maps the segments
\begin{equation*}
l_{k}=\left\{ re^{i\frac{k\theta _{0}}{m}},r\in \lbrack 0,1]\right\} ,k=0,1,
\end{equation*}%
homeomorphically onto the interval $[0,1],$ respectively; $f_{1}$ maps the
segments
\begin{equation*}
L_{k}=\left\{ re^{i\frac{k\theta _{0}}{m}},r\in \lbrack 1,+\infty ]\right\}
,k=0,1,
\end{equation*}%
homeomorphically onto the segments
\begin{equation*}
l^{-}=\left\{ 1+re^{-i\theta _{0}},r\in \lbrack 0,+\infty ]\right\}
\end{equation*}%
and%
\begin{equation*}
l^{+}=\left\{ 1+re^{i\theta _{0}},r\in \lbrack 0,+\infty ]\right\} ,
\end{equation*}%
respectively; and $f_{1}$ restricted to the domain $A_{1}$ is a holomorphic
mapping that covers the domain $D_{l_{0}}$ two times and covers the domain $%
\mathbb{C}\backslash \overline{D_{l_{0}}}$ one times.

Let
\begin{equation*}
A_{1}^{\ast }=A_{1}^{\circ }\cup l_{0}^{\circ }\cup l_{1}^{\circ },
\end{equation*}%
where $A_{1}^{\circ }$ is the interior of $A_{1},$ and let
\begin{equation*}
A_{k}^{\ast }=e^{i\frac{2\left( k-1\right) \theta _{0}}{m}}A_{1}^{\ast
}=\{e^{i\frac{2\left( k-1\right) \theta _{0}}{m}}z;z\in A_{1}^{\ast
}\},k=2,\dots ,m.
\end{equation*}%
Then for each $k=1,\dots ,m$ we can defined a continuous function $f_{k}$ on
$A_{k}^{\ast }$ inductively: $f_{k+1}$ is obtained from $f_{k}$ by Schwarz
symmetry principle cross the symmetry axis
\begin{equation*}
l_{k}=\{re^{i\frac{2k\theta _{0}}{m}},r\in \lbrack 0,1]\}.
\end{equation*}

Let $H^{+}$ be the upper half plane $\mathrm{Im}z>0$ in $\mathbb{C}$ and let
\begin{equation*}
K=H^{+}\backslash \cup _{k=1}^{m-1}\{re^{i\frac{2k\theta _{0}}{m}},r\in
\lbrack 1,+\infty )\}.
\end{equation*}%
Then $K=\left( \cup _{k=1}^{m}A_{k}^{\ast }\right) \backslash (l_{0}\cup
l_{m})$ and $f_{1},\dots ,f_{m}$ can be patched to be a holomorphic function
$f$ defined on $K\cup (-\infty ,+\infty )$, by the Schwarz symmetric
principle$.$ It is clear that $K$ is a simply connected domain and there is
a conformal mapping $h$ from $\Delta $ onto $K$ such that $h$ can be
extended to be a continuous function $\widetilde{h}$ such that when $z$ goes
around $\partial \Delta $ once, $\widetilde{h}(z)$ describes the boundary
section $(-\infty ,+\infty )$ of $K$ once and the boundary sections $\{re^{i%
\frac{k\theta _{0}}{m}},r\in \lbrack 1,+\infty )\}$ twice for $k=1,\dots
,m-1.$

Let $g=f\circ \widetilde{h}.$ Then $g:\overline{\Delta }\rightarrow
\overline{\mathbb{C}}$ is a continuous mapping that is holomorphic in $%
\Delta $ and when we regard $g$ as a mapping from $\overline{\Delta }$ to $S,
$ $g$ restricted to $\Delta $ covers $S\backslash D_{l_{0}}$ by $m$ times
and covers $D_{l_{0}}$ by $2m$ times and the boundary curve $\Gamma
_{g}=g(z),z\in \partial \Delta ,$ covers $\partial D_{l_{0}}$ by $m$ times
and covers the shortest path $\overline{0,1}$ in $S$ from $0$ to $1$ by $2$
times.

Then we have
\begin{equation*}
A(g,\Delta )=4m\pi +mA(D_{l_{0}})
\end{equation*}%
and
\begin{equation*}
L(g,\partial \Delta )=2L(\overline{0,1})+mL(\partial D_{l_{0}})=\pi +ml_{0},
\end{equation*}%
and then%
\begin{equation*}
\frac{A(g,\Delta )}{L(g,\partial \Delta )}=\frac{4m\pi +mA(D_{l_{0}})}{\pi
+mL(\partial D_{l_{0}})}<\frac{4\pi +A(D_{l_{0}})}{l_{0}}=h_{0}.
\end{equation*}%
It is clear that as $m\rightarrow +\infty ,$ $\frac{A(g,\Delta )}{%
L(g,\partial \Delta )}=\frac{4m\pi +mA(D_{l_{0}})}{\pi +mL(\partial
D_{l_{0}})}$ converges to $h_{0}.$ This completes the proof.
\end{proof}

\end{document}